\documentclass[12pt,twoside,a4paper]{amsart}
\usepackage{amssymb}
\usepackage{latexsym}
\usepackage{amsmath}
\usepackage{euscript}
\usepackage{graphics}
\usepackage{enumitem}
\usepackage{imakeidx}

\newcommand{\ba}{\begin{array}}
\newcommand{\ea}{\end{array}}
\newcommand{\bea}{\begin{eqnarray}}
\newcommand{\eea}{\end{eqnarray}}
\newcommand{\bead}{\begin{eqnarray*}}
\newcommand{\eead}{\end{eqnarray*}}
\newcommand{\be}{\begin{equation}}
\newcommand{\ee}{\end{equation}}
\newcommand{\bed}{\begin{displaymath}}
\newcommand{\eed}{\end{displaymath}}
\newcommand{\bl}{\begin{lemma}}
\newcommand{\el}{\end{lemma}}
\newcommand{\bp}{\begin{proposition}}
\newcommand{\ep}{\end{propostion}}
\newcommand{\bt}{\begin{theorem}}
\newcommand{\et}{\end{theorem}}
\newcommand{\Label}{\label}
\newcommand{\bc}{\begin{corollary}}
\newcommand{\ec}{\end{corollary}}
\newcommand{\la}{\Label}

\newcommand{\br}{\begin{remark}}
\newcommand{\er}{\end{remark}}
\newcommand{\bd}{\begin{definition}}
\newcommand{\ed}{\end{definition}}
\newcommand{\bex}{\begin{example}}
\newcommand{\eex}{\end{example}}

\def\senki{{\lbrack\negthinspace [\bot ]\negthinspace\rbrack}}
\def\senki+{{\lbrack\negthinspace [+] \negthinspace\rbrack}}

\newcommand\R{{\mathbb{R}}}
\newcommand\C{{\mathbb{C}}}

\newcommand\N{{\mathbb{N}}}

\def\sh{{\mathfrak h}}
\def\sfl{{\mathfrak l}}
\def\st{{\mathfrak t}}

\def\sD{{\mathfrak D}}      
   \def\sH{{\mathfrak H}}   
      
   \def\sN{{\mathfrak N}}

      \def\dC{{\mathbb C}}
\def\dD{{\mathbb D}}

   \def\dN{{\mathbb N}}   
      \def\dR{{\mathbb R}}

   \def\cB{{\mathcal B}}   \def\cC{{\mathcal C}}
      
   \def\cH{{\mathcal H}}   
\def\cJ{{\mathcal J}}      \def\cL{{\mathcal L}}
      
\def\cP{{\mathcal P}}      \def\cR{{\mathcal R}}
      \def\cU{{\mathcal U}}

\newcommand\ZA{{A}}

\newcommand\gH{{\mathfrak{H}}}

\newcommand\rD{{\rm{d}}}
\newcommand\rH{{\rm{H}}}
\newcommand{\gA}{{\alpha}}

\newcommand\ctg{\cot}

\def\h#1{{{\hat #1} }}
\def\wt#1{{{\widetilde #1} }}
\def\wh#1{{{\widehat #1} }}

\def\bm\chi{\mbox{\boldmath$\chi$}}
\def\half{{\frac{1}{2}}}

\def\RE{{\rm Re\,}}
\def\IM{{\rm Im\,}}
\def\Ext{{\rm Ext\,}}
\def\ran{{\rm ran\,}}
\def\cran{{\rm \overline{ran}\,}}
\def\dom{{\rm dom\,}}
\def\mul{{\rm mul\,}}

\def\cdom{{\rm \overline{dom}\,}}
\def\clos{{\rm clos\,}}

\def\dim{{\rm dim\,}}
\def\diag{{\rm diag\,}}

\let\xker=\ker \def\ker{{\xker\,}}

\def\cspan{{\rm \overline{span}\, }}

\def\supp{{\rm supp\,}}

\def\cmr{{\dC \setminus \dR}}
\def\uphar{{\upharpoonright\,}}

\DeclareMathOperator{\hplus}{\, \widehat + \,}

\newtheorem{theorem}{Theorem}[section]
\newtheorem{proposition}[theorem]{Proposition}
\newtheorem{corollary}[theorem]{Corollary}
\newtheorem{lemma}[theorem]{Lemma}

\theoremstyle{definition}
\newtheorem{assumption}[theorem]{Assumption}
\newtheorem{example}[theorem]{Example}
\newtheorem{remark}[theorem]{Remark}
\newtheorem{definition}[theorem]{Definition}

\numberwithin{equation}{section}

\makeindex

\begin{document}

\title[Boundary triples, Weyl functions and inverse problems]
{Generalized boundary triples, Weyl functions\\ and inverse
problems}

\author{Vladimir Derkach}
\author{Seppo Hassi}
\author{Mark Malamud}

\address{Department of Mathematics, National Pedagogical Dragomanov University,
Kiev, Pirogova 9, 01601, Ukraine} \email{derkach.v@gmail.com}

\address{Department of Mathematics and Statistics \\
University of Vaasa \\
P.O. Box 700, 65101 Vaasa \\
Finland} \email{sha@uwasa.fi}

\address{Institute of Applied Mathematics and Mechanics \\
National Academy of Science of Ukraine \\
Slovyansk \\
Ukraine}
\email{malamud3m@gmail.com}

\thanks{The research was partially supported by a grant from the Vilho,
Yrj\"o and Kalle V\"ais\"al\"a Foundation of the Finnish Academy of
Science and Letters. V.~D. and M.~M. gratefully also acknowledge
financial support by the University of Vaasa.
The research of V.~D. was also supported by a Volkswagen Stiftung grant} 

\keywords{Symmetric operator, selfadjoint extension, resolvent,
boundary value problem, boundary triple, trace operator, Green's
identities, Dirichlet-to-Neumann type map,
Weyl function, Weyl family.} %

\subjclass[2010]{Primary 34B08, 35J25, 47A05, 47A10, 47A20, 47B15,
47B25; Secondary 34B40, 34L40, 35J05, 35J10, 46C20, 47A48, 47A57,
47B36.}

\begin{abstract}
With a closed symmetric operator $\ZA$ in a Hilbert space $\sH$ a
triple $\Pi=\{\cH,\Gamma_0,\Gamma_1\}$ of a Hilbert space $\cH$ and
two abstract trace operators $\Gamma_0$ and $\Gamma_1$ from $\ZA^*$
to $\cH$ is called a generalized boundary triple for $\ZA^*$ if an
abstract analogue of the second Green's formula holds. Various
classes of generalized boundary triples are introduced and
corresponding Weyl functions $M(\cdot)$ are investigated. The most
important ones for applications are specific classes of
(essentially) unitary boundary triples for which Green's second
identity admits a certain maximality property which guarantees that
the Weyl functions of (the closures of the regularized versions of)
boundary triples are Nevanlinna functions on $\cH$, i.e.
$M(\cdot)\in \cR(\cH)$, or at least they belong to the class $\wt
\cR(\cH)$ of Nevanlinna families on $\cH$. The boundary condition
$\Gamma_0f=0$ determines a reference operator $A_0\,(=\ker
\Gamma_0)$. The case where $A_0$ is selfadjoint implies a relatively
simple analysis, as the joint domain of the trace mappings
$\Gamma_0$ and $\Gamma_1$ admits a von Neumann type decomposition
via $A_0$ and the defect subspaces of $A$. The case where $A_0$ is
only essentially selfadjoint is more involved, but appears to be of
great importance, for instance, in applications to boundary value
problems e.g. in PDE setting or when modeling differential operators
with point interactions. Wide classes of generalized boundary
triples will be characterized in purely analytic terms via the Weyl
function $M(\cdot)$ and close interconnections between different
classes of boundary triples and the corresponding
transformed/renormalized Weyl functions are investigated. These
characterizations involve solving direct and inverse problems for
specific classes of (unbounded) operator functions $M(\cdot)$. Most
involved ones concern operator functions $M(\cdot)\in \cR(\cH)$ for
which
\[
 \tau_{M(z)}(f,g)=(2i\,\IM z)^{-1}[(M(z)f,g)-(f,M(z)g)], \quad f,g\in\dom
 M(z),
\]
defines a closable nonnegative form on $\cH$. It appears that the
closability of $\tau_{M(z)}(f,g)$ does not depend on $z\in\dC_{\pm}$
and, moreover, that the closure then is a form domain invariant
holomorphic function on $\dC_{\pm}$, while $\tau_{M(z)}(f,g)$ itself
is possibly defined only at the point $z$ due to the fact that the
equality $\dom M(z)\cap \dom M(\lambda)=\{0\}$ can occur for all
points $z\neq \lambda$ from the same halfplane. One of the main
results connects these delicate properties of $M(\cdot)$ to the
simple geometric condition that $A_0$ is essentially selfadjoint. In
this study we also derive several additional new results, for
instance, Kre\u{\i}n-type resolvent formulas are extended to the
most general setting of unitary and isometric boundary triples
appearing in the present work. All the main results are shown to
have applications in the study of (ordinary and partial)
differential operators. More specifically we treat Laplacian
operator on bounded domains with smooth, Lipschitz, or even rough
boundary, a mixed boundary value problem for the Laplacian as well
as momentum, Schr\"odinger, and Dirac operators with infinite
sequences of local point interactions.
\end{abstract}

\maketitle



\section{A description of key concepts and an outline of main
results}\label{sec1}

\subsection{Ordinary boundary triples and Weyl functions}
Let $\sH$ be a separable Hilbert space, let $A$ be a not necessarily
densely defined closed symmetric operator in $\sH$ with equal \index{Operator}
deficiency indices $n_+(A)=n_-(A) \le \infty$. The adjoint $A^*$ of
the operator $\ZA$ is a linear relation, i.e., a subspace of vectors
$\wh g=\binom{g}{g'} \in  \sH^2$ such that \index{Linear relation!adjoint}
\[
(Af, g)-(f, g') =0 \quad\textup{ for all }\quad f\in\dom A,
\]
see~\cite{Ben}, \cite{Arens}. In what follows the operator $\ZA$ will
be identified with its graph, so that the set
$\mathcal{C}(\mathcal{H})$ of closed linear operators will be
considered as a subset of $\widetilde{\mathcal{C}}(\mathcal{H})$ of
closed linear relations in $\cH$. Then $\ZA$ is symmetric precisely\index{Operator!symmetric}
when $A\subseteq A^*$. The defect subspaces  $\frak N_z$ of $\ZA$ are
related to $A^*$ by the equality
 $\frak N_z := \ker(A^*-z)$, $z\in\dC\setminus\dR$ and $n_\pm(A):=\dim\sN_{\pm i}$.


In the beginning of thirties J. von Neumann~\cite{Neum32} created
the  extension theory of symmetric operators in Hilbert spaces. His
approach relies on two fundamental formulas and allowed a
description of all selfadjoint ($m$-dissipative) extensions by means \index{Operator!selfadjoint}\index{Extensions!$m$-dissipative}
of isometric (contractive) operators \index{Operator!isometric}\index{Operator!contractive} from $\frak N_i$ onto $\frak
N_{-i}$ (see in this connection monographs \cite{AG,Co,Alb_Kur_00}).
Later on it has been realized that this approach has some drawbacks
and was not convenient when, for instance, treating applications to
boundary value problems (BVP's) for ordinary and especially to
partial differential equations (ODE and PDE).

During four last decades a new approach to the extension theory has
been developed (\cite{Vi52}, \cite{Gr68}, \cite{RB69}, \cite{Gor71},
\cite{GG}, see also~\cite{DM91}, \cite{DM95}, \cite{BrGeiPan08}).
This approach relies on  concepts of abstract boundary mappings and
abstract Green's identity and was introduced independently in
\cite{Koc75,Bruk76}; in this connection it is also necessary to
point out the early paper by J.W. Calkin~\cite{Cal39}. Some further
discussion on Calkin's paper is given below.
%
%
  \begin{definition}\label{een}
A collection  ${\Pi}=\{ \cH, \Gamma_0, \Gamma_1\}$  consisting of a
Hilbert space $\cH$
and two linear mappings $\Gamma_0$ and $\Gamma_1$ from $A^*$ to
$\cH$, is said to be an {\it ordinary boundary triple}\index{Boundary triple!ordinary} for $A^*$ if:
\begin{enumerate}[label=\textbf{\thedefinition.\arabic*}]
\item \label{1_def_een}
The abstract Green's identity  \begin{equation}\label{Greendef1}
  (f', g)-(f, g')
  = (\Gamma_1\wh f, \Gamma_0\wh g)_{\cH}- (\Gamma_0\wh f, \Gamma_1\wh g)_{\cH}
  \end{equation}
holds for all $\wh f=\begin{pmatrix}f\\f'\end{pmatrix},\wh
g=\begin{pmatrix}g\\g'\end{pmatrix} \in A^*$;
\item \label{2_def_een}
The mapping $\Gamma:=
                       \begin{pmatrix}
                         \Gamma_0 \\
                         \Gamma_1
                       \end{pmatrix}
:\,A^*\to \cH^2$ is surjective.
\end{enumerate}
\end{definition}

Note that   in the ODE setting formula \eqref{Greendef1}  turns into
the classical Lagrange identity being a key  tool in treatment of
BVP's. Advantage of this approach becomes obvious in applications to
BVP's for elliptic equations where formula  \eqref{Greendef1}
becomes a second Green's identity. However, in this case the second
assumption~\ref{2_def_een} is violated and this circumstance was
overcome in the classical papers by M. Visik \cite{Vi52} and G.
Grubb \cite{Gr68} (see also \cite{Gr09}). Namely, relying on the
Lions-Magenes trace theory (\cite{LionsMag72}, \cite{Agran2015},
\cite{Gr09}) they regularized the classical Dirichlet and Neumann
trace mappings to get a proper version of Definition~\ref{een}.

The operator $\Gamma$ in Definition~\ref{een} is called the {\it reduction operator}
\index{Reduction operator}
(in the terminology of~\cite{Cal39}).
Definition~\ref{een}  immediately yields a  parametrization   of
the set of all selfadjoint
extensions $\wt A$ of $\ZA$  by means of abstract boundary conditions via 
\begin{equation}\label{eq:SA_ext}
 \wt A  =A_{\Theta} :=\{\wh f\in A^*: \,\Gamma\wh f\in\Theta\},
\end{equation}
where $A_\Theta$ ranges over the set of all selfadjoint extensions\index{Extensions!selfadjoint}
of $\ZA$ when $\Theta$ ranges over the set of all selfadjoint
relations in $\cH$. This correspondence is bijective and in this
case
\begin{equation}\label{eq:SA_ext2}
  \Theta := \Gamma (\wt A) . 
\end{equation}
 Two following selfadjoint extensions of $\ZA$  are of
particular interest:
\begin{equation}\label{eq:A01}
    A_0:=\ker\Gamma_0 =  A_{\Theta_\infty} \quad \text{and} \quad  A_1:=\ker\Gamma_1 =
    A_{\Theta_1};
\end{equation}
here $\Theta_\infty = \{0\}\times \cH$ and ${\Theta_1} = \Bbb O$.
These extensions  are {\it  disjoint}\index{Extensions!disjoint}, i.e. $A_0\cap A_1=A$, and
{\it transversal}\index{Extensions!transversal}, i.e. $A_0\wh + A_1=A^*$.    In what follows $A_0$
is considered as a reference extension of $\ZA$.

It is emphasized that the parametrization \eqref{eq:SA_ext} is
convenient in applications to concrete BVP's as it provides an
explicit description of boundary conditions associated with the
extension $A_\Theta$; this natural connection cannot be obtained via
the second J. von Neumann formula,~\cite{AG}.

The main analytical tool in description of spectral properties of
selfadjoint extensions of $\ZA$ is the abstract Weyl function,
introduced and investigated in~\cite{DM85,DM87, DM91}. Let $A_0$ be
a reference operator given by~\eqref{eq:A01}, let $\rho(A_0)$ be the resolvent set of $A_0$,
let $\sN_{\lambda}:=\ker(\ZA^*-\lambda)$, $\lambda\in\dC\setminus\dR$, be the
{\it defect subspace} \index{Defect subspace}of $\ZA$ and let
\begin{equation}\label{eq:whf_l}
    \wh\sN_{\lambda}:=\left\{\wh f_\lambda=\begin{pmatrix}f_{\lambda}\\
\lambda f_{\lambda}\end{pmatrix}:\,f_{\lambda}\in
\sN_{\lambda}\right\}.
\end{equation}

\begin{definition}[\cite{DM85,DM87, DM91}] \label{W00}
The abstract \textit{Weyl function} \index{Weyl function} and the \textit{$\gamma$-field}
\index{Gamma-field}
of $\ZA$, corresponding to an ordinary boundary triple
$\Pi=\{\cH,\Gamma_0,\Gamma_1\}$ are defined by
\begin{equation}\label{0.2}
M(\lambda)\Gamma_0\wh f_{\lambda}= \Gamma_1\wh f_{\lambda},
 \quad
  \gamma(\lambda)\Gamma_0\wh f_{\lambda}=f_{\lambda}, \quad
 \wh f_\lambda\in \wh\sN_{\lambda},\quad  \lambda\in\rho(A_0),
\end{equation}
where $\wh f_\lambda$ is given by~\eqref{eq:whf_l}.
  \end{definition}
Notice that when the symmetric operator $\ZA$ is densely defined its
adjoint is a single-valued operator and Definitions~\ref{een}
and~\ref{W00} can be used in a simpler form by treating $\Gamma_0$
and $\Gamma_1$ as operators from $\dom A^*$ to $\cH$,
see~\cite{Koc75}, \cite{GG}, \cite{DM91}. In what follows this
convention will be tacitly used in most of our examples.

\begin{example}
Let $\ZA$ be a minimal symmetric operator associated in $L^2(\Bbb
R_+)$ with Sturm-Liouville differential expression
\[
 \cL := -\frac{d^2}{dx^2} + q(x), \qquad  q = \overline{q} \in L^1_{loc}([0,\infty)).
\]
Assume the limit-point case at infinity, i.e. assume that
$n_{\pm}(\ZA) = 1.$ The defect subspace $\frak N_{\lambda}$  is
spanned by the Weyl solution $\psi(\cdot,\lambda)$ of the equation
$\cL f = \lambda f$ which is given by
\[
 \psi(x,\lambda) = c(x,\lambda) + m(\lambda)s(x,\lambda) \in L^2(\Bbb R_+),
\]
where $c(\cdot,\lambda)$ and $s(\cdot,\lambda)$ are cosine and sine
type solutions of the equation $\cL f = \lambda f$ subject to the
initial conditions
\[
 c(0,\lambda)=1,\quad c'(0,\lambda)=0;
 \qquad s(0,\lambda)=0,\quad s'(0,\lambda)=1.
\]
The function  $m(\cdot)$ is called the Titchmarsh-Weyl coefficient\index{Titchmarsh-Weyl coefficient}
of $\cL$.

In this case a boundary triple $\Pi=\{\Bbb C,\Gamma_0,\Gamma_1\}$
can be defined as $\Gamma_0f = f(0),$  $\Gamma_1f = f'(0).$ The
corresponding Weyl function $M(\lambda)$ coincides with the classical
Titchmarsh-Weyl coefficient, $M(\lambda) = m(\lambda)$.
   \end{example}
In this connection let us mention that the role of the Weyl function
$M(\lambda)$ in the extension theory of symmetric operators is
similar to that of the classical Titchmarsh-Weyl coefficient
$m(\lambda)$ in the spectral theory of Sturm-Liouville operators\index{Sturm-Liouville operator}.
For instance, it is known (see \cite{KL73}, \cite{DM91}) that if
$\ZA$ is simple, i.e. $\ZA$ does not admit orthogonal decompositions
with a selfadjoint summand, then the Weyl function $M(\lambda)$
determines the boundary triple $\Pi$, in particular, the pair
$\{\ZA, A_0\}$, uniquely up to unitary equivalence. Besides, when
$\ZA$ is simple, the spectrum of $A_\Theta$ coincides with the
singularities of the operator function $(\Theta - M(z))^{-1}$; see
\cite{DM91}.

As was shown in~\cite{DM91, DM95} and \cite{MM2}  the  Weyl function
$M(\cdot)$ and the $\gamma$-field $\gamma(\cdot)$ both  are well defined
and  holomorphic on the resolvent set $\rho(A_0)$ of the  operator $A_0$.
Moreover,  the $\gamma$-field $\gamma(\cdot)$ and the  Weyl
function $M(\cdot)$  satisfy the identities
    \begin{equation}\label{eq:gamma}
 \gamma(\lambda)=[I+(\lambda-\mu)(A_0-\lambda)^{-1}]\gamma(\mu),
 \quad
 \lambda,\mu\in\rho(A_0),
   \end{equation}
  \begin{equation}\label{eq:Q_fun}
{M(\lambda)-M(\mu)}^*=(\lambda-\bar\mu)\gamma(\mu)^*\gamma(\lambda),
\qquad
    \lambda,\mu\in\rho(A_0).
\end{equation}
This means that $M(\cdot)$ is a $Q$-function of the operator $\ZA$
in the sense of Kre\u{\i}n and Langer~\cite{KL71}.
%

Denote by $\cB(\cH)$ the set of bounded linear operators in $\cH$
and denote by $\cR[\cH]$ the class of Herglotz-Nevanlinna functions\index{Herglotz-Nevanlinna function},
i.e., operator valued functions $F(\lambda)$ with values in
$\cB(\cH)$, which are holomorphic on $\dC\setminus\dR$ and satisfy
the conditions
   \begin{equation}\label{eq:Weyl-sym}
F(\lambda) = F(\bar\lambda)^*\quad \mbox{and}\quad \IM F(\lambda)\ge
0\quad\mbox{for all}\quad\lambda\in\dC\setminus\dR,
   \end{equation}
see \cite{KacK}. It follows from \eqref{eq:gamma}
and~\eqref{eq:Q_fun} that $M$ belongs to the Herglotz-Nevanlinna
class $\cR[\cH]$.

Furthermore, since $\gamma(\lambda)$  isomorphically maps $\cH$
onto $\sN_{\lambda}$, the relation~\eqref{eq:Q_fun} ensures  that
the imaginary part $\IM M (z)$ of  $M(z)$ is positively definite,
i.e. $M(\cdot)$ belongs to the subclass $\cR^u[\cH]$ of  {\it
uniformly strict} Herglotz-Nevanlinna  functions\index{Herglotz-Nevanlinna function!uniformly strict},:
\[
 \cR^u[\cH]:= \left\{F(\cdot)\in \cR[\cH]:\, 0\in\rho(\IM  F(i))\right\}.
\]
The converse is also true.

 \begin{theorem}[\cite{KL73,DM95}]\label{thm:WF_Ord_BT}
The set of Weyl functions corresponding to  ordinary boundary
triples coincides with the class $\cR^u[\cH]$ of  uniformly strict
Herglotz-Nevanlinna functions.
   \end{theorem}



\subsection{$B$-generalized  boundary triples}
In BVP's for Sturm-Liouville operator with operator potential, for
partial differential operators~\cite{DHMS09}, and in point
interaction theory it seems natural to consider more general
boundary triples by weakening the surjectivity assumption
\ref{2_def_een} in Definition~\ref{een}. The following notion of
$B$-\emph{generalized boundary triple} was introduced
in~\cite{DM95}.
  \begin{definition}\label{genBT}
Let $\ZA$ be a closed symmetric operator in a Hilbert space $\sH$ with
equal deficiency indices and let $\ZA_*$ be a linear relation in $\sH$
such that $\ZA\subset \ZA_*\subset \overline{ \ZA_*} = A^*$. Then the
collection $\Pi = \{\cH,\Gamma_0,\Gamma_1\}$, where $\cH$ is a
Hilbert space and $\Gamma = \{\Gamma_0,\Gamma_1\}$ is a
single-valued linear mapping from $\ZA_*$ into ${\cH}^2$, is said to
be a \emph{$B$-generalized  boundary triple} for $A^*$, \index{Boundary triple!B-generalized} if:
   \begin{enumerate}[label=\textbf{\thedefinition.\arabic*}]
\item \label{1_def_genBT}
the abstract Green's identity~\eqref{Greendef1} holds
            for all $\wh f=\begin{pmatrix}f\\f'\end{pmatrix},\wh g=\begin{pmatrix}g\\g'\end{pmatrix} \in  \ZA_*$;
\item \label{2_def_genBT}
$\ran\Gamma_0=\cH$;
\item \label{3_def_genBT}
$A_0:=\ker\Gamma_0$ is a selfadjoint relation in $\sH$.
\end{enumerate}
  \end{definition}
The Weyl function $M(\lambda)$ corresponding to a $B$-generalized
boundary triple is defined by the same equality~\eqref{0.2} where
$\wh f_\lambda$ runs through   $\wh\sN_\lambda\cap \ZA_*$, a proper
subset of $\wh\sN_\lambda$. For every $\lambda\in\rho(\ZA)$ the Weyl
function $M(\lambda)$ takes values in $\cB(\cH)$ and this justifies
the present usage of the term $B$-generalized boundary triple, where
``$B$'' stands for a \emph{bounded Weyl function}, i.e., a function
whose values are bounded operators.

\begin{example}\label{example1.6}
Let $\Omega$ be a bounded domain in $\Bbb R^n$ with smooth boundary
$\partial\Omega$. Consider the Laplace operator $-\Delta$ in
$L^2(\Omega)$. Let $\gamma_D$ and $\gamma_N$ be the Dirichlet and
Neumann trace mappings. It is well known (see \cite{Agran2015,
Gr09,LionsMag72,Maurin65,Tri} and Section \ref{sec7.3} for detail
and further references) that the mappings $\gamma_D: H^{3/2}(\Omega)
\to H^{1}(\partial\Omega)$ and $\gamma_N: H^{3/2}(\Omega) \to
H^{0}(\partial\Omega)=L^2(\partial\Omega)$ are well defined and
surjective.

Consider a pre-maximal operator $A_*$ as the restriction of the maximal Laplace operator $\ZA_{\max}$ to
the domain
%
  \begin{equation}\label{dom-Lap0}
\dom \ZA_* = H_\Delta^{3/2}(\Omega):=  H^{3/2}(\Omega)\cap \dom
\ZA_{\max} = \left\{f\in H^{3/2}(\Omega):\Delta f\in
L^2(\Omega)\right\}.
   \end{equation}
Using the key mapping properties of $\gamma_D$ and $\gamma_N$ one
can extend the classical Green's formula to the domain $\dom \ZA_*$.
Notice that the condition $\gamma_N f=0$, $f\in \dom A_*$,
determines the Neumann realization $\Delta_N$ of the Laplace
operator. Since $\Delta_N$ is selfadjoint and
$\gamma_N(\dom \ZA_*)=H^{0}(\partial\Omega)$, the triple $\Pi =
\{L^2(\partial\Omega), \Gamma_0, \Gamma_1\}$ with
\[
\Gamma_0 =
\gamma_N\uphar{\dom \ZA_*}\quad\textup{and}\quad
\Gamma_1 = \gamma_D\uphar{\dom
\ZA_*}
\]
is a $B$-generalized boundary triple for $A^*$ with $\dom\Gamma =
\dom \ZA_*$. Besides, the corresponding Weyl function $M(\cdot)$
coincides with the inverse of the classical Dirichlet-to-Neumann map
$\Lambda(\cdot)$, i.e. $M(\cdot) = \Lambda(\cdot)^{-1}$ (see
Chapter~\ref{sec7.3} for details).
\end{example}
It should be noted that the Weyl function $M(\cdot)$ corresponding
to a $B$-generalized boundary triple, satisfies the properties
\eqref{eq:gamma}--\,\eqref{eq:Weyl-sym}. However, instead of the
property $0\in\rho(\IM M(i))$ one has a weaker condition
$0\not\in\sigma_p(\IM M(i))$. This motivates the following
definition.

Denote by $\cR^s[\cH]$ the class of strict Nevanlinna functions,\index{Herglotz-Nevanlinna function!strict}
that is
\[
 \cR^s[\cH]:=\left\{ F(\cdot)\in \cR[\cH]:\, 0\not\in\sigma_p(\IM
 F(i))\right\}.
\]
In fact, it was also shown in \cite[Chapter 5]{DM95}  that every
$M(\cdot)\in \cR^s[\cH]$ can be realized as the Weyl function of a
certain $B$-generalized boundary triple and hence the following
statement holds.
 \begin{theorem}[\cite{DM95}]\label{thm:WF_BG_BT}
The set of Weyl functions corresponding to $B$-generalized  boundary
triples coincides with the class $\cR^s[\cH]$ of  strict
Herglotz-Nevanlinna functions.
   \end{theorem}


This realization result as well as the technique of $B$-generalized
boundary triples has recently been applied also e.g. to problems in
scattering theory (see~\cite{BMN16}), in analysis of discrete and
continuous time system theory, and in boundary control theory (for
some recent papers, see e.g. \cite{BKSZ15}, \cite{HdSSz2012},
\cite{MalStaf2006}).


\subsection{Unitary boundary triples}
A general class of boundary triples, to be called here unitary
boundary triples, was introduced in~\cite{DHMS06}. In fact, the
appearance of this concept was motivated  by the inverse problem for
the most general class of Nevanlinna functions: realize each
Nevanlinna function as the Weyl function of an appropriate type
generalized boundary triple.

To this end denote by $\cR(\cH)$ the Nevanlinna class of all
operator valued holomorphic functions on $\dC_+$ (in the resolvent
sense) with values in the set of maximal dissipative (not
necessarily bounded) linear operators in $\cH$. Each $M(\cdot) \in
\cR(\cH)$ is extended to $\dC_-$ by symmetry with respect to the
real line $M(\lambda) = M(\bar\lambda)^*$; see~\cite{KL71},
\cite{DHMS06}. Analogous to the subclass $\cR^s[\cH]$ of bounded
Nevanlinna functions $\cR[\cH]$, the class $\cR(\cH)$ contains a
subclass $\cR^s(\cH)$ of strict (in general unbounded) Nevanlinna
functions which satisfy the condition
\begin{equation}\label{stric-unb}
\cR^s(\cH) := \left\{F(\cdot)\in \cR(\cH):\, \IM (F(i)h,h) =0
\Longrightarrow h=0, \quad h\in \dom F(i)\right\}.
\end{equation}

%
%
In order to present the definition of a unitary boundary triple,
introduce the fundamental symmetries
\begin{equation}\label{jh}
 J_\sH:=\begin{pmatrix} 0 & -iI_\sH \\ i I_\sH & 0\end{pmatrix},\quad
  J_\cH:=\begin{pmatrix} 0 & -iI_\cH \\ i I_\cH & 0\end{pmatrix},
\end{equation}
and the associated Kre\u{\i}n spaces $(\sH^2,J_\sH)$ and
$(\cH^2,J_\cH)$ (see~\cite{AI86}, \cite{Bog68}) obtained by endowing
the Hilbert spaces $\sH^2$ and $\cH^2$ with the following indefinite
inner products
\begin{equation}\label{eq:leftik_sH}
    [ \wh f,\wh g]_{\sH^2}=
\bigl(J_\sH \wh f,\wh g\bigr)_{\sH^2},\quad[ \wh h,\wh
k]_{\cH^2}= \bigl(J_\cH \wh h,\wh k\bigr)_{\cH^2},\quad \wh
f, \wh g\in \sH^2,\quad \wh h, \wh k \in \cH^2.
\end{equation}
%
%
This allows to rewrite the Green's identity~\eqref{Greendef1} in the
form
\begin{equation}\label{eq:J_isom}
 [ \wh f,\wh g]_{\sH^2}=[ \Gamma\wh f,\Gamma \wh
 g]_{\cH^2},
\end{equation}
which expresses the fact that a mapping $\Gamma$ which satisfies the
Green's identity~\eqref{Greendef1} is in fact a
$(J_\sH,J_\cH)$-isometric mapping from the Kre\u{\i}n space
$(\sH^2,J_\sH)$ to the Kre\u{\i}n space $(\cH^2,J_\cH)$. If
$\Gamma^{[*]}$ denotes the Kre\u{\i}n space adjoint of the operator
$\Gamma$ (see definition~\eqref{def:J_unit}), then~\eqref{eq:J_isom}
can be simply rewritten as $\Gamma^{-1}\subset \Gamma^{[*]}$. The
surjectivity of $\Gamma$ implies that
 $   \Gamma^{-1}= \Gamma^{[*]}$.
Following Yu.L. Shmuljan~\cite{Sh0} a linear operator
$\Gamma:(\sH^2,J_\sH)\to(\cH^2,J_\sH)$ will be called
$(J_\sH,J_\cH)$-unitary, if $\Gamma^{-1}= \Gamma^{[*]}$.
   \begin{definition}[\cite{DHMS06}]\label{unitBT}\index{Boundary triple!unitary}
   \index{Boundary triple!isometric}
A collection $\{\cH,\Gamma_0,\Gamma_1\}$ is called a \textit{unitary
(resp. isometric) boundary triple} for $A^*$,
if $\cH$ is a Hilbert space and $\Gamma=\begin{pmatrix}\Gamma_0\\
\Gamma_1\end{pmatrix}$ is a linear operator from $\sH^2$ to $\cH^2$
such that:
\begin{enumerate}[label=\textbf{\thedefinition.\arabic*}]
\item\label{1_def_unitBT}
$\ZA_*:=\dom\Gamma$ is dense in $A^*$ with respect to the topology on $\sH^2$;
  \item \label{2_def_unitBT} 
  The operator $\Gamma$ is $(J_\sH,J_\cH)$-unitary (resp. isometric).
\end{enumerate}
\end{definition}
The Weyl function $M(\lambda)$ corresponding to a unitary  boundary triple $\Pi$ is
defined again by the same formula~\eqref{0.2}.
The {\it transposed boundary triple}\index{Boundary triple!transposed}
$\Pi^\top:=\{\cH,\Gamma_1,-\Gamma_0\}$ to a unitary boundary triple
$\Pi$ is also a unitary  boundary triple, the corresponding Weyl
function takes the form $M^\top(\lambda)=-M(\lambda)^{-1}$.

The main realization theorem in \cite{DHMS06} gave a solution to the
inverse problem mentioned above.
 \begin{theorem}[\cite{DHMS06}]\label{TAMS_Th}
The class of Weyl functions corresponding to unitary boundary
triples coincides with the class $\cR^s(\cH)$ of (in general
unbounded) strict  Nevanlinna functions.
   \end{theorem}

In fact, in \cite[Theorem 3.9]{DHMS06} a stronger result is stated
showing that the class $\cR^s(\cH)$ can be replaced by the class
$\cR(\cH)$ or even by the class $\widetilde \cR(\cH)$ of Nevanlinna
pairs when one allows multivalued linear mappings $\Gamma$ in
Definition \ref{unitBT}; see Theorem~\ref{GBTNP} in Section
\ref{sec3.2}. Theorem \ref{TAMS_Th} plays a key role in the
construction of generalized resolvents in the framework of coupling
method that was originally introduced in \cite{DHMS1} and developed
in its full generality in~\cite{DHMS09}.

In connection with Definition \ref{unitBT} we wish to make some
comments on a seminal paper \cite{Cal39} by J.W. Calkin, where a
concept of the \emph{reduction operator} is introduced and
investigated. Although no proper geometric machinery appears in the
definition of Calkin's reduction operator this notion in the case of
a densely defined operator $\ZA$ essentially coincides with concept
of a unitary operator between Kre\u{\i}n spaces as in
Definition~\ref{unitBT}. An overview on the early work of Calkin and
some connections to later developments can be found from the papers
in the monograph \cite{HdSSz2012}; for a further discussion see also
Section \ref{sec3.5}.

Despite of the complete solution to the realization problem given in
Theorem \ref{TAMS_Th}, more specifically determined subclasses of
unitary boundary triples and closely connected isometric triples
together with the associated subclasses of Weyl functions are both of
theoretical and practical interest since they naturally appear, for
instance, in various problems of mathematical physics. Such
questions, sometimes of rather delicate nature, lead to inverse as
well as direct problems for $R$-functions in the classes $\cR(\cH)$
and $\wt\cR(\cH)$, where precisely prescribed analytic properties of
operator functions have to be connected to appropriately determined
geometric properties of associated boundary triples and vice versa.
These problems have motivated the present work and resulted in
several further subclasses of holomorphic operator functions,
included or at least closely related to the class of $\cR(\cH)$ of
Hergloz-Nevanlinna functions, all of them occurring in boundary
value problems in the ODE and PDE settings.

The notions of ordinary boundary triples and $B$-generalized
boundary triples turned out to be unitary boundary triples;
see~\cite{DHMS06}. In particular, a unitary boundary triple $\Pi$ is
an ordinary ($B$-generalized) boundary triple if and only if the
corresponding Weyl function $M(\cdot)$ belongs to the class
$\cR^u[\cH]$ (resp. $M(\cdot)\in \cR^s[\cH]$). In Section~\ref{sec5}
we consider two further subclasses of unitary boundary triples:
$S$-generalized and $ES$-generalized boundary triples. For deriving
some of the main results on unitary boundary triples we have
established some new facts on the interaction between
$(J_\sH,J_\cH)$-unitary relations and unitary colligations appearing
e.g. in system theory and in the analysis of Schur functions, see
Section \ref{sec5.1}; a background for this connection can be found
from~\cite{BHS09}. On the other hand, in Section~\ref{sec5.2} we
extend Kre\u{\i}n's resolvent formula to the general setting of
unitary boundary triples $\{\cH,\Gamma_0,\Gamma_1\}$. Namely, for
any proper extension $A_\Theta\in {\Ext}_S$ satisfying
$A_\Theta\subset \dom \Gamma$ the following Kre\u{\i}n-type formula
holds:
\begin{equation}\label{resol_Intro}
(A_\Theta-\lambda)^{-1}-(A_0-\lambda)^{-1}
=\gamma(\lambda)\bigl(\Theta-M(\lambda)\bigr)^{-1}\gamma(\bar\lambda)^*,\qquad
\lambda\in\cmr.
\end{equation}
It is emphasized that in this formula $A_\Theta$ \emph{is not
necessarily closed and it is not assumed that
$\lambda\in\rho(A_\Theta)$}, in particular, here the inverses
$(A_\Theta-\lambda)^{-1}$ and $(\Theta-M(\lambda))^{-1}$ are
understood in the sense of relations, see
Theorem~\ref{Kreinformula2}; for an analogous formula see also
Theorem~\ref{Kreinformula}.

\subsection{$S$-generalized boundary triples}
Following~\cite{DHMS06}  we consider a special class of unitary
boundary triples singled out by the condition that
$A_0:=\ker\Gamma_0$ is a selfadjoint extension of $\ZA$.
 \begin{definition}[\cite{DHMS06}]\label{SgenBT}\index{Boundary triple!S-generalized}
A unitary boundary triple $\Pi = \{\cH,\Gamma_0,\Gamma_1\}$ is said
to be an \emph{$S$-generalized  boundary triple} 
for $A^*$ if the assumption~\ref{3_def_genBT} holds, i.e. $A_0 := \ker
\Gamma_0$  is a selfadjoint extension of $\ZA$.
\end{definition}

Next following~\cite[Theorem~7.39]{DHMS12} and
\cite[Theorem~4.13]{DHMS06} (see also an extension given in
Theorem~\ref{prop:C6} below) we present a complete characterization
of the Weyl functions $M(\cdot)$ corresponding to $S$-generalized
boundary triples.

 \begin{theorem}\label{prop:C6B}{\rm (\cite{DHMS06,DHMS12})}
Let  $\Pi=\{\cH,\Gamma_0,\Gamma_1\}$ be a unitary boundary triple
for $A^*$ and let $M(\cdot)$ and $\gamma(\cdot)$ be the
corresponding Weyl function and $\gamma$-field, respectively. Then
the following statements are equivalent:
\begin{enumerate}[label={\normalfont (\roman*)}]
\item  $A_0 = \ker \Gamma_0$ is selfadjoint, i.e. $\Pi$ is an $S$-generalized boundary triple;
\item $A_*=A_0\hplus \wh\sN_{\lambda}$ and $A_*=A_0\hplus \wh\sN_{\mu}$ for some (equivalently for all) $\lambda\in\dC_+$ and $\mu\in\dC_-$;
\item $\ran\Gamma_0 = \dom M(\lambda)= \dom M(\mu)$ for some (equivalently for all) $\lambda\in\dC_+$ and $\mu\in\dC_-$;
\item $\gamma(\lambda)$ and $\gamma(\mu)$ are bounded and densely defined in $\cH$
    for some (equivalently for all) $\lambda\in\dC_+$ and $\mu\in\dC_-$;
\item $\IM M(\lambda)$ is bounded and densely defined for some (equivalently for all)
    $\lambda\in\dC\setminus\dR$;
\item the Weyl function $M(\cdot)$ belongs to $\cR^s(\cH)$ and it admits a  representation
\begin{equation}\label{eq:DomInv}
 M(\lambda) = E + M_0(\lambda), \quad M_0(\cdot)\in \cR[\cH], \quad
  \lambda\in\dC\setminus\dR,
\end{equation}
where $E=E^*$ is a selfadjoint (in general unbounded)  operator in
$\cH$.
\end{enumerate}
\end{theorem}
Here the symbol $\hplus $ means the componentwise sum of two linear relations, see~\eqref{eq:hsum}.
Notice that, for instance, the implications (i)$\Rightarrow$ (ii),
(iii) are immediate from the following decomposition of
$A_*:=\dom\Gamma$:
\begin{equation}\label{eq:A_*}
  A_*=A_0\hplus\wh\sN_\lambda(A_*),\quad \lambda\in\rho(A_0).
\end{equation}
In accordance with \eqref{eq:DomInv}  the Weyl function
corresponding to an $S$-generalized boundary triple is an operator
valued Herglotz-Nevanlinna function with domain invariance property:
$\dom M(\lambda) = \dom E = \ran\Gamma_0$,  \  $\lambda \in \Bbb
C_{\pm}$. It takes values in the set $\cC(\cH)$ of closed (in general  unbounded) operators while values of the imaginary parts $\IM
M(\lambda)$ are bounded operators.

As an example we mention that the transposed boundary triple
$\Pi^\top = \{L^2(\partial\Omega), \Gamma_1, -\Gamma_0\}$ from the
PDE Example~\ref{example1.6} is an $S$-generalized boundary triple.
The corresponding Weyl function coincides (up to sign change) with
the Dirichlet-to-Neumann map $\Lambda(\cdot)$, i.e. $M(\cdot)^{\top}
= -\Lambda(\cdot)$; see Proposition~\ref{prop:Laplacian}.



  \subsection{$ES$-generalized boundary triples and form domain invariance}

Next we discuss one of the main objects appearing in the present
work; it is easily introduced by the following definition.
   \begin{definition}\label{EgenBT}\index{Boundary triple!ES-generalized}
A \emph{unitary} boundary triple $\{\cH,\Gamma_0,\Gamma_1\}$ for
$\ZA^*$ is said to be an \emph{essentially selfadjoint generalized
boundary triple}, in short, \emph{$ES$-generalized boundary triple}
for $\ZA^*$, if:
  \begin{enumerate}[label=\textbf{\thedefinition.\arabic*}]
\item
\label{1_def_EgenBT} $A_0:=\ker\Gamma_0$ is an essentially selfadjoint linear relation in $\sH$.
  \end{enumerate}
  \end{definition}

To characterize the class of $ES$-generalized boundary triples in
terms of the corresponding Weyl functions we associate with each
${M(\cdot)}$ a family of nonnegative quadratic forms
$\st_{M(\lambda)}$ in $\cH$:
   \begin{equation}\label{eq:W-form}
\st_{M(\lambda)}[u,v] :=
\frac{1}{\lambda-\bar\lambda}\,[(M(\lambda)u,v)-(u,M(\lambda)v)],
\quad u,v\in\dom(M(\lambda)), \quad \lambda\in\cmr.
 \end{equation}
The forms $\st_{M(\lambda)}$  are not necessarily  closable.
However, it is shown that if $\st_{M(\lambda_0)}$ is  closable at
one point $\lambda_0\in \dC_{\pm}$, then $\st_{M(\lambda)}$ is
closable  for each $\lambda\in\dC_{\pm}$. In the latter case the
domain of the closure $\overline\st_{M(\lambda)}$  does not depend
on $\lambda \in \dC_{\pm}$, i.e. the  family  $\overline\st_{M(\cdot)}$ is domain invariant.

In what follows one of the main results established in this
connection reads as follows (cf. Theorem \ref{essThm1}).
  \begin{theorem}\label{main_theorem}
Let  $\Pi = \{\cH,\Gamma_0,\Gamma_1\}$ be a unitary boundary  triple
for $\ZA^*$. Let also $M(\cdot)$ and $\gamma(\cdot)$  be the
corresponding Weyl function and $\gamma$-field, respectively. Then
the following statements are equivalent:
\begin{enumerate}[label={\normalfont (\roman*)}]
\item $\Pi$ is  $ES$-generalized boundary triple for $\ZA^*$;
\item $\gamma(\pm i)$ is closable;
\item $\gamma(\lambda)$ is closable for  every $\lambda\in\dC_{\pm}$ and
$\dom\gamma(\lambda) = \dom\gamma(\pm i)$, $\lambda\in \Bbb
C_{\pm}$;
\item the form $\st_{M(\pm i)}$ is closable;
\item the form $\st_{M(\lambda)}$ is closable for each $\lambda\in\dC_\pm$ and
$\dom{\overline\st_{M(\lambda)}} = \dom {\overline\st_{M(\pm i)}}$,
$\lambda\in \Bbb C_{\pm}$;
\item the Weyl function $M(\cdot)$ belongs to $ \cR^s(\cH)$ and is form domain invariant in $\dC_\pm$.\index{Weyl function!form domain invariant}
%
\end{enumerate}
  \end{theorem}
If $\{\cH,\Gamma_0,\Gamma_1\}$ is an $ES$-generalized, but not
$S$-generalized, boundary triple for $A^*$, then the
equality~\eqref{eq:A_*} fails to hold
and turns out to be an inclusion
\begin{equation}\label{eq:A_*A_0}
  A_0\hplus\wh\sN_\lambda(A_*)\subsetneq A_* \subset A^*=\overline{A_0} \hplus\wh\sN_\lambda(A^*),\quad \lambda\in\rho(A_0).
\end{equation}
Indeed, since $A_0$ is not selfadjoint (while it is essentially
selfadjoint), the decomposition $A_*=A_0\hplus\wh\sN_\lambda(A_*)$
doesn't hold; see \cite[Theorem~4.13]{DHMS06}.
Then there
clearly exist $\wh f\in A_*$ which does not belong to
$A_0\hplus\wh\sN_\lambda(A_*)$, so that $\Gamma_0 \wh f \neq 0$ as
well as $\Gamma_0 \wh f\not\in \Gamma_0(\wh\sN_\lambda(A_*))=\dom
M(\lambda)$. In particular, in this case a strict inclusion $\dom
M(\lambda)\subsetneq \ran \Gamma_0$ holds and, consequently, the
Weyl function $M(\lambda)$ can loose e.g. the domain invariance
property. However, the domain of the closure $\overline{\Gamma}_0$
contains the selfadjoint relation $\overline{A_0}$ and admits the
decomposition
\begin{equation}\label{eq:clos_Gamma_0}
  \dom\overline{\Gamma}_0 =\overline{A_0}\hplus(\dom(\overline{\Gamma}_0)\cap\wh\sN_\lambda(A^*)),
  \quad \lambda\in\rho(\overline{A_0}).
\end{equation}
This yields the equality
\[
\dom\overline{\gamma(\lambda)}=\overline{\Gamma}_0(\dom(\overline{\Gamma}_0)\cap\wh\sN_\lambda(A^*))=\ran
\overline{\Gamma}_0,
 \]
which in combination with the equality
$\dom{\overline\st_{M(\lambda)}}=\dom{\overline{\gamma(\lambda)}}$
yields the form domain invariance property for $M$:
\begin{equation}\label{eq:dom_t_M}
 \dom{\overline\st_{M(\lambda)}}=\ran \overline{\Gamma}_0.
\end{equation}

Passing from the case of an $S$-generalized boundary triple to the
case of an $ES$-generalized triple (which is not $S$-generalized) means that $A_0 \not = A_0^*$.
Then, in particular, conditions (ii) and (iii) in
Theorem~\ref{prop:C6B} are necessary violated. We split the
situation into two different cases:
\begin{assumption}\label{ass:ii_1}
$M(\lambda)$ is not domain invariant, i.e. $\dom M(\lambda_1)\not =
\dom M(\lambda_2)$ at least for two points $\lambda_1,  \lambda_2\in
\Bbb C_+$,  $\lambda_1\not =\lambda_2$, while it is form domain
invariant, i.e. $\dom{\overline\st_{M(\lambda)}} = \dom
{\overline\st_{M(\pm i)}}$, $\lambda\in \Bbb C_{\pm}$.
\end{assumption}
\begin{assumption}\label{ass:ii_2}
$\dom M(\lambda) =  \dom M(\pm i)$, $ \lambda\in\Bbb C_{\pm}$,  while  $\dom M(\pm i) \subsetneqq  \ran\Gamma_0$.
\end{assumption}

In the next two examples we demonstrate that both possibilities
appear in the spectral theory. It is first shown that $R$-functions
satisfying Assumption~\ref{ass:ii_1} naturally arise in the theory
of differential operators with boundary conditions involving
$\lambda$-depending spectral parameters.

     \begin{example}
Let $\varphi(\cdot)$  be  a scalar $R$-function and $\cH =L^2(0,
\infty).$
 Define  an operator valued function $G(\cdot) = G_\varphi(\cdot)$  by setting
\[
G_\varphi(z) f = -i\frac{d^2u}{dx^2},\qquad  \dom (G_\varphi(z)) =
\{u\in W^2_2(\mathbb R_+) :\ u'(0) = \varphi(z)u(0)\}, \quad z \in
\C_+.
\]
Clearly, $G_\varphi(z)$ is densely defined,
$\rho(G_\varphi(z))\not = \emptyset$ for each $z \in \C_+$ and the
family $G_\varphi(\cdot)$ is holomorphic in $\C_+$ in the resolvent
sense.
%
%
%
%
Integrating by parts one obtains
$$
{\mathfrak t}_{G(z)}[u] := \IM (G_\varphi(z)u,u) =
\int_{\R_+}|u'(x)|^2\,dx + \IM \varphi(z)|u(0)|^2, \quad u\in \dom
{\mathfrak t}_{G(z)} = \dom(G_\varphi(z)).
$$
Hence  the form ${\mathfrak t}_{G(z)}$  is nonnegative  and
$G_\varphi(z)$ is $m$-dissipative for each $z \in \C_+$. Moreover,
$G(\cdot)\in R^s(\cH)$ since $\ker{\mathfrak t}_{G(z)} = \{0\}$.
Therefore, by Theorem  \ref{TAMS_Th}, there exists a certain unitary
boundary triple such that  the corresponding Weyl function coincides
with $G(\cdot)$.

The form  ${\mathfrak t}_{G(z)}$ is   closable   and its  closure
is  given by
  \begin{equation}\label{eq:Form-clos}
{\overline {\mathfrak t}}_{G(z)}[u]  = \int_{\R_+}|u'(x)|^2dx + \IM
\varphi(z)|u(0)|^2, \qquad \dom {\overline {\mathfrak t}}_{G(z)} =
W^1_2(\mathbb R_+),  \  z \in \mathbb C_+.
   \end{equation}
Thus, form domain  $\dom ({\overline {\mathfrak t}}_{G(z)})=
W^1_2(\mathbb R_+)$  does not depend on $z \in \mathbb C_+$ while
the domain  $\dom G(z)$ does, i.e. $G$ satisfies the Assumption~\ref{ass:ii_1}.
%
%
%
%
The operator associated with the form ${\overline {\mathfrak
t}}_{\varphi(z)}$ (the imaginary part of  $G_{\varphi}(z)$) is given
by
%
%
\[
G_{\varphi,I}(z) u = - \frac{d^2u}{dx^2},\quad  \dom
(G_{\varphi,I}(z) = \{u\in W^2_2(\mathbb R_+): u'(0) =  (\IM
\varphi(z)) u(0)\}, \  z \in \C_+.
\]
   \end{example}
Next we present an example of the Weyl function satisfying
Assumption~\ref{ass:ii_2}. Such $R$-functions arise in the theory of
Schr\"odinger operators with local point interactions.
     \begin{example}\label{Ex:Schrod} 
Let $X=\{x_n\}^{\infty}_1$ be a strictly increasing sequence of
positive numbers 
satisfying $\lim_{n\to\infty}x_n = \infty.$
Denote $x_0=0$,
\begin{equation}\label{dn**}
 d_n:=x_{n}-x_{n-1}>0,\qquad 0\leq d_*:= \inf_{n\in\N}d_n, \qquad d^*:= \sup_{n\in\N}d_n \leq \infty.
\end{equation}

Let also $\rH_n$  be a minimal operator associated with the
expression $-\frac{\rD^2}{\rD x^2}$ in $L^{2}_0[x_{n-1},x_{n}]$.
Then $\rH_n$ is a symmetric operator  with deficiency indices
$n_{\pm}(\rH_n)=2$ and domain $\dom(\rH_{n})=
W^{2,2}_0[x_{n-1},x_{n}].$ Consider in $L^2(\dR_+)$  the direct sum
of symmetric  operators $\rH_n$,
      \[
\rH := {\rm H}_{\min}  =\bigoplus_{n=1}^{\infty} \rH_n,  
\qquad  \dom(\rH_{\min})=W^{2,2}_0(\dR_+\setminus X) =
\bigoplus_{n=1}^{\infty} W^{2,2}_0[x_{n-1},x_{n}].
       \]
It is easily seen that a boundary  triple  $\Pi_n=\{\Bbb C^2,
\Gamma_0^{(n)},\Gamma_1^{(n)}\}$  for  $\rH_n^*$ can be chosen  as
      \begin{equation}\label{IV.1.1_05}
\Gamma_0^{(n)}f:= \left(\begin{array}{c}
    f'(x_{n-1}+)\\
    f'(x_{n}-)
\end{array}\right),
\qquad  \Gamma_1^{(n)}f:= \left(\begin{array}{c}
    -f(x_{n-1}+)\\
    f(x_{n}-)
\end{array}\right), \qquad
f\in W_2^2[x_{n-1},x_n].
\end{equation}
The corresponding Weyl  function $M_n$ is given by
\begin{equation}\label{eq:M_n}
    M_n(z) = \dfrac{-1}{\sqrt{z}}
    \begin{pmatrix}
        \cot(\sqrt{z}d_n) & -\frac{1}{\sin(\sqrt{z}d_n)}\\
        -\frac{1}{\sin(\sqrt{z}d_n)} & \cot(\sqrt{z}d_n)
    \end{pmatrix}.
\end{equation}
Clearly, $\rH = \rH_{\min}$ is a  closed symmetric operator in $L^2(\dR_+)$.  Next we put
\begin{equation}\label{eq:Gamma_Point_Int}
\cH={l}^2(\dN)\otimes \dC^2, \qquad
  \Gamma=\left(
           \begin{array}{c}
             \Gamma_0 \\
             \Gamma_1 \\
           \end{array}
         \right)  := \bigoplus_{n=1}^{\infty} \left(
           \begin{array}{c}
             \Gamma_0^{(n)} \\
             \Gamma_1^{(n)} \\
           \end{array}
         \right)
\end{equation}
and note that in accordance with the definition of the direct sum of
linear mappings
   \begin{equation*}
 \dom \Gamma := \bigl\{f =
\oplus^{\infty}_{n=1} f_n \in\dom \ZA^*: \ \sum_{n\in
\N}\|\Gamma^{(n)}_j f_n\|^2_{\cH_n} <\infty, \, j\in\{0,1\}\bigr\}.
   \end{equation*}
We also put $\overline \Gamma_j := \oplus_{n=1}^{\infty}
\Gamma_j^{(n)}$ and note that it is a closure of
$\Gamma_j=\overline\Gamma_j\uphar \dom \Gamma$, $j=1,2$.
As stated in the next theorem the orthogonal sum $\Pi :=
\oplus_{n=1}^{\infty} \Pi_n$ of the boundary triples $\Pi_n$
determines an $ES$-generalized boundary triple of desired type.
\end{example}

Notice that the minimal operator $\rH$ as well as the corresponding
triple ${\Pi}$ for $\rH^*$ naturally arise when treating the
Hamiltonian  $\rH_{X,\gA}$ with $\delta$-interactions in the
framework of extension theory. The latter have appeared in various
physical problems as exactly solvable models that describe
complicated physical phenomena (see e.g. \cite{Alb_Ges_88,
Alb_Kur_00, Exn_04, KosMMM_2013} for  details).

The next theorem completes the results from~\cite{KosMMM} regarding
the non-regularized boundary triple ${\Pi}=\oplus_{n\in\dN}{\Pi}_n$
in the case $d_*=0$.
The proof is postponed until Section \ref{sec8.3}.

  \begin{theorem}\label{th_criterion(bt)Ex}
Let $\Pi := \oplus_{n=1}^{\infty} \Pi_n = \{\cH, \Gamma_0,\Gamma_1
\}$ be the direct sum of boundary triples $\Pi_n$ defined
by~\eqref{IV.1.1_05}, \eqref{eq:Gamma_Point_Int}, let $M(\cdot)$ be the corresponding Weyl
function, and let $d_*=0$  and   $d^* < \infty$. Then the following
statements hold:
\begin{enumerate}[label={\normalfont (\roman*)}]
\item  The triple $\Pi$ is an ES-generalized boundary triple for $H_{\min}^*$ such that $A_0\not =A_0^*$.
\item 
The Weyl function $M(\cdot)$ is domain invariant and its domain is
given by
\[
\dom M(z)= \left\{ \left\{\binom{a_n}{b_n}\right\}_{n=1}^\infty \in
{l}^2(\dN)\otimes \dC^2:\, \{a_n- b_n\}_{n=1}^\infty\in
{l}^2(\dN;\{d_n^{-2}\})\right\}, \quad z\in \C_{\pm}.
\]
\item
The  range of $\Gamma_0$ is given
by
\[
 \ran\Gamma_0
 = \left\{ \binom{a_n}{b_n}_{n=1}^\infty \in {l}^2(\dN) \otimes
\dC^2:\, \{a_n - b_n\}_{n=1}^\infty\in
 {l}^2(\dN;\{d_n^{-1}\})\right\} \supsetneqq \dom M(\pm i).
\]
\item The domain of the form $\st_{M(z)}$ generated by the imaginary part $\IM M(z)$ is given by
\[
 \dom \st_{M(z)} = \ran\Gamma_0, \quad z\in \C_{\pm}.
\]
\item  The  transposed triple $\Pi^\top$  is an S-generalized boundary triple for $H_{\min}^*$, i.e. $A_1 =A_1^*$. However, it is not
      a B-generalized boundary triple for $H_{\min}^*$.
  \end{enumerate}
%
  \end{theorem}
Combining statements (ii) and (iii) of Theorem
\ref{th_criterion(bt)Ex} yields that in the case $d_* =0$ \emph{the
Weyl function  $M(\cdot)$ corresponding to the triple $\Pi =
\oplus_{n=1}^{\infty} \Pi_n$ satisfies Assumption~\ref{ass:ii_2},
i.e. it is domain invariant,\index{Weyl function!domain invariant} $\dom M(z) = \dom M(i), \ z\in \Bbb
C_{\pm}$,  while   $\dom M(i) \subsetneqq \ran\Gamma_0$}.

Hence,  by Theorem~\ref{prop:C6B},  $A_0\ne A_0^*$ and  $\Pi = \oplus_{n=1}^{\infty}
\Pi_n$ being $ES$-generalized, is not an $S$-generalized boundary triple for $\rH^*$.

It is emphasized that in this case we compute  $\dom M(z)$,  $\dom
\st_{M(z)}$, and $\ran\Gamma_0$  \textit{explicitly}. Notice also
that in this case  the Weyl function $M(\cdot)$ as well as its
imaginary part $\IM M(\cdot)$ take values in the set of unbounded
operators. In addition to the Schr\"odinger operators introduced in
Example~\ref{Ex:Schrod} analogous results for \emph{moment and Dirac
operators with local point interactions} are established in
Section~\ref{sec8}.


Before closing this subsection we wish to mention that other type of
examples for $ES$-generalized boundary triples are the
\emph{Kre\u{\i}n - von Neumann Laplacian} appearing in Sections
\ref{sec7.1}, \ref{sec7.3} and the \emph{Zaremba Laplacian} for
mixed boundary value problem treated in Section~\ref{ex:Mixed}.

\subsection{$AB$-generalized boundary triples and quasi boundary triples}
The following definition of an \emph{$AB$-generalized boundary
triple} is a weakening of the notions of $B$-generalized and
$S$-generalized boundary triples as well as of the class of quasi
boundary triples.
\begin{definition}\label{ABGtriple00}\index{Boundary triple!AB-generalized}
A collection $\{\cH,\Gamma_0,\Gamma_1\}$ is said to be an
\emph{almost $B$-generalized boundary triple}, or briefly, an
\emph{AB--generalized boundary triple} for $\ZA^*$, if $\ZA_*=\dom
\Gamma$ is dense in $\ZA^*$ and the following conditions  are
satisfied:
\begin{enumerate}[label=\textbf{\thedefinition.\arabic*}]
\item\label{1_def_ABGtriple00} the abstract Green's identity~\eqref{Greendef1} holds
            for all $\wh f,  \wh g \in  \ZA_*$;
\item \label{2_def_ABGtriple00} $\ran \Gamma_0$ is dense in $\cH$;
\item\label{3_def_ABGtriple00} $A_0:=\ker\Gamma_0$ is a selfadjoint relation in $\sH$.
   \end{enumerate}
\end{definition}

It is shown in Proposition \ref{ABGcor2} that the Weyl function
$M(\cdot)$ of an $AB$-generalized boundary triple admits a similar
characterization as the Weyl function corresponding to an
$S$-generalized boundary triple.
%

   \begin{proposition}\label{ABG_Intro}
The class of $AB$-generalized boundary triples coincides with the class of
isometric boundary triples such that the corresponding  Weyl
functions satisfy the condition \eqref{stric-unb} and they are of
the form
  \begin{equation}\label{WeylABG2_Intro}
 M(\lambda)=E+M_0(\lambda),\quad M_0(\cdot)\in \cR[\cH],  \quad \lambda\in\cmr,
  \end{equation}
with $E$ a symmetric densely defined operator in $\cH$. In
particular, every function $M(\cdot)$ of the form
\eqref{WeylABG2_Intro} such that $\ker \IM M_0(\lambda)\cap \dom
E=\{0\}$ is a Weyl function of a certain $AB$-generalized boundary triple.

Moreover,  $E = E^*$ if and only if  $\Gamma$ is
$(J_\sH,J_\cH)$-unitary (see Definition~\ref{unitBT}), i.e. if and
only if the $AB$-generalized boundary triple
$\{\cH,\Gamma_0,\Gamma_1\}$ is also $S$-generalized.
  \end{proposition}

Further properties of $AB$-generalized boundary triples are studied
in Section~\ref{sec4}. 

A connection between $ES$-~and $AB$-generalized boundary triples is
established in Theorem~\ref{essThm2}. More precisely, it is shown
that for every strict form domain invariant operator valued
Nevanlinna function $M\in \cR^s(\cH)$ there exist a bounded operator
$G\in[\cH]$ with $\ker G=\ker G^*=\{0\}$, a closed symmetric densely
defined operator $E$ in $\cH$, and a bounded Nevanlinna function
$M_0\in \cR[\cH]$, with the property
\begin{equation}\label{MclH_Intro}
 \cH=\clos\sD_\lambda:=\clos \{ h\in\cH:\, (\IM M_0(\lambda))^{\half}h\in \ran G^*\}, \quad
 \lambda\in\cmr,
\end{equation}
such that
\begin{equation}\label{Mess_Intro}
  M(\lambda)= G^{-*}(E+M_0(\lambda))G^{-1}, \quad \lambda\in\cmr.
\end{equation}
Conversely, every Nevanlinna function $M(\cdot)$ of the form
\eqref{Mess_Intro} is form domain invariant on $\cmr$, whenever
$E\subset E^*$, $G\in\cB(\cH)$, $\ker G=\ker G^*=\{0\}$, and $M_0\in
\cR[\cH]$ satisfies~\eqref{MclH_Intro}.

Theorem~\ref{essThm2} offers a \emph{renormalization procedure}
which produces from a form domain invariant Weyl function a domain
invariant Weyl function, whose imaginary part in standard operator
sense becomes a well-defined and bounded operator function on
$\cmr$, i.e., the renormalized boundary triple is $AB$-generalized.
Some related results, showing how $B$-generalized boundary triples
give rise to $ES$-generalized boundary triples, are established in
Section \ref{sec6.1}. These results are applied in the analysis of
\emph{regularized trace operators} for Laplacians.

On the other hand, since every $AB$-generalized boundary triple can
be regularized to produce a $B$-generalized boundary triple, see
Theorem \ref{QBTthm}, and every $B$-generalized boundary triple can
be regularized to produce an ordinary boundary triple, see
Theorem~\ref{thm:2.1}, there is a controlled connection from
$ES$-generalized boundary triples to ordinary boundary triples. In
this way these abstract \emph{regularization procedures} open an
avenue e.g. to complete spectral analysis and related
well-established investigations in extension theory of symmetric
operators and its various applications. It should be pointed out
that, only for ordinary boundary triples, the pair of boundary
mappings $\Gamma=\{\Gamma_0,\Gamma_1\}$ provides a topological
isomorphism between the set of all linear relations in the parameter
space $\cH$ and the complete class of intermediate extensions $\wt
A$ lying between $A$ and its adjoint $A^*$.

The class of $AB$-generalized boundary triples contains the class of
so-called quasi boundary triples, which has been studied in J.
Behrndt and M. Langer \cite{BeLa07}.

  \begin{definition}[\cite{BeLa07}]\label{def-Jussi}\index{Boundary triple!quasi}
Let $\ZA$ be a densely defined symmetric operator in $\frak H.$ A
triple $\Pi = \{\cH,\Gamma_0,\Gamma_1\}$ is said to be a \emph{quasi
boundary triple} for $\ZA^*$, if the following conditions are
satisfied:
\begin{enumerate}[label=\textbf{\thedefinition.\arabic*}]
  \item \label{1_def_def-Jussi}$\ZA_*:=\dom\Gamma$ is dense in $\ZA^*$ with respect to the topology on $\sH^2$ and
  the Green's identity~\eqref{Greendef1} holds  for all
  $\wh f, \wh g  \in  \ZA_*$;
  \item \label{2_def_def-Jussi}the  range of
$\Gamma$ 
is dense in $\cH\times\cH$;
    \item\label{3_def_def-Jussi}   $A_0 = \ker \Gamma_0$ is a selfadjoint operator  in $\sH$.
\end{enumerate}
    \end{definition}

In the definition of a quasi boundary triple Assumption~\ref{2_def_ABGtriple00} is replaced by the stronger
assumption that the joint range of $\Gamma=\{\Gamma_0,\Gamma_1\}$ is
dense in $\cH\times\cH$. The Weyl function corresponding to a quasi
boundary triple is again defined by the same formula~\eqref{0.2}.
The notion of quasi boundary triple proved to be useful in elliptic
theory~\cite{BeLa07}, see also~\cite{BeMi14,Post16}. A connection
between quasi boundary triples and $AB$-generalized boundary triples
is given in Corollary~\ref{quasicor}. A joint feature in
$AB$-generalized boundary triples and quasi boundary triples is that
without additional assumptions on the mapping
$\Gamma=\{\Gamma_0,\Gamma_1\}$ these boundary triples are not
unitary. Consequently, their Weyl functions need not belong to the
class of Nevanlinna functions; i.e. the values $M(\lambda)$ need not
be maximal dissipative (accumulative) in $\dC_+$ ($\dC_-$). More
explicitly, the defect numbers of the operator $E$ in
\eqref{WeylABG2_Intro} need not be equal; in which case even after
taking closures of $\Gamma$ and $M(\lambda)$ this situation is not
changed.
A complete characterization (a realization result) of the set of
Weyl functions corresponding to $AB$-generalized and quasi boundary
triples is given by formula \eqref{eq:DomInv} with $M_0(\cdot)$
belonging to the class $\cR[\cH]$ while $E$ is \emph{a symmetric,
but not necessarily selfadjoint, operator}, such that
\[
 \dom E^*\cap \ker \IM {M_0(\lambda)}=\{0\};
\]
the role of this last condition is connected with the assumption~\ref{2_def_def-Jussi}.

Different applications of quasi boundary triples in boundary value
problems including applications to  elliptic  theory and  trace
formulas can be found e.g. in~\cite{BeLa07},~\cite{BeMi14},
\cite{BeGeMM15}, \cite{HdSSz2012}, and~\cite{Post16}.

\begin{remark}\label{quasiremark}
A connection between $B$-generalized boundary triples and quasi
boundary triples appears in \cite[Theorem~7.57]{DHMS12} and in a
more precise form is given in
\cite[Propositions~5.1,~5.3]{Wietsma13}; see also
\cite{WietsmaThesis2012}. A slightly more general result can be
found in Corollary~\ref{quasicor}. Further results and a more
detailed discussion on these connections are given in
Sections~\ref{sec4} and~\ref{sec5}.
\end{remark}

\subsection{Preparatory results for applications}
Section \ref{sec6} is devoted to the study of certain specific types
of boundary triples offering also applicable abstract frameworks for
including trace operators in a boundary triple environment. In
Section \ref{sec6.1} it is shown how certain simple transforms of
$B$-generalized boundary generate $ES$-generalized boundary triples;
in Section \ref{sec7} such transforms are identified and made
explicit in the Laplace setting.

In typical applications to elliptic PDE's the minimal operator $A$
is nonnegative and the Dirichlet-to-Neumann map $\Lambda(\cdot)$ is
constructed at the origin $z=0$ or at some point $x<0$ on the real
line in such a way that $\Lambda(x)$ is a nonnegative selfadjoint
operator in the boundary space $L^2(\partial\Omega)$. Under weak
additional assumptions this implies that the corresponding boundary
triple is not only isometric but, in fact, unitary. In this
connection we offer the following~\emph{analytic extrapolation
principle for Weyl functions} initially defined only at one real
regular point and then extended in an appropriate manner to the
complex half-planes. In particular, this result can be used to check
whether a pair of boundary mappings $\{\Gamma_0,\Gamma_1\}$
satisfying Green's identity~\eqref{Greendef1} determines a unitary
boundary triple.

\begin{proposition}\label{Mrealreg0}
Let $\Pi = \{\cH,\Gamma_0,\Gamma_1\}$ be an isometric boundary
triple for $A^*$, let $H$ with $\dom H\subset \dom A_*(=\dom
\Gamma)$ be a selfadjoint extension of $A$ admitting a real regular
point $x\in\rho(H)$, and let the mapping $M(x)$ be defined by the
formula~\eqref{0.2} for all $\wh f_x\in \wh\sN_{x}(A_*)$.

If $M(x)$ satisfies the conditions
\begin{equation}\label{Mext}
 M(x)=M(x)^* \quad\text{and}\quad 0 \in \rho(M(x)+x),
\end{equation}
then $M(x)$ admits an analytic extrapolation $M(z)$ to the
half-planes $z\in\dC_\pm$ defining a function in the class
$\cR(\cH)$ of Nevanlinna functions. Furthermore, under conditions
~\eqref{Mext} the boundary triple $\Pi$ is unitary and $M(\cdot)$ is
the Weyl function of $\Pi$.
\end{proposition}

This result is contained in a somewhat more general statement proved
in Theorem~\ref{MRealreg}. For partial differential operator
\emph{the first Green's formula} typically implies the conditions
appearing in Proposition \ref{Mrealreg0}. As to applications of
Proposition \ref{Mrealreg0} let us mention that in
Section~\ref{example8.3} we construct a unitary boundary pair for
Laplacians on rough domains, see Proposition \ref{ABGrough}. In
Section \ref{sec6.4} the result is applied to associate unitary
boundary triples with the concepts of \emph{boundary pairs for
nonnegative forms} appearing in \cite{Arlin96,ArHa15} and in the
most general form in the paper \cite{Post16} of O. Post. The
connection of various classes of boundary pairs for nonnegative
forms as defined in \cite{Post16} to the present subclasses of
unitary boundary triples is established in Theorem~\ref{ThmFpair}.

\subsection{A short description of the contents}
In Section~\ref{sec2} we recall basic concepts of linear relations
(sums of relations, componentwise sums, defect subspaces, etc.) as
well as unitary and isometric relations in Kre\u{\i}n space. We also
introduce the concepts of Nevanlinna functions and families.

In Section~\ref{sec3} we discuss unitary and isometric boundary
pairs and triples. We also introduce the notions of Weyl functions
and families and discuss their properties. A general version of the
main realization result, Theorem~\ref{GBTNP}, is presented therein,
too. It completes and improves Theorem~\ref{TAMS_Th}. Besides
certain isometric transforms of boundary triples are discussed.

In Section~\ref{sec4} we investigate $AB$-generalized boundary
triples. In particular, we  present the proof of
Proposition~\ref{ABG_Intro} and discuss its generalizations and
consequences.
In Theorem~\ref{QBTthm}  we find a connection between
$B$-generalized  and $AB$-generalized boundary triples by means of
triangular isometric transformations. In Theorem~\ref{Kreinformula}
it is shown that every $AB$-generalized boundary triple admits a
Kre\u{\i}n type resolvent formula.
%
%


In Section \ref{sec5} we prove Theorem~\ref{main_theorem} (see
Theorem \ref{essThm1}). Here the connection between unitary boundary
triples and unitary colligations is systematically used. In
particular, this connection is applied to extend
Theorem~\ref{prop:C6B} to the case of  $S$-generalized boundary
pairs (see Theorem~\ref{prop:C6}). In this case
representation~\eqref{eq:DomInv}  for the Weyl function remains
valid with $M_0\in \cR[\cH_0]$ and $\cH_0 \subseteq \cH$ instead of
$M_0\in \cR^s[\cH]$.
Besides, in Theorem~\ref{essThm2} a connection between
$ES$-generalized boundary triples and$AB$-generalized boundary
triples is established via an isometric transform introduced in
Lemma~\ref{isomtrans} (see formula~\eqref{isom-transf}).

Section~\ref{sec6} contains a couple of further useful results which
are of preparatory nature for applications of unitary and, in
particular, $ES$-generalized boundary triples. Sections \ref{sec7}
and \ref{sec8} are devoted to applications of the general results in
the PDE setting by treating Laplace operators and in the ODE setting
by studying differential operators with local point interactions, it
is these applications that initially acted as a motivation for the
present work.

\section{Preliminary concepts}\label{sec2}

\subsection{Linear relations in Hilbert spaces}
\index{Linear relation}
A linear relation $T$ from $\sH$ to $\sH'$ is a linear subspace of
$\sH \times \sH'$. Systematically a linear operator $T$ will be
identified with its graph. It is convenient to write $T:\sH\to\sH'$
and interpret the linear relation $T$ as a multivalued linear
mapping from $\sH$ into $\sH'$. If $\sH'=\sH$ one speaks of a linear
relation $T$ in $\sH$. Many basic definitions and properties
associated with linear relations can be found
in~\cite{Arens,Ben,Co}.

The following notions appear throughout this paper. For a linear
relation $T:\sH \to \sH'$ the symbols $\dom T$, $\ker T$, $\ran T$,
$\mul T$ and $\overline{T}$ stand for the domain, kernel, range,
multivalued part, and  closure, respectively. The inverse $T^{-1}$
is a relation from $\sH'$ to $\sH$ defined by
$\{\,\{f',f\}:\,\{f,f'\}\in T\,\}$. The adjoint $T^*$ is the closed
linear relation from $\sH'$ to $\sH$ defined by 
\[
 T^*=\left\{\,\begin{pmatrix}h\\k\end{pmatrix} \in \sH' \oplus \sH :\,
     (k,f)_{\sH}=(h,g)_{\sH'}, \, \begin{pmatrix}f\\g\end{pmatrix}\in T \,\right\}.
\]
The sum $T_1+T_2$ and the componentwise sum $T_1 \wh + T_2$ of two
linear relations $T_1$ and $T_2$ are defined by
\begin{equation}
\label{eq:sum}
  T_1+T_2=\left\{\,\begin{pmatrix}f\\g+k\end{pmatrix} :\, \begin{pmatrix}f\\g\end{pmatrix} \in T_1, \begin{pmatrix}f\\k\end{pmatrix} \in T_2\,\right\},
\end{equation}
\begin{equation}
\label{eq:hsum}
  T_1 \hplus T_2=\left\{\, \begin{pmatrix}f+h\\g+k\end{pmatrix} :\, \begin{pmatrix}f\\g\end{pmatrix} \in T_1, \begin{pmatrix}h\\k\end{pmatrix}\in
T_2\,\right\}.
\end{equation}
If the componentwise sum is orthogonal it will be denoted by $T_1
\oplus T_2$. If  $T$ is closed, then the null spaces of $T-\lambda$,
$\lambda \in \dC$, defined by
\begin{equation}
\label{defect0} \sN_\lambda(T)=\ker (T-\lambda), \quad \wh
\sN_\lambda(T) =\left\{\, \begin{pmatrix}f\\\lambda f\end{pmatrix}
\in T :\, f\in\sN_\lambda(T)\,\right\},
\end{equation}
are also closed. Moreover, $\rho(T)$ ($\hat\rho(T)$) stands for the
set of regular (regular type) points of $T$.

\index{Linear relation!symmetric}\index{Linear relation!dissipative}\index{Linear relation!accumulative}\index{Linear relation!selfadjoint}
 Recall that a linear relation $T$ in $\sH$ is called
\textit{symmetric}, \textit{dissipative}, or \textit{accumulative}
if $\mbox{Im }(h',h)=0$, $\ge 0$, or $\le 0$, respectively, holds
for all $\{h,h'\}\in T$. These properties remain invariant under
closures. By polarization it follows that a linear relation $T$ in
$\sH$ is symmetric if and only if $T \subset T^*$. A linear relation
$T$ in $\sH$ is called \textit{selfadjoint} if $T=T^*$, and it is
called \textit{essentially selfadjoint} \index{Linear relation!essentially selfadjoint}if $\overline{T}=T^*$. A
dissipative (accumulative) linear relation $T$ in $\sH$ is called
\emph{maximal dissipative} \index{Linear relation!maximal dissipative}(\emph{maximal accumulative}) if it has
no proper dissipative (accumulative) extensions.

If the relation $T$ is maximal dissipative (accumulative), then
$\mul T=\mul T^*$ and  the orthogonal decomposition $\sH=(\mul
T)^\perp \oplus \mul T$ induces an orthogonal decomposition of $T$
as
\begin{equation}
\label{eq0}
 T=\textup{gr }T_{\rm op} \oplus (\{0\} \times \cH_\infty), \quad \cH_\infty= \mul T, \quad
 \textup{gr }T_{\rm op}=\left\{\, \begin{pmatrix}f\\ g\end{pmatrix}\in T:\, g\in \cH\ominus \cH_\infty\,\right\},
\end{equation}
where $T_\infty:=\{0\} \times \cH_\infty$ is a purely multivalued selfadjoint relation in $\cH_\infty$ and $T_{\rm
op}$ is a densely defined maximal dissipative (resp. accumulative)
operator in $\cH\ominus \cH_\infty$. In particular, if $T$ is a selfadjoint
relation, then there is such a decomposition, where $T_{\rm op}$ is a
densely defined selfadjoint operator in $\cH\ominus \cH_\infty$.

A family of linear relations $M(\lambda)$, $\lambda \in \cmr$, in a
Hilbert space $\cH$ is called a \textit{Nevanlinna family}\index{Nevanlinna family} if:
\begin{enumerate}
\def\labelenumi{\textit{(\roman{enumi})}}
\item for every $\lambda \in \dC_+ (\dC_-)$
      the relation $M(\lambda)$ is maximal dissipative
      (resp. accumulative);
\item $M(\lambda)^*=M(\bar \lambda)$, $\lambda \in \cmr$;
\item for some, and hence for all, $\mu \in \dC_+ (\dC_-)$ the
      operator family $(M(\lambda)+\mu)^{-1} (\in [\cH])$ is
      holomorphic for all $\lambda \in \dC_+ (\dC_-)$.
\end{enumerate}
By the maximality condition, each relation $M(\lambda)$, $\lambda
\in \cmr$, is necessarily closed. The \textit{class of all
Nevanlinna families} in a Hilbert space is denoted by $\wt
\cR(\cH)$. If the multivalued part $\mul M(\lambda)$ of $M \in \wt
\cR(\cH)$ is nontrivial, then it is independent of $\lambda \in
\cmr$, so that
\begin{equation} \label{ml}
 M(\lambda)=\textup{gr }M_{\rm op}(\lambda) \oplus M_\infty
   \quad \cH_\infty=\mul M(\lambda), \quad \lambda\in\dC\setminus\dR,
\end{equation}
where $M_\infty=\{0\} \times \cH_\infty$ is a purely multivalued linear relation in $\cH_\infty:=\mul M(\lambda)$ and $M_{\rm op}(\cdot)\in\cR(\cH\ominus\cH_\infty)$, cf.
\cite{KL71,KL73,LT77}. Identifying operators in $\cH$ with their
graphs one can consider classes
\[
\cR^u[\cH]\subset \cR^s[\cH]\subset 
\cR^s(\cH)\subset \cR(\cH)
\]
introduced in Section~\ref{sec1} as subclasses of $\wt \cR(\cH)$.

In addition, a Nevanlinna family $M(\lambda)$, $\lambda\in\cmr$,
which admits a holomorphic extrapolation to the negative real line
$(-\infty,0)$ (in the resolvent sense as in item (iii) of the above
definition) and whose values $M(x)$ are nonnegative (nonpositive)
selfadjoint relations for all $x<0$ is called a \textit{Stieltjes
family}\index{Stieltjes family} (an \textit{inverse Stieltjes family}\index{Stieltjes family!inverse}, respectively).

%
%

\begin{definition}
A symmetric linear relation $\ZA$ in $\sH$ is called \textit{simple}
if there is no nontrivial orthogonal decomposition of $\sH=\sH_1
\oplus\sH_2$ and no corresponding orthogonal decomposition $\ZA=\ZA_1
\oplus \ZA_2$ with $\ZA_1$ a symmetric relation in $\sH_1$ and $\ZA_2$ a
selfadjoint relation in $\sH_2$.
\end{definition}

The decomposition~\eqref{eq0} for $\ZA=\ZA_{\rm op}\oplus \ZA_\infty$
shows that a simple closed symmetric relation is necessarily an
operator. Recall that (cf. e.g. \cite{LT77}) a closed symmetric
linear relation $\ZA$ in a Hilbert space $\sH$ is simple if and only
if
\[
 \sH=\cspan \{\, \sN_\lambda(\ZA^*) :\, \lambda \in \cmr\,\}.
\]

\subsection{Unitary and isometric  relations in Kre\u{\i}n spaces}
Let $\sH$ and $\cH$ be Hilbert spaces and let $(\sH^2,J_{\sH})$ and
$(\cH^2,J_{\cH})$ be Kre\u{\i}n spaces with fundamental symmetries
$J_{\sH}$, $J_{\cH}$ and indefinite inner products $[\cdot,\cdot]_\sH$,
$[\cdot,\cdot]_\cH$ defined in~\eqref{jh} and~\eqref{eq:leftik_sH},
respectively.

If  $\Gamma$ is a linear relation from the Kre\u{\i}n space
$(\sH^2,J_\sH)$ to the Kre\u{\i}n space $(\cH^2,J_\cH)$, then the
adjoint linear relation $\Gamma^{[*]}$ is defined by\index{Linear relation!adjoint}
\begin{equation}\label{def:J_unit}
  \Gamma^{[*]}=\left\{\left(
                        \begin{array}{c}
                          \wh k \\
                          \wh{g} \\
                        \end{array}
                      \right)\in
                      \left(
                        \begin{array}{c}
                          \cH^2 \\
                          \sH^2 \\
                        \end{array}
                      \right):
     [ \wh f,\wh g]_{\sH^2}=[ \wh h, \wh
 k]_{\cH^2}\, \text{ for all }
 \left(
                        \begin{array}{c}
                          \wh f \\
                          \wh{h} \\
                        \end{array}
                      \right)\in\Gamma
  \right\}.
\end{equation}

\begin{definition}\label{def2.5}{(\rm \cite{Sh0})}
\ A linear relation $\Gamma$ from the Kre\u{\i}n space
$(\sH^2,J_\sH)$ to the Kre\u{\i}n space $(\cH^2,J_\cH)$ is said to
be \textit{$(J_\sH,J_\cH)$--isometric} if $\Gamma^{-1}\subset \Gamma^{[*]}$
and \textit{$(J_\sH,J_\cH)$--unitary}, if $\Gamma^{-1}=\Gamma^{[*]}$.
\end{definition}

The following two statements are due to Yu.L.~Shmul'jan \cite{Sh0};
see also~\cite{DHMS06}.

\begin{proposition}\label{UNIT}
Let $\Gamma$ be a $(J_\sH,J_\cH)$-unitary relation from the Kre\u{\i}n
space $(\sH^2,J_\sH)$ to the Kre\u{\i}n space $(\cH^2,J_\cH)$. Then:
\begin{enumerate}
\def\labelenumi{\textit{(\roman{enumi})}}
\item $\dom \Gamma$ is closed if and only if $\ran \Gamma$ is closed; \item
the following equalities hold:
\[
\ker \Gamma=(\dom \Gamma)^{[\perp]}, \quad \mul \Gamma=(\ran
\Gamma)^{[\perp]}.
\]
\end{enumerate}
\end{proposition}

%
%

A {$(J_\sH,J_\cH)$-unitary relation
$\Gamma:(\sH^2,J_\sH)\to(\cH^2,J_\cH)$ may be multivalued,
nondensely defined, and unbounded. It is the graph of an operator if
and only if its range is dense. In this case it need not be densely
defined or bounded; and even if it is bounded it need not be densely
defined.
%

\section{Unitary and isometric boundary pairs and associated Weyl families}
\label{sec3}

\subsection{Definitions and basic properties}

Let $\ZA$ be a closed symmetric linear relation in the Hilbert space
$\sH$. It is not assumed that the defect numbers of $\ZA$ are equal or
finite. Following \cite{DHMS06,DHMS12} a unitary/isometric boundary pair for $\ZA^*$ is defined as follows.

\begin{definition}
\label{GBT} Let $\ZA$ be a closed symmetric linear relation in a
Hilbert space $\sH$, let $\cH$ be an auxiliary Hilbert space and let $\Gamma$ be a linear relation from the Kre\u{\i}n space
$(\sH^2,J_\sH)$
to the Kre\u{\i}n space $(\cH^2,J_\cH)$.
Then\index{Boundary pair!unitary} \index{Boundary pair!isometric}
$\{\cH,\Gamma\}$ is called a \textit{unitary/isometric boundary pair for $\ZA^*$}, if:
\begin{enumerate}[label=\textbf{\thedefinition.\arabic*}]
  \item\label{1_def_GBT}$\ZA_*:=\dom\Gamma$ is dense in $\ZA^*$ with respect to the topology on $\sH^2$;
  \item \label{2_def_GBT} the linear relation $\Gamma$ is $(J_\sH,J_\cH)$-unitary/isometric.
\end{enumerate}
\end{definition}
In particular, it follows from this definition that for all vectors
$\{\wh f,\wh h\},\, \{\wh g,\wh k\}\in\Gamma$ of the
form~\eqref{eq:leftik_sH} the abstract \textit{Green's identity}index{Green's identity}
(cf. Definition \ref{een}) holds
\begin{equation}
\label{Green1}
 (f',g)_\sH-(f,g')_\sH=({h'},k)_{\cH}-(h,k')_{\cH}.
\end{equation}

Let $\{\cH,\Gamma\}$ be a unitary boundary pair for $\ZA^*$ and let
$\ZA_*=\dom\Gamma$. According to~\cite[Proposition~2.12]{DHMS06} the
domain ${\ZA_*}$ of $\Gamma$ is a linear relation in $\sH$, such that
\[
 \ZA\subset {\ZA_*}\subset \ZA^*,
 \quad
 \overline{\ZA_*}=\ZA^*.
\]
The eigenspaces ${\sN}_{\lambda}({\ZA_*})$ and
$\wh{\sN}_{\lambda}({\ZA_*})$ of ${\ZA_*}$ are defined as in
~\eqref{defect0},
\[
 \sN_\lambda({\ZA_*})=\ker ({\ZA_*}-\lambda),
\quad
 \wh \sN_\lambda({\ZA_*})
  =\left\{\,  \begin{pmatrix}f_\lambda\\ \lambda f_\lambda\end{pmatrix} \in {\ZA_*} :\, f_\lambda\in\sN_\lambda({\ZA_*})\,\right\}.
\]

\begin{definition}
\label{Weylfam} The \textit{Weyl family} \index{Weyl family} $M$ of $\ZA$ corresponding to
the unitary/isometric boundary pair $\{\cH,\Gamma\}$ is defined by
$M(\lambda):=\Gamma(\wh{\sN}_\lambda(\ZA_*))$, i.e.,
\begin{equation}
\label{Weylf}
 M(\lambda) :=
  \left\{\,\wh h\in\cH^2:\,\{\wh f_\lambda,\wh h\}\in\Gamma \mbox{
for some } \wh f_\lambda= \begin{pmatrix}f_\lambda\\ \lambda f_\lambda\end{pmatrix}\in\sH^2\,\right\}\quad (\lambda \in \cmr),
\end{equation}
In the case where $M$ is single-valued it is called the \textit{Weyl
function} of $\ZA$ corresponding to $\{\cH,\Gamma\}$. The
\textit{$\gamma$-field} of $\ZA$ corresponding to the
unitary/isometric boundary pair $\{\cH,\Gamma\}$ is defined by
\begin{equation}
\label{gfield}
 \gamma(\lambda) :=
  \left\{\, \{h,f_\lambda\}\in \cH\times\sH :\,\left\{\begin{pmatrix}f_\lambda\\ \lambda f_\lambda\end{pmatrix},
  \begin{pmatrix}h\\  h'\end{pmatrix}\right\}\in\Gamma \mbox{
for some } h'\in\cH
\right\},
\end{equation}
where $\lambda \in \cmr$. Moreover, $\wh\gamma(\lambda)$ stands for
\begin{equation}
\label{ghatfield}
 \wh\gamma(\lambda):
  =\left\{\, \{h,\wh f_\lambda\}\in \cH\times\sH^2:\,\left\{\wh f_\lambda,
  \begin{pmatrix}h\\  h'\end{pmatrix}\right\}\in\Gamma \mbox{
for some } h'\in\cH
\right\}.
\end{equation}
\end{definition}

With $\gamma(\lambda)$ the relation $\Gamma\uphar\wh\sN_\lambda(\ZA_*)$
can be rewritten as follows
\begin{equation}
\label{ggam1}
 \Gamma\uphar\wh\sN_\lambda(\ZA_*)
 :=\left\{\,\left\{\begin{pmatrix}\gamma(\lambda)h\\
 \lambda\gamma(\lambda)h\end{pmatrix},\begin{pmatrix}h\\h'\end{pmatrix}\right\}:\,
     \begin{pmatrix}h\\h'\end{pmatrix}\in M(\lambda)\,\right\},
\quad \lambda\in\cmr.
\end{equation}

Associate with $\Gamma$ the following two linear relations which are not
necessarily closed:
\begin{equation}
\label{g01}
 \Gamma_0 =\left\{ \, \{\wh f,h\} :\, \{\wh f, \wh h\} \in \Gamma,
\,
           \wh h=\begin{pmatrix}h\\h'\end{pmatrix}\,\right\}, \quad
           \Gamma_1 =\left\{ \, \{\wh f,h'\} :\, \{\wh f, \wh h\} \in \Gamma,
\,
             \wh h=\begin{pmatrix}h\\h'\end{pmatrix}\,\right\}.
\end{equation}

The $\gamma$-field $\gamma(\cdot)$ associated with $\{\cH,\Gamma\}$
is the first component of the mapping $\wh\gamma(\lambda)$ in~\eqref{ghatfield}. Observe, that
\[
 \wh\gamma(\lambda):=(\Gamma_0\uphar\wh\sN_\lambda({\ZA_*}))^{-1},
 \quad \lambda\in\cmr,
\]
is a linear mapping from $\Gamma_0(\wh\sN_\lambda({\ZA_*}))=\dom
M(\lambda)$ onto $\wh\sN_\lambda({\ZA_*})$: it is single-valued in view
of~\eqref{Green1}; cf.~\eqref{Green2},~\eqref{Green2B}.
Consequently, the $\gamma$-field is a
single-valued mapping from $\dom M(\lambda)$ onto $\sN_\lambda({\ZA_*})$
and it satisfies $\gamma(\lambda)\Gamma_0\wh f_\lambda=f_\lambda$ for
all $\wh f_\lambda\in \wh\sN_\lambda({\ZA_*})$.

If $\Gamma$ is single-valued then these component mappings decompose $\Gamma$,
$\Gamma=\Gamma_0\times\Gamma_1$, and the triple $\{\cH,\Gamma_0,\Gamma_1\}$
will be called a \textit{unitary/isometric boundary triple} for $\ZA^*$.
In this case the Weyl function corresponding to the unitary/isometric boundary triple
$\{\cH,\Gamma_0,\Gamma_1\}$ can be also defined via
\begin{equation}\label{eq:Mgen}
     M(\lambda)\Gamma_0\wh f_\lambda=\Gamma_1\wh f_\lambda, \quad \wh f_\lambda\in \wh\sN_\lambda(\ZA_*).
\end{equation}
When $A$ admits real regular type points it is useful to extend
Definition~\ref{Weylfam} of the Weyl family to the points on the
real line by setting $M(x):=\Gamma(\wh{\sN}_x(\ZA_*))$ or, more
precisely,
\begin{equation}
\label{Weylreal}
 M(x) :=
  \left\{\,\wh h\in\cH^2:\,\{\wh f_x,\wh h\}\in\Gamma\,\, \text{ for some } \wh f_x = \begin{pmatrix}f_x\\ x
f_x\end{pmatrix}\in\sH^2,\; x \in \dR\,\right\}.
\end{equation}

\subsection{Unitary boundary pairs and unitary boundary
triples}\label{sec3.2} The following theorem shows that the set of
all Weyl families of unitary boundary pairs coincides with $\wt
\cR(\cH)$ (see~\cite[Theorem~3.9]{DHMS06}).
Recall that a unitary boundary pair $\{\cH,\Gamma\}$ for $\ZA^*$ is said to be \textit{minimal}, if
\[
 \sH=\sH_{\min}:=\cspan\{\,\sN_\lambda({\ZA_*}):\,\lambda\in\dC_+\cup\dC_-\,\}.
\]

\begin{theorem}\label{GBTNP}
Let $\{\cH,\Gamma\}$ be a unitary boundary pair for $\ZA^*$. Then
the corresponding Weyl family $M$ belongs to the class of Nevanlinna
families $\wt \cR(\cH)$.

Conversely, if $M$ belongs to the class $\wt \cR(\cH)$, then there
exists a unique (up to a unitary equivalence) minimal unitary
boundary pair $\{\cH,\Gamma\}$ whose Weyl function coincides with
$M$.
\end{theorem}

Notice that Theorem~\ref{TAMS_Th} contains a general analytic
criterion for an isometric boundary triple to be unitary; the Weyl
function should be a Nevanlinna function, cf. Theorem~\ref{TAMS_Th}.

 \begin{corollary}\label{cor:TAMS_Th}
The class of Weyl functions corresponding to unitary boundary
triples coincides with the class $\cR^s(\cH)$ of (in general
unbounded) strict  Nevanlinna functions.
   \end{corollary}
\begin{proof}
The statement is immediate when combining  Theorem~\ref{GBTNP} with
Proposition~4.5 from \cite{DHMS06}.
\end{proof}

As a consequence of~\eqref{Green1} and~\eqref{ggam1} the following
identity holds (cf.~\eqref{ml})
\begin{equation}\label{Green3U}
 (\lambda-\bar\mu)(\gamma(\lambda)h,\gamma(\mu)k)_\sH
 =(M_{\rm op}(\lambda)h,k)_\cH - (h,M_{\rm op}(\mu)k)_\cH,
\end{equation}
where $h\in\dom M(\lambda)$ and $k\in \dom M(\mu)$,
$\lambda,\mu\in\cmr$.

As was already mentioned in Section~\ref{sec1} every operator valued
function $M$ from $\cR^u[\cH]$ ($\cR^s[\cH]$) can be realized as a
Weyl function of some ordinary boundary triple ($B$-generalized
boundary triple, respectively).

The multivalued analog for the notion of $B$-generalized boundary
triple was introduced in \cite[Section~5.3]{DHMS06}, a formal
definition reads as follows.

\begin{definition}\label{Btriple}
Let $\ZA$ be a symmetric operator (or relation) in the Hilbert space
$\sH$ and let $\cH$ be another Hilbert space. Then a linear relation
$\Gamma:\ZA^*\to \cH\oplus\cH$ with dense domain in $\ZA^*$ is said
to be a \emph{$B$-generalized boundary pair} for $\ZA^*$,\index{Boundary pair!B-generalized} if the
following three conditions are satisfied:
\begin{enumerate}[label=\textbf{\thedefinition.\arabic*}]
\item \label{1_def_Btriple} the abstract Green's identity~\eqref{Green1} holds;
\item \label{2_def_Btriple} $\ran \Gamma_0=\cH$;
\item  \label{3_def_Btriple} $A_0=\ker \Gamma_0$ is selfadjoint,
\end{enumerate}
where $\Gamma_0$ stands for the first component of $\Gamma$; see~\eqref{g01}.
\end{definition}
As was shown in~\cite[Proposition~5.9]{DHMS06} every Weyl function
of a $B$-generalized boundary pair belongs to the class $\cR[\cH]$
and, conversely, every operator valued function $M\in \cR[\cH]$ can
be realized as the Weyl function of  a $B$-generalized boundary
pair.

\subsection{Isometric boundary pairs and isometric boundary triples}\label{sec3.3}


Let $\Gamma$ be a $(J_\sH,J_\cH)$-isometric relation from the
Kre\u{\i}n space $(\sH^2,J_\sH)$ to the Kre\u{\i}n space
$(\cH^2,J_\cH)$. In view of~\eqref{jh}-\eqref{eq:J_isom} this just
means that the abstract Green's identity~\eqref{Green1} holds. It
follows from~\eqref{Green1} that
\[
 \ker \Gamma \subset (\dom \Gamma)^{[\perp]}, \quad \mul \Gamma \subset (\ran \Gamma)^{[\perp]},
\]
compare Proposition~\ref{UNIT}. Let  $\Gamma_0$  and $\Gamma_1$ be
the linear relations determined by~\eqref{g01}. The kernels
$A_0:=\ker \Gamma_0$ and $A_1:=\ker \Gamma_1$ need not be closed,
but they are symmetric extensions of $\ker \Gamma$ which are
contained in the domain $\ZA_*=\dom\Gamma$ of $\Gamma$; cf.
\cite[Proposition~2.13]{DHMS06}. If $\ZA_*=\dom \Gamma$ is dense in
$\ZA^*$ then the pair  $\{\cH,\Gamma\}$ is viewed as an isometric
boundary pair for  $\ZA^*$; cf. Definition~\ref{GBT}. In general
$\ZA:=(\ZA_*)^*=(\dom \Gamma)^{[\perp]}$ is an extension of $\ker
\Gamma$ which need not belong to $\dom \Gamma$; for some sufficient
conditions for the equality $\ZA=\ker \Gamma$, see
\cite[Section~2.3]{DHMS09} and \cite[Section~7.8]{DHMS12}.

With $\{\wh f_\lambda,\wh h\},\{\wh g_\mu,\wh k\}\in\Gamma$,
$\lambda,\mu\in\dC$, the Green's identity~\eqref{Green1} gives, cf.
~\eqref{Green3U},
\begin{equation}\label{Green2}
 (h',k)_\cH - (h,k')_\cH=
 (\lambda-\bar\mu)(f_\lambda,g_\mu)_\sH.
\end{equation}
In particular, with $\mu=\lambda$~\eqref{Green2} implies that $\IM (h',h)_\cH= \IM \lambda\|f_\lambda\|^2$.
Hence, for all $\lambda\in\cmr$
\begin{equation}\label{Green2B}
\ker (\Gamma\upharpoonright\wh\sN_\lambda(\ZA_*))=\{0\}
\quad\mbox{and }\quad
\ker(\Gamma_j\upharpoonright\wh\sN_\lambda(\ZA_*))=\{0\}
\quad(j=0,1).
\end{equation}
Moreover, with $\mu=\bar\lambda$~\eqref{Green2} implies that
\begin{equation}\label{GreenM}
 M(\bar\lambda)\subseteq M(\lambda)^*, \quad \lambda\in\cmr.
\end{equation}
Here equality does not hold if $\Gamma$ is not unitary.
%
However, with the Weyl family the multivalued part of $\Gamma$ can
be described explicitly; see \cite[Lemma~7.57]{DHMS12}, cf. also
\cite[Lemma~4.1]{DHMS06}.

\begin{lemma}\label{Weylmul}
Let $\{\cH,\Gamma\}$ be an isometric boundary pair with the Weyl
family $M$. Then the following equalities hold for all
$\lambda\in\cmr$:
\begin{enumerate}\def\labelenumi{\textit{(\roman{enumi})}}
\item $M(\lambda)\cap M(\lambda)^*=\mul \Gamma$;
\item $\ker M(\lambda)\times\{0\}=\mul \Gamma\cap (\cH\times \{0\})$;
\item  $\{0\}\times\mul M(\lambda)=\mul \Gamma\cap(\{0\}\times\cH)$;
\item $\ker (M(\lambda)-M(\lambda)^*)=\mul \Gamma_0$;
\item $\ker (M(\lambda)^{-1}-M(\lambda)^{-*})=\mul \Gamma_1$.
\end{enumerate}
\end{lemma}

If $\Gamma$ itself is single-valued, then the Weyl family $M$ is an
operator valued function, i.e. $\mul (M(\lambda)=0$, in the class
$\cR^s(\cH)$, see \cite[Proposition~4.5]{DHMS06}. Moreover, $\ker
\IM (M(\lambda))=\{0\}$ and $\ker \IM (M(\lambda)^{-1})=\{0\}$, in
particular, $\ker M(\lambda)=0$. Recall that when $\Gamma$ is
single-valued $M(\lambda)$ can equivalently be defined by the
equality~\eqref{eq:Mgen}. Hence, if $h\in\cH$ is given and $h\in
\Gamma_0(\wh\sN_\lambda(\ZA_*))$, then $\gamma(\lambda)h$ solves a
boundary eigenvalue problem, i.e., $\gamma(\lambda)h\in \ker
(\ZA^*-\lambda)$ and $\Gamma_0\wh\gamma(\lambda)h=h$, while
$\Gamma_1\wh\gamma(\lambda)h=M(\lambda)h$.
Also for an operator valued $M(\cdot)$ the identity~\eqref{Green2}
can be rewritten in the form
\begin{equation}\label{Green3}
 (\lambda-\bar\mu)(\gamma(\lambda)h,\gamma(\mu)k)_\sH
 =(M(\lambda)h,k)_\cH - (h,M(\mu)k)_\cH,
\end{equation}
where $h\in\dom M(\lambda)$ and $k\in \dom M(\mu)$,
$\lambda,\mu\in\cmr$. This is an analog of~\eqref{Green3U} for an
isometric boundary triple.

Let $\Gamma$ be an isometric relation and let $A_0=\ker \Gamma_0$.
Then $A_0$ is a symmetric, not necessarily closed, relation and one
can write for every $\lambda\in\cmr$,
\[
 A_0=\left\{\begin{pmatrix}(A_0-\lambda)^{-1}h\\ h+\lambda(A_0-\lambda)^{-1}h\end{pmatrix}:\,
 h\in \ran(A_0-\lambda)  \right\}.
\]
The linear mapping
\begin{equation}
\label{Hlambda}
 H(\lambda): h\to \left\{\begin{pmatrix}(A_0-\lambda)^{-1}h\\ h+\lambda(A_0-\lambda)^{-1}h\end{pmatrix}\right\}
\end{equation}
from $\ran(A_0-\lambda)$ onto $A_0$ is clearly bounded with bounded
inverse.

\begin{lemma}\label{cor:GH}
Let $\{\cH,\Gamma\}$ be an isometric boundary pair and let $A_0=\ker \Gamma_0$.
Then the following assertions hold:
\begin{enumerate}
\def\labelenumi{\textit{(\roman{enumi})}}

\item $\Gamma_1H(\lambda)$ is closable for one (equivalently for all) $\lambda\in\cmr$ if and only if $\Gamma_1\uphar
A_0$ is closable;

\item $\Gamma_1H(\lambda)$ is closed for one (equivalently for all) $\lambda\in\cmr$ if and only if $\Gamma_1\uphar
A_0$ is closed;

\item $\Gamma_1H(\lambda)$ is bounded operator for one (equivalently for all) $\lambda\in\cmr$ if and only if $\Gamma_1\uphar
A_0$ is a bounded operator;

\item $\dom\Gamma_1H(\lambda)$ is dense in $\sH$ for some (equivalently for all) $\lambda,\bar\lambda\in\cmr$ if and only if $A_0$ is essentially
selfadjoint;

\item $\dom\Gamma_1H(\lambda)=\sH$ for some (equivalently for all) $\lambda,\bar\lambda\in\cmr$ if and only if $A_0$ is selfadjoint.
\end{enumerate}
\end{lemma}

\begin{proof}
By definition $A_0=\ker\Gamma_0 \subset \dom \Gamma_1$, so that
$\dom\Gamma_1H(\lambda)=\ran(A_0-\lambda)$, $\lambda\in\cmr$. Since
$H(\lambda):\ran(A_0-\lambda)\to A_0$ is bounded with bounded
inverse, all the statements are easily obtained by means of the
equality $\Gamma_1\uphar A_0=(\Gamma_1H(\lambda))H(\lambda)^{-1}$.
\end{proof}

Similar facts can be stated for the restriction $\Gamma_0\uphar A_1$,
where $A_1=\ker \Gamma_1$.

The inclusion~\eqref{Iso1} in the next proposition was
stated for a single-valued $\Gamma$ with dense range in
\cite[Proposition~7.59]{DHMS12}; here a direct proof
for this inclusion is given in the general case.

\begin{lemma}\label{isoHlem}
Let $\{\cH,\Gamma\}$ be an isometric boundary pair, let
$\gamma(\lambda)$ be its $\gamma$-field, and let $H(\lambda)$ be as
defined in~\eqref{Hlambda}. Then
\begin{equation}
\label{Iso}
 \Gamma H(\lambda)\subset\binom{0}{\gamma(\bar\lambda)^*} \hplus (\{0\}\times\mul \Gamma), \quad \lambda\in\cmr,
\end{equation}
where the adjoint $\gamma(\bar\lambda)^*$ of $\gamma(\bar\lambda)$ is in general a linear relation.
In particular,
\begin{equation}
\label{Iso1}
\Gamma_1 H(\lambda)\subset
 \gamma(\bar\lambda)^* \hplus (\{0\}\times\mul \Gamma_1), \quad \lambda\in\cmr,
\end{equation}
and if, in addition, $\mul \Gamma_1=\{0\}$, then
\begin{equation}
\label{Iso1B}
 \Gamma_1 H(\lambda)\subset \gamma(\bar\lambda)^*, \quad \lambda\in\cmr.
\end{equation}
Furthermore, the following statements hold:
\begin{enumerate}\def\labelenumi{\textit{(\roman{enumi})}}
\item if $\gamma(\bar\lambda)$ is densely defined for some
$\bar\lambda\in\cmr$,
then $\gamma(\bar\lambda)^*$ is a closed operator and if, in
addition, $\mul \Gamma_1=\{0\}$, then $\Gamma_1 H(\lambda)$ is a
closable operator;
\item if $A_0=\ker\Gamma_0$ is essentially selfadjoint, then $\gamma(\bar\lambda)$ is closable for all $\lambda\in\cmr$;
\item  if $A_0=\ker\Gamma_0$ is selfadjoint, then $\dom \gamma(\bar\lambda)^*=\sH$ and $\gamma(\bar\lambda)$
is a bounded operator for all $\lambda\in\cmr$.
\end{enumerate}
\end{lemma}
\begin{proof}
Let $h\in \dom \gamma(\bar\lambda)=\dom M(\bar\lambda)$ and
$k_\lambda\in \ran(A_0-\lambda)$. Then $\{\wh
\gamma(\bar\lambda)h,\{h,h'\}\}\in \Gamma$ and, since
$H(\lambda)k_\lambda\in A_0=\ker \Gamma_0$, one has
$\{H(\lambda)k_\lambda,\{0,k''\}\}\in\Gamma$ for some $k''\in\cH$.
On the other hand, $\{k_\lambda,k'\}\in\Gamma_1 H(\lambda)$ means
that $\{k_\lambda,\{k,k'\}\}\in\Gamma H(\lambda)$ for some $k\in\cH$
which combined with $\{H(\lambda)k_\lambda,\{0,k''\}\}\in\Gamma$
implies that $\{\{0,0\},\{k,k'-k''\}\}\in\Gamma$.

Now applying Green's identity~\eqref{Green1} shows that
\[
 (\bar\lambda\gamma(\bar\lambda)h,(A_0-\lambda)^{-1}k_\lambda)
 -(\gamma(\bar\lambda)h,(I+\lambda(A_0-\lambda)^{-1})k_\lambda)
 =0-(h,k'')_\cH.
\]
This identity can be rewritten equivalently in the form
\[
 (\gamma(\bar\lambda)h,k_\lambda)=(h,k'')_\cH
\]
for all $h\in\dom \gamma(\bar\lambda)$ and $k_\lambda\in
\ran(A_0-\lambda)$. This proves that $\{k_\lambda,k''\}\in
\gamma(\bar\lambda)^*$. Hence, if $\left\{k_\lambda,\begin{pmatrix}k\\k'\end{pmatrix}\right\}\in\Gamma
H(\lambda)$ then
\begin{equation}
\label{G1H}
 \left\{k_\lambda,\begin{pmatrix}k\\k'\end{pmatrix}\right\}=
 \left\{k_\lambda,\begin{pmatrix}0\\k''\end{pmatrix}\right\}+
 \left\{0,\begin{pmatrix}k\\k'-k''\end{pmatrix}\right\},\,\,
  \{k_\lambda,k''\}\in \gamma(\bar\lambda)^*,\,\,
  \begin{pmatrix}k\\k'-k''\end{pmatrix}\in\mul \Gamma,
\end{equation}
from which the formulas~\eqref{Iso} and~\eqref{Iso1} follow.
If $\mul \Gamma_1=\{0\}$, then $\begin{pmatrix}k\\k'-k''\end{pmatrix}\in\mul \Gamma$ implies that $k'=k''$
and therefore the above argument shows that $\{k_\lambda,k'\}\in
\gamma(\bar\lambda)^*$ for all $\{k_\lambda,k'\}\in\Gamma_1
H(\lambda)$; i.e.~\eqref{Iso1B} is satisfied.

It remains to prove the statements (i)--(iii).

(i) If $\gamma(\bar\lambda)$ is densely defined then clearly
$\gamma(\bar\lambda)^*$ is a closed operator and if $\Gamma_1$ is
single-valued then~\eqref{Iso1B} shows that $\Gamma_1 H(\lambda)$ is
closable as a restriction of $\gamma(\bar\lambda)^*$.

(ii) By Lemma~\ref{cor:GH} $A_0$ is essentially selfadjoint if and only
if $\Gamma_1 H(\lambda)$ is densely defined, in which case also
$\gamma(\bar\lambda)^*$ is densely defined, so that
$\gamma(\bar\lambda)$ is closable.

(iii) If $A_0$ is selfadjoint, then $\dom \Gamma_1 H(\lambda)=\sH$
and, therefore, also $\dom\gamma(\bar\lambda)^*=\sH$. In addition
$\gamma(\bar\lambda)$ is closable, thus $\clos\gamma(\bar\lambda)$
and $\gamma(\bar\lambda)$ are bounded operators.
\end{proof}

\begin{proposition}\label{A0graph}
Let $\ZA$ be a closed symmetric relation in the Hilbert space $\sH$
and let $\{\cH,\Gamma\}$ be an isometric boundary pair whose domain
$\ZA_*$ is dense in $\ZA^*$, let $M(\cdot)$ and $\gamma(\cdot)$ be
the corresponding Weyl function and the $\gamma$-field and, in
addition, assume that $A_0=\ker \Gamma_0$ is selfadjoint. Then:
\begin{enumerate}\def\labelenumi{\textit{(\roman{enumi})}}
\item $\ZA_*:=\dom \Gamma$ admits the decomposition $\ZA_*=A_0 \wh{+} \wh\sN_\lambda(\ZA_*)$
and $\wh\sN_\lambda(\ZA_*)$ is dense in $\wh\sN_\lambda(\ZA^*)$ for all $\lambda\in\cmr$;

\item with a fixed $\lambda\in\cmr$ the graph of $\Gamma$ admits the following representation:
\[
 \Gamma=  \Gamma A_0 \hplus \left\{\left\{ \binom{\gamma(\lambda)h}{\lambda\gamma(\lambda)h},\binom{h}{h'} \right\}:
  \, \binom{h}{h'}\in M(\lambda)\right\};
\]

\item if $\wt\Gamma:(\sH^2,J_\sH)\to(\cH^2,J_\cH)$ is an isometric extension of $\Gamma$
with the Weyl function $\wt M$ and the $\gamma$-field $\wt
\gamma(\cdot)$ such that $\wt \ZA_*:=\dom \wt\Gamma\subset \ZA^*$, then
with a fixed $\lambda\in\cmr$ the following equivalence holds:
\[
 \wt \Gamma=\Gamma \quad\Leftrightarrow\quad \wt M(\lambda)=M(\lambda).
\]
\end{enumerate}
\end{proposition}
\begin{proof}
(i) By von Neumann's formula $\ZA^*=A_0\hplus \wh\sN_\lambda(\ZA^*)$. Since $\ZA_*:=\dom\Gamma$ is dense in
$\ZA^*$ and $A_0\subset \ZA_*$, it follows that $\ZA_*=A_0\hplus \wh\sN_\lambda(\ZA_*)$ and that $\wh\sN_\lambda(\ZA_*)$
is dense in $\wh\sN_\lambda(\ZA^*)$ for every $\lambda\in\cmr$.

(ii) In view of (i) for every $\{\wh f,\wh k\}\in\Gamma$
there exist unique elements $\wh f_0\in A_0$ and $\wh f_\lambda\in\wh\sN_\lambda(\ZA_*)$,
$\lambda\in\cmr$, such that $\wh f=\wh f_0+\wh f_\lambda$.
Moreover, if $\{\wh f_\lambda,\wh h\}\in \Gamma$ then $\wh h=\{h,h'\}\in M(\lambda)$
and one can write (uniquely) $\wh f_\lambda=\wh\gamma(\lambda)h$; see~\eqref{ggam1}.
The stated representation for $\Gamma$ is now clear.

(iii) It follows from $\Gamma\subset\wt\Gamma$ that $A_0\subset \ker\wt\Gamma_0$. Since
$\ker\wt\Gamma_0$ is symmetric and $A_0$ is selfadjoint, the equality $A_0=\ker\wt\Gamma_0$ holds.
Now recall that two linear relations with $H_1\subset H_2$ are equal precisely when the equalities
$\dom H_1=\dom H_2$ and $\mul H_1=\mul H_2$ hold; see \cite{Arens}. By Lemma~\ref{Weylmul} (i)
$\mul \Gamma=M(\lambda)\cap M(\lambda)^*$. Therefore, $\wt M(\lambda)=M(\lambda)$ implies
that $\mul\wt\Gamma=\mul\Gamma$.
Moreover, we have $\dom\wt M(\lambda)=\dom M(\lambda)$
and, since $\widehat{\wt\gamma}(\lambda)$ maps
$\dom \wt M(\lambda)$ onto $\wh\sN_\lambda(\wt \ZA_*)$ and $\wh \gamma(\lambda)$ maps
$\dom M(\lambda)$ onto $\wh\sN_\lambda(\ZA_*)$, we conclude from (i) that $\dom \wt \Gamma=\dom\Gamma$.
Therefore, $\wt M(\lambda)=M(\lambda)$ implies $\wt\Gamma=\Gamma$.
The reverse implication is clear.
\end{proof}

The Weyl function of an isometric or unitary boundary pair is in
general unbounded and multivalued operator. In what follows Weyl
functions $M(\lambda)$, whose domain (or form domain) does not
dependent on $\lambda\in\cmr$ are of special interest. Here a
characterization for domain invariant Weyl families will be
established. We start with the next lemma concerning the domain
inclusion $\dom M(\lambda)\subset \dom M(\mu)$.

\begin{lemma}\label{domM}
Let $\{\cH,\Gamma\}$ be an isometric boundary pair with $\ZA_*=\dom
\Gamma$, let $M$ and $\gamma(\cdot)$ be the corresponding Weyl
family and $\gamma$-field, and let $A_0=\ker\Gamma_0$. Then for each
fixed $\lambda,\mu\in\cmr$ with $\lambda\neq \mu$ the inclusion
\begin{equation}\label{domM1}
 \dom M(\mu)\subset \dom M(\lambda),
\end{equation}
is equivalent to the inclusion
\begin{equation}\label{domM2}
 \ran \gamma(\mu)\subset \ran(A_0-\lambda).
\end{equation}
If one of these conditions is satisfied, then the $\gamma$-field
$\gamma(\cdot)$ satisfies the identity
\begin{equation}
\label{ggam3}
 \gamma(\lambda)h=[I+(\lambda-\mu)(A_0-\lambda)^{-1}]\gamma(\mu)h, \quad h\in\dom \gamma(\mu).
\end{equation}
\end{lemma}

\begin{proof}
By Definitions~\ref{Weylfam} $\dom M(\lambda)=\dom
\gamma(\lambda)=\Gamma_0(\wh{\sN}_\lambda(\ZA_*))$ and, moreover,
$\ran \gamma(\lambda)=\sN_\lambda(\ZA_*)$, $\lambda\in\cmr$.

Now assume that~\eqref{domM1} holds and let $h\in\dom
M(\mu)\subset\dom M(\lambda)$. It follows from~\eqref{ghatfield} that
there exist $h',h''\in\cH$ such that
\[
\left\{\binom{\gamma(\lambda)h}{\lambda\gamma(\lambda)h},\binom{h}{h'}\right\}
     \in\Gamma\uphar \wh\sN_\lambda(T)\subset \Gamma, \quad
\left\{\binom{\gamma(\mu)h}{\mu\gamma(\mu)h},\binom{h}{h''}\right\}
     \in\Gamma\uphar \wh\sN_\mu(T)\subset \Gamma.
\]
This implies
\[
 \left\{\binom{(\gamma(\lambda)-\gamma(\mu))h}{(\lambda\gamma(\lambda)-\mu\gamma(\mu))h},
 \binom{0}{h'-h''}\right\} \in\Gamma,
\]
and hence
\begin{equation}
\label{ii1}
 \binom{(\gamma(\lambda)-\gamma(\mu))h}{(\lambda\gamma(\lambda)-\mu\gamma(\mu))h}\in A_0
 \,\text{ and }\,
 \binom{(\gamma(\lambda)-\gamma(\mu))h}{(\lambda-\mu)\gamma(\mu)h} \in A_0-\lambda.
\end{equation}
Therefore, $\gamma(\mu)h \in \ran (A_0-\lambda)$ for every $h \in\dom
M(\mu)$ and thus~\eqref{domM2} follows.

Conversely, assume that~\eqref{domM2} holds and let $h \in \dom
M(\mu)=\dom \gamma(\mu)$. This implies that
\begin{equation}
\label{kaks}
 \left\{\binom{\gamma(\mu)h}{\mu \gamma(\mu)h}, \binom{h}{h'}\right\} \in \Gamma
\end{equation}
for some $h'\in \cH$. Moreover, since $\gamma(\mu)h \in \ran
(A_0-\lambda)$, there exists an element $k \in \sH$ such that $\{k,
\gamma(\mu)h+\lambda k\} \in A_0=\ker\Gamma_0$. Consequently, there exists
$\varphi \in \cH$ such that
\begin{equation}
\label{uks}
 \left\{\binom{(\lambda-\mu) k}{(\lambda-\mu)\gamma(\mu)h+\lambda (\lambda-\mu) k},
 \binom{0}{\varphi} \right\} \in \Gamma.
\end{equation}
It follows from~\eqref{kaks} and~\eqref{uks} that
\[
 \left\{\binom{\gamma(\mu)h+(\lambda-\mu) k}{\lambda(\gamma(\mu)h+(\lambda-\mu) k)},
\binom{h}{h'+\varphi}\right\} \in \Gamma.
\]
Therefore, $h \in \Gamma_0(\wh \sN_\lambda(\ZA_*))=\dom
M(\lambda)$. This proves the inclusion~\eqref{domM1}.

Finally, observe that the assumption~\eqref{domM1} implies~\eqref{ii1}.
Since $A_0$ is symmetric, $(A_0-\lambda)^{-1}$ is a bounded operator
on $\ran (A_0-\lambda)$ and, thus,~\eqref{ii1} leads to~\eqref{ggam3}.
\end{proof}

The next result characterizes domain invariance of the Weyl family
corresponding to an arbitrary isometric boundary pair
$\{\cH,\Gamma\}$. In the special case of a unitary boundary
pair $\{\cH,\Gamma\}$ 
items (i) and
(iii) contain \cite[Proposition~4.11, Corollary~4.12]{DHMS06}.

\begin{proposition}\label{domMchar}
Let the assumptions and notations be as in Lemma~\ref{domM}. Then
the following statements hold:
\begin{enumerate}
\def\labelenumi{\textit{(\roman{enumi})}}
\item $\dom M(\lambda)$ is independent from
      $\lambda \in \dC_+$ (resp. from $\lambda \in\dC_-$) if and only if
\[
 \sN_\mu(\ZA_*) \subset \ran (A_0-\lambda) \quad \text{for all } \lambda, \mu \in \dC_+
      \quad (\text{resp. for all } \lambda, \mu \in \dC_-), \quad \lambda \ne
      \mu,
\]
in this case the $\gamma$-field $\gamma(\cdot)$ satisfies
\[
 \gamma(\lambda)=[I+(\lambda-\mu)(A_0-\lambda)^{-1}]\gamma(\mu),
 \quad
 \lambda,\mu\in\dC_+ (\dC_-);
\]

\item if $A_0$ is selfadjoint, then $\dom M(\lambda)$ does not dependent on $\lambda \in
\cmr$;

\item if $\dom M(\lambda)$ does not dependent on $\lambda \in \cmr$,
then $A_0$ is essentially selfadjoint.
\end{enumerate}
\end{proposition}
\begin{proof}
The assertions (i) and (ii) follow directly from Lemma~\ref{domM}.

To see (iii) one can use the same argument that is presented in
\cite[Corollary~4.12]{DHMS06}.
\end{proof}

\subsection{Some transforms of boundary triples}\label{sec3.4}

In this subsection a specific transform of isometric boundary
triples is treated. In what follows such transforms are used
repeatedly and, in fact, they appear also in concrete boundary value
problems in ODE and PDE settings. To formulate a general result in
the abstract setting consider in the Kre\u{\i}n space
$(\cH^2,J_\cH)$ the transformation operator $V$ whose action is
determined by the triangular operator
\begin{equation}\label{Vtrans}
 V=\begin{pmatrix} G^{-1} & 0 \\ EG^{-1} & G^*\end{pmatrix}, \quad E\subset E^*,\,\cdom E=\cdom G=\cran G=\cH,\,  \ker G=\{0\}.
\end{equation}
By assumptions on $G$ one has $\ker G^*=\mul G^*=\{0\}$, so that the adjoint $G^*$ is an injective operator in $\cH$.
To keep a wider generality, $G$ is not assumed to be a closed operator, while in applications
that will often be the case. In particular, it is possible that $G^*$ is not densely defined and also its range need not be dense.
Since $E$ is a densely defined symmetric operator, it is closable and its closure $\overline{E}\subset E^*$
is also symmetric.
With the assumptions on $V$ in~\eqref{Vtrans} a direct calculation shows that
\[
 (J_\cH Vf,Vg)_{\cH^{2}}=(J_\cH f,g)_{\cH^{2}}, \quad f,g\in \dom V.
\]
Hence, $V$ is an isometric operator in the Kre\u{\i}n space $(\cH^2,J_\cH)$. Moreover, $V$ is injective.
These observations lead to the following (unbounded) extension of \cite[Proposition~3.18]{DHMS09}.

\begin{lemma}\label{isomtrans}
Let $\{\cH,\Gamma_0,\Gamma_1\}$ be an isometric boundary triple for
$\ZA^*$ such that $\ker \Gamma=A$, let $\gamma(\lambda)$ and
$M(\lambda)$ be the corresponding $\gamma$-field and the Weyl
function, and let $V$ be as defined in~\eqref{Vtrans}. Then $V$ is
isometric in the Kre\u{\i}n space $(\cH^2,J_\cH)$ and moreover:
\begin{enumerate}
\def\labelenumi{\textit{(\roman{enumi})}}

\item the transform $\wt \Gamma=V\circ\Gamma$
\begin{equation}\label{isom-transf}
 \begin{pmatrix} \wt\Gamma_0\wh f \\ \wt\Gamma_1\wh f \end{pmatrix}
 =\begin{pmatrix} G^{-1} \Gamma_0\wh f \\ EG^{-1}\Gamma_0\wh f+G^*\Gamma_1\wh f \end{pmatrix},\quad
 \wh f\in\dom \Gamma,
\end{equation}
defines an isometric boundary triple with domain $\wt \ZA_*:=\dom
\wt \Gamma$ and kernel $\ker\wt\Gamma=A$;

\item the $\gamma$-field and the Weyl function of $\wt\Gamma$ are in general unbounded nondensely defined operators
given by
\[
 \wt\gamma(\lambda)k=\gamma(\lambda)Gk,
 \quad \wt M(\lambda)k=E k+G^*M(\lambda)Gk, \quad k\in \dom \wt M(\lambda),
  \quad \lambda\in\cmr.
\]
\end{enumerate}
\end{lemma}

\begin{proof}
(i) By the assumptions in~\eqref{Vtrans}  $V$ is an isometric
operator in the Kre\u{\i}n space $(\cH^2,J_\cH)$ and since $\Gamma$
is an isometric operator from $(\sH^2,J_\sH)$ to $(\cH^2,J_\cH)$ the
composition operator $V\circ\Gamma$ is also an isometric operator
from $(\sH^2,J_\sH)$ to $(\cH^2,J_\cH)$. Since $V$ is injective, one
has $\ker \wt\Gamma=\ker \Gamma=A$. In general $V$ is not everywhere
defined, so that $\wt \ZA_*$ is typically a proper linear subset of
$\ZA_*=\dom \Gamma$ which is not necessarily dense in $\ZA^*$.

(ii) By Definition \ref{Weylfam} the Weyl function $\wt M(\lambda)$ of $\wt\Gamma$ is given by $\wt M(\lambda)=V\circ M(\lambda)$
or, more explicitly, by
\[
\begin{split}
 \wt M(\lambda)
 &=\left\{\,\begin{pmatrix} G^{-1}h\\EG^{-1}h+G^*M(\lambda)h \end{pmatrix}
  :\, h\in \dom EG^{-1}\cap \dom G^*M(\lambda) \,\right\} \\
 &=\left\{\, \begin{pmatrix} k\\Ek+G^*M(\lambda)Gk \end{pmatrix}
 :\, \begin{array}{c}
        h=Gk\in \dom G^*M(\lambda), \\
       k\in\dom G\cap \dom E
     \end{array}
\,\right\} =E+G^*M(\lambda)G.
\end{split}
\]
Similarly,  $(G^{-1}\Gamma_0\uphar\wh\sN_\lambda({\wt \ZA_*}))^{-1}=(\Gamma_0\uphar\wh\sN_\lambda({\wt \ZA_*}))^{-1}G$
implies that $\wt\gamma(\lambda)=\gamma(\lambda)G$ with $\dom \wt\gamma(\lambda)=\dom \wt M(\lambda)$.
\end{proof}

\begin{example}
(i) If $G=I_\cH$ then the condition $\wt \Gamma_1\wh f=0$ reads as $\Gamma_1\wh f+E\Gamma_0\wh f=0$.
In applications such conditions are called Robin type boundary conditions. This corresponds to the transposed boundary triple
$\{\cH,\Gamma_1+E\Gamma_0,-\Gamma_0\}$ which is also isometric and has $-(M(\lambda)+E)^{-1}$ as its Weyl function.

(ii) As indicated $G$ need not be closable.
An extreme situation appears when $G$ is a \textit{singular operator}; cf.~\cite{Kosh99}.
By definition this means that $\dom G\subset \ker \overline{G}$ or, equivalently, that
$\ran G\subset \mul \overline{G}$. Thus, in this case $\dom G^*=\ran G^*=\{0\}$.
If, for instance, $\Gamma$ is an ordinary boundary triple for $\ZA^*$ then
$A_0=\ker \Gamma_0$ and $A_1=\ker \Gamma_1$ are selfadjoint.
It is easy to check that
\[
  \wt \ZA_*=\bigl\{\,\wh f\in \ZA^*:\, \Gamma_1\wh f =0\,\bigr\}
  =\ker\Gamma_1=A_1, \quad \ker \wt \Gamma_0=A_0\cap A_1=A.
\]
Moreover, $\ran \wt\Gamma=E\uphar\dom G$ is a symmetric operator in
$\cH$ and $\dom \wt M(\lambda)=\dom \wt\gamma(\lambda)$ is trivial.
\end{example}

\subsection{Some additional remarks}\label{sec3.5}
Despite of the fact that the paper~\cite{Cal39} has been quoted by
M.G.~Kre\u{\i}n~\cite{Kr47} and a discussion on \cite{Cal39} appears in
the monograph~\cite{DanSch71} the actual results of Calkin on
reduction operators remained widely unknown among experts in
extension theory. Apparently this was caused by the fact that the
paper~\cite{Cal39} was ahead of time -- it was using the new
language of binary linear relations with hidden ideas on geometry of
indefinite inner product spaces, concepts which were not well
developed at that time. The concept of a bounded reduction operator
investigated therein (see \cite[Chapter~IV]{Cal39}) essentially
covers the notion of an ordinary boundary triple in
Definition~\ref{een} as well as the notion of $D$-boundary triple
introduced in \cite{Mog06} for symmetric operators with unequal
defect numbers. An overview on the early work of Calkin and more
detailed description on its connections to boundary triples and
unitary boundary pairs (boundary relations) can be found from the
monograph \cite{HdSSz2012}. In fact, \cite{HdSSz2012} contains a
collection of articles reflecting various recent activities in
different fields of applications with related realization results
for Weyl functions, including analysis of differential operators,
continuous time state/signal systems and boundary control theory
with interconnection analysis of port-Hamiltonian systems involving
Dirac and Tellegen structures etc.

\section{$AB$-generalized boundary pairs and $AB$-generalized boundary triples}\label{sec4}

In this section  present a new generalization of the class of
$B$-generalized boundary triples from \cite{DM95} (see
Definition~\ref{genBT}).
  \begin{definition}\label{ABGtriple}
Let $\ZA$ be a symmetric operator (or relation) in the Hilbert space
$\sH$ and let $\cH$ be another Hilbert space. Then a linear relation
$\Gamma: \ZA^*\to \cH\oplus\cH$ with domain dense in $\ZA^*$ is said
to be an \emph{almost $B$-generalized boundary pair}, in short,
\emph{$AB$-generalized boundary pair} \index{Boundary pair!AB-generalized}  for $\ZA^*$, if the following
three conditions are satisfied:
  \begin{enumerate}
[label=\textbf{\thedefinition.\arabic*}]
\item \label{1_def_ABGtriple} the abstract Green's identity~\eqref{Green1} holds;
\item \label{2_def_ABGtriple} $\ran \Gamma_0$ is dense in $\cH$;
\item \label{3_def_ABGtriple} $A_0=\ker \Gamma_0$ is selfadjoint,
\end{enumerate}
where $\Gamma_0 = \pi_0\Gamma$ stands for the first component mapping of $\Gamma$; see
~\eqref{g01}.

A \emph{single-valued $AB$-generalized boundary pair} is also said
to be an \emph{almost $B$-generalized boundary triple}, shortly, an
\emph{$AB$-generalized boundary triple} for $\ZA^*$.
\end{definition}


If $\Gamma$ is an $AB$-generalized boundary pair for $\ZA^*$, then
the same is true for its closure. Indeed, since $\overline \Gamma$
is an extension of $\Gamma$, it is clear that $\dom
\overline{\Gamma}$ is dense in $\ZA^*$ and $\ran
(\overline{\Gamma})_0$ is dense in $\cH$. By Assumption
\ref{1_def_ABGtriple} $\Gamma$ is isometric (in the Kre\u{\i}n space
sense), i.e. $\Gamma^{-1}\subset \Gamma^{[*]}$. Thus, clearly
$\overline{\Gamma}^{\,-1}\subset
\Gamma^{[*]}=\overline{\Gamma}^{[*]}$. Hence, the closure satisfies
the Green's identity~\eqref{Green1} and this implies that the
corresponding kernels $\ker (\overline{\Gamma})_0\supset \ker
\Gamma_0=A_0$ and $\ker (\overline{\Gamma})_1\supset\ker
\Gamma_1=A_1$ are symmetric. Therefore, $\ker
(\overline{\Gamma})_0=A_0$ must be selfadjoint.

\subsection{Characteristic properties of $AB$-generalized boundary pairs and triples}

The next theorem describes some central properties of an
$AB$-generalized boundary pair.

\begin{theorem}\label{ABGthm}
Let $\ZA$ be a closed symmetric relation in  $\sH$, let $\{\cH,\Gamma\}$ be
an $AB$-generalized boundary pair for $\ZA^*$, and let $\Gamma_0$
and $\Gamma_1$ be the corresponding component mappings from $\dom
\Gamma$ into $\cH$. Moreover, let $\gamma(\cdot)$ and $M(\cdot)$ be
the corresponding  $\gamma$-field and the Weyl function.  Then:
\begin{enumerate}
\def\labelenumi{\textit{(\roman{enumi})}}

\item $\ker \Gamma=\ZA$;

\item $\ZA_*:=\dom \Gamma$ admits the decomposition $\ZA_*=A_0 \wh{+} \wh\sN_\lambda(\ZA_*)$
and $\wh\sN_\lambda(\ZA_*)$ is dense in $\wh\sN_\lambda(\ZA^*)$ for all $\lambda\in\cmr$;

\item the $\gamma$-field $\gamma(\lambda)$ is a densely defined bounded operator from $\ran\Gamma_0$ onto
$\wh\sN_\lambda(\ZA_*)$. It is domain invariant and one has
$$
\dom\gamma(\lambda)=\ran\Gamma_0,\quad  \ker\gamma(\lambda) = \mul\Gamma_0;
$$
\item
the adjoint $\gamma(\lambda)^*$ is a bounded everywhere defined operator and, moreover,
equalities hold in~\eqref{Iso},~\eqref{Iso1},
   \begin{equation}
\label{ABGG1}
\Gamma H(\lambda)=\binom{0}{\gamma(\bar\lambda)^*} \hplus (\{0\}\times\mul \Gamma),\quad
 \Gamma_1 H(\lambda)=
 \gamma(\bar\lambda)^* \hplus (\{0\}\times\mul \Gamma_1), \quad \lambda\in\cmr;
\end{equation}

\item the closure of the $\gamma$-field $\overline{\gamma(\lambda)}$ is a bounded operator from $\cH$ into $\wh\sN_\lambda(\ZA^*)$
satisfying  the identity
\[
 \overline{\gamma(\lambda)}=[I+(\lambda-\mu)(A_0-\lambda)^{-1}]\overline{\gamma(\mu)},
 \quad \lambda, \mu \in \cmr;
\]

\item the Weyl function $M$ is a densely defined operator valued function which is domain invariant,
$\dom M(\lambda)=\ran \Gamma_0$, $M(\lambda)\subset M(\bar\lambda)^*$, and the imaginary part
$\IM M(\lambda)=(M(\lambda)-M(\lambda)^*)/2i$ is bounded with $\dom \IM M(\lambda)=\ran \Gamma_0$ and
$\ker \IM M(\lambda)=\mul\Gamma_0$. Furthermore, $M(\lambda)$ admits the
following representation
\begin{equation}\label{WeylABG}
 M(\lambda)=E+M_0(\lambda), \quad \lambda\in\cmr,
\end{equation}
where $E=\RE M(\mu)$ is a symmetric densely defined operator in $\cH$ and $M_0(\cdot)$ is a bounded
Nevanlinna function (defined on $\dom E$), i.e., $\overline{ M_0(\cdot)}\in \cR[\cH]$.
\end{enumerate}
\end{theorem}
\begin{proof}
(i) It is clear from the Green's identity that $\ker\Gamma\subset
(\dom \Gamma)^*=(\ZA_*)^*=\ZA$; cf. \cite[Lemma~7.3]{DHMS12}. To
prove the reverse inclusion, the property that $\gamma(\lambda)$,
$\lambda\in\cmr$, is densely defined will be used (and this is
independently proved in (iii) below). Assumption
\ref{3_def_ABGtriple} implies that $\ZA=(\ZA_*)^*\subset
A_0^*=A_0=\ker \Gamma_0\subset\dom \Gamma$. On the other hand, if
$k_\lambda\in \ran(A-\lambda)$ then by Lemma~\ref{isoHlem}
$\{k_\lambda,k''\}\in \gamma(\bar\lambda)^*$ for some $k''$ and thus
for all $h\in\dom\gamma(\bar\lambda)$ one has
\[
  (k'',h)_\cH=(k_\lambda,\gamma(\bar\lambda)h)=0.
\]
Assumption \ref{2_def_ABGtriple} combined with $\dom
\gamma(\bar\lambda)=\ran \Gamma_0$ (see proof of (iii) below) shows
that $\gamma(\bar\lambda)$ is densely defined and, hence,
$\gamma(\bar\lambda)^*$ is an operator and
$k''=\gamma(\bar\lambda)^*k_\lambda=0$. Now apply the formula
~\eqref{G1H} in the proof of Lemma~\ref{isoHlem} to $k_\lambda\in\ran
(A-\lambda)$: therein $\{k_\lambda,\{k,k'\}\}\in\Gamma H(\lambda)$
and $k''=0$ so that~\eqref{G1H} reads as
\[
  \left\{k_\lambda,\begin{pmatrix}k\\k'\end{pmatrix}\right\}
  =\left\{k_\lambda,\begin{pmatrix}0\\0\end{pmatrix}\right\}
  +\left\{0,\begin{pmatrix}k\\k'\end{pmatrix}\right\},\quad \begin{pmatrix}k\\k'\end{pmatrix}\in\mul \Gamma.
\]
Hence, $H(\lambda)k_\lambda\in\ker\Gamma$ and $\ZA=H(\lambda)(\ran (\ZA-\lambda))\subset \ker \Gamma$.
Therefore, $\ker\Gamma=\ZA$.

(ii) This holds by Proposition~\ref{A0graph} (i).

(iii) \& (iv) The decomposition of $\ZA_*$ in (ii) combined with
$A_0=\ker \Gamma_0$ implies that
\[
 \Gamma_0(\ZA_*)=\Gamma_0(\wh\sN_\lambda(\ZA_*))=\dom M(\lambda)=\dom \gamma(\lambda), \quad \lambda\in\cmr.
\]
Hence, $\dom M(\lambda)=\dom \gamma(\lambda)=\ran \Gamma_0$ does not
depend on $\lambda\in\cmr$. Now Assumption \ref{2_def_ABGtriple}
shows that $\gamma(\lambda)$ and $M(\lambda)$ are densely defined
for all $\lambda\in\cmr$. Moreover, according to Lemma
\ref{isoHlem}~(iii) $\gamma(\lambda)$ is a bounded operator and the
equality $\dom \gamma(\lambda)^*=\sH$ holds for all
$\lambda\in\cmr$. Since $\gamma(\lambda)$ is densely defined in
$\cH$, the adjoint $\gamma(\lambda)^*$ is a bounded everywhere
defined operator from $\sH$ into $\Gamma_1(A_0)$. Since
$M(\bar\lambda)\subset M(\lambda)^*$, see~\eqref{GreenM}, the
adjoint $M(\lambda)^*$ and the closure of $M(\bar\lambda)$ are also
densely defined operators. In view of~\eqref{Green3} one has
\begin{equation}\label{gWeyl01}
 (\lambda-\bar\mu)(\gamma(\lambda)h,\gamma(\mu)k)_\sH = ((M(\lambda)-M(\mu)^*)h,k)_\cH,
 \quad \lambda,\mu\in\cmr,
\end{equation}
for all $h,k\in \dom \gamma(\lambda)=\ran\Gamma_0$.
In particular, $2i\IM \lambda\|\gamma(\lambda)h\|_\sH^2=(
(M(\lambda)-M(\lambda)^*)h,h )_\cH$ holds for all $h\in
\dom\gamma(\lambda)= \dom M(\lambda)$, $\lambda\in\cmr$. By
Lemma~\ref{Weylmul}~\eqref{gWeyl01} implies that
\[
 \ker\gamma(\lambda)=\ker(M(\lambda)-M(\lambda)^*)=\mul \Gamma_0, \quad \lambda\in\cmr.
\]
It remains to prove~\eqref{ABGG1}. Observe, that
$\dom\Gamma_1H(\lambda)=\dom \gamma(\bar\lambda)^*=\sH$ and clearly
the multivalued parts on both sides of the inclusion in~\eqref{Iso},~\eqref{Iso1}
are equal. Hence, the inclusions~\eqref{Iso},~\eqref{Iso1} must prevail
actually as equalities (by the criterion from \cite{Arens}).

(v) Since $\dom M(\lambda)=\dom\gamma(\lambda)=\ran\Gamma_0$ does
not depend on $\lambda\in\cmr$, the equality
\begin{equation}\label{gWeyl02}
 (I+(\lambda-\mu)(A_0-\lambda)^{-1})\gamma(\mu)h=\gamma(\lambda)h
\end{equation}
holds for all $\lambda,\mu\in\cmr$ and
$h\in\ran\Gamma_0$ by Proposition~\ref{domMchar}.
According to (iii) $\gamma(\lambda)$ is bounded and densely defined, so that its
closure $\overline{\gamma(\lambda)}$ is bounded and defined
everywhere on $\cH$. The formula in (iv) is obtained by taking
closures in~\eqref{gWeyl02}.

(vi) It suffices to prove the representation~\eqref{WeylABG} for
$M(\lambda)$, since all the other assertions were already shown
above when proving (iii) \& (iv). It follows from~\eqref{gWeyl01}
and~\eqref{gWeyl02} that
\[
\begin{split}
 (M(\lambda)h,k)
 & =(M(\mu)^*h,k)+(\lambda-\bar\mu)
  ((I+(\lambda-\mu)(A_0-\lambda)^{-1})\gamma(\mu)h,\gamma(\mu)k) \\
 & =(\RE M(\mu)h,k)+(((\lambda-\RE \mu)+
  (\lambda-\mu)(\lambda-\bar\mu)(A_0-\lambda)^{-1})\gamma(\mu)h,\gamma(\mu)k),
\end{split}
\]
$h,k\in\dom \gamma(\lambda)=\dom M(\mu)=\ran \Gamma_0$,
$\lambda,\mu\in \cmr$. Here $2\RE M(\mu)=M(\mu)+M(\mu)^*$ and hence
$2(\RE M(\mu))^*\supset M(\mu)^*+\overline{M(\mu)}\supset 2\RE M(\mu)$, so
that $E:=\RE M(\mu)$ is a symmetric operator with $\dom E=\dom
M(\mu)=\ran \Gamma_0$. On the other hand, since
$\overline{\gamma(\lambda)}$ and its adjoint $\gamma(\lambda)^*$ are
bounded everywhere defined operators, it follows that the closure of
\begin{equation}
\label{DefM0}
 M_0(\lambda):
 =\gamma(\mu)^*((\lambda-\RE \mu)+(\lambda-\mu)(\lambda-\bar\mu)(A_0-\lambda)^{-1}){\gamma(\mu)}
\end{equation}
is a bounded holomorphic operator valued Nevanlinna function acting on
$\cH$, such that $M(\lambda)=E+M_0(\lambda)$. This completes the
proof.
\end{proof}

For an $AB$-generalized boundary pair it is possible to describe the
graph of $\Gamma$, $(\ran\Gamma)^{[\perp]}$, and the closure of
$\ran \Gamma$ explicitly.

\begin{corollary}\label{ABGgraph}
Let $\Gamma$ be an $AB$-generalized boundary pair for $\ZA^*$ and
let $\gamma(\cdot)$ and $M(\cdot)=E+M_0(\cdot)$ be the corresponding
$\gamma$-field and Weyl function as in Theorem~\ref{ABGthm} with
$E=\RE M(\mu)$ for some fixed $\mu\in\cmr$. Then:
\begin{enumerate}\def\labelenumi {\textit{(\roman{enumi})}}
\item with a fixed $\lambda\in\cmr$ the graph of $\Gamma$ admits the following representation:
\[
 \Gamma=  \left\{\left\{ H(\lambda)k_\lambda, \binom{0}{\gamma(\bar{\lambda})^*k_\lambda} \right\}
  + \left\{ \binom{\gamma(\lambda)h}{\lambda\gamma(\lambda)h},\binom{h}{M(\lambda)h} \right\};
  \, \begin{array}{c} k_\lambda\in\ran(A_0-\lambda) \\ h\in \dom M(\lambda)\end{array} \right\};
\]
\item the range of $\Gamma$ satisfies
\[
 (\ran\Gamma)^{[\perp]}=E^*\upharpoonright\ker\overline{\gamma(\lambda)}
 \quad \text{and}\quad \cran\Gamma=(E^*\upharpoonright\ker\overline{\gamma(\lambda)})^*,
\]
and here $\ker\overline{\gamma(\lambda)}=\ker(\overline{ M_0(\lambda)})$ does not depend on $\lambda\in\cmr$.
In particular, $\ran\Gamma$ is dense in $\cH$ if and only if $\dom E^*\cap\ker\overline{\gamma(\lambda)}=\{0\}$
for some or, equivalently, for every $\lambda\in\cmr$.

\item  $\Gamma$ is a single-valued mapping if and only if $\mul \Gamma_0=\{0\}$ or, equivalently,
if and only if $\ker \IM M(\lambda)\,(=\ker \gamma(\lambda))=0$,
$\lambda\in\cmr$.
\end{enumerate}
\end{corollary}
\begin{proof}
(i) Using the representation of $\Gamma H(\lambda)$ in~\eqref{ABGG1}, the inclusion
$\mul\Gamma\subset M(\lambda)$ in Lemma \ref{Weylmul}, and the fact that by Theorem~\ref{ABGthm}
$M(\lambda)$ is an operator, one concludes that the representation of $\Gamma$ given in
Proposition~\ref{A0graph}~(ii) can be rewritten in the form as stated in (i).

(ii) The description in (i) shows that
\begin{equation}
\label{rangeG}
 \ran\Gamma=\Gamma(A_0)\hplus M(\lambda) =\binom{0}{\ran\gamma(\bar\lambda)^*}\hplus M(\lambda),
\end{equation}
for all $\lambda\in\cmr$.
Therefore, $(\ran\Gamma)^{[\perp]}=(\{0\}\times \ran\gamma(\bar\lambda)^*)^{[\perp]}\cap M(\lambda)^*$.
Hence  $\wh k=\{k,k'\}\in (\ran\Gamma)^{[\perp]}$ if and only if $\wh k\in M(\lambda)^*$
and $k\in(\ran\gamma(\bar\lambda)^*)^\perp=\ker\overline{\gamma(\bar\lambda)}$.
Since $E=\RE M(\mu)$, one has $\RE M_0(\mu)=0$ and hence $\ker\overline{\gamma(\mu)}=\ker\overline{(\IM M_0(\mu))}=\ker(\overline{ M_0(\mu)})$
and this kernel does not depend on $\mu\in\cmr$ due to $\overline{M_0(\cdot)}\in \cR[\cH]$; cf. Theorem~\ref{ABGthm} (v).
This proves that
\[
 (\ran\Gamma)^{[\perp]}=M(\lambda)^*\uphar \ker \overline{\gamma(\lambda)}
 =(E^*+M_0(\lambda)^*)\uphar \ker \overline{\gamma(\lambda)}=E^*\uphar \ker \overline{\gamma(\lambda)}.
\]
As to the closure observe that
\[
 \cran\Gamma=((\ran\Gamma)^{[\perp]})^{[\perp]}=((\ran\Gamma)^*)^*=(E^*\uphar \ker \overline{\gamma(\lambda)})^*.
\]
Thus, $\cran\Gamma=\cH\times\cH$ if and only if $E^*\uphar \ker \overline{\gamma(\lambda)}=\{0,0\}$
or, equivalently, $\dom E^*\cap\ker\overline{\gamma(\lambda)}=\{0\}$, since $E^*$ together with $E\,(\subset E^*)$
is a densely defined operator in $\cH$.

(iii)
In view of (i) this follows from $\mul\Gamma_0=\ker \IM
M(\lambda)=\ker \gamma(\lambda)$; see Lemma~\ref{Weylmul}.
\end{proof}

Corollary \ref{ABGgraph} shows that for an $AB$-generalized boundary
pair the inclusion $\mul\Gamma\subset (\ran \Gamma)^{[\perp]}$ is in
general strict. In particular, the range of $\Gamma$ for a single-valued
$AB$-generalized boundary pair, i.e., an $AB$-generalized boundary
triple, need not be dense in $\cH\times\cH$. Notice that an
$AB$-generalized boundary pair with the surjectivity condition $\ran
\Gamma_0=\cH$ is called a $B$-generalized boundary pair for $\ZA^*$;
see Definition~\ref{Btriple}. The next result gives a connection
between $AB$-generalized boundary pairs and $B$-generalized boundary
pairs.

\begin{theorem}\label{QBTthm}
Let $\{\cH,\Gamma\}$ be a $B$-generalized boundary pair for $\ZA^*$,
and let $M(\cdot)$ and $\gamma(\cdot)$ be the corresponding Weyl
function and $\gamma$-field. Let also  $E$ be a symmetric densely
defined operator in $\cH$ and let $\Gamma=\{\Gamma_0,\Gamma_1\}$
where $\Gamma_i=\pi_i\Gamma$, $i=0,1$, be the corresponding
components of $\Gamma$ as in~\eqref{g01}. Then the transform
\begin{equation}\label{Qeq:1}
   \begin{pmatrix}
   \wt\Gamma_0 \\ \wt\Gamma_1
   \end{pmatrix}=
   \begin{pmatrix}
   I & 0 \\
   E & I
   \end{pmatrix}\begin{pmatrix}
   \Gamma_0 \\ \Gamma_1
   \end{pmatrix}
\end{equation}
defines an $AB$-generalized boundary pair for $\ZA^*$. The
corresponding Weyl function $\wt M(\cdot)$ and
$\wt\gamma(\cdot)$-field are connected by
\[
 \wt M(\lambda)=E+M(\lambda), \quad \wt\gamma(\lambda)=\gamma(\lambda)\uphar\dom E,\quad \lambda\in\cmr.
\]
Furthermore, $\wt\Gamma:=\{\wt\Gamma_0,\wt\Gamma_1\}$ in~\eqref{Qeq:1} is closed if and
only if $E$ is a closed symmetric operator in $\cH$, in particular, the closure of
$\wt\Gamma$ is given by~\eqref{Qeq:1} with $E$ replaced by its closure $\overline{E}$.

Conversely, if $\{\cH,\wt\Gamma\}$ is an $AB$-generalized boundary
pair for $\ZA^*$ then there exists a $B$-generalized boundary pair
$\{\cH,\Gamma\}$  for $\ZA^*$ and a densely defined symmetric
operator $E$ in $\cH$ such that $\wt\Gamma$ is given by
~\eqref{Qeq:1}.
\end{theorem}

\begin{proof}
($\Rightarrow$) By Lemma~\ref{isomtrans} the block
triangular transformation $V$ in~\eqref{Qeq:1} acting on $\cH\times\cH$ is an isometric
operator. Consequently, $\wt \Gamma=V\circ\Gamma$ is isometric. It is clear from
~\eqref{Qeq:1} 
that $A_0:=\ker \Gamma_0\subset \ker\wt\Gamma_0$, which by symmetry
of $\ker\wt\Gamma_0$ implies that $\ker\wt\Gamma_0=A_0$. Clearly
$\ran \wt\Gamma_0$ is dense in $\cH$, since $\ran\Gamma_0=\cH$ and
$E$ is densely defined. Thus $\wt\Gamma$ admits all the properties
in Definition~\ref{ABGtriple}. Since in addition $\ker
\wt\Gamma=\ker \Gamma$, it follows from Theorem \ref{ABGthm}~(i)
that $\wt \ZA_*=\dom \wt\Gamma$ is dense in $\ZA^*$. Therefore,
$\{\cH,\wt\Gamma\}$ is an $AB$-generalized boundary pair for
$\ZA^*$. The connections between the Weyl functions and
$\gamma$-fields are clear from the definitions; cf.
Lemma~\ref{isomtrans}.

To treat the closedness properties of $\wt\Gamma$ consider the
representation of $\wt\Gamma$ in Corollary~\ref{ABGgraph}. Let
$\lambda\in\cmr$ be fixed and assume that the sequence $\{\wh
f_n,\wh k_n\}\in\wt\Gamma$ converges to $\{\wh f,\wh k\}$. Then $\wh
f_n=H(\lambda)k_{n,\lambda}+\wh \gamma(\lambda)h_n$ with unique
$k_{n,\lambda}\in\ran(A_0-\lambda)$ and $h_n\in\dom \wt
M(\lambda)=\dom E$ and, since the angle between the graphs of $A_0$
and $\wh\sN_\lambda(\wt \ZA_*)$ is positive, it follows that
$k_{n,\lambda}\to k_{\lambda}\in\ran(A_0-\lambda)$. Moreover, the
representation of $\{\wh f_n,\wh k_n\}\in \wt\Gamma$ in
Corollary~\ref{ABGgraph} shows that $h_n\to h\in\cH$. According to
Theorem~\ref{ABGthm} $\gamma(\lambda)$ and $\gamma(\bar\lambda)^*$
are bounded operators and, since  $\wt M(\lambda)=E+M(\lambda)$,
where $M(\lambda)$ is bounded (see \cite[Proposition~3.16]{DHMS09}),
it follows from Corollary~\ref{ABGgraph} that 
\[
 \left\{ H(\lambda)k_{n,\lambda}, \binom{0}{\gamma(\bar{\lambda})^*k_{n,\lambda}} \right\}
  + \left\{ \binom{\gamma(\lambda)h_n}{\lambda\gamma(\lambda)h_n},\binom{h_n}{Eh_n+M(\lambda)h_n} \right\}\in \wt\Gamma
\]
converges to
\begin{equation}\label{elem1}
 \left\{ H(\lambda)k_\lambda, \binom{0}{\gamma(\bar{\lambda})^*k_\lambda} \right\}
  + \left\{ \binom{\gamma(\lambda)h}{\lambda\gamma(\lambda)h},\binom{h}{h''+M(\lambda)h} \right\}\in \clos\wt\Gamma,
\end{equation}
where $\{h,h''\}\in\overline{E}$. It is also clear that the limit
element in~\eqref{elem1} belongs to $\wt\Gamma$ if and only if
$\lim_{n\to\infty}\{h_n,Eh_n\}=\{h,h''\}\in E$. Therefore,
$\wt\Gamma$ is closed if and only if $E$ is closed and, moreover,
the closure of $\wt\Gamma$, which is also an $AB$-generalized
boundary pair for $\ZA^*$ (as stated after
Definition~\ref{ABGtriple}), is given by~\eqref{Qeq:1} with $E$
replaced by its closure $\overline{E}$.

($\Leftarrow$) Let $\{\cH,\wt\Gamma\}$ be an $AB$-generalized
boundary pair. According to Theorem~\ref{ABGthm} the corresponding
Weyl function $\wt M$ is of the form $\wt M=E+M$, where $\overline{
M}$ is a bounded Nevanlinna function and $E\,(=\RE \wt M(\mu))$ is a
symmetric densely defined operator in $\cH$.

To construct $\wt\Gamma$ directly from an associated $B$-generalized
boundary pair, define
\begin{equation}\label{Qeq:6}
   \begin{pmatrix}
   \wh\Gamma_0 \\ \wh\Gamma_1
   \end{pmatrix}:=
   \begin{pmatrix}
   I & 0 \\
   -E & I
   \end{pmatrix}\begin{pmatrix}
   \wt\Gamma_0 \\ \wt\Gamma_1
   \end{pmatrix}.
\end{equation}
Since $\wt M(\lambda)=\wt\Gamma(\wh\sN_\lambda(\wt \ZA_*))\subset\ran\wt\Gamma$,
where $\wt \ZA_*=\dom\wt\Gamma$, and $\dom \wt M(\lambda)=\dom E$, it
follows that the graph of $\wt M(\lambda)$ belongs to the domain of
the block operator $$\begin{pmatrix}
   I & 0 \\
   -E & I
\end{pmatrix},$$
i.e., $\wh\sN_\lambda(\wt \ZA_*)\subset\dom\wh\Gamma$ for all
$\lambda\in\cmr$. Moreover,
\[
\wh\Gamma(\wh\sN_\lambda(\wt \ZA_*))=-E+\wt
M(\lambda)=M(\lambda)\upharpoonright \dom E\subset\ran\wh\Gamma\]
for all $\lambda\in\cmr$. Since $\overline{M}\in \cR[\cH]$ this
implies that $\ran\wh\Gamma_0$ is dense in $\cH$. Clearly,
$\ker\wh\Gamma_0=\ker\wt\Gamma_0=A_0$ and since $\wt
\ZA_*=A_0\wh{+}\wh\sN_\lambda(\wt \ZA_*)$ one concludes that $\wt
\ZA_*=\dom\wh\Gamma=\dom\wt\Gamma$ is dense in $\ZA^*$. Thus,
$\wh\Gamma$ is also an $AB$-generalized boundary pair for $\ZA^*$
and, consequently, also its closure is an $AB$-generalized boundary
triple for $\ZA^*$, too. Denote the closure of $\wh\Gamma$ by
$\Gamma^{(0)}$. Then the corresponding Weyl function
$M^{(0)}(\cdot)$ is an extension of $M$ and its closure is equal to
$\overline{M}$. Since $\Gamma^{(0)}$ is closed, it must be unitary
by \cite[Theorem~7.51]{DHMS12} (cf. \cite[Proposition~3.6]{DHMS06}).
In particular, $M^{(0)}(\cdot)$ is also closed, i.e.,
$M^{(0)}(\cdot)=\overline{M}\in \cR[\cH]$. Thus,
$\ran\Gamma_0^{(0)}=\dom M^{(0)}(\cdot)=\cH$ and hence
$\Gamma^{(0)}$ is a $B$-generalized boundary pair for $\ZA^*$; see
Definition~\ref{Btriple}. Finally, in view of~\eqref{Qeq:6} one has
\[
   \begin{pmatrix}
   \wt\Gamma_0 \\ \wt\Gamma_1
   \end{pmatrix}=
   \begin{pmatrix}
   I & 0 \\
   E & I
   \end{pmatrix}\begin{pmatrix}
   \wh\Gamma_0 \\ \wh\Gamma_1
   \end{pmatrix}
   \subset    \begin{pmatrix}
   I & 0 \\
   E & I
   \end{pmatrix}
   \Gamma^{(0)}=:\wt \Gamma^{(0)}.
\]
Here equality $\wt\Gamma=\wt\Gamma^{(0)}$ holds by Proposition
\ref{A0graph} (iii), since $\wt
M^{(0)}(\cdot)=E+\overline{M}(\cdot)=\wt M(\cdot)$.
\end{proof}

The proof of Theorems \ref{ABGthm} contains also the following result.

\begin{corollary}\label{ABGcor1}
If $\{\cH,\wt\Gamma\}$ 
is an $AB$-generalized boundary pair for $\ZA^*$ with the Weyl
function $\wt M(\cdot)$ and $E=\RE \wt M(\mu)$ for some $\mu\in
\rho(M)$, then the closure of $\Gamma=\begin{pmatrix}
   I & 0 \\
   -E & I
   \end{pmatrix}\wt\Gamma$
defines a $B$-generalized boundary pair for $\ZA^*$ with the bounded
Weyl function $M(\cdot)=\clos(\wt M(\cdot)-E)$.
\end{corollary}

Theorems \ref{ABGthm} and~\ref{QBTthm} imply the following
characterization for the Weyl functions corresponding to
$AB$-generalized boundary pairs.
  \begin{corollary}\label{ABGcor2}
The class of $AB$-generalized boundary pairs coincides with the
class of isometric boundary pairs whose Weyl function is of the form
\begin{equation}\label{WeylABG2}
 M(\lambda)=E+M_0(\lambda),\quad \lambda\in\cmr,
\end{equation}
with $E$ a symmetric densely defined operator in $\cH$ and
$M_0(\cdot)\in \cR[\cH]$. In particular, every function $M$ of the
form~\eqref{WeylABG2} is a Weyl function of some $AB$-generalized
boundary pair.
  \end{corollary}
\begin{proof}
By Theorem \ref{ABGthm} the Weyl function $M$ of an $AB$-generalized
boundary pair $\{\cH,\wt\Gamma\}$ is of the form~\eqref{WeylABG2},
where $E\subset E^*$ is densely defined and $M_0(\cdot)\in
\cR[\cH]$.

Conversely, if $M$ is given by~\eqref{WeylABG2} with $M_0(\cdot)\in
\cR[\cH]$, then by \cite[Proposition~3.16]{DHMS09} $M_0(\cdot)$ is
the Weyl function of a $B$-generalized boundary pair
$\{\cH,\Gamma\}$ for $\ZA^*$. Now according to the first part of
Theorem \ref{QBTthm} the transform $\wt\Gamma$ of $\Gamma$ defined
in~\eqref{Qeq:1} is an $AB$-generalized boundary pair for $\ZA^*$
such that the corresponding Weyl function is equal to
~\eqref{WeylABG2}.
\end{proof}

By Definition~\ref{Btriple} every $B$-generalized boundary pair is
also an $AB$-generalized boundary pair. Hence, the notions of
$AB$-generalized boundary pairs and $AB$-generalized boundary
triples generalize the earlier notions of ($B$-)generalized boundary
triples as introduced in \cite{DM95} and boundary triples of bounded
type as defined in \cite[Section~5.3]{DHMS06}. It is emphasized that
$B$-generalized boundary pairs are not only isometric: they are also
unitary in the Kre\u{\i}n space sense, see Definition \ref{GBT}. The
characteristic properties of the classes of $B$-generalized boundary
triples and pairs can be found in
\cite[Propositions~5.7,~5.9]{DHMS06} and
\cite[Proposition~3.16]{DHMS09}. In particular, recall that the
class of functions $M\in \cR[\cH]$ coincides with the class of Weyl
functions of $B$-generalized boundary pairs and the class of
functions $M\in \cR^s[\cH]$ coincides with the class of
$B$-generalized boundary triples. Some further characterizations
connected with $AB$-generalized boundary pairs are given in the next
corollary.

\begin{corollary}\label{QBRcor}
Let $\{\cH,\wt\Gamma\}$ be an $AB$-generalized boundary pair for
$\ZA^*$ as in Theorem~\ref{QBTthm} and let $E$ be a symmetric
densely defined operator in $\cH$ as in~\eqref{Qeq:1}. Then:
\begin{enumerate}\def\labelenumi {\textit{(\roman{enumi})}}
\item $\{\cH,\wt\Gamma\}$ is a unitary boundary pair (boundary relation) for $\ZA^*$ if and only if
the operator $E$ is selfadjoint;

\item $\{\cH,\wt\Gamma\}$ has an extension to a unitary boundary pair for $\ZA^*$
if and only if the operator $E$ has equal defect numbers and in this
case the formula~\eqref{Qeq:1} defines a unitary extension of
$\wt\Gamma$ when $E$ is replaced by some selfadjoint extension $E_0$ of $E$;

\item $\{\cH,\wt\Gamma\}$ is a $B$-generalized boundary pair for $\ZA^*$ if and only if
the operator $E$ is bounded and everywhere defined (hence selfadjoint);

\item $\{\cH,\wt\Gamma\}$ is an ordinary boundary triple for $\ZA^*$ if and only if $\ran\Gamma=\cH\oplus\cH$, or equivalently,
if and only if $\ran \Gamma$ is closed, $E$ is bounded, and $\ker\IM M(\lambda)=0$, $\lambda\in\cmr$.
\end{enumerate}
\end{corollary}

\begin{proof}
(i) By Theorem~\ref{QBTthm} $\wt\Gamma$ is closed if and only if $E$
is closed. Moreover, $E=E^*$ if and only if $M$ is a Nevanlinna
function. Now the statement follows from
\cite[Proposition~3.6]{DHMS06} (or \cite[Theorem~7.51]{DHMS12}).

(ii) This is clear from part (i) and Theorem~\ref{QBTthm}.

(iii) This follows from Theorem \ref{ABGthm}~(v) by the equalities
$\ran \wt\Gamma_0= \dom \wt M=\dom E\,(=\cH)$.

(iv) The first equivalence is contained in \cite[Proposition~5.3]{DHMS06}.
To prove the second criterion, we apply Corollary~\ref{ABGgraph}, in particular,
the representation of $\ran \Gamma$ in~\eqref{rangeG}:
\begin{equation}\label{rangeG2}
 \ran\Gamma=\Gamma(A_0)\hplus M(\lambda) =(\{0\}\times \ran\gamma(\bar\lambda)^*)\hplus M(\lambda).
\end{equation}
Clearly, $E$ is bounded precisely when $M(\lambda)$,
$\lambda\in\cmr$, is bounded. In this case the angle between the
last two subspaces in~\eqref{rangeG2} is positive and then $\ran
\Gamma$ is closed if and only if $\ran\gamma(\bar\lambda)^*$ and
$M(\lambda)$ both are closed. By Theorem~\ref{ABGthm}
$\gamma(\lambda)$ is bounded and $\dom \gamma(\lambda)=\dom M(\lambda)=\cH$,
when $M(\lambda)$ is closed and bounded. Then $\gamma(\lambda)$ is
closed and $(\ran\gamma(\bar\lambda)^*)^\perp=\ker
\gamma(\lambda)=\ker \IM M(\lambda)$. Therefore, the conditions
$\ran\Gamma$ is closed, $E$ is bounded, and $\ker\IM M(\lambda)=0$
imply that $\ran\Gamma$ is also dense in $\cH\times\cH$ and, thus,
$\Gamma$ is surjective. The converse is clear.
\end{proof}

The notions of $AB$-generalized boundary pairs and $AB$-generalized
boundary triples generalize also the class of so-called
\textit{quasi boundary triples}, which has been studied in J.
Behrndt and M. Langer \cite{BeLa07}. In the definition of a quasi
boundary triple it is assumed that $\Gamma=\begin{pmatrix}
   \Gamma_0 \\ \Gamma_1
   \end{pmatrix}$ is single-valued and Assumption \ref{2_def_ABGtriple}
in Definition \ref{ABGtriple} is replaced by the stronger condition
that the joint range of $\Gamma=\begin{pmatrix}
   \Gamma_0 \\ \Gamma_1
   \end{pmatrix}$ is dense in $\cH\times\cH$.

Corollary \ref{ABGgraph} gives the following characterization for
quasi boundary triples.

\begin{corollary}\label{quasicor}
An $AB$-generalized boundary triple $\{\cH,\Gamma_0,\Gamma_1\}$ for $\ZA^*$ with
the Weyl function $M=E+M_0(\cdot)$ represented in the form
~\eqref{WeylABG} is a quasi boundary triple (with single-valued
$\Gamma$) 
for $\ZA^*$ if and only if $\ran\Gamma$ is dense in $\cH\oplus\cH$,
or equivalently,
\begin{equation}\label{quasichar}
 \dom E^*\cap \ker \overline{\IM M(\lambda)} \,\left(=\dom E^*\cap \ker \IM \overline{
 M_0(\lambda)}\right)
 =\{0\},
\end{equation}
for some or, equivalently, for every $\lambda\in\cmr$.
\end{corollary}

\begin{remark}\label{quasiremark2}
A connection between
$B$-generalized boundary triples and quasi boundary triples can be
found from \cite[Theorem~7.57]{DHMS12} and
\cite[Propositions~5.1,~5.3]{Wietsma13}. In fact, each of them is
special case of Theorem~\ref{QBTthm}. Moreover, it should be noted
that in the formulation of the converse part in
\cite[Theorem~7.57]{DHMS12} one should use a $B$-generalized
boundary pair $\{\cH,\Gamma\}$, instead of a $B$-generalized
boundary triple $\{\cH,\Gamma_0,\Gamma_1\}$, since $\ker
\gamma(\lambda)=\ker \IM M(\lambda)=0$ ($M$ is strict) does not
imply in general that $\ker\overline{\gamma(\lambda)}=\ker
\overline{\IM M(\lambda)}=\ker \IM \overline{M_0(\lambda)}=0$, i.e.
$\overline{M_0}\in \cR[\cH]$ as in the proof of Theorem~\ref{QBTthm}
above: only the factor mapping $\Gamma/\mul\Gamma$ (see \cite{AI86},
\cite[eq.~(2.15)]{HaWi12}) becomes single-valued (equivalently the
corresponding Weyl function is strict, cf.
\cite[Proposition~4.7]{DHMS06}). It should be also noted that a
condition which is equivalent to~\eqref{quasichar} appears in
\cite[Section 5.1]{Wietsma13}; see also \cite{WietsmaThesis2012}.
For some further related facts, see Corollary~\ref{Scor} and
Remark~\ref{rem:Error_Scrit} in Section~\ref{sec5}.
\end{remark}

The next result describes a connection between $B$-generalized
boundary pairs and ordinary boundary triples. In the special case of
$B$-generalized boundary triples the corresponding result was
established in \cite[Theorem~7.24]{DHMS12}.

\begin{theorem} \label{thm:2.1}
Let $\{\cH,\Gamma\}$ be a $B$-generalized boundary pair for $\ZA^*$
and let $M(\cdot)$ be the corresponding Weyl function. Then there
exists an ordinary boundary triple
$\{\cH_{s},\Gamma_0^{0},\Gamma_1^{0}\}$ with $\cH_s=\cran\IM
M(\lambda)$, $\lambda\in\cmr$, and operators $E=E^*\in\cB(\cH)$  and
$G\in\cB(\cH,\cH_s)$ with $\ker G=\cH\ominus\cH_s$ such that
\begin{equation}\label{eq:3}
    \begin{pmatrix}
    \Gamma_0 \\ \Gamma_1
    \end{pmatrix}=
    \begin{pmatrix}
    G^{-1} & 0 \\
    E G^{-1}      & G^*
    \end{pmatrix}\begin{pmatrix}
    \Gamma_0^{0} \\ \Gamma_1^{0}
    \end{pmatrix},
\end{equation}
where $G^{-1}$ stands for the inverse of $G$ as a linear relation. If $M_0(\cdot)$ is
the Weyl function corresponding to the ordinary boundary triple
$\{\cH_s,\Gamma_0^{0},\Gamma_1^{0}\}$, then
  \begin{equation}\label{eq:MM0}
M(\lambda)=G^*M_0(\lambda)G+E,\quad\lambda\in\rho(A_0).
  \end{equation}
\end{theorem}
  \begin{proof}
The proof is based on \cite[Theorem~7.24]{DHMS12} and
\cite[Propositions~3.18,~4.1]{DHMS09}.

Let $E=\RE M(i)$. Then by \cite[Propositions~3.18]{DHMS09} (cf. Lemma~\ref{isomtrans})
the transform
\begin{equation}\label{trans01}
    \wt\Gamma=\left\{\left\{\wh f,\begin{pmatrix} h \\ -Eh + h' \end{pmatrix}\right\}:\, \{\wh f,\wh h\}\in \Gamma\,\right\}
\end{equation}
defines a new $B$-generalized boundary pair for $\ZA^*$ with the
Weyl function $M(\cdot)-E$ and the original $\gamma$-field
$\gamma(\cdot)$ of $\{\cH,\Gamma\}$.

Let $P_s$ be the orthogonal projection onto $\cH_s:=\cran\IM
M(\lambda)$. Then according to \cite[Proposition~4.1]{DHMS09} the
transform $\Gamma^{(s)}:\sH^2\to (\cH_s)^2$ given by
\begin{equation}\label{trans02}
    \Gamma^{(s)}=\left\{\left\{\wh f,\begin{pmatrix} k \\ P_s k' \end{pmatrix}\right\}:\, \{\wh f,\wh k\}\in \wt\Gamma,\, (I-P_s)k=0\, \right\}
\end{equation}
determines a $B$-generalized boundary pair $\{\cH_s,\Gamma^{(s)}\}$
for $(A^{(s)})*$, where $A^{(s)}$ is defined by
\begin{equation}\label{trans03}
 A^{(s)}:=\ker \Gamma^{(s)}.
\end{equation}
The corresponding Weyl function and $\gamma$-field are given by
\[
 M^{(s)}(\lambda)=P_s(M(\lambda)-E)\uphar \cH_s, \quad \gamma^{(s)}(\lambda)=\gamma(\lambda)\uphar\cH_s.
\]
Recall that $\ker (M(\lambda)-E)=\ker \IM M(\lambda)$ does not
depend on $\lambda\in\cmr$. Consequently,
$M(\lambda)-E=M^{(s)}(\lambda)\oplus 0_{\cH\ominus\cH_s}$. Since
$\ker \gamma(\bar\lambda)=\ker (M(\bar\lambda)-E)=\ker P_s$ one has
$\ran \gamma(\bar\lambda)^*\subset \cH_s$ and it follows from
Corollary~\ref{ABGgraph} that $\ran \wt\Gamma_1\subset \cH_s$.
Therefore,~\eqref{trans02} implies that $A^{(s)}$ defined in
~\eqref{trans03} coincides with $\ZA$: $\ker \Gamma^{(s)}=\ker
\Gamma=\ZA$. By construction $M^{(s)}(\cdot)\in \cR^s[\cH_s]$ and
hence $\Gamma^{(s)}$ is single-valued; i.e.
$\{\cH_s,\Gamma^{(s)}_0,\Gamma^{(s)}_1\}$ is in fact a
$B$-generalized boundary triple for $\ZA^*$; cf.
\cite[Proposition~4.7]{DHMS06}.

One can now apply \cite[Theorem~7.24]{DHMS12} with $R=\RE
M^{(s)}(i)=0$ and $K=(\IM M^{(s)}(i))^{1/2}$ to conclude that there
exists an ordinary boundary triple $\{\cH_s,\Gamma_0^0,\Gamma_1^0\}$
with the Weyl function $M_0(\cdot)$ such that
$\Gamma^{(s)}_0=K^{-1}\Gamma_0^0$, $\Gamma^{(s)}_1=K\Gamma_1^0$, and
\[
 M^{(s)}(\lambda)=KM_0(\lambda)K, \quad \lambda\in\cmr.
\]
In particular, $M(i)=E+i\,K^2P_s$ and
$M(\lambda)=E+P_sKM_0(\lambda)KP_s$. The statement follows by taking
$G=KP_s$. Indeed, since $\ran \wt\Gamma_1\subset \cH_s$ and $\mul
\wt\Gamma_0=\ker\IM M(\lambda)=\ker P_s$ (see Lemma~\ref{Weylmul})
~\eqref{trans02} shows that $\dom \Gamma^{(s)}=\dom \wt\Gamma$ and
\[
 \wt \Gamma=\Gamma^{(s)}\oplus \left\{\left\{\wh 0,\begin{pmatrix} k \\ 0 \end{pmatrix}\right\}:\, P_sk=0\, \right\}
          = \left\{\left\{\wh k,\begin{pmatrix}  P_s^{-1}\Gamma_0^{(s)}\wh k\\ P_s\Gamma_1^{(s)}\wh k\end{pmatrix} \right\}:\, \wh k\in\dom \Gamma^{(s)} \,\right\}.
\]
Finally, using $G^{-1}=P_s^{-1}K^{-1}=K^{-1}\oplus (\{0\}\times \ker
P_s)$ and~\eqref{trans01} yields the formulas~\eqref{eq:3} and
~\eqref{eq:MM0}.
\end{proof}

The notion of an $AB$-generalized boundary pair introduced in
Definition~\ref{ABGtriple} appears to be useful in characterizing
the class of unbounded Nevanlinna functions (and multivalued
Nevanlinna families) whose imaginary parts generate closable forms
$\tau_{M(\lambda)}=[(M(\lambda)\cdot,\cdot)-(\cdot,M(\lambda)\cdot)]/2i$
via~\eqref{Green3U} and whose closures are domain invariant. All
such functions, after renormalization by a bounded operator
$G\in[\cH]$, turn out to be Weyl functions of $AB$-generalized
boundary triples, i.e., for a suitable choice of $G$, $G^*MG$ is a
function of the form~\eqref{WeylABG}: see Theorem~\ref{essThm2} and
Corollary~\ref{essCor1} in Section~\ref{sec5}.

\subsection{A Kre\u{\i}n type formula for $AB$-generalized boundary triples}
\label{sec4.2}

In this section a Kre\u{\i}n type (resolvent) formula for
$AB$-generalized boundary triples will be established. We refer to
\cite[Proposition~7.27]{DHMS12} where the case of $B$-generalized
boundary triples was treated, and \cite{BeLa07,BeLa2012} for the
case of quasi boundary triples.

If $A_0=\ker\Gamma_0$ is selfadjoint, then it follows from the first von Neumann's formula
that for each $\lambda\in\rho(A_0)$ the domain of $\Gamma$ can be decomposed as follows:
\begin{equation}\label{domGamma2}
 \dom \Gamma=A_0 \hplus (\dom \Gamma\cap \wh \sN_{\lambda}(\ZA^*)).
\end{equation}
Now let $\Gamma$ be a single-valued and let $\Gamma$ be decomposed
as $\Gamma=\{\Gamma_0,\Gamma_1\}$. Let $\wt A$ be an extension of
$\ZA$ which belongs to the domain of $\Gamma$ and let $\Theta$ be a
linear relation in $\cH$ corresponding to $\ZA$:
\begin{equation}\label{ATheta}
 \Theta=\Gamma(\wt A), \quad \wt A\subset \dom \Gamma
 \quad \Leftrightarrow \quad \wt A=A_\Theta:=\Gamma^{-1}(\Theta), \quad \Theta\subset \ran\Gamma.
\end{equation}

\begin{theorem}\label{Kreinformula}
Let $\ZA$ be a closed symmetric relation, let $\Pi =
\{\cH,\Gamma_0,\Gamma_1\}$ be an $AB$-generalized boundary triple
for $\ZA^*$ with $A_0=\ker\Gamma_0$, and let  $M(\cdot)$ and
$\gamma(\cdot)$ be the corresponding  Weyl function and
$\gamma$-field, respectively. Then for any extension $A_\Theta\in
\Ext_A$ satisfying $A_\Theta\subset \dom \Gamma$ the following
Kre\u{\i}n-type formula\index{Kre\u{\i}n-type formula} holds
\begin{equation}\label{resol}
(A_\Theta-\lambda)^{-1}=(A_0-\lambda)^{-1}
+\gamma(\lambda)\bigl(\Theta-M(\lambda)\bigr)^{-1}\gamma(\bar\lambda)^*,\qquad
\lambda\in\rho(A_0).
\end{equation}
Here the inverses in the first and last terms are taken in the sense of linear
relations.
\end{theorem}

\begin{proof}
Let $A_\Theta\subset \dom \Gamma$ be an extension of $\ZA$ with $\Theta$ as in~\eqref{ATheta}
and let $\lambda\in\rho(A_0)$. Assume that $\{g,g'\} \in (A_\Theta-\lambda)^{-1}$ or, equivalently, that
$\wh g_\Theta:=\{g',g+\lambda g'\}\in A_\Theta$. Then
\begin{equation}\label{graph0}
 \wh g_0:= 
 \{(A_0-\lambda)^{-1}g,(I+\lambda(A_0-\lambda)^{-1})g\} \in A_{0} \subset \dom\Gamma,
\end{equation}
and $\wh g_\Theta-\wh g_0\in \dom\Gamma$. Furthermore,
\[
\wh g_\Theta-\wh g_0 
 =\{g'-(A_0-\lambda)^{-1}g,\lambda (g'-(A_0-\lambda)^{-1}g)\},
\]
so that $\wh g_\Theta-\wh g_0\in \wh \sN_{\lambda}(\ZA_*)$. Since
$A_0$ is selfadjoint, $\dom\wh\gamma(\lambda)=\dom
M(\lambda)=\ran\Gamma_0$ and $\wh \gamma(\lambda)$ maps its domain
onto $\wh \sN_{\lambda}(\ZA_*)\subset \dom\Gamma$. In particular,
one can write
\begin{equation}\label{graph}
\wh g_\Theta= \wh g_0 + \wh \gamma(\lambda) \varphi
\end{equation}
for some $\varphi \in \ran\Gamma_0$. By Definition~\ref{Weylfam}
$\Gamma \wh \gamma(\lambda)\varphi=\{\varphi, M(\lambda)\varphi\}$
and it follows from Theorem \ref{ABGthm} that $\Gamma \wh g_{0}=
\{0,\gamma(\bar\lambda)^*g\}$, since $\mul\Gamma=\{0\}$. Now
an application of $\Gamma$ to~\eqref{graph} shows
that
\begin{equation}\label{graphG1}
 \{0,\gamma(\bar\lambda)^*g\}+\{\varphi,M(\lambda)\varphi\}
 =\Gamma \wh g_0 +\Gamma  \wh \gamma(\lambda)\varphi
 =\Gamma \wh g_{\Theta} \in \Theta,
\end{equation}
see~\eqref{ATheta}.
Thus $\{\varphi,\gamma(\bar\lambda)^*g+M(\lambda)\varphi \}\in
\Theta$ and $\{\varphi,\gamma(\bar\lambda)^*g\}\in \Theta-
M(\lambda)$ or, equivalently, $\{g,\varphi\}\in(\Theta-
M(\lambda))^{-1}\gamma(\bar\lambda)^*$. Substituting this into
~\eqref{graph} and taking the first components leads to the inclusion
\[
(A_\Theta-\lambda)^{-1} \subset (A_0-\lambda)^{-1}+
  \gamma(\lambda)\bigl(\Theta-M(\lambda)\bigr)^{-1}\gamma(\bar\lambda)^*.
\]

To prove the reverse inclusion, let $g\in\sH$ such that
$\{\gamma(\bar\lambda)^*g,\varphi\}\in(\Theta-M(\lambda))^{-1}$.
Equivalently, this means that
\begin{equation}\label{graphG2}
 \{g,\gamma(\lambda)\varphi\}\in\gamma(\lambda)(\Theta-M(\lambda))^{-1}\gamma(\bar\lambda)^*,
\end{equation}
since $\dom(\Theta-M(\lambda))\subset \ran\Gamma_0=\dom \gamma(\lambda)$ with $\lambda\in\rho(A_0)$;
see~\eqref{ATheta}. It follows from $\{\varphi,\gamma(\bar\lambda)^*g+M(\lambda)\varphi \}\in
\Theta$ and~\eqref{ATheta} that $\Gamma \wh g_{\Theta}=\{\varphi,\gamma(\bar\lambda)^*g+M(\lambda)\varphi \}$
for some $\wh g_{\Theta}\in A_\Theta$. The decomposition~\eqref{domGamma2} implies that
$\wh g_{\Theta}=\wh v_0+\wh \gamma(\lambda)\psi$, where $\wh v_0\in A_0$ and $\psi\in\dom \gamma(\lambda)$ are
unique since the decomposition in~\eqref{domGamma2} is direct and $\ker \gamma(\lambda)=\{0\}$ due to $\mul\Gamma_0=\{0\}$;
see Theorem~\ref{ABGthm}. Associate with $g$ the element $\wh g_0\in A_0$ as in~\eqref{graph0}.
Then by applying $\Gamma$ to $\wh v_0+\wh \gamma(\lambda)\psi$
we conclude that necessarily (cf.~\eqref{graphG1})
\[
 \psi=\varphi, \quad \wh v_0-\wh g_0\in \ZA.
\]
Since $A\subset A_\Theta$, this implies that for all
$\{\gamma(\bar\lambda)^*g,\varphi\}\in(\Theta-M(\lambda))^{-1}$ one
has $\wh g_0+\wh \gamma(\lambda)\varphi\in A_\Theta$ or,
equivalently, that
\[
 \{g,(A_0-\lambda)^{-1}g+\gamma(\lambda)\varphi\}\in (A_\Theta-\lambda)^{-1}.
\]
In view of~\eqref{graphG2} this proves the reverse inclusion ``$\supset$'' in~\eqref{resol}.
\end{proof}
  \begin{remark}
We emphasize that in the Kre\u{\i}n-type formula~\eqref{resol} it is not
assumed that $\lambda\in\rho(A_\Theta)$. In particular,
$A_\Theta-\lambda$ need not be invertible; $A_\Theta$ and $\Theta$
need not even be closed.
  \end{remark}
The following  statement  is an immediate consequence of Theorem \ref{Kreinformula}.

\begin{corollary}
Let the assumptions be as in Theorem \ref{Kreinformula} and let $\lambda\in\rho(A_0)$.
Then:
\begin{enumerate}
\def\labelenumi{\textit{(\roman{enumi})}}
\item $\ker(A_\Theta-\lambda)=\gamma(\lambda)\ker(\Theta-M(\lambda))$;

\item if $(\Theta-M(\lambda))^{-1}$ is a bounded operator, then the same is true for
$(A_\Theta-\lambda)^{-1}$;

\item if $0\in \rho(\Theta-M(\lambda))$ then $\lambda\in\rho(A_\Theta)$.
\end{enumerate}
\end{corollary}

\section{Some classes of unitary boundary triples and their Weyl functions}
\label{sec5}

\subsection{Unitary boundary pairs and unitary colligations}\label{sec5.1}
Some formulas from Section~\ref{sec3} can be essentially improved
when using the interrelations between unitary relations and unitary
colligations, see~\cite{BHS09}. Let $\{\cH,\Gamma\}$ be a unitary
boundary pair corresponding to a single-valued unitary relation
$\Gamma:(\sH^2,J_\sH)\to(\cH^2,J_\cH)$. The Kre\u{\i}n spaces
$(\sH^2,J_\sH)$ and $(\cH^2,J_\cH)$ admit fundamental decompositions
\[
\sH^2=P_+\sH^2[+]P_-\sH^2,\quad \cH^2=P_+\cH^2[+]P_-\cH^2
\]
with respect to the fundamental symmetries $J_\sH$ and $J_\cH$, where $P_+$ and $P_-$ are the orthogonal projections
\[
P_+=\frac12\left(
      \begin{array}{cc}
        I & -iI \\
        iI & I \\
      \end{array}
    \right),\quad
P_-=\frac12\left(
      \begin{array}{cc}
        I & iI \\
        -iI & I \\
      \end{array}
    \right)
\]
each acting on its own space. The Potapov-Ginzburg transform $\omega'(\Gamma)$
\[
\omega'(\Gamma)=\left\{\left\{\left(
                                \begin{array}{c}
                                  P_-\wh f \\
                                  P_+\wh h \\
                                \end{array}
                              \right),
                              \left(
                                \begin{array}{c}
                                  P_+\wh f \\
                                  P_-\wh h \\
                                \end{array}
                              \right)
\right\}:                \left\{\left(
                   \begin{array}{c}
                     f \\
                     f' \\
                   \end{array}
                 \right),
\left(
                   \begin{array}{c}
                     h \\
                     h' \\
                   \end{array}
                 \right)\right\}\in\Gamma \right\}
\]
of $\Gamma$ (see~\cite{AI86}) is unitarily equivalent to the transform $\omega(\Gamma)$
\begin{equation}\label{eq:Col}
\omega(\Gamma):=\left\{\left\{\left(
                   \begin{array}{c}
                     f+if' \\
                     h-ih' \\
                   \end{array}
                 \right),
\left(
                   \begin{array}{c}
                     -f+if' \\
                     -h-ih' \\
                   \end{array}
                 \right)\right\}:\,
                 \left\{\left(
                   \begin{array}{c}
                     f \\
                     f' \\
                   \end{array}
                 \right),
\left(
                   \begin{array}{c}
                     h \\
                     h' \\
                   \end{array}
                 \right)\right\}\in\Gamma \right\},
\end{equation}
which  is the graph of a unitary operator
$\cU:\left(
                                            \begin{array}{c}
                                              \sH \\
                                              \cH \\
                                            \end{array}
                                          \right)
\to
\left(
                                            \begin{array}{c}
                                              \sH \\
                                              \cH \\
                                            \end{array}
                                          \right)$.
This operator will be also called the Potapov-Ginzburg transform \index{Potapov-Ginzburg transform}of $\Gamma$.
The transform $\omega:\Gamma\mapsto\cU$ establishes a one-to-one correspondence between
the set of unitary boundary pairs and the set of unitary colligations. The inverse transform
$\Gamma=\omega^{-1}(\cU)$ takes the form
\begin{equation}\label{eq:8.1}
    \Gamma=\left\{\left\{\left(
                   \begin{array}{c}
                     g'-g \\
                     i(g'+g) \\
                   \end{array}
                 \right),
\left(
                   \begin{array}{c}
                     u'-u \\
                     -i(u'+u) \\
                   \end{array}
                 \right)\right\}:\,
                 \left\{\left(
                   \begin{array}{c}
                     g \\
                     u \\
                   \end{array}
                 \right),
\left(
                   \begin{array}{c}
                     g' \\
                     u' \\
                   \end{array}
                 \right)\right\}\in\mbox{gr }\cU \right\}.
\end{equation}
Let us consider the unitary operator $\cU$ and the pair of Hilbert spaces $\sH$ and
$\cH$ as a unitary colligation \index{Unitary colligation} written in the block form (see~\cite{Br78})
\begin{equation}\label{eq:Col_U}
\cU=\left(
      \begin{array}{cc}
        T & F \\
        G & H \\
      \end{array}
    \right)\in\cB(\sH\oplus\cH), \quad\mbox{where}\quad T\in\cB(\sH),\quad H\in\cB(\cH).
\end{equation}
Then the representation~\eqref{eq:8.1} for $\Gamma$ takes the form
\begin{equation}\label{eq:8.2}
    \Gamma=
    \left\{\left\{\left(
                   \begin{array}{c}
                     (T-I)g+Fu \\
                     i(T+I)g+iFu \\
                   \end{array}
                 \right),
\left(
                   \begin{array}{c}
                     Gg+(H-I)u \\
                     -iGg-i(H+I)u\\
                   \end{array}
                 \right)\right\}:\,
                 g\in\sH,\,u\in\cH \right\}.
\end{equation}
Since $\cU=(\cU^*)^{-1}$, then
\begin{equation}\label{eq:8.5}
    \cU= \left\{\left\{\left(
                   \begin{array}{c}
                     T^*g'+G^*u' \\
                     F^*g'+H^*u'\\
                   \end{array}
                 \right),
                 \left(
                   \begin{array}{c}
                     g' \\
                     u' \\
                   \end{array}
                 \right)\right\}:\,
                 g'\in\sH,\,u'\in\cH\right\}
    \end{equation}
and hence $\Gamma$ admits a dual representation
\begin{equation}\label{eq:8.2A}
    \Gamma=
    \left\{\left\{\left(
                   \begin{array}{c}
                     (I-T^*)g'-G^*u' \\
                     i(I+T^*)g'+iG^*u' \\
                   \end{array}
                 \right),
\left(
                   \begin{array}{c}
                     -F^*g'+(I-H^*)u' \\
                     -iF^*g'-i(I+H^*)u'\\
                   \end{array}
                 \right)\right\}:\,
                 g'\in\sH,\,u'\in\cH \right\}.
\end{equation}
Let us collect some formulas concerning $\Gamma$ and $\cU$ which are immediate from~\eqref{eq:8.2} and~\eqref{eq:8.2A} (see also~\cite{BHS09}).
\begin{proposition}\label{Wlemma}
Let  $\{\cH,\Gamma\}$ be a unitary boundary pair for $\ZA^*$ with $\Gamma$ given by~\eqref{eq:8.2}, and let ${\ZA_*}=\dom\Gamma$, $A_0=\ker\Gamma_0$.
Then:
\begin{equation}\label{eq:\ZA_*}
    {\ZA_*}=\ran\left(
                   \begin{array}{cc}
                     T-I &F \\
                     i(T+I) & iF \\
                   \end{array}
                 \right)=
                  \ran\left(
                   \begin{array}{cc}
                     (I-T)^* &-G^* \\
                     i(I+T)^* & iG^* \\
                   \end{array}
                 \right),
\end{equation}
\begin{equation}\label{eq:8.8}
    \mul {\ZA_*}=
    (I-T)^{-1}\ran F=(I-T^*)^{-1}\ran G^*;
\end{equation}
\begin{equation}\label{eq:8.8S}
    \mul \ZA=\ker(I-T)=\ker(I-T^*);
\end{equation}

\begin{equation}\label{eq:8.9}
\begin{split}
A_0&=
    \left\{\left(
                   \begin{array}{c}
                     (T-I)g+Fu \\
                     i(T+I)g+iFu \\
                   \end{array}
                 \right):\,
                     Gg+(H-I)u =0,\,
                 g\in\sH,\,u\in\cH \right\}\\
&= \left\{\left(
                   \begin{array}{c}
                     (I-T^*)g'-G^*u' \\
                     i(I+T^*)g'+iG^*u' \\
                   \end{array}
                 \right):\,
                     F^*g'+(H^*-I)u'=0,\,
                 g'\in\sH,\,u'\in\cH \right\}
\end{split}
\end{equation}
\begin{equation}\label{eq:8.11}
    \ran\Gamma_0=\ran(I-H)+\ran G=\ran(I-H^*)+\ran F^*;
\end{equation}
\[
    \mul\Gamma=\left\{\left(
                   \begin{array}{c}
                     (H-I)u \\
                     -i(H+I)u \\
                   \end{array}
                 \right):\,  u\in\ker F\right\}
                 =\left\{\left(
                   \begin{array}{c}
                     (I-H^*)u' \\
                     -i(I+H^*)u' \\
                   \end{array}
                 \right):\,  u'\in\ker G^*\right\},
\]
in particular,
\begin{equation}\label{eq:mul_Gamma}
    \mul\Gamma=\{0\}\Longleftrightarrow\ker F=\{0\}\Longleftrightarrow\ker G^*=\{0\}.
\end{equation}
\end{proposition}
The characteristic function \index{Characteristic function} (or
transfer function\index{Transfer function}) of the unitary
colligation $\cU$ (see~\cite{Br78})
\[
\theta(\zeta)=H+\zeta G(I-\zeta T)^{-1}F\quad (\zeta\in\dD)
\]
is holomorphic in $\dD$ and takes values in the set of contractive operators in $\cH$.
\begin{proposition}\label{prop:C2}
Let  $\{\cH,\Gamma\}$ be a unitary boundary pair for $\ZA^*$ with $\Gamma$ given by~\eqref{eq:8.2},
let $\lambda\in\dC_+$ and let
$\zeta=\frac{\lambda-i}{\lambda+i}$. Then:
\begin{equation}\label{eq:gamma+}
\gamma(\lambda)=\{\{ (\theta(\zeta)-I)u,(1-\zeta)(I-\zeta
T)^{-1}Fu\}:\,u\in\cH\},
\end{equation}
\begin{equation}\label{eq:gamma-}
\gamma(\bar{\lambda})=\{\{ (\theta(\zeta)^*-I)u,(1-\bar{\zeta})(I-\bar{\zeta}
T^*)^{-1}G^*u\}:\,u\in\cH\}.
\end{equation}
In particular,
\begin{equation}\label{eq:gamma_pmi}
\gamma(i)=\{\{ (H-I)u,Fu\}:\,u\in\cH\},\quad
\gamma(-i)=\{\{ (H^*-I)u,G^*u\}:\,u\in\cH\}.
\end{equation}
The Weyl function $M$ corresponding to the boundary pair
$\{\cH,\Gamma\}$ and the characteristic function $\theta$ are
connected by
\begin{equation}\label{eq:M_theta}
    M(\lambda)=i(I+\theta(\zeta))(I-\theta(\zeta))^{-1},\quad
    M(\bar{\lambda})=-i(I+\theta(\zeta)^*)(I-\theta(\zeta)^*)^{-1}
\end{equation}
If $\Gamma$ is single-valued then $\dom\gamma(\lambda)$ and $\dom\gamma(\bar\lambda)$ are dense in $\cH$.
\end{proposition}
\begin{proof}
Since 
$\zeta=\frac{\lambda-i}{\lambda+i}\in\dD$  the operator $(I-\zeta
T)$ has a bounded inverse. Using substitution $g=f+\zeta(I-\zeta
T)^{-1}Fu$ one can rewrite~\eqref{eq:8.2} in the form
\begin{equation}\label{eq:8.2_lambda}
    \Gamma=
    \left\{\left\{\left(
                   \begin{array}{c}
                     (T-I)f+(1-\zeta)(I-\zeta T)^{-1}Fu \\
                     i(T+I)f+i(1+\zeta)(I-\zeta T)^{-1}Fu \\
                   \end{array}
                 \right),
\left(
                   \begin{array}{c}
                     Gf+(\theta(\zeta)-I)u \\
                     -iGf-i(\theta(\zeta)+I)u\\
                   \end{array}
                 \right)\right\}:
                 \begin{array}{c}
                 f\in\sH\\
                 u\in\cH
                 \end{array}\right\}.
\end{equation}
Since $\lambda=i\frac{1+\zeta}{1-\zeta}$ then setting in~\eqref{eq:8.2_lambda} $f=0$ one obtains
\begin{equation}\label{eq:8.12}
\Gamma\uphar\wh\sN_\lambda=
    \left\{\left\{\left(
                   \begin{array}{c}
                     (1-\zeta)(I-\zeta T)^{-1}Fu \\
                     \lambda (1-\zeta)(I-\zeta T)^{-1}Fu \\
                   \end{array}
                 \right),
                 \left(
                   \begin{array}{c}
                   (\theta(\zeta)-I)u\\
                   -i(\theta(\zeta)+I)u
                   \end{array}
                 \right)\right\}:\,u\in\cH
                 \right\}
\end{equation}
and hence~\eqref{eq:gamma+} and the first equalities in
~\eqref{eq:gamma_pmi} and~\eqref{eq:M_theta} follow.

Similarly, substitution $g'=f'+\bar{\zeta}(I-\bar{\zeta}
T^*)^{-1}G^*u'$ in~\eqref{eq:8.2} shows that the linear relation
$\Gamma$ coincides with the set of vectors
\begin{equation}\label{eq:8.2A_lambda}
 \left\{\left(
                   \begin{array}{c}
                     (I-T^*)f'+(\bar{\zeta}-1)(I-\bar{\zeta} T^*)^{-1}G^*u' \\
                     i(I+T^*)f'+i(\bar{\zeta}+1)(I-\bar{\zeta} T^*)^{-1}G^*u' \\
                   \end{array}
                 \right),
\left(
                   \begin{array}{c}
                     -F^*f'+(I-\theta(\zeta)^*)u' \\
                     -iF^*f'-i(I+\theta(\zeta)^*)u'\\
                   \end{array}
                 \right)\right\},
\end{equation}
where $ f'\in\sH$,   $u'\in\cH$. Hence with $f'=0$ one obtains from~\eqref{eq:8.2A_lambda}
\begin{equation}\label{eq:8.12A}
\Gamma\uphar\wh\sN_{\bar\lambda}=
    \left\{  \left\{\left(
                   \begin{array}{c}
                     (\bar{\zeta}-1)(I-\bar{\zeta} T^*)^{-1}G^*u' \\
                     \bar\lambda(\bar{\zeta}-1)(I-\bar{\zeta} T^*)^{-1}G^*u' \\
                   \end{array}
                 \right),
\left(
                   \begin{array}{c}
                     (I-\theta(\zeta)^*)u' \\
                     i(I+\theta(\zeta)^*)u'\\
                   \end{array}
                 \right)\right\}:\,u'\in\cH
                 \right\}               .
\end{equation}
Now the formula~\eqref{eq:gamma-} and the second equalities in
~\eqref{eq:gamma_pmi} and~\eqref{eq:M_theta} are implied
by~\eqref{eq:8.12A}.

If $\mul \Gamma=\{0\}$ then using the fact that $\gamma(\pm i)$ is
single-valued, i.e., $\ker(H-I)\subset \ker F$ and
$\ker(H^*-I)\subset \ker G^*$, it follows from Proposition
\ref{Wlemma} that $\ker(I-H)=\ker (I-H^*)=\{0\}$ and hence
$\dom\gamma(-i)=\ran(I-H^*)$ and $\dom\gamma(i)=\ran(I-H)$ are dense
in $\cH$. Equivalently, $\dom \gamma(\lambda)$ is dense in $\cH$ for
all $\lambda\in\cmr$.
\end{proof}

\begin{proposition}\label{gammaclos}
Let $\{\cH,\Gamma_0,\Gamma_1\}$ be a unitary
boundary triple for $\ZA^*$. Then the closure of the $\gamma$-field
is given by
\begin{equation}
\label{ceq0}
 \overline{\wh\gamma(\lambda)}=(\overline \Gamma_0\uphar{\wh\sN_\lambda(\ZA^*)})^{-1},
 \quad \lambda \in \cmr.
\end{equation}
In particular,
\[
  \ker\overline{\wh\gamma(\lambda)}=\mul \overline \Gamma_0,
  \quad
  \mul \overline{\wh\gamma(\lambda)}=(\ker \overline \Gamma_0)\cap \wh \sN_\lambda(\ZA^*),
\]
and
\[
  \ran \overline{\wh\gamma(\lambda)}=(\dom \overline \Gamma_0)\cap \wh \sN_\lambda(\ZA^*),
\quad
  \dom\overline{\wh\gamma(\lambda)}=\overline \Gamma_0((\dom \overline \Gamma_0)\cap \wh \sN_\lambda(\ZA^*)).
\]
\end{proposition}

\begin{proof}
By definition $\wh\gamma(\lambda)=(\Gamma_0\uphar
\wh\sN_\lambda({\ZA_*}))^{-1}=(\Gamma_0\cap (\wh\sN_\lambda({\ZA_*})\times
\cH))^{-1}$, which implies that
\[
 \overline{\wh\gamma(\lambda)}^{\,-1}\subset \overline\Gamma_0\uphar{\wh\sN_\lambda(\ZA^*)}, \quad \lambda\in\cmr.
\]
To prove the reverse inclusion, assume that $\{\wh f_\lambda,h\}\in
\overline\Gamma_0\cap (\wh\sN_\lambda(\ZA^*)\times \cH)$. With
$\lambda\in\dC_+$ it follows from~\eqref{eq:8.2_lambda} that there
are sequences $f_n\in\sH$ and $u_n\in\cH$, such that
\begin{equation}\label{eq:8.2B}
    \left\{\left(
                   \begin{array}{c}
                     (T-I)f_n+(1-\zeta)(1-\zeta T)^{-1}Fu_n \\
                     i(T+I)f_n+i(1+\zeta)(1-\zeta T)^{-1}Fu_n \\
                   \end{array}
                 \right),
                     Gf_n+(\theta(\zeta)-I)u_n
                 \right\}
    \to
    \left\{\left(
                   \begin{array}{c}
                     f_\lambda \\
                     \lambda f_\lambda \\
                   \end{array}
                 \right), h
    \right\}.
\end{equation}
This implies that $(I-\zeta T)f_n\to 0$ and hence $f_n\to 0$, since
$\lambda\in\dC_+$ or, equivalently, $\zeta\in \dD$. Thus
\[
    \left\{\left(
                   \begin{array}{c}
                     (1-\zeta)(1-\zeta T)^{-1}Fu_n \\
                     \lambda(1-\zeta)(1-\zeta T)^{-1}Fu_n \\
                   \end{array}
                 \right),
                     (\theta(\zeta)-I)u_n
                 \right\}
    \to
    \left\{\left(
                   \begin{array}{c}
                     f_\lambda \\
                     \lambda f_\lambda \\
                   \end{array}
                 \right), h
    \right\},
\]
which by~\eqref{eq:gamma+} in Proposition~\ref{prop:C2} means that
$\{\wh f_\lambda,h\}\in \overline{\wh\gamma(\lambda)}^{\,-1}$.

Similarly, with $\bar{\lambda}\in\dC_-$ it follows from
~\eqref{eq:8.2A_lambda} that for every $\{\wh f_{\bar{\lambda}},h\}\in
\overline\Gamma_0\cap (\wh\sN_{\bar{\lambda}}(\ZA^*)\times \cH)$ there
exists a sequence $u'_n\in\cH$ such that
\[
    \left\{\left(
                   \begin{array}{c}
                     (\bar{\zeta}-1)(1-\bar{\zeta} T)^{-1}G^*u'_n \\
                     \bar{\lambda}(\bar{\zeta}-1)(1-\bar{\zeta} T)^{-1}G^*u'_n \\
                   \end{array}
                 \right),
                     (I-\theta(\zeta)^*)u'_n
                 \right\}
    \to
    \left\{\left(
                   \begin{array}{c}
                     f_{\bar{\lambda}} \\
                     \bar{\lambda} f_{\bar{\lambda}} \\
                   \end{array}
                 \right), h
    \right\},
\]
which by~\eqref{eq:gamma-} in Proposition~\ref{prop:C2} means that
$\{\wh f_{\bar{\lambda}},h\}\in \overline{\wh\gamma(\bar{\lambda})}^{\,-1}$. This
completes the proof of~\eqref{ceq0} and the remaining statements
follow easily from this identity.
\end{proof}

\begin{corollary}\label{cor:G0closable}
Let $\{\cH,\Gamma_0,\Gamma_1\}$ be a unitary boundary triple for
$\ZA^*$ and let $M(\cdot)$ be the corresponding Weyl function. Then
the mapping $\Gamma_0$ is closable if and only if for some,
equivalently for every, $\lambda\in\cmr$ the Weyl function satisfies
the following condition:
\begin{equation}\label{clIM}
 h_n\to h \text{ in } \cH \quad \text{and } \quad \IM (M(\lambda)h_n,h_n) \to 0 \quad
 (n\to\infty)
 \quad \Longrightarrow \quad h=0.
\end{equation}
\end{corollary}
\begin{proof}
By Lemma \ref{Weylmul} $M(\cdot)$ is an operator valued function
with $\ker (M(\lambda)-M(\lambda)^*)=\{0\}$. In this case
~\eqref{Green3} implies that
\[
 (\lambda-\bar\lambda)\|\gamma(\lambda)h\|^2_\sH
 =2i\, \IM (M(\lambda)h,h)_\cH,
\]
$h\in\dom M(\lambda)$, $\lambda,\mu\in\cmr$. From this formula it is
clear that the condition~\eqref{clIM} is equivalent to $\ker
\overline{\gamma(\lambda)}=\{0\}$. Therefore, the result follows
from Proposition \ref{gammaclos}.
\end{proof}

Clearly, the condition~\eqref{clIM} is stronger than the condition
~\eqref{stric-unb} appearing in the definition of strict Nevanlinna
functions. If $M(\cdot)\in \cR[\cH]$ then the condition~\eqref{clIM}
simplifies to $\ker {\IM M(\lambda)}=\{0\}$, i.e., for bounded
Nevanlinna functions the conditions~\eqref{clIM} and
~\eqref{stric-unb} are equivalent. Hence, if
$\{\cH,\Gamma_0,\Gamma_1\}$ is a $B$-generalized boundary triple
then $\Gamma_0$ is closable. However, when $M(\cdot)$ is an
unbounded Nevanlinna function, the condition in Corollary
\ref{cor:G0closable} need not be satisfied. Example
\ref{ex:G0notclosable} shows that already for $S$-generalized
boundary triples $\{\cH,\Gamma_0,\Gamma_1\}$ the mapping $\Gamma_0$
need not be closable.


\begin{proposition}\label{prop:Gamma_1H}
Let  $\{\cH,\Gamma_0,\Gamma_1\}$ be a unitary
boundary triple, let
$\cU=\omega(\Gamma)$ be its Potapov-Ginzburg transform,
let $\lambda\in\dC_+$, $\zeta=\frac{\lambda-i}{\lambda+i}$, and let $H(\lambda)$ be defined by~\eqref{Hlambda}.
Then:
\begin{equation}\label{eq:GammaH+}
\Gamma_1H(\lambda)=\{\{
g,v\}:\,(\theta(\zeta)-I)v+(\zeta-1)G(I-\zeta T)^{-1}g=0,\,\,
g\in\sH,\,v\in\cH\},
\end{equation}
\begin{equation}\label{eq:GammaH-}
\Gamma_1H(\bar{\lambda})=\{\{g',v'\}:
(\theta(\zeta)^*-I)v'+(\bar{\zeta}-1)F^*(I-\bar{\zeta}
T^*)^{-1}g'=0,\,\,g'\in\sH,\,v'\in\cH\}.
\end{equation}
In particular, 
the linear operators $\Gamma_1H(\lambda)$
have constant ranges for all $\lambda\in\dC_\pm$:
\begin{equation}\label{eq:ranGammaHi+}
 \ran(\Gamma_1H(\lambda))=(H-I)^{-1}\ran G=(H^*-I)^{-1}\ran F^*=\ran(\Gamma_1H(\bar\lambda)).
\end{equation}
\end{proposition}
\begin{proof}
It follows from~\eqref{eq:8.2_lambda} and~\eqref{eq:8.2A_lambda}
that
\begin{equation}\label{eq:A0_lambda}
    A_0-\lambda=
    \left\{\left(
                   \begin{array}{c}
                     (T-I)f+(1-\zeta)(I-\zeta T)^{-1}Fu \\
                     \frac{2i}{1-\zeta}(I-\zeta T)f \\
                   \end{array}
                 \right):\,
                   \begin{array}{c}
                   f\in\sH,\,\, u\in\cH\\
                     Gf+(\theta(\zeta)-I)u=0
                   \end{array}
                 \right\}.
\end{equation}
\begin{equation}\label{eq:A0_lambda*}
 A_0-\bar{\lambda}=\left\{\left(
                   \begin{array}{c}
                     (I-T^*)f'+(\bar{\zeta}-1)(I-\bar{\zeta} T^*)^{-1}G^*u' \\
                     \frac{-2i}{1-\bar{\zeta}}(I-\bar{\zeta} T^*)f' \\
                   \end{array}
                 \right):
                   \begin{array}{c}
                   f'\in\sH,\,\, u'\in\cH\\
                     -F^*g'+(I-\theta(\zeta)^*)u'
                   \end{array}
                 \right\}                 .
\end{equation}
In particular, using~\eqref{eq:A0_lambda} and the equality
$g=2i(1-\zeta T)(1-\zeta)^{-1}$ one obtains
\[
h:=(A_0-\lambda)^{-1}g=\frac{1-\zeta}{2i}(T-I)(I-\zeta
T)^{-1}g+(1-\zeta)(I-\zeta T)^{-1}Fu,
\]
where $u\in\cH$ satisfies the equality
\[
\frac{1-\zeta}{2i}G(I-\zeta T)^{-1}g+(\theta(\zeta)-I)u=0,
\]
or, equivalently,
\[
(\theta(\zeta)-I)(-2iu)+(\zeta-1)G(I-\zeta T)^{-1}g=0.
\]

 On the other hand, if
$
\left(
\begin{array}{c}
h \\
 h'
  \end{array}
  \right)=H(\lambda)g
$
then by~\eqref{eq:8.2_lambda}
\begin{equation}\label{eq:8.12B}
\left\{ \left(
\begin{array}{c}
h \\
 h'
  \end{array}
  \right),-2iu\right\}=
  \left\{ \left(
\begin{array}{c}
h \\
 h'
  \end{array}
  \right),-iGf-i(\theta(\zeta)+I)u\right\}\in\Gamma_1
\end{equation}
and hence $\{g,-2iu\}\in \Gamma_1H(\lambda)$. This
proves~\eqref{eq:GammaH+}.

Similarly, the formula~\eqref{eq:GammaH-} is based
on~\eqref{eq:8.2A} and~\eqref{eq:A0_lambda*}.

Let $u\in (H-I)^{-1}(\ran G\cap\ran(H-I))$ and $\lambda\in\dC_+$. Then $(H-I)u\in\ran
G$ and hence
\begin{equation}\label{eq:I-H}
    (\theta(\zeta)-I)u=(H-I)u+\zeta G(I-\zeta T)^{-1}Fu\in\ran G.
\end{equation}
In view of~\eqref{eq:GammaH+} this proves that
$u\in\ran(\Gamma_1H(\lambda))$.

Conversely, if $u\in\ran(\Gamma_1H(\lambda))$ then in view of~\eqref{eq:GammaH+} and~\eqref{eq:I-H}  $(H-I)u\in\ran
G$.

Similarly, for $\lambda\in\dC_-$ the equality~\eqref{eq:ranGammaHi+} is implied by~\eqref{eq:GammaH-}.
\end{proof}
Notice that for a single-valued $\Gamma$ one has $\ker(I-H)=\{0\}$, since otherwise $\ker F$ and $\ker G^*$ are nontrivial, which by~\eqref{eq:8.11} contradicts  the assumption that $\mul\Gamma=\{0\}$.
\begin{proposition}\label{prop:C3}
Let  $\{\cH,\Gamma_0,\Gamma_1\}$ be a unitary
boundary triple,  let
$\cU=\omega(\Gamma)$ be its Potapov-Ginzburg transform given
by~\eqref{eq:Col} and~\eqref{eq:Col_U}, and let $H(\lambda)$ be defined by~\eqref{Hlambda}. Then:
\begin{enumerate}\def\labelenumi{\textit{(\roman{enumi})}}
\item  $\Gamma_1H(\lambda)=\gamma(\bar\lambda)^*$ 
for all
$\lambda\in\dC\setminus\dR$;

\item
The range of linear relation
 $\gamma({\lambda})^*$
does not depend on $\lambda\in\dC_+$, $\bar\lambda\in\dC_-$:
\begin{equation}\label{eq:ranGammaHi+2}
\ran\gamma(\bar\lambda)^*=(I-H)^{-1}(\ran G),\quad
\ran\gamma(\lambda)^*=(I-H^*)^{-1}(\ran F^*).
\end{equation}
\end{enumerate}
\end{proposition}
\begin{proof}
(i) By Proposition~\ref{prop:C2} $\dom \gamma(\bar{\lambda})$ is
dense in $\cH$ and hence $\gamma(\bar{\lambda})^*$ is a
single-valued operator from $\sH$ to $\cH$  for all
$\lambda\in\dC\setminus\dR$. By Lemma~\ref{isoHlem} (cf. \cite[Lemma
7.38]{DHMS12}) one has
$\Gamma_1H(\lambda)\subseteq \gamma(\bar\lambda)^* $. 
Now let $\lambda\in\dC_+$ and assume that
$v=\gamma(\bar{\lambda})^*g$ for some $g\in\sH$. Then
by~\eqref{eq:gamma-}
\[
((1-\bar{\zeta})(I-\bar{\zeta} T^*)^{-1}G^*u,g)=((\theta(\zeta)^*-I)u,v)_\cH
\]
for all $u\in\cH$, and hence
\[
(\theta(\zeta)-I)v+(\zeta-1)G(I-\zeta T)^{-1}g=0.
\]
In view of~\eqref{eq:GammaH+} one obtains
$\{g,v\}\in\Gamma_1H(\lambda)$, so that $\gamma(\bar\lambda)^*
\subset \Gamma_1H(\lambda)$. This proves (i) for $\lambda\in\dC_+$.
Similarly, the assumption $\{g',v'\}\in\gamma(\lambda)^*$
$(\lambda\in\dC_+)$ yields in view of~\eqref{eq:gamma-}
\[
(\theta(\zeta)^*-I)v+(\bar{\zeta}-1)F^*(I-\bar{\zeta} T^*)^{-1}g'=0.
\]
By~\eqref{eq:GammaH-} this proves that
$\{g',v'\}\in\Gamma_1H(\bar{\lambda})$ and hence (i) is in force
also for $\lambda\in\dC_-$.

(ii) The formulas~\eqref{eq:GammaH+} and~\eqref{eq:GammaH-}
for $\lambda=i$ take the form
\begin{equation}\label{eq:GammaHi+}
\Gamma_1H(i)=\{\{ g,v\}:\,(H-I)v+Gg=0, \,\,g\in\sH,\,v\in\cH\},
\end{equation}
\begin{equation}\label{eq:GammaHi-}
\Gamma_1H(-i)=\{\{g',v'\}: (H^*-I)v'+F^*g'=0,\,\,g'\in\sH,\,v'\in\cH\}.
\end{equation}
Then~\eqref{eq:ranGammaHi+2} is immediate by item (i) using the
formulas~\eqref{eq:GammaHi+},~\eqref{eq:GammaHi-} and
Proposition~\ref{prop:Gamma_1H}.
\end{proof}

\begin{corollary}\label{cor:C3}{\rm (\cite{DHMS12})}
Let the assumptions of Proposition~\ref{prop:C3} be in force. 
Then the operator $\Gamma_1H(\lambda)$ is bounded or, equivalently, $\Gamma_1\uphar
A_0$ is bounded if and only if $A_0$ is closed.
\end{corollary}

\begin{proof}
By Lemma~\ref{cor:GH} the operator $\Gamma_1H(\lambda)$
is bounded if and only if the restriction $\Gamma_1\uphar A_0$ is
bounded. Since $\Gamma_1H(\lambda)=\gamma(\bar{\lambda})^*$ by Proposition~\ref{prop:C3}, this mapping
is closed and it follows from the closed graph theorem that $\Gamma_1\uphar A_0$ is bounded
if and only if its domain $\dom (\Gamma_1\uphar A_0)=A_0$ is closed.
\end{proof}
\begin{remark}\label{rem:C3}
If  $\{\cH,\Gamma\}$ is a unitary boundary pair,  then it also
admits the representations~\eqref{eq:8.2} and~\eqref{eq:8.2A} in
terms of its Potapov-Ginzburg transform $\cU=\omega(\Gamma)$. Then,
for instance,
\begin{equation}\label{eq:ranGammaHi+1}
\mul ( \Gamma_1H(\lambda))=(I+H)\ker F=(I+H^*)\ker G^*,
\end{equation}
and
the statements (i) and (ii) in Proposition~\ref{prop:C3} take the form:
\begin{equation}\label{eq:G1H}
    \Gamma_1H(\lambda)=\gamma(\bar\lambda)^* \hplus (\{0\}\times\mul \Gamma_1)\quad\mbox{ for all}\quad  \lambda\in\dC\setminus\dR
\end{equation}
\begin{equation}\label{eq:ranGammaHi+3}
\ran\gamma(\lambda)^*=(I-H)^{-1}(\ran G\cap\ran(I-H)),
\end{equation}
and e.g. the adjoint of the $\gamma$-field for $\lambda\in\dC_+$
admits the formula
\[
 \gamma(\bar\lambda)^*=\{ \{f,f'\}:\, (I-H)f'=Gf \}.
\]
\end{remark}


\subsection{A Kre\u{\i}n type formula for unitary boundary
triples}\label{sec5.2}

In this section Kre\u{\i}n's resolvent formula is extended to the
setting of general unitary boundary triples. It is analogous to the
formula established in Section \ref{sec4.2}. Recall from
\cite{DHMS06} that for a unitary boundary triple the kernel
$A_0=\ker\Gamma_0$ need not be selfadjoint, it is in general only a
symmetric extension of $A$ which can even coincide with $A$; see
e.g. the transposed boundary triple treated in Example
\ref{example6.5} below. For simplicity the next result is formulated
for nonreal points $\lambda\in\cmr$; these points are regular type
points for $A_0$.

As in Section \ref{sec4.2}, let $\wt\ZA$ be an extension of $\ZA$
which belongs to the domain of $\Gamma$ and let $\Theta$ be a linear
relation in $\cH$ corresponding to $\wt\ZA$:
\begin{equation}\label{ATheta2}
 \Theta=\Gamma(\wt A), \quad \wt A\subset \dom \Gamma
 \quad \Leftrightarrow \quad \wt A=A_\Theta:=\Gamma^{-1}(\Theta), \quad \Theta\subset \ran\Gamma.
\end{equation}

\begin{theorem}\label{Kreinformula2}
Let $\ZA$ be a closed symmetric relation, let $\Pi =
\{\cH,\Gamma_0,\Gamma_1\}$ be a unitary boundary triple for $\ZA^*$
with $A_0=\ker\Gamma_0$, and let  $M(\cdot)$ and $\gamma(\cdot)$ be
the corresponding  Weyl function and $\gamma$-field, respectively.
Then for any extension $A_\Theta\in \Ext_A$ satisfying
$A_\Theta\subset \dom \Gamma$ the following equality holds for every
$\lambda\in\cmr$,
\begin{equation}\label{resol2}
(A_\Theta-\lambda)^{-1}-(A_0-\lambda)^{-1} =
 \gamma(\lambda)\bigl(\Theta-M(\lambda)\bigr)^{-1}\gamma(\bar\lambda)^*,
\end{equation}
where the inverses in the first and last term are taken in the sense
of linear relations.
\end{theorem}

\begin{proof}
We first prove the inclusion ``$\subset$'' in~\eqref{resol2}. Let
$A_\Theta\subset \dom \Gamma$ be an extension of $\ZA$ with $\Theta$
as in~\eqref{ATheta2}. Since $A_0$ is symmetric,
$(A_0-\lambda)^{-1}$ is a bounded, in general nondensely defined,
operator for every fixed $\lambda\in\cmr$. Now assume that
$\{g,g''\}\in (A_\Theta-\lambda)^{-1}-(A_0-\lambda)^{-1}$. Then
$g\in \dom (A_\Theta-\lambda)^{-1} \cap \dom (A_0-\lambda)^{-1}$ and
$\{g,g'\} \in (A_\Theta-\lambda)^{-1}$ for some $g'\in\sH$, so that
$g''=g'-(A_0-\lambda)^{-1}g$. Hence $\wh g_\Theta:=\{g',g+\lambda
g'\}\in A_\Theta\subset \dom \Gamma$,
\begin{equation}\label{A0g0}
 \wh g_0:= 
 \{(A_0-\lambda)^{-1}g,(I+\lambda(A_0-\lambda)^{-1})g\} \in A_{0} \subset \dom\Gamma,
\end{equation}
and
\[
\wh g_\Theta-\wh g_0 
 =\{g'-(A_0-\lambda)^{-1}g,\lambda (g'-(A_0-\lambda)^{-1}g)\},
\]
so that $\wh g_\Theta-\wh g_0\in \wh \sN_{\lambda}(\ZA_*)$. Recall
that $\wh \gamma(\lambda)$ maps $\dom\wh\gamma(\lambda)$ onto $\wh
\sN_{\lambda}(\ZA_*)\subset \dom\Gamma$ and hence there exists
$\varphi \in \dom\wh\gamma(\lambda)=\dom M(\lambda)$ such that
\begin{equation}\label{graph2g}
 \wh g_\Theta - \wh g_0 = \wh \gamma(\lambda) \varphi,
 \quad \Gamma \wh \gamma(\lambda)\varphi=\{\varphi,
M(\lambda)\varphi\},
\end{equation}
see~\eqref{ghatfield},~\eqref{ggam1}; notice that $M(\lambda)$ is an
operator, since $\mul\Gamma=\{0\}$. Clearly $\Gamma_0 \wh g_{0}=0$
and according to Proposition~\ref{prop:C3} one has $\Gamma_1\wh
g_0=\Gamma_1 H(\lambda)g=\gamma (\overline\lambda)^*g$, where $
H(\lambda)$ is defined by~\eqref{Hlambda}. Observe, that here
$\gamma(\bar\lambda)^*$ is an operator since $H(\lambda)$ and
$\Gamma_1$ are operators. Now it follows from~\eqref{graph2g} that
\begin{equation}\label{graph2gg}
 \{0,\gamma(\bar\lambda)^*g\}+\{\varphi,M(\lambda)\varphi\}
 =\Gamma \wh g_0 +\Gamma  \wh \gamma(\lambda)\varphi
 =\Gamma \wh g_{\Theta} \in \Theta,
\end{equation}
see~\eqref{ATheta2}. Consequently,
$\{\varphi,\gamma(\bar\lambda)^*g+M(\lambda)\varphi \}\in \Theta$
and $\{\varphi,\gamma(\bar\lambda)^*g\}\in \Theta- M(\lambda)$ or,
equivalently, $\{g,\varphi\}\in(\Theta-
M(\lambda))^{-1}\gamma(\bar\lambda)^*$ and hence~\eqref{graph2g}
shows that
\[
 \{g,g''\}=
 \{g,\gamma(\lambda)\varphi\}\in\gamma(\lambda)(\Theta-M(\lambda))^{-1}\gamma(\bar\lambda)^*,
\]
which proves the first inclusion in~\eqref{resol2}.

To prove the reverse inclusion ``$\supset$'' in~\eqref{resol2}
assume that $\{g,g''\}\in
\gamma(\lambda)(\Theta-M(\lambda))^{-1}\gamma(\bar\lambda)^*$. Since
$\dom(\Theta-M(\lambda))\subset \dom M(\lambda)=\dom\gamma(\lambda)$
the assumption on $\{g,g''\}$ means that for some $\varphi\in\cH$
one has $\{\gamma(\bar\lambda)^*g,\varphi\}\in
(\Theta-M(\lambda))^{-1}$ and
\[
 \{g,g''\}=\{g,\gamma(\lambda)\varphi\}\in\gamma(\lambda)(\Theta-M(\lambda))^{-1}\gamma(\bar\lambda)^*.
\]
It follows from $\{\varphi,\gamma(\bar\lambda)^*g+M(\lambda)\varphi
\}\in \Theta$ and~\eqref{ATheta2} that $\Gamma \wh
g_{\Theta}=\{\varphi,\gamma(\bar\lambda)^*g+M(\lambda)\varphi \}$
for some $\wh g_{\Theta}\in A_\Theta$. By Proposition~\ref{prop:C3}
$\Gamma_1H(\lambda)=\gamma(\bar\lambda)^*$, which shows that
$g\in\ran(A_0-\lambda)$, $\lambda\in\cmr$; see~\eqref{Hlambda}. Now
associate with $g$ the element $\wh g_0$ as in~\eqref{A0g0}. Since
$\Gamma \wh g_{0}=\{0,\gamma(\bar\lambda)^*g\}$ and $\Gamma\wh
\gamma(\lambda)\varphi=\{\varphi,M(\lambda)\varphi\}$ we conclude
that~\eqref{graph2gg} is satisfied. Therefore, $\wh g_0+\wh
\gamma(\lambda)g- \wh g_{\Theta}\in \ker \Gamma = A$ and thus $\wh
g_0+\wh \gamma(\lambda)\varphi\in A_\Theta$ or, equivalently,
\[
 \{g,(A_0-\lambda)^{-1}g+\gamma(\lambda)\varphi\}\in (A_\Theta-\lambda)^{-1}.
\]
Hence,
\[
\{g,g''\}= \{g,\gamma(\lambda)\varphi\} \in
(A_\Theta-\lambda)^{-1}-(A_0-\lambda)^{-1}.
\]
This proves the reverse inclusion in~\eqref{resol2} and completes
the proof.
\end{proof}
  \begin{remark}
Again notice the generality of the formulas in~\eqref{resol2}, in
particular, that $\lambda$ need not belong to $\rho(A_\Theta)$.
Observe also that in the formula~\eqref{resol2} the operator
$(A_0-\lambda)^{-1}$ cannot be shifted to the righthand side without
loosing the stated equality. Indeed, in that case only the following
inclusion remains valid:
\[
(A_\Theta-\lambda)^{-1}\supset (A_0-\lambda)^{-1} -
 \gamma(\lambda)\bigl(\Theta-M(\lambda)\bigr)^{-1}\gamma(\bar\lambda)^*.
\]
  \end{remark}

By considering the multivalued parts we obtain the following
statement for the point spectrum of $A_\Theta$ from
Theorem~\ref{Kreinformula2}.

\begin{corollary}
With the assumptions in Theorem \ref{Kreinformula2} one has
$\lambda\in\sigma_p(A_\Theta)$ if and only if
$0\in\sigma_p\bigl(\Theta-M(\lambda)\bigr)$, in which case
\[
 \ker(A_\Theta-\lambda)=\gamma(\lambda)\ker(\Theta-M(\lambda)),\quad
 \lambda\in\cmr.
\]
\end{corollary}


\subsection{$S$-generalized boundary triples}\label{sec5.3}
Here we extend Definition~\ref{SgenBT} to the case of boundary
pairs.
 \begin{definition}\label{SgenBP}\index{Boundary pair!S-generalized }
A unitary boundary pair  $\{\cH,\Gamma\}$ is  said to be an
\emph{$S$-generalized boundary pair}, if $A_0$ is a selfadjoint
linear relation in $\sH$.
\end{definition}
In the following proposition some special boundary triples/pairs
are characterized in terms of  its Potapov-Ginzburg transform.
\begin{proposition}\label{Wlemma2}
Let  $\{\cH,\Gamma\}$ be a unitary boundary pair,  let
$\cU=\omega(\Gamma)$ be its Potapov-Ginzburg transform given
by~\eqref{eq:Col} and~\eqref{eq:Col_U}, and let ${\ZA_*}=\dom\Gamma$, $A_0=\ker\Gamma_0$.
Then:
\begin{enumerate}\def\labelenumi {\textit{(\roman{enumi})}}
\item $\{\cH,\Gamma_0,\Gamma_1\}$ is an ordinary boundary triple if and only if
\[
 \ran G=\cH\quad\Longleftrightarrow\quad \ran F^*=\cH;
\]
\item $\{\cH,\Gamma_0,\Gamma_1\}$ is a $B$-generalized boundary triple if and
only if
  \[\left\{\begin{array}{c}
         \ker F=\{0\} \\
         \ran (I-H)=\cH
       \end{array}\right.
       \Longleftrightarrow\quad \left\{\begin{array}{c}
         \ker G^*=\{0\} \\
         \ran (I-H^*)=\cH
       \end{array}\right.
;\]
\item  $\{\cH,\Gamma\}$ is a $B$-generalized
boundary pair if and only if
  \[
  \Gamma_0|\wh\sN_i=\cH
  \quad\Longleftrightarrow\quad \ran (I-H)=\cH
   \quad\Longleftrightarrow\quad \ran (I-H^*)=\cH;
\]
\item  $\Gamma_0$ is surjective if and only if
  \[
  \ran(I-H)+\ran G=\cH\quad\Longleftrightarrow\quad
  \ran (I-H^*)+\ran F^*=\cH;
  \]
\item $\{\cH,\Gamma\}$ is a $S$-generalized
boundary pair if and only if
\[
\ran G\subset\ran(I-H)\quad\mbox{ and }\quad\ran F^*\subset\ran (I-H^*).
\]
\end{enumerate}
\end{proposition}
\begin{proof}
The statements (i)--(iii) can be found in \cite[Proposition~5.9, Corollaries 5.11 and 5.12]{BHS09}.

(iv) This is implied by~\eqref{eq:8.11}.

(v) This statement follows from the equalities
\begin{equation}\label{eq:8.9i}
\begin{split}
A_0-i&=
    \left\{\left(
                   \begin{array}{c}
                     (T-I)g+Fu \\
                     2ig \\
                   \end{array}
                 \right):\,
                     Gg+(H-I)u =0,\,
                 g\in\sH,\,u\in\cH \right\}\\
A_0+i&= \left\{\left(
                   \begin{array}{c}
                     (I-T^*)g'-G^*u' \\
                     2ig' \\
                   \end{array}
                 \right):\,
                     F^*g'+(H^*-I)u'=0,\,
                 g'\in\sH,\,u'\in\cH \right\}
\end{split}
\end{equation}
which, in turn, are implied by~\eqref{eq:8.9}.
\end{proof}

\begin{remark}
Observe also that $A_0$ is a maximal symmetric operator if at least one of the conditions
$\ran G\subset\ran(I-H)$ or $\ran F^*\subset\ran (I-H^*)$
is satisfied. The statement (v) in Proposition \ref{Wlemma2} is contained in \cite[Prop]{Cal39};
see also~\cite[Prop]{HaWi12}, where it is formulated in terms of angular representation of $A_0$.
An example of a unitary boundary triple $\{\cH,\Gamma\}$, such that $A_0$ is selfadjoint and
$\Gamma_0$ is not surjective is presented in~\cite[Example~6.6]{DHMS06}.
\end{remark}

The following lemma shows that the conditions (iv) and (v) in Proposition~\ref{Wlemma2} are not unrelated.

\begin{lemma}\label{lem:ranGamma}
Let $\cU$ be a unitary colligation of the form~\eqref{eq:Col_U}.
Then the following conditions are equivalent:
\begin{enumerate}\def\labelenumi {\textit{(\roman{enumi})}}
\item  $\ran(I-H)+\ran G=\cH$;
\item   $\ran(I-H^*)+\ran F^*=\cH$;
\item   $\ran(I-H)=\cH$;
\item   $\ran(I-H^*)=\cH$.
\end{enumerate}
\end{lemma}
\begin{proof}
The equivalence of (i) and (ii) is implied by~\eqref{eq:8.11}.

Since $\ran(I-H)\subseteq\ran(I-H)+\ran G$ and
$\ran(I-H^*)\subseteq\ran(I-H^*)+\ran F^*$ it remains to prove the
implications (i) $\Rightarrow$ (iii) and (ii) $\Rightarrow$ (iv).

Assume that $\ran(I-H)+\ran G=\cH$. Then using \cite{FW71} and the identity $HH^*+GG^*=I$ one obtains
\[
\begin{split}
\ran(I-H)+\ran
G&=\ran(((I-H)(I-H^*))^{1/2})+\ran((GG^*)^{1/2})\\
&=\ran(((I-H)(I-H^*)+GG^*)^{1/2})
\\
&=\ran((I-2\RE H+HH^*+GG^*)^{1/2})=\ran((I-\RE H)^{1/2}).
\end{split}
\]
This implies the equality $\ran(I-\RE H)=\cH$ and hence $-I\le\RE
H\le qI$ for some $q<1$. Therefore, the numerical range of $H$ is
contained in the half-plane $\RE z\le q$ and hence $1\in\rho(H)$.
This proves (iii). The implication (ii) $\Rightarrow$ (iv) is proved
similarly.
\end{proof}

\begin{corollary}\label{cor:C5}
If $\Pi=\{\cH,\Gamma_0,\Gamma_1\}$ is a unitary boundary triple with
$\ran\Gamma_0=\cH$, then $A_0=A_0^*$ and $\Pi$ is necessarily a
$B$-generalized boundary triple.
\end{corollary}

\begin{remark}\label{rem:OBT}
If $\{\cH,\Gamma_0,\Gamma_1\}$ is an ordinary boundary triple, then
$\Gamma$ and, consequently, $\Gamma_0$ and $\Gamma_1$ are
surjective. Hence, $A_0=A_0^*$ and $A_1=A_1^*$. This conclusion can
be made directly also from Proposition \ref{Wlemma2}. Indeed, the
assumption $\ran G=\cH$ implies $0\in\rho(GG^*)$. In view of the
identity $GG^*=I-HH^*$ this implies $1\in\rho(HH^*)$ and hence
$1\in\rho(H)$. By Proposition \ref{Wlemma2} (v) this condition yields
$A_0=A_0^*$.
\end{remark}

%

We are now ready to prove Theorem~\ref{prop:C6B} in a more general
setting, where $\{\cH,\Gamma\}$ is an arbitrary unitary boundary
pair. It gives a complete characterization of the Weyl functions
$M(\cdot)$ of $S$-generalized boundary pairs. In its present general
form it completes and extends \cite[Theorem~4.13]{DHMS06} and
\cite[Theorem~7.39]{DHMS12}.


\begin{theorem}\label{prop:C6}
Let  $\Pi=\{\cH,\Gamma\}$ be a unitary boundary pair and let $M(\cdot)$ and
$\gamma(\cdot)$ be the corresponding Weyl family and the $\gamma$-field. Then the
following statements are equivalent:
\begin{enumerate}\def\labelenumi {\textit{(\roman{enumi})}}
\item $A_0$ is selfadjoint, i.e. $\Pi$ is an $S$-generalized boundary pair;
\item $A_*=A_0\hplus \wh\sN_{\lambda}$ and $A_*=A_0\hplus \wh\sN_{\mu}$ for some (equivalently for all) $\lambda\in\dC_+$ and $\mu\in\dC_-$;
\item $\ran\Gamma_0 = \dom M(\lambda)= \dom M(\mu)$ for some (equivalently for all) $\lambda\in\dC_+$ and $\mu\in\dC_-$;
\item $\gamma(\lambda)$ and $\gamma(\mu)$ are bounded for some (equivalently for all) $\lambda\in\dC_+$ and $\mu\in\dC_-$;
\item $\IM M_{\rm op}(\lambda)$ is bounded with dense domain in $\cdom M(\lambda)$ for some (equivalently for all) $\lambda\in\dC_+$ and $\mu\in\dC_-$;
\item The Weyl family $M(\lambda)$, $\lambda\in\cmr$, admits the representation
\begin{equation}\label{eq:M_S_gen}
   M(\lambda)=E+M_0(\lambda),
\end{equation}
where $E=E^*$ is a selfadjoint relation in $\cH$  and $M_0\in
\cR[\cH_0]$, with $\cH_0={\cdom E}$.
\end{enumerate}
\end{theorem}
\begin{proof}
(i) $\Leftrightarrow$ (ii) This equivalence and the independence
from $\lambda\in\dC_+$ and $\mu\in\dC_-$ is proved in
\cite[Theorem~4.13]{DHMS06}.

(i) $\Leftrightarrow$ (iii) This can also be obtained from
\cite[Theorem~4.13]{DHMS06}, but we present here a different proof.
Indeed, it follows from~\eqref{eq:8.2_lambda} that for all
$\lambda\in\dC_+ $ and $\zeta=\frac{\lambda-i}{\lambda+i}$
\begin{equation}\label{eq:A0sa}
    \ran \Gamma_0=\ran G+\ran (\theta(\zeta)-I).
\end{equation}
If $A_0=A_0^*$ then by Proposition~\ref{Wlemma2} $\ran
G\subset\ran(I-H)$ and~\eqref{eq:A0_lambda} yields $\ran G\subset
\ran (\theta(\zeta)-I)$. By~\eqref{eq:A0sa},~\eqref{eq:gamma+}, and
$\dom\gamma(\lambda)=\dom M(\lambda)$ one obtains
\[
\ran \Gamma_0=\ran (\theta(\zeta)-I)=\dom M(\lambda)\quad\mbox{for
all}\quad \lambda\in\dC_+.
\]

Similarly, it follows from~\eqref{eq:A0sa} and~\eqref{eq:A0_lambda*}
that
\[
\ran \Gamma_0=\ran F^*+\ran (\theta(\zeta)^*-I)=\dom
M(\bar{\lambda}) \quad\mbox{for all}\quad \lambda\in\dC_+.
\]

Conversely, if for some $\lambda\in\dC_+$ one has $\ran
\Gamma_0=\dom M(\lambda)=\dom\gamma(\lambda)$, then~\eqref{eq:8.11}
implies, in particular, that $\ran G\subset \ran (\theta(\zeta)-I)$.
Hence, it follows from~\eqref{eq:A0_lambda} that $\ran
(A_0-\lambda)= \sH$. Similarly the identities $\ran \Gamma_0=\dom
M(\bar{\lambda})=\dom\gamma(\bar{\lambda})$ imply that $\ran
(A_0-\bar\lambda)= \sH$ and, thus, $A_0=A_0^*$.

(i) $\Rightarrow$ (iv) This implication was proved in
Theorem~\ref{ABGthm}~(iv), (v).

(iv) $\Rightarrow$ (i) If some
$\gamma(\lambda):{\cdom\gamma(\lambda)}\to\sH$ is bounded then
$\dom\gamma(\lambda)^*=\sH$. Then by~\eqref{eq:G1H}
\[
 \ran(A_0-\lambda)=\dom\Gamma_1 H(\lambda)=\dom\gamma(\lambda)^*=\sH.
\]
Similarly if $\gamma(\mu)$ is bounded then $\ran(A_0-\mu)=\sH$.
Thus, $A_0$ is a selfadjoint relation in $\sH$.

(iv) $\Rightarrow$ (v),~(vi) Consider the decomposition~\eqref{ml}
$M(\lambda)=\textup{gr }M_{\rm op}(\lambda)\oplus M_\infty$ of the Weyl family
$M(\lambda)$ with the operator part $M_{\rm op}\in\cR(\cH_0)$, where
$\cH_0={\cdom M(\lambda)}$. As was already shown, now $A_0=A_0^*$
and $\dom M_{\rm op}(\lambda)=\ran\Gamma_0$ for all
$\lambda\in\dC\setminus\dR$. It follows from the equality $M_{\rm
op}(\lambda)^*=M_{\rm op}(\bar\lambda)$ that the operator $E_0=\RE
M_{\rm op}(\lambda_0)$ ($\lambda_0\in\dC_+$) is selfadjoint with the
domain $\dom E_0=\ran\Gamma_0$. Moreover, since the operator
$\gamma(\lambda)$ is bounded for all $\lambda\in\dC\setminus\dR$ it
follows from the equality~\eqref{Green3U} that the operator
\begin{equation}\label{eq:Im_M_0}
    \IM M_{\rm op}(\lambda_0)=\IM \lambda_0 \gamma(\lambda_0)^*\gamma(\lambda_0)
\end{equation}
is also bounded in $\cH_0$ and hence the operator $M_{\rm
op}(\lambda)-E_0$ is bounded in $\cH_0$ at $\lambda_0$. Therefore,
its closure, denoted now by $M_0(\lambda)$, is bounded in $\cH_0$ at
$\lambda_0$ and then also for all $\lambda\in\dC\setminus\dR$; see
e.g. \cite[Proposition~4.18]{DHMS09}, \cite[Theorem~3.9]{DHM15}.
Finally, by setting $E=E_0\oplus M_\infty$ one arrives
at~\eqref{eq:M_S_gen}.

Finally, the implication (vi) $\Rightarrow$ (v) is clear and (v)
$\Rightarrow$ (iv) (for $\mu=\bar\lambda)$ follows easily
from~\eqref{Green3U}.
\end{proof}

Theorem~\ref{prop:C6} implies Theorem~\ref{prop:C6B}. In the case
that $\Gamma$ is single-valued $M(\lambda)$ is an operator valued
Nevanlinna function with $\ker \IM M(\lambda)=\ker
(M(\lambda)-M(\lambda)^*)=\{0\}$, i.e., $M(\cdot)\in \cR^s(\cH)$;
see~\eqref{stric-unb} and Lemma~\ref{Weylmul}.

\begin{corollary}\label{Scor}
Let $\{\cH,\Gamma\}$ be an $S$-generalized boundary pair with the
Weyl family $M(\cdot)=E+M_0(\cdot)$ as in Theorem~\ref{prop:C6}.
Then $\ran\Gamma$ is dense in $\cH\times\cH$, i.e., $\Gamma$ defines
an $S$-generalized boundary triple if and only if $E\,(=\RE M(\mu))$
is a selfadjoint operator and
\begin{equation}\label{Soper}
\dom
E\cap\ker\overline{\gamma(\lambda)}=E\cap\ker\IM{M_0(\lambda)}=\{0\},
\quad \lambda\in\cmr.
\end{equation}
\end{corollary}
\begin{proof}
This follows from Lemma~\ref{Weylmul} and Corollary~\ref{ABGgraph}.
\end{proof}

Corollary~\ref{Scor} can be used to give an example of an
$S$-generalized boundary triple $\{\cH,\Gamma_0,\Gamma_1\}$ such
that the mapping $\Gamma_0$ is not closable; cf.
Corollary~\ref{cor:G0closable}.

\begin{example}\label{ex:G0notclosable}
Let $M_0(\cdot)\in \cR[\cH]$ be a bounded Nevanlinna function such
that $\ker \IM M_0(\lambda)$ is nontrivial and let $E$ be an
unbounded selfadjoint operator in $\cH$ with $\dom E\cap \ker \IM
M_0(\lambda)=\{0\}$. Then the function
\[
 M(\lambda)=E+M_0(\cdot), \quad \lambda\in\cmr,
\]
is a domain invariant Nevanlinna function. Moreover, it follows from
Corollary~\ref{Scor} and Theorems~\ref{TAMS_Th},~\ref{prop:C6B} that
$M(\cdot)$ can be realized as the Weyl function of some
$S$-generalized boundary triple $\{\cH,\Gamma_0,\Gamma_1\}$.
However, $\IM (M(\lambda)h,h)=\IM (M_0(\lambda)h,h)$, $h\in \dom
M(\lambda)$, does not satisfy the condition~\eqref{clIM} in
Corollary~\ref{cor:G0closable}, since $\ker \overline{\IM
M(\lambda)}=\ker \IM M_0(\lambda)$ is nontrivial by construction.
\end{example}

\begin{remark}\label{rem:Error_Scrit}
Observe that in Theorems~\ref{prop:C6B} and~\ref{prop:C6} the
function $M_0(\cdot)$ can be considered as the closure of
$M(\cdot)-E$. In Theorem \ref{prop:C6} $M(\cdot)$ is an operator
valued function if and only if $E$ is an operator. By
Corollary~\ref{Scor} even in this case $\Gamma$ can still be
multivalued if the kernel $\ker M_0(\lambda)$ or $\ker \IM
M_0(\lambda)=\ker\overline{\gamma(\lambda)}$ is nontrivial and the
condition~\eqref{Soper} is violated. In fact, any bounded Nevanlinna
function in $\cH$ with $\ker \IM M_0(\lambda)\neq \{0\}$ combined
with an unbounded selfadjoint operator $E$ in $\cH$ satisfying the
condition~\eqref{Soper} is associated with an $S$-generalized
boundary triple $\{\cH,\Gamma_0,\Gamma_1\}$ with the Weyl function
$M=E+M_0(\cdot)$. If such a function $M$ is regularized by
subtracting the unbounded constant operator $E$, the function
$M_0(\cdot)=M(\cdot)-E$ corresponds to an $AB$-generalized boundary
triple $\{\cH,\wt\Gamma_0,\wt\Gamma_1\}$ whose range $\ran\wt\Gamma$
is not dense in $\cH^2$. In particular, $\wt\Gamma$ whose Weyl
function is the regularized function $M(\cdot)-E$ is not a quasi
boundary triple. The closure $M_0(\cdot)$ of $M(\cdot)-E$ is the
Weyl function of the closure of $\wt \Gamma$ which in this case is
always a (multivalued) $B$-generalized boundary pair. An example of
an $S$-generalized boundary triple with $\ker \IM M_0(\lambda)\neq
\{0\}$ satisfying the property~\eqref{Soper} appears in
\cite[Proposition~2.17]{BeMi14}.
\end{remark}

\subsection{$ES$-generalized
boundary triples and form domain invariant Nevanlinna functions}
Recall, see~Definition~\ref{EgenBT}, that a unitary boundary  triple
$\{\cH,\Gamma_0,\Gamma_1\}$  for $\ZA^*$ is called
\emph{$ES$-generalized}, if the extension $A_0$ is essentially
selfadjoint in $\sH$.

As the main result of this section it will be shown that the class
of Weyl functions of $ES$-generalized boundary  triples coincides
with the class of form domain invariant Nevanlinna functions.

\begin{definition}\label{DefFormdomain}
A Nevanlinna function $M\in \cR(\cH)$ is said to be form domain
invariant in $\dC_+(\dC_-)$, if
the quadratic form $\st_{M(\lambda)}$ in $\cH$ generated by the imaginary part of $M(\lambda)$ via
\begin{equation}\label{eq:5.18D}
 \st_{M(\lambda)}(u,v)=\frac{1}{\lambda-\bar\lambda}\,[(M(\lambda)u,v)-(u,M(\lambda)v)], 
\end{equation}
is closable for all $\lambda\in\dC_+(\dC_-)$ and the closure of the
form $\st_{M(\lambda)}$ has a constant domain. A Nevanlinna family
$M\in \wt \cR(\cH)$ is said to be form domain invariant in
$\dC_+(\dC_-)$, if its operator part $M_{\rm op}(\cdot)$ in the
decomposition~\eqref{ml} is form domain invariant in $\dC_+(\dC_-)$.
\end{definition}
The following three lemmas are preparatory for the main result.

\begin{lemma}\label{prop:C11}
Let $\{\cH,\Gamma_0, \Gamma_1\}$ be a unitary boundary triple.
Then the following statements are equivalent:
\begin{enumerate}
    \item [(i)] $\ran(A_0-\lambda)$ is dense in $\sH$ for some or, equivalently, for every $\lambda\in\dC_+(\dC_-)$;
    \item [(ii)] $\gamma(\bar{\lambda})$ admits a single-valued closure
    $\overline{\gamma(\bar{\lambda})}$ for some or, equivalently, for every $\lambda\in\dC_+(\dC_-)$.
\end{enumerate}
\end{lemma}
\begin{proof}
(i)$\Leftrightarrow$(ii)
In view of~\eqref{eq:G1H} in Remark~\ref{rem:C3} for every
$\lambda\in\dC_+(\dC_-)$
\[
\dom\gamma(\bar{\lambda})^*=\dom(\Gamma_1H(\lambda))=\ran(A_0-\lambda).
\]
Therefore, $\gamma(\bar{\lambda})$ admits a single-valued closure
for $\lambda\in\dC_+\,(\dC_-)$ if and only if $\ran(A_0-\lambda)$ is
dense in $\sH$.
\end{proof}

%

\begin{lemma}\label{prop:C111}
Let $\{\cH,\Gamma_0, \Gamma_1\}$ be an $ES$-generalized boundary
triple. Then:
\begin{enumerate}
        \item [(i)] $\ker \overline{\Gamma}_0=\overline{A_0}$ is selfadjoint and
the domain of $\overline{\Gamma}_0$ admits the decomposition
\begin{equation}
\label{ceq1a}
 \dom \overline{\Gamma}_0=\overline{A_0}\dot{+} (\dom \overline{\Gamma}_0\cap \wh\sN_\lambda(\ZA^*))
 =\overline{A_0}\dot{+} \ran\overline{\wh\gamma(\lambda)},
 \quad \lambda\in\cmr;
\end{equation}
    \item [(ii)] $\gamma({\lambda})$ admits a single-valued closure
    $\overline{\gamma({\lambda})}$ for every $\lambda\in\dC\setminus\dR$;
    \item [(iii)]
    the closure of the $\gamma$-field satisfies
\begin{equation}
\label{ceq1b}
 \ran \overline{\Gamma}_0= \dom\overline{\gamma(\lambda)}=\dom\overline{\gamma(\mu)}, \quad
 \lambda,\mu\in\cmr;
\end{equation}
\item [(iv)] $\overline{\gamma(\lambda)}$ and $\overline{\gamma(\mu)}$ are connected by
\begin{equation}
\label{ceq1}
 \overline{\gamma(\lambda)}=[I+(\lambda-\mu)(\overline{A_0}-\lambda)^{-1}]\overline{\gamma(\mu)},
 \quad \lambda, \mu \in \cmr;
\end{equation}
\end{enumerate}
\end{lemma}
\begin{proof}
(i) As a closed linear relation $\overline{\Gamma}_0$ has a closed
kernel, which implies that $\overline A_0\subset \ker
\overline{\Gamma}_0$. Since $\overline{A_0}$ is selfadjoint, the first von
Neumann's formula shows that $\ZA^*=\overline{A_0}\dot{+}
\wh\sN_\lambda(\ZA^*)$ for all $\lambda\in\cmr$. Consequently,
\[
 \overline{A_0}\subset \dom \overline{\Gamma}_0 \subset
  \overline{A_0}\dot{+} \wh\sN_\lambda(\ZA^*),
 \quad \lambda\in\cmr,
\]
and this implies the first equality in~\eqref{ceq1a}. The second
equality in~\eqref{ceq1a} holds by Proposition~\ref{gammaclos}.
Finally, according to Proposition~\ref{gammaclos} $\ker
\overline{\Gamma}_0\cap \wh\sN_\lambda(\ZA^*)=\mul
\overline{\wh\gamma(\lambda)}=\{0\}$, since $\gamma(\lambda)$ or,
equivalently, $\wh\gamma(\lambda)$ is closable by Lemma~\ref{prop:C11}. Since
$\overline{A_0}\subset \ker \overline{\Gamma}_0$, the identity $\ker
\overline{\Gamma}_0\cap \wh\sN_\lambda(\ZA^*)=\{0\}$ combined with the
first equality in~\eqref{ceq1a} implies the equality $\overline{A_0}=\ker
\overline{\Gamma}_0$.

(ii) The statement (ii) is implied by Lemma~\ref{prop:C11}.

(iii) Since $\overline{A_0}$ is selfadjoint, the defect subspaces of $\ZA$ are
connected by
\[
 \sN_\lambda(\ZA^*)=[I+(\lambda-\mu)(\overline{A_0}-\lambda)^{-1}]\sN_\mu(\ZA^*),
 \quad \lambda,\mu \in \cmr.
\]
Hence, if $f_\lambda=[I+(\lambda-\mu)(\overline{A_0}-\lambda)^{-1}]f_\mu$,
then $\wh f_\mu=\{f_\mu,\mu f_\mu\}\in\wh\sN_\mu(\ZA^*)$ precisely
when
\begin{equation}
\label{ceq4}
 \wh f_\lambda= \{f_\lambda,\lambda f_\lambda\}
 =\wh f_\mu+(\lambda-\mu)\overline{H(\lambda)}f_\mu \in \wh
 \sN_\lambda(\ZA^*),
\end{equation}
where
$\overline{H(\lambda)}f_\mu=\{(\overline{A_0}-\lambda)^{-1}f_\mu,(I+\lambda(\overline{A_0}-\lambda)^{-1})f_\mu\}\in
\overline{A_0}$. Since $\overline{A_0}\subset \dom \overline{\Gamma}_0$, it
follows from~\eqref{ceq4} that $\wh f_\mu\in\dom(\overline{\Gamma}_0)\cap
\wh\sN_\mu(\ZA^*)$ if and only if $\wh f_\lambda \in
\dom(\overline{\Gamma}_0)\cap \wh\sN_\lambda(\ZA^*)$ and
\[
 \{\wh f_\mu,h\} \in \overline{\Gamma}_0\cap (\wh\sN_\mu(\ZA^*)\oplus \cH)
 \, \Leftrightarrow \,
 \{\wh f_\lambda,h\} \in \overline{\Gamma}_0\cap (\wh\sN_\lambda(\ZA^*)\oplus\cH)
\]
for some $h\in\cH$.
Now, using (i) and Proposition~\ref{gammaclos} one gets
\[
 \dom\overline{\wh\gamma(\lambda)}=\overline{\Gamma}_0(\dom \overline{\Gamma}_0\cap
 \wh\sN_\lambda(\ZA^*)) = \ran \overline{\Gamma}_0=
 \overline{\Gamma}_0(\dom \overline{\Gamma}_0\cap
 \wh\sN_\mu(\ZA^*))=\dom\overline{\wh\gamma(\mu)},
\]
Clearly
$\dom\overline{\wh\gamma(\lambda)}=\dom\overline{\gamma(\lambda)}$,
$\lambda\in\cmr$, and hence (iii) is proved.

(iv) The proof of (iii) shows that $\{h,\wh f_\mu\}\in
\overline{\wh\gamma(\mu)}$ if and only if $\{h,\wh f_\lambda\}\in
\overline{\wh\gamma(\lambda)}$.
Consequently,
$\{h,f_\mu\}\in \overline{\gamma(\mu)}$ if and only if
$\{h,f_\lambda\}=\{h,
[I+(\lambda-\mu)(\overline{A_0}-\lambda)^{-1}]f_\mu\}\in
\overline{\gamma(\lambda)}$ and, since $\overline{\gamma(\mu)}$ and
$\overline{\gamma(\lambda)}$ are operators, this means that
~\eqref{ceq1} is satisfied.
\end{proof}

\begin{lemma}\label{gg01}
Let $M$ be the Weyl family of some
unitary boundary triple $\{\cH,\Gamma_0,\Gamma_1\}$ of $\ZA^*$ and let
$\gamma(\cdot)$ be the corresponding $\gamma$-field. Then:
\begin{enumerate}\def\labelenumi{\textit{(\roman{enumi})}}
\item for all $h\in\dom M(\lambda)$, $k\in\dom M(\mu)$, and $\lambda,\mu\in\cmr$
one has
\[
 \frac{(M(\lambda)u,v)_\cH-(u,M(\mu)v)_\cH}{\lambda-\bar\mu}
   =(\gamma(\lambda)u,\gamma(\mu)v)_\sH;
\]
\item for all $\lambda\in\cmr$ one has $\ker\gamma(\lambda)=\{0\}$;
\item the form $\st_{M(\lambda)}$, $\lambda\in\cmr$, is closable if and only if
$\ran(A_0-\bar{\lambda})$ is dense in $\sH$, hence closability of $\st_{M(\lambda)}$ does not depend
on the representing unitary relation $\Gamma$ used for a realization of $M$ as its Weyl family.
\end{enumerate}
\end{lemma}
\begin{proof}
(i) This is a direct consequence of the Green's identity when
applied to the elements in $\wh \sN_\lambda({\ZA_*})$; see
~\eqref{Green2} and~\eqref{Green3}; cf. also
\cite[Proposition~4.8]{DHMS06}.

(ii) is implied by the definition of $\gamma(\lambda)$
in~\eqref{gfield}
and the equality $\ker\gamma(\lambda)=\mul\Gamma_0=\{0\}$.

(iii) Part (i) gives the following representation for $\st_{M(\lambda)}$:
\[
 \st_{M(\lambda)}[u,v]=(\gamma(\lambda)u,\gamma(\lambda)v)_\sH.
\]
It is well-known (see e.g. \cite[Chapter~VI]{Kato}) that the form $(\gamma(\lambda)u,\gamma(\mu)v)_\sH$
is closable precisely when the operator $\gamma(\lambda)$ is closable. Now the statement is obtained from
Lemma~\ref{prop:C11}.
\end{proof}

\begin{theorem}\label{essThm1}
Let $\Pi=\{\cH,\Gamma_0,\Gamma_1\}$  be a unitary boundary triple
for $\ZA^*$  and let $M$ and $\gamma(\cdot)$ be the corresponding Weyl
function and the $\gamma$-field. Then the following statements are
equivalent:
\begin{enumerate}
    \item [(i)] $\ran(A_0-\lambda)$ is dense in $\sH$ for some or, equivalently, for every $\lambda\in\dC_+(\dC_-)$;
    \item [(ii)]   $\gamma(\lambda)$ admits a single-valued closure     $\overline{\gamma(\lambda)}$ for one $\lambda\in\dC_+(\dC_-)$ with a domain dense in $\cH$;
    \item [(iii)]   $\gamma(\lambda)$ admits a single-valued closure
    $\overline{\gamma(\lambda)}$ for every $\lambda\in\dC_+(\dC_-)$ is domain invariant with a constant domain dense in $\cH$;
    \item [(iv)] the form $\st_{M(\lambda)}$ is closable for one $\lambda\in\dC_+(\dC_-)$;
    \item [(v)]  the Weyl function $M$ belongs to $ \cR^s(\cH)$ and is form domain invariant in $\dC_+(\dC_-)$.
\end{enumerate}
In particular, if statements {\rm(i)--(v)} are satisfied both in
$\dC_+$ and $\dC_-$ then $\Pi$ is an $ES$-generalized boundary
triple and the Weyl function $M$ is form domain invariant with
\begin{equation}\label{EssformDom}
 \dom \overline{\st_{M(\lambda)}}=\dom \overline{\gamma(\lambda)}=\ran \overline{\Gamma}_0,
 \quad \lambda\in \dC\setminus\dR.
\end{equation}
\end{theorem}
\begin{proof}
The equivalence (i) $\Leftrightarrow$ (ii) is obtained from
Lemma~\ref{prop:C11}. The fact that the domain of $\overline{\gamma(\lambda)}$ is dense in $\cH$ follows from Proposition~\ref{prop:C2}.

The equivalences (i) $\Leftrightarrow$ (iv), (v) and (ii)
$\Leftrightarrow$ (iii) follow from Lemma~\ref{prop:C111}.

In particular, part (iii) of Lemma~\ref{gg01} shows that the form
$\st_{M(\lambda)}$ is closable for some (and then for every)
$\lambda\in\dC_+$ and for some (and then for every) $\mu\in \dC_-$
if and only if $A_0$ is essentially selfadjoint. In this case the
closure of the form $\st_{M(\lambda)}$ is given by
\begin{equation}
\label{ceq2a}
 \overline{\st_{M(\lambda)}}[u,v]=(\overline{\gamma(\lambda)}u,\overline{\gamma(\lambda)}v)_\sH,
\end{equation}
in particular, $\dom \overline{\st_{M(\lambda)}}=\dom
\overline{\gamma(\lambda)}$. According to Lemma~\ref{prop:C111} this
domain does not depend on $\lambda\in\cmr$ when $A_0$ is essentially
selfadjoint. The last equality in~\eqref{EssformDom} is obtained
from~\eqref{ceq1b}. This completes the proof.
\end{proof}

%
%
\begin{remark}\label{rem:Mderiv}
    Let  $\{\cH,\Gamma_0, \Gamma_1\}$ be an $ES$-generalized boundary triple, and assume that $(\alpha,\beta)\subset\rho(\overline{A}_0)$. Then:
    \begin{enumerate}
      \item [(i)] for every $\mu\in(\alpha,\beta)$ $\gamma({\mu})$ admits a single-valued closure $\overline{\gamma({\mu})}$ such that
~\eqref{ceq1b}     and~\eqref{ceq1} hold for all $\lambda,\,\mu\in(\dC\setminus\dR)\cup(\alpha,\beta)$;
     \item [(ii)] for every $\mu\in(\alpha,\beta)$ and $u,v\in\cH$ there exists a limit
\begin{equation}
\label{ceq2b}
      \overline{\st_{M(\mu)}}[u,v]=\lim_{\nu\downarrow 0}\overline{\st_{M(\mu+i\nu)}}[u,v]=(\overline{\gamma(\mu)}u,\overline{\gamma(\mu)}v)_\cH.
\end{equation}
    \end{enumerate}
    The proof of the first statement is precisely the same as the proof of Lemma~\ref{prop:C111}. The statement (ii) is implied by the equality
    ~\eqref{ceq2a}, and the continuity of $\overline{\gamma(\mu)}u$ with respect to
    $\mu\in(\alpha,\beta)$; see~\eqref{ceq1}.
\end{remark}

In general, the closure of the mapping $\gamma(\lambda)$ is not
single-valued, see e.g \cite[Example~6.7]{DHMS06}. A simple example
of a unitary boundary triple whose Weyl function is form domain
invariant and $\gamma$-field is unbounded can be obtained as follows
(see also \cite[Example~6.5]{DHMS06}).

\begin{example}\label{example5.1}
Let $H$ be a nonnegative selfadjoint operator in the Hilbert space
$\sH$ with $\ker H=\{0\}$. Let $\ZA_*=\ran H^{1/2}\times \dom
H^{1/2}$, so that $\ZA:=(\ZA_*)^*=\{0,0\}$ and $(\ZA)^{*}=\sH^2$,
and define
\[
 \Gamma_0 \wh f = H^{-1/2}f,\quad \Gamma_1 \wh f = H^{1/2}f', \quad
 \text{where } \wh f=\{f,f'\}\,  \text{ with }  f\in\ran H^{1/2}, \,\, f'\in \dom H^{1/2}.
\]
Then $\{\sH,\Gamma_0,\Gamma_1\}$ is a unitary boundary triple for
$\ZA^*=\overline{ \ZA_*}$. Indeed, Green's
identity~\eqref{Greendef1} is satisfied,
and $\ran \Gamma$ is dense in $\sH^2$. Moreover, it is
straightforward to check that $\Gamma$ is closed, since $H^{1/2}$ is
selfadjoint and, in particular, closed. Observe, that $\wh
f_\lambda:=\{f_\lambda,\lambda f_\lambda\}\in \ZA_*$ if and only if
$f_\lambda=H^{1/2}k$ and $\lambda f_\lambda=H^{-1/2}g$, with
$k\in\dom H^{1/2}$ and $g\in\ran H^{1/2}$, are connected by
$H^{-1/2}g=\lambda H^{1/2}k$. Then $k\in\dom H$ and
\[
 \Gamma_0 \wh f_\lambda = k, \quad \Gamma_1 \wh f_\lambda = \lambda Hk.
\]
These formulas imply that $\gamma(\lambda)=H^{1/2}$ and
$M(\lambda)=\lambda H$, $\lambda\in\dC$. In particular, the Weyl
function is a Nevanlinna function. According to
\cite[Proposition~3.6]{DHMS06} this implies that $\Gamma$ is in fact
$J_\sH$-unitary.

Note that $M(\lambda)$ and its inverse are domain invariant, but in
general unbounded Nevanlinna functions with unbounded imaginary
parts. Clearly, $A_0=\ker \Gamma_0=\{0\}\times \dom H^{1/2}$ and
$A_1=\ker \Gamma_1=\ran H^{1/2}\times \{0\}$ are essentially
selfadjoint and $A_0$ is selfadjoint ($A_1$ selfadjoint) if and only
if $M(\lambda)=\lambda H$ ($-M(\lambda)^{-1}=-\lambda^{-1} H^{-1}$,
respectively) is a bounded Nevanlinna function (cf.
Theorem~\ref{prop:C6B}).
\end{example}

In this example the Weyl function is also domain invariant. In fact,
domain invariance of a Nevanlinna function $M$ implies its form
domain invariance.

\begin{proposition}
Let $M$ be a Nevanlinna function in the Hilbert space $\cH$. If the
equality $\dom M(\lambda)=\dom M(\bar\lambda)$ holds for some
$\lambda\in\cmr$, then $M$ is form domain invariant.

In particular, if $M$ is domain invariant, then it is also form
domain invariant.
\end{proposition}
\begin{proof}
If $\dom M(\lambda)=\dom M(\bar\lambda)$ for some $\lambda\in\cmr$, then one can write
\[
 \st_{M(\lambda)}[u,v]
 =\left( \frac{M(\lambda)-M(\lambda)^*}{\lambda-\bar\lambda}u,v \right)_\cH
 =(\gamma(\lambda)u,\gamma(\lambda)v)_\sH, \quad u,v\in\dom M(\lambda).
\]
Hence, the operator
\[
 N(\lambda):= \frac{M(\lambda)-M(\lambda)^*}{\lambda-\bar\lambda}
\]
is nonnegative and densely defined in $\cH \ominus \mul M(\lambda)$.
Therefore, the form $\st_{M(\lambda)}$ is closable for
$\lambda\in\cmr$; see~\cite{Kato}. By applying the same reasoning to
$\bar\lambda$ it is seen that also the form $\st_{M(\bar\lambda)}$
is closable. Now by applying Lemma~\ref{gg01} it is seen that $A_0$
is essentially selfadjoint (by item (iii)) and hence by
Theorem~\ref{essThm1}  $M$ is form domain invariant.
\end{proof}


The converse statement does not hold. In fact, in~\cite{DHM15} an
example  of a form domain invariant Nevanlinna function is
constructed, such that the domains of $M(\lambda)$ and $M(\mu)$ have
a zero intersection:
\[
 \dom M(\lambda)\cap \dom M(\mu)=\{0\} \text{ for all }
 \lambda,\mu \in \dC_+.
\]
\begin{remark}\label{essrem1}\index{Boundary pair!ES-generalized}
A unitary boundary pair  $\{\cH,\Gamma\}$ for $\ZA^*$  is said to be
$ES$-generalized if $\overline{A_0}=A_0^*$. $ES$-generalized
boundary pairs can be characterized by the following equivalent
conditions:
\begin{enumerate}
    \item [(i)]  for every $\lambda\in\cmr$, $\gamma(\lambda)$ admits a single-valued closure
    $\overline{\gamma(\lambda)}$ with a constant domain;
    \item [(ii)] the Weyl family $M\in \cR(\cH)$ is form domain invariant,
    i.e. its operator part $M_{\rm op}(\cdot)$ in the decomposition~\eqref{ml} is form domain invariant.
\end{enumerate}
Notice, that in the case when (i)-(ii) are in force and $\mul\Gamma$
is nontrivial it may happen that the domain of the form
$\overline{\st_{M_{\rm
op}(\lambda)}}[u,v]=(\overline{\gamma(\lambda)}u,
\overline{\gamma(\lambda)}v)_\sH$ is not dense in $\cH$,
$\lambda\in\dC\setminus \dR$; for an example involving differential
operators; see Example~\ref{Example_third}.
\end{remark}

\subsection{Renormalizations of form domain invariant Nevanlinna functions}
\index{Renormalization of form domain invariant functions}
The next theorem shows that form domain invariant Nevanlinna
functions $M$ in $\cH$ can be renormalized with a bounded operator
$G$ such that the renormalized function $G^*MG$ becomes domain
invariant.

\begin{theorem}\label{essThm2}
Let $\{\cH,\Gamma_0,\Gamma_1\}$ be a unitary boundary triple for
$\ZA^*$ with the $\gamma$-field $\gamma(\cdot)$ and the Weyl function
$M$, and assume that $A_0=\ker \Gamma_0$ is essentially selfadjoint.
Then:
\begin{enumerate}
  \item [(1)] There exists a bounded operator $G$ in $\cH$ with $\ran G=\dom
\overline{\gamma(\lambda)}$, $\lambda\in\cmr$, and $\ker G=\{0\}$,
such that
\begin{equation}\label{ess01}
 \begin{pmatrix}\wt \Gamma_0 \\ \wt \Gamma_1 \end{pmatrix}=\clos
 \begin{pmatrix} G^{-1} & 0 \\ 0 & G^* \end{pmatrix}
 \begin{pmatrix}\Gamma_0 \\ \Gamma_1 \end{pmatrix}
\end{equation}
defines an $AB$-generalized boundary pair $\{\cH,\wt\Gamma\}$ for
$\ZA^*$.
  \item [(2)] The corresponding Weyl function $\wt M$ is
domain invariant and it is given by
\[
 \wt M(\lambda)=E+ M_0(\lambda),
\]
where $E$ is a closed densely defined symmetric operator in $\cH$
and $M_0(\cdot)$ is a bounded Nevanlinna function (defined on $\dom
E$).
  \item [(3)] Furthermore, $\overline{G^*M(\lambda)G}$ is also a Weyl
function of a closed $AB$-generalized boundary pair and it satisfies
\[
 \overline{G^*M(\lambda)G}= E_0+ M_0(\lambda)\subset \wt M(\lambda),\quad \lambda\in\cmr,
\]
where $E_0\subset E$ is a closed densely defined symmetric
restriction of $E$.
\end{enumerate}
\end{theorem}
\begin{proof}
The proof is divided into steps.
\bigskip

{\bf 1.} {\it Construction of a bounded operator $G$ with the properties}
\begin{equation}\label{eq:Prop_G}
  \ker G=\{0\},\quad \ran G=\dom
\overline{\gamma(\mu)} \quad\textup{and}\quad \overline{\dom
\gamma(\mu)G}=\cH, \quad\textup{for some}\quad\mu\in\cmr.
\end{equation}
Since $A_0$ is essentially selfadjoint, $\gamma(\lambda)$ is
closable and the dense subspace
$\cH_0=\dom\overline{\gamma(\lambda)}$ of $\cH$ does not depend on
$\lambda\in\cmr$; see Theorem~\ref{essThm1}. Since $\cH_0$ is an
operator range there exists a bounded selfadjoint operator $G=G^*$
with $\ran G=\cH_0$ and $\ker G=\{0\}$; for instance, one can fix
$\mu \in \cmr$ and then take
$G=(\gamma(\mu)^*\overline{\gamma(\mu)}+I)^{-1/2}$. Namely, $\dom
\gamma(\mu)=\dom M(\mu)$ is dense in $\cH$, since $\mul\Gamma=\{0\}$
by assumption, and hence $\gamma(\mu)^*\overline{\gamma(\mu)}$ is a
selfadjoint operator satisfying $\dom \overline{\gamma(\mu)}=\dom
(\gamma(\mu)^*\overline{\gamma(\mu)})^{1/2}=\dom(\gamma(\mu)^*\overline{\gamma(\mu)}+I)^{1/2}$.
With this choice of $G$ the domain of $\gamma(\mu)G$ is dense in
$\cH$ since $\dom \gamma(\mu)$ is a core for the form $\st_{M(\mu)}$
and due to $\dom (\st_{M(\mu)}+I)=\ran G$ one concludes that $\dom
\gamma(\mu)$ is also a core for the operator
$G^{-1}=(\gamma(\mu)^*\overline{\gamma(\mu)}+I)^{1/2}$.

{\bf 2.} {\it Construction of an isometric boundary triple $\{\cH,\Gamma^G_0,\Gamma^G_1\}$  such that
the corresponding $\gamma$-field $ \gamma^G(\lambda) $ is  a bounded densely defined operator}.


Introduce the transform $\{\cH,\Gamma^G_0,\Gamma^G_1\}$ of the
boundary triple $\{\cH,\Gamma_0,\Gamma_1\}$ by setting
\begin{equation}\label{ess02}
 \begin{pmatrix}\Gamma^G_0 \\ \Gamma^G_1 \end{pmatrix}=
  \begin{pmatrix} G^{-1} & 0 \\ 0 & G^* \end{pmatrix}
 \begin{pmatrix}\Gamma_0 \\ \Gamma_1 \end{pmatrix},
\end{equation}
where $G$ has the properties stated above. The block operator is
isometric (in the Kre\u{\i}n space $(\cH^2,J_\cH)$) and hence
$\Gamma^G$ is isometric as a composition of isometric mappings; i.e.
$\Gamma^G$ satisfies the Green's identity~\eqref{Green1} ({\it
Assumption~\ref{2_def_GBT}}). Since $(\Gamma_0\uphar
\wh\sN_\lambda({\ZA_*}))=\wh\gamma(\lambda)^{-1}$ one has
\[
 \Gamma_0\wh\sN_\lambda({\ZA_*})=\dom\gamma(\lambda)\subset
 \dom\overline{\gamma(\lambda)}=\ran G,
\]
which implies that $\ran\wh\gamma(\lambda)=
\wh\sN_\lambda({\ZA_*})=\wh\sN_\lambda(\ZA^*)\cap \dom \Gamma\subset\dom
\Gamma^G$ (here ${\ZA_*}=\dom \Gamma$), and hence
$\wh\sN_\lambda({\ZA_*}^G)=\wh\sN_\lambda({\ZA_*})$. Moreover, it is clear
that $\ker\Gamma^G_0=\ker \Gamma_0=A_0$ is essentially selfadjoint.
Since the closure of $A_0\hplus \wh\sN_\lambda({\ZA_*})$ is
$\overline{A_0}\hplus \wh\sN_\lambda(\ZA^*)=\ZA^*$ one gets $\cdom \Gamma^G=\ZA^*$ ({\it Assumption~\ref{1_def_GBT}}). The
corresponding $\gamma$-field is given by
\[
 \gamma^G(\lambda)=(\Gamma^G_0\uphar \wh\sN_\lambda({\ZA_*}^G))^{-1}=\gamma(\lambda)G,
\quad \lambda\in\cmr;
\]
see Lemma~\ref{isomtrans}. Since $\gamma(\lambda)$ is closable and
$\gamma(\lambda)G\subset \overline{\gamma(\lambda)}G$, it follows
from $\ran G=\dom \overline{\gamma(\lambda)}$ that the closed
operator $\overline{\gamma(\lambda)}G$ is everywhere defined and,
hence, bounded by the closed graph theorem. Thus also
$\gamma(\lambda)G$ is a bounded operator with bounded closure
$\overline{\gamma(\lambda)G}\subset \overline{\gamma(\lambda)}G$.

Next recall the operator $H(\lambda)$, $\lambda\in\cmr$, from
~\eqref{Hlambda}; see also Lemma~\ref{cor:GH}.
Since $A_0$ is essentially selfadjoint,
$\dom (\Gamma_1 H(\lambda))=\dom H(\lambda)=\ran (A_0-\lambda)$ is dense in $\sH$.
Since $\ker\Gamma^G_0=A_0\subset\dom
\Gamma^G_1$ and $\mul \Gamma^G_1=\{0\}$, it follows from Proposition~\ref{prop:C3} that
\begin{equation}\label{ess03}
 \Gamma^G_1 H(\lambda)=G^*\Gamma_1 H(\lambda)
 =G^*\gamma(\bar\lambda)^*\subset (\gamma(\bar\lambda)G)^*,
 \quad \lambda\in \cmr.
\end{equation}
By the construction of $G$ the domain of $\gamma(\mu)G$ is dense in
$\cH$ for some $\mu\in\cmr$. Therefore,~\eqref{ess03} implies that
$\Gamma^G_1 H(\bar\mu)$ is a bounded densely defined operator for
some $\mu\in\cmr$ and, since $A_0$ is essentially selfadjoint,
Lemma~\ref{cor:GH} shows that $\Gamma^G_1 H(\lambda)$ is bounded and
densely defined for all $\lambda\in\cmr$.

{\bf 3.} {\it Verification of (1):}
Now consider the closure $\wt\Gamma$ of $\Gamma^G$ in~\eqref{ess02}.
It is shown below that $\ker \wt\Gamma_0=\overline{A_0}$, which means that
$\ker \wt\Gamma_0$ is selfadjoint ({\it Assumption~\ref{1_def_EgenBT}}), since $A_0$ is essentially
selfadjoint by assumption. By construction $\Gamma^G$ is defined via
the transform $\Gamma^G=\{G^{-1}\Gamma_0,G^*\Gamma_1\}$ of
$\{\Gamma_0,\Gamma_1\}$. It follows from Lemma~\ref{isoHlem} (see
also Remark~\ref{rem:C3}) that the graph of $\Gamma^G$ contains all
elements of the form
\begin{equation}\label{ess04}
 \{\wh f,\wh k\}
 = \left\{ H(\lambda)k_\lambda, \binom{0}{G^*\gamma(\bar{\lambda})^*k_\lambda} \right\}
  + \left\{ \binom{\gamma(\lambda)Gh}{\lambda\gamma(\lambda)Gh},\binom{h}{G^*M(\lambda)Gh} \right\},
\end{equation}
where $k_\lambda\in\ran(A_0-\lambda)$ and $h\in \dom G^*
M(\lambda)G=\dom\gamma(\lambda)G$, $\lambda\in\cmr$. Let $\wh
h=\{h,h'\}\in \overline{A_0}$ and let $k\in\ran
(\overline{A_0}-\lambda)$ be such that $\wh h=
\overline{H(\lambda)}k$, where $\overline{H(\lambda)}$ corresponds
to the graph of $\overline{A_0}$; see~\eqref{Hlambda}. Moreover, let
$k_n\in \ran (A_0-\lambda)$ be a sequence such that $k_n\to k$ as
$n\to \infty$. Then $H(\lambda)k_n\to \wh h \in \overline{A_0}$,
since $H(\lambda)$ is bounded. Moreover, by boundedness of
$\Gamma^G_1 H(\lambda)=G^*\gamma(\bar\lambda)^*$
\[
 \wh h_n=\Gamma^G H(\lambda)k_n=\{ 0,G^*\gamma(\bar\lambda)^*k_n\} \to \{0,g\},\quad g\in \cH.
\]
Since $\wt\Gamma$ is closed, it follows that $\{\wh h;\{0,g\}\}\in
\wt\Gamma$ which shows that $\wh h\in \ker \wt\Gamma_0$. Hence,
$\overline{A_0}\subset \ker \wt\Gamma_0$ and since $\ker\wt\Gamma_0$
is symmetric this implies that $\ker
\wt\Gamma_0=\overline{A_0}=A_0^*$.

Since $\cdom \Gamma^G=\ZA^*$, the closure $\wt \Gamma$ has dense
domain in $\ZA^*$ ({\it Assumption~\ref{1_def_unitBT}}). Clearly,
$\dom G^*M(\mu)G=\dom \gamma(\mu)G\subset \ran \Gamma^G_0$ and hence
the ranges of $\Gamma^G_0$ and $\wt\Gamma_0$ are dense in $\cH$
({\it Assumption~\ref{2_def_ABGtriple00}}). Furthermore, $\wt\Gamma$
as the closure of $\Gamma^G$ is also isometric, i.e., Green's
formula~\eqref{Green1} holds for $\wt\Gamma$ ({\it
Assumption~~\ref{1_def_ABGtriple00}}). According to
Definition~\ref{ABGtriple} this means that $\wt \Gamma$ is an
$AB$-generalized boundary pair for $\ZA^*$.

{\bf 4.} {\it Verification of (2):}
The form of the Weyl function $\wt M(\lambda)=E+ M_0(\lambda)$ of
$\wt \Gamma$ is obtained from Theorem \ref{ABGthm}. Furthermore, by
Theorem~\ref{QBTthm} $\wt\Gamma$ is closed if and only if $E$ is
closed or, equivalently, $\wt M(\lambda)$, $\lambda\in\cmr$, is
closed.

{\bf 5.} {\it Verification of (3):}
Since $\Gamma^G_1 H(\lambda)=G^*\gamma(\bar\lambda)^*$ and
$\gamma_G(\lambda):=\gamma(\lambda)G$ are bounded and densely
defined for each $\lambda\in\cmr$, if follows from~\eqref{ess04}
that
\[
 \wh\Gamma^G:=
 \left\{\left\{ \overline{H(\lambda)}k_\lambda, \binom{0}{\overline{G^*\gamma(\bar{\lambda})^*} k_\lambda} \right\}
  + \left\{ \binom{\overline{\gamma_G(\lambda)}h}{\lambda\overline{\gamma_G(\lambda)}h},\binom{h}{\overline{M_G(\lambda)}h} \right\};
   \begin{array}{c} k_\lambda\in\ran(\overline{A}_0-\lambda), \\ h\in \dom \overline{M_G(\lambda)} \end{array} \right\}
  \subset \wt\Gamma,
\]
where $\gamma_G(\lambda)$ and $M_G(\lambda):=G^*M(\lambda)G$ are the
$\gamma$-field and the Weyl function of $\Gamma^G$. Notice that
$\overline{A_0}\subset\dom\wh\Gamma^G$ and, as shown above, $\ran
\wh\Gamma^G_0\supset \dom \gamma(\mu)G$ is dense in $\cH$ ({\it
Assumption~\ref{2_def_ABGtriple00}-\ref{3_def_ABGtriple00}}). Due to
$\wh\Gamma^G\subset \wt\Gamma$ also Green's identity~\eqref{Green1}
is satisfied ({\it Assumption~\ref{1_def_ABGtriple00}}). Therefore,
$\wh\Gamma^G$ is also an $AB$-generalized boundary pair whose Weyl
function is clearly $\overline{M_G(\lambda)}$, which is closed. Now
by Theorem~\ref{QBTthm} the $AB$-generalized boundary pair
$\wh\Gamma^G$ is also closed and, since $\wh\Gamma^G\subset
\wt\Gamma$, one has
\[
 \overline{G^*M(\lambda)G}\subset\wt M(\lambda)=E+
 M_0(\lambda),\quad \lambda\in\cmr.
\]
Now $\overline{G^*M(\lambda)G}$ as a closed restriction of $E+
 M_0(\lambda)$ is of the form $\overline{G^*M(\lambda)G}=E_0+
 M_0(\lambda)$, $\lambda\in\cmr$, where $E_0$ is a closed densely defined restriction of
$E$; cf. Theorem~\ref{ABGthm}. This proves the last statement.
\end{proof}


Theorem \ref{essThm2} remains valid for all form domain invariant
Nevanlinna functions $M(\cdot)\in \cR(\cH)$ that need not be strict.
The only essential difference appearing in the proof of Theorem
\ref{essThm2} in this case is that $\ker \gamma(\lambda)=\mul
\Gamma_0$ is nontrivial, and then also, $\ker \gamma_G(\lambda)=\ker
\gamma(\lambda)G$ is nontrivial. Notice that even if $\ker
\gamma(\lambda)=\{0\}$ (i.e. $M(\cdot)\in \cR^s(\cH)$) then the
$\gamma$-field $\wt\gamma(\lambda)$ as well as its closure
$\overline{\wt\gamma(\lambda)}=\overline{\gamma_G(\lambda)}$ can
have a nontrivial kernel. This explains why the constructed boundary
pair $\wt\Gamma$ can in general be multivalued even if the original
boundary triple $\Gamma=\{\Gamma_0,\Gamma_1\}$ is single-valued.

Theorem \ref{essThm2} combined with the next lemma yields an
explicit representation for the class of form domain invariant
Nevanlinna functions as well as form domain invariant Nevanlinna
families.


\begin{lemma}\label{essLemma}
Let $G$ be a bounded operator in the Hilbert space $\cH$ with $\ker
G=\ker G^*=\{0\}$, let $H$ be a closed symmetric densely defined
operator on $\cH$ and let $M_0(\cdot)\in \cR[\cH]$. Then the function
\[
 M(\lambda)= G^{-*}(H+M_0(\lambda))G^{-1}, \quad \lambda\in\cmr,
\]
is form domain invariant if and only if for some, equivalently for
every, $\lambda\in\cmr$ the set
\[
 \sD_\lambda:=\{ h\in\cH:\, (\IM M_0(\lambda))^{\half}h\in \ran G^*\}
\]
is dense in $\cH$.
\end{lemma}
\begin{proof}
To calculate the form $\st_{M(\lambda)}$ let $\lambda\in\cmr$ be
fixed and let $u,v\in \dom M(\lambda)$. Then $u,v\in \dom G^{-1}$
and hence
\[
\begin{split}
 \st_{M(\lambda)}[u,v] & =\frac{1}{\lambda-\bar\lambda}\,
 [\left((H+M_0(\lambda))G^{-1}u,G^{-1}v\right)-\left(G^{-1}u,(H+M_0(\lambda))G^{-1}v\right)]\\
  &=\frac{1}{\IM \lambda}\,\left((\IM M_0)(\lambda)G^{-1}u,G^{-1}v\right)
  =\frac{1}{\IM \lambda}\,\left((\IM M_0(\lambda))^{\half}G^{-1}u,(\IM M_0(\lambda))^{\half}G^{-1}v\right)
\end{split}
\]
where symmetry of $H$ has been used. This form is closable precisely
when the operator $(\IM M_0(\lambda))^{\half}G^{-1}$ is closable or,
equivalently, its adjoint $((\IM
M_0(\lambda))^{\half}G^{-1})^*=G^{-*}(\IM M_0(\lambda))^{\half}$ is
densely defined. Since $\sD_\lambda=\dom G^{-*}(\IM
M_0(\lambda))^{\half}$, the closability of $\st_{M(\lambda)}$ is
equivalent for $\sD_\lambda$ to be dense in $\cH$.

To prove that this criterion does not depend on $\lambda\in\cmr$
consider $M_0(\cdot)$ as the Weyl function of some $B$-generalized
boundary pair $(\cH,\Gamma)$. Let $\gamma_0(\cdot)$ be the
corresponding $\gamma$-field and let $A_0=\ker \Gamma_0$ be the
associated selfadjoint operator. Then the form
$\st_{M(\lambda)}[u,v]$ can be also rewritten in the form
\[
 \st_{M(\lambda)}[u,v]=\left(\gamma_0(\lambda)G^{-1}u,\gamma_0(\lambda)G^{-1}v\right)
\]
and hence the form $\st_{M(\lambda)}[u,v]$ is closable if and only
if $\gamma_0(\lambda)G^{-1}$ is a closable operator. Now for any
$\lambda,\mu\in\cmr$ one has
\[
  (I+(\lambda-\mu)(A_0-\lambda)^{-1})\gamma_0(\mu)G^{-1}=\gamma_0(\lambda)G^{-1},
\]
and since $I+(\lambda-\mu)(A_0-\lambda)^{-1}$ bounded with bounded
inverse, one concludes that $\gamma_0(\mu)G^{-1}$ is closable
exactly when $\gamma_0(\lambda)G^{-1}$ is closable and that the
closures are connected by
\[
  (I+(\lambda-\mu)(A_0-\lambda)^{-1})\overline{\gamma_0(\mu)G^{-1}}=\overline{\gamma_0(\lambda)G^{-1}}.
\]
Therefore, if $\st_{M(\mu)}[u,v]$ is closable for some $\mu \in
\cmr$ then $\st_{M(\lambda)}[u,v]$ is closable for all $\lambda \in
\cmr$ and the form domains of these closures coincide. This
completes the proof.
\end{proof}

\begin{proposition}\label{essCor1}
Let $M$ be a strict form domain invariant operator valued Nevanlinna
function in the Hilbert space $\cH$. Then there exist a bounded
operator $G\in[\cH]$ with $\ker G=\ker G^*=\{0\}$, a closed
symmetric densely defined operator $E$ in $\cH$, and a bounded
Nevanlinna function $M_0(\cdot)\in \cR[\cH]$ with the property
\begin{equation}\label{MclH}
 \cH=\clos\sD_\lambda:=\clos \{ h\in\cH:\, (\IM M_0(\lambda))^{\half}h\in \ran G^*\}, \quad
 \lambda\in\cmr,
\end{equation}
such that $M(\cdot)$ admits the representation
\begin{equation}\label{Mess}
 M(\lambda)= G^{-*}(E+M_0(\lambda))G^{-1}, \quad \lambda\in\cmr.
\end{equation}
Conversely, every Nevanlinna function $M(\cdot)$ of the form
~\eqref{Mess} is form domain invariant $\cmr$, whenever $E\subset
E^*$ and $G\in\cB(\cH)$, $\ker G=\ker G^*=\{0\}$, and $M_0(\cdot)\in
\cR[\cH]$ satisfy the condition~\eqref{MclH}.
\end{proposition}
\begin{proof}
Let the Nevanlinna function $M\in \cR(\cH)$ be realized as the Weyl
function of some boundary pair $\{\cH,\Gamma\}$ (see
Theorem~\ref{essThm1}, \cite[Theorem~3.9]{DHMS06}). Since $M$ is
form domain invariant, $A_0$ is essentially selfadjoint by
Theorem~\ref{essThm1}. Since $M$ is an operator valued Nevanlinna
function, one can apply Theorem \ref{essThm2} (see also the
discussion after Theorem \ref{essThm2}), which shows that the
inclusion ${G^*M(\lambda)G}\subset E+ M_0(\lambda)$ holds for every
$\lambda\in\cmr$. This implies that
\begin{equation}\label{ess06}
 M(\lambda)=  G^{-*}G^*M(\lambda)GG^{-1} \subset G^{-*}(E+M_0(\lambda))G^{-1},
\end{equation}
where $G$ is a bounded operator with $\ker G=\ker G^*=\{0\}$ (cf.
proof of Theorem \ref{essThm2} where $\cran G=\cH$ by construction).
Clearly, the function $G^{-*}(E+M_0(\lambda))G^{-1}$ is dissipative
for $\lambda\in\dC_+$ and accumulative for $\lambda\in\dC_-$. Since
$M$ is Nevanlinna function, it is maximal dissipative in $\dC_+$ and
maximal accumulative in $\dC_-$. Therefore, the inclusion in
~\eqref{ess06} prevails as an equality. Since $M(\cdot)$ is form
domain invariant Lemma \ref{essLemma} shows that the condition
~\eqref{MclH} holds for every $\lambda\in\cmr$.

Conversely, if $M(\cdot)$ is a Nevanlinna function of the form
~\eqref{Mess}, where $E$, $G$ and $M_0(\cdot)$ are as indicated and
the condition~\eqref{MclH} holds for some $\lambda\in\cmr$, then by
Lemma \ref{essLemma} $M(\cdot)$ is form domain invariant and the
condition holds for every $\lambda\in\cmr$.
\end{proof}

\begin{remark}
As to the renormalization in Theorem~\ref{essThm2} we do not know if
the renormalized function $\wt M(\cdot)=E+M_0(\cdot)$ belongs to the
class of Nevanlinna functions.

However, the representation of $M(\cdot)$ in
Proposition~\ref{essCor1} combined with $E\subset E^*$ leads to
\[
 M(\lambda)=M(\bar\lambda)^* \supset G^{-*}(E^*+M_0(\lambda))G^{-1}
  \supset M(\lambda), \quad \lambda\in\cmr.
\]
Hence, $M(\lambda)$ can also be represented with $E^*$ instead of
$E$ as follows:
\[
 M(\lambda)=G^{-*}(E^*+M_0(\lambda))G^{-1}, \quad \lambda\in\cmr.
\]
In particular, if $\wt E$ is any maximal symmetric extension of $E$
then one has also the representation
\[
 M(\lambda)=G^{-*}(\wt E+M_0(\lambda))G^{-1}, \quad \lambda\in\cmr.
\]
\end{remark}

\begin{remark}
The result in Proposition~\ref{essCor1} remains valid also for form
domain invariant Nevanlinna families. In this case there exist a
bounded operator $G\in[\cH]$ with $\ker G=\ker G^*=\mul M(\lambda)$,
$\lambda\in\cmr$, a closed symmetric densely defined operator $E$ in
$\cH$, and a bounded Nevanlinna function $M_0(\cdot)\in \cR[\cH]$
satisfying~\eqref{MclH}, such that
\[
 M(\lambda)= G^{-*}(E+M_0(\lambda))G^{-1}, \quad \lambda\in\cmr.
\]
To see this, decompose $M(\lambda)=\textup{gr }M_{\rm op}(\lambda)+M_\infty$,
where $M_\infty=\{0\}\times \mul M(\lambda)$, $\lambda\in\cmr$, see
~\eqref{ml}. Now as in the proof of Proposition \ref{essCor1} the
operator part $M_{\rm op}(\lambda)$ admits the representation
$M_{\rm op}(\lambda)=G_0^{-*}(E+M_0(\lambda))G_0^{-1}$ with some
operator $G_0\in[\cH_0]$ in $\cH_0=\cH \ominus \mul M(\lambda)$ with
$\ker G_0=\ker G_0^*=\{0\}$. The desired representation of $M$ is
obtained by letting $G$ to be the zero continuation of $G_0$ from
$\cH_0$ to $\cH=\cH_0\oplus \mul M(\lambda)$.
\end{remark}

The next example contains a wide class of $ES$-generalized boundary
triples and demonstrates the regularization procedure formulated in
Theorem~\ref{essThm2}.

\begin{example}\label{example6.1}
Let $\{\cH,\Gamma_0^{0},\Gamma_1^{0}\}$  be an ordinary boundary
triple $\Pi^0=\{\cH,\Gamma_0^{0},\Gamma_1^{0}\}$ for $\ZA^*$ with
$A_0^0=\ker \Gamma_0^0$, $A_1^0=\ker\Gamma_1^0$, let $M_0(\cdot)$
and $\gamma_0(\cdot)$ be the corresponding Weyl function and the
$\gamma$-field, and let $G\in \cB(\cH)$ with $\ker G=\ker
G^*=\{0\}$. Then the transform
\begin{equation}\label{eq:001}
    \begin{pmatrix}
    \Gamma_0 \\ \Gamma_1
    \end{pmatrix}=
    \begin{pmatrix}
    G & 0 \\
    0 & G^{-*}
    \end{pmatrix}\begin{pmatrix}
    \Gamma_0^{0} \\ \Gamma_1^{0}
    \end{pmatrix},
\end{equation}
where $G^{-*}$ stands for $(G^{-1})^*=(G^{*})^{-1}$, defines an
$ES$-generalized boundary triple $\Pi=\{\cH,\Gamma_0,\Gamma_1\}$ for
$\ZA^*$. Indeed, since $G\in\cB(\cH)$ the transform $V$ in
~\eqref{Vtrans} is unitary in the Kre\u{\i}n space $\{\cH^2,J_\cH\}$
and it follows from \cite[Theorem~2.10~(ii)]{DHMS09} that the
composition $\Gamma=V\circ \Gamma^0$ is unitary. By Lemma
\ref{isomtrans} one has $\ker \Gamma=\ZA$ and, since $\Gamma$ is
unitary, $\ZA_*:=\dom \Gamma$ is dense in $\ZA^*$. Since $\Pi^0$ is
an ordinary boundary triple, $\cH\times \{0\}\subset \ran \Gamma^0$
and hence one concludes from~\eqref{eq:001} that
\[
 \ran \Gamma_0=\ran G, \quad A_0:=\ker \Gamma_0=A_0^0\cap \ZA_*.
\]
Consequently, $\ran\Gamma_0$ is dense in $\cH$ and $A_0$ is
essentially selfadjoint. Moreover, $A_1:=\ker\Gamma_1=\ker
\Gamma_1^0=A_1^0$ and $\ran \Gamma_1=\dom G^*=\cH$: this means that
the transposed boundary triple $\{\cH,\Gamma_1,-\Gamma_0\}$ is
$B$-generalized. Observe, that $A_0$ is selfadjoint if and only if
$\ran G=\cH$ or, equivalently, when $\Pi$ is an ordinary boundary
triple for $\ZA^*$, too.

Next the form domain of the Weyl function $M$ is calculated. By
Lemma~\ref{isomtrans} $M(\cdot)=G^{-*}M_0(\cdot)G^{-1}$ and
$\gamma(\cdot)=\gamma_0(\cdot)G^{-1}$. Let $\lambda\in\cmr$ be fixed
and let $u,v\in\dom M(\lambda)$. Then
\[
\begin{array}{ll}
 \st_{M(\lambda)}[u,v]
 &=\frac{1}{\lambda-\bar\lambda}\, [(G^{-*}M_0(\lambda)G^{-1}u,v)-(u,G^{-*}M_0(\lambda)G^{-1}v)]\\
 &=\frac{1}{\lambda-\bar\lambda}\, [(M_0(\lambda)G^{-1}u,G^{-1}v)-(M_0(\lambda)^*G^{-1}u,G^{-1}v)]\\
 &=(\gamma_0(\lambda)G^{-1}u,\gamma_0(\lambda)G^{-1}v).
 \end{array}
\]
Since $\Pi^0$ is an ordinary boundary triple,
$\gamma_0(\lambda):\cH\to\ker(\ZA^*-\lambda)$ is bounded and
surjective, i.e., the inverse of this mapping is also bounded. Hence
$\gamma_0(\lambda)G$ is closed, when considered on its natural
domain $\dom \gamma_0(\lambda)G^{-1}=\ran G\,(\supset \dom
M(\lambda))$. Therefore, the closure of the form $\st_{M(\lambda)}$
is given by
\[
 \overline{\st_{M(\lambda)}}[u,v]=(\gamma(\lambda)G^{-1}u,\gamma(\lambda)G^{-1}v), \quad
 u,v\in\ran G.
\]
In particular, $M(\lambda)$ is a form domain invariant Nevanlinna
function whose form domain is equal to $\ran G$. Since $G$ is
bounded, one can use $G$ to produce a regularized function $\wt M$:
\[
 \wt
 M=G^*MG=G^{*}(G^{-*}M_0(\cdot)G^{-1})G=M_0(\lambda),
\]
so that $\wt M$ coincides with the Nevanlinna function $M_0(\cdot)$
which belongs to the class $\cR^u[\cH]$.

It is emphasized that when $G$ is not surjective, the form domain
invariant function $M(\cdot)=G^{-*}M_0(\cdot)G^{-1}$ need not be
domain invariant. In fact, in \cite{DHM15} an example of a form
domain invariant Nevanlinna function $M$ was given, such that
\[
 \dom M(\lambda)\cap \dom M(\lambda)=\{0\}, \quad \lambda\neq\mu \quad
 (\lambda,\mu\in\cmr),
\]
and the corresponding regularized function $\wt M$ therein still
belongs to the class $\cR^u[\cH]$.
\end{example}

In Example~\ref{example6.1} the boundary triple $\Pi$ is
$ES$-generalized and the transposed boundary triple
$\Pi^\top:=\{\cH,\Gamma_1,-\Gamma_0\}$ is $B$-generalized.
Therefore, according to \cite[Theorem~7.24]{DHMS12}, or
Theorem~\ref{thm:2.1}, there exist an ordinary boundary triple $\wt
\Pi^0$ and operators $R=R^*, K\in \cB(\cH)$, $\ker K=\ker
K^*=\{0\}$, such that $\Pi^\top$ is the transform~\eqref{eq:3} of
$\wt \Pi^0$. Recall that one can take e.g. $R=\RE(-M(i)^{-1})$,
$K=(\IM (-M(i)^{-1}))^{1/2}$. In particular, this yields the
following connections between the associated Weyl functions:
\[
 -M^{-1}(\cdot)=K^* \wt M_0(\cdot)K+R.
\]
In particular, if $R=0$ then one obtains $M(\cdot)=K^{-1}(-\wt
M_0(\cdot)^{-1})K^{-*}$, where $-\wt M_0(\cdot)^{-1}$ belongs to the
class $\cR^u[\cH]$.

This together with Example~\ref{example6.1} essentially
characterizes those $ES$-generalized boundary triples $\Pi$ for
$\ZA^*$ whose transposed boundary triple $\Pi^\top$ is
$B$-generalized. \\

Recall that Weyl functions of $S$-generalized boundary pairs are
domain invariant, but converse does not hold; for explicit examples
see Section~\ref{sec7}. As shown in the next proposition a domain
invariant Nevanlinna function can be always renormalized by means of
a fixed bounded operator to a bounded Nevanlinna function.

\begin{proposition}\label{Domrenorm}
Let $M(\cdot)$ be a domain invariant operator valued Nevanlinna
function in the Hilbert space $\cH$. Moreover, let $G$ with $\ker
G=\ker G^*=\{0\}$ be a bounded operator in $\cH$ such that $\ran
G=\dom M(\lambda)$, $\lambda\in\cmr$. Then the renormalized function
\begin{equation}\label{boundedM_G}
 M_G(\lambda)= G^{*}M(\lambda)G, \quad \lambda\in\cmr,
\end{equation}
is a bounded Nevanlinna function. Moreover, $M_G(\cdot)\in
\cR^s[\cH]$ precisely when $M(\cdot)\in \cR^s(\cH)$.
\end{proposition}
\begin{proof}
By assumptions the equality $\dom G^{*}M(\lambda)G=\dom
M(\lambda)G=\cH$ holds for all $\lambda\in\cmr$. Consequently, the
adjoint $M_G(\lambda)^*$ is a closed operator and in view of
\[
 M_G(\lambda)^*=(G^{*}M(\lambda)G)^*\supset G^{*}M(\bar\lambda)G
\]
one has $\dom M_G(\lambda)^*=\cH$. Consequently, the equality
$M_G(\lambda)^*=G^{*}M(\bar\lambda)G$ holds for all
$\lambda\in\cmr$. Now clearly $\IM M_G(\lambda)=G^*\IM M(\lambda)G$,
which implies that $M_G(\cdot)\in \cR[\cH]$ and also shows the last
statement in the proposition.
\end{proof}

The assumption $\ran G=\dom M(\lambda)$ in
Proposition~\ref{Domrenorm} (or more generally the inclusion $\dom
M(\lambda)\subset \ran G$) guarantees that $M(\cdot)$ can be
recovered from $M_G(\cdot)$ in~\eqref{boundedM_G} similarly as was
done in Proposition~\ref{essCor1}:
\[
  M(\lambda)= G^{-*}G^*M(\lambda)GG^{-1}=G^{-*}M_G(\lambda)G^{-1}, \quad \lambda\in\cmr.
\]

\subsection{Examples on renormalization}

The following examples demonstrate different renormalizations of
some form domain invariant Nevanlinna functions. In the first
example the real part of $M(i)$ is strongly subordinated with
respect to its imaginary part. In this case the renormalized
function $\wt M(\cdot)$ is a bounded Nevanlinna function.

   \begin{example}\label{Example_first}
Let $S$ be a positively definite closed symmetric operator in $\cH,$
$S\ge \varepsilon I.$ Let
\begin{equation}\label{5.35_Nev_func}
M(z)= z S^*S + S, \quad  \dom M(z)=\dom S^*S, \quad z\in\Bbb C.
\end{equation}
Replacing if necessary $S$ by  $S +aI$ we can  assume that
$\varepsilon
>1.$

First notice that
    \begin{equation*}
  \|f\|^2\le\varepsilon^{-2}\|S f\|^2 = \varepsilon^{-2}(S^* Sf,f) \le
\varepsilon^{-2}\|S^* Sf\| \cdot \|f\|,\quad  f\in\dom S^*S,
    \end{equation*}
i.e. $\|S^* S f\|\ge\varepsilon^2\|f\|$. It follows that $S$ is
strongly subordinated with respect to $S^* S$, i.e.
   \begin{equation*}
\|S f\|^2 = (Sf,Sf) = (S^* Sf, f) \le \|S^* Sf\| \cdot \|f\| \le
\varepsilon^{-2}\|S^* Sf\|^2, \quad  f\in\dom S^*S.
   \end{equation*}
Since $\dom S^*S \subset\dom S\subset\dom S^*$, one easily proves
that $S^*$ is also strongly subordinated with respect to $S^* S.$
Now, these inequalities imply that  both operators $S/z$ and $S^*/z$
are also strongly subordinated to $S^* S$ for $|z|\ge 1.$ Therefore, 
%
%
    \begin{equation*}
M(\overline z)^* = (\overline z S^* S+S)^* = z S^*S + S^* = z S^*S +
S = M(z).
    \end{equation*}
Since $M(\cdot)$ is dissipative in $\Bbb C_+$, it follows that
$M(z)$ is maximal dissipative for $z\in \Bbb C_+, |z|\ge 1,$ and
maximal accumulative for $z\in \Bbb C_-,\ |z|\ge 1$. In turn, the
latter implies  that $M(z)$ being holomorphic and dissipative is
$m$-dissipative for each $z\in \Bbb C_+$. Summing up we conclude
that $M(\cdot)$ is an entire Nevanlinna function with values in
$\cC(\cH)$.

Furthermore,
  \begin{equation*}
    {\mathfrak t}_{M(z)}(f,g) = \frac{(M(z)f,g)-(f,M(z)g)}{z-\bar{z}}
    = (Sf, Sg),\quad
 f,g\in\dom S^*S, \quad z\in\dC.
\end{equation*}
The form is closable because so is the operator $S$. Taking the
closure we obtain the closed form  ${\overline {\mathfrak
t}}_{M(z)}(f,g)  = (Sf, Sg),\
 f,g\in\dom S, \ z\in\dC$, with constant domain.
In other words,  $M(\cdot)$ is  a form domain invariant Nevanlinna
function and the (selfadjoint) operator associated with ${\overline
{\mathfrak t}}_{M(z)}$ in accordance with the second representation
theorem is $(S^*S)^{1/2}$.

Now consider the renormalization of $M(\cdot)$ as in
Theorem~\ref{essThm2}. The operator $G=(S^*S)^{-\half}$ is bounded
and $\ran G=\dom {\overline{\mathfrak t}}_{M(z)}$. Moreover,
$G^*(S^*S)G=I\uphar{\dom (S^*S)^\half}$ and $G^*SG=G^*U$, where
$U:\cran (S^*S)^\half=\cH\to \cran S$ is the (partial) isometry from
the polar decomposition $S=U(S^*S)^{\half}$. Consequently,
$C:=G^*SG$ is a bounded selfadjoint operator in $\cH$. By
Theorem~\ref{essThm2} one has $\wt M(z)\supset\clos(G^*M(z)G)=zI+C$.
Thus, $\wt M(z)=zI+C$ is a bounded Nevanlinna function.
\end{example}

In the next example we change the roles of the real and imaginary
parts of $M(i)$ of the function treated in Example
\ref{Example_first}. This leads to a renormalized Nevanlinna
function which is unbounded.

\begin{example}\label{Example_second}
Consider the entire operator function
\[
 M_1(z)=S^*S+z S, \quad z\in \dC.
\]
Then $M_1(z)$ is a Nevanlinna function; cf.
Example~\ref{Example_first}. It is domain invariant and also form
domain invariant:
\[
  (M_1(z)f,g)-(f,M_1(z)g)=(z-\bar z)(Sf,g),
  \quad f,g\in \dom S^*S.
\]
The closure of this form is given by
\[
  \tau_{M_1(z)}(f,g)=(S_F^{1/2}f,S_F^{1/2}g), \quad f,g\in \dom (S_F)^{\half}.
\]
The operator $S_F^{-1/2}$ is bounded, injective, and clearly $\ran
S_F^{-1/2}=\dom \tau_{M_1(z)}$. Consider the renormalization of this
function determined by $G=S_F^{-1/2}$:
\[
  G^*M_1(z)G=S_F^{-1/2}S^*S S_F^{-1/2} +z\, S_F^{-1/2}S S_F^{-1/2}.
\]
Here
\[
 \dom (S_F^{-1/2}S^*S S_F^{-1/2})\subset \dom (S_F^{-1/2}S S_F^{-1/2})
  =S_F^{1/2}(\dom S).
\]
By the first representation theorem the operator $S_F^{-1/2}S
S_F^{-1/2}$ is densely defined and
\[
 S_F^{-1/2}S S_F^{-1/2}\subset S_F^{-1/2}S_F S_F^{-1/2}=I_{\dom
 S_F^{1/2}}.
\]
Hence $T_1:=S_F^{-1/2}S S_F^{-1/2}$ is bounded and its closure is
the identity operator on $\cH$. On the other hand, for
$T_0:=S_F^{-1/2}S^*S S_F^{-1/2}$ one has
\[
 T_0=S_F^{-1/2}S^*S S_F^{-1}S_F^{1/2}
 =S_F^{-1/2}S^*S S^{-1}S_F^{1/2}=S_F^{-1/2}(S^*\uphar
 \ran S)S_F^{1/2}.
\]
Here $H:=S^*\uphar \ran S$ is a closed restriction of $S^*$ with
nondense domain in $\cH$ and its adjoint is given by
\[
 H^*=(S^*\uphar \ran S)^*
 =S \hplus \left( \{0\}\times \ker S^* \right).
\]
Consequently,
\[
 T_0=S_F^{-1/2}H S_F^{1/2},\quad T_0^*=(S_F^{-1/2}H S_F^{1/2})^*=S_F^{1/2}H^* S_F^{-1/2}.
\]
One can rewrite $T_0^*$ as follows
\[
 T_0^*=S_F^{1/2}\left[S \hplus \left( \{0\}\times \ran S\right)\right]S_F^{-1/2}.
\]
Here $\ker S^* \cap \dom S_F^{1/2}\subset \dom S_F=\{0\}$, since
$0\in\rho(S_F)$. Hence $T_0^*$ is an operator and $T_0$ is a densely
defined nonnegative operator in $\cH$. Moreover, $\ran T_0^*=\cH$
and $\ran T_0=S_F^{-\half}(\ran H)=\dom S_F^{\half}$ is dense in
$\cH$, so that $\ran (T_0)^{**}=\cH$. Hence, $T_0$ is essentially
selfadjoint. From Theorem~\ref{essThm2} one gets
\[
 \wt M_1(z)\supset\overline{G^*M_1(z)G}=\clos(T_0+z\, T_1)=(T_0)^{**}+z\, I, \quad
 z\in\dC.
\]
Thus, $\wt M_1(z)$ is an unbounded domain invariant Nevanlinna
function, whose imaginary part is bounded.
\end{example}

As a comparison we consider another renormalization of the function
$M(z)= z S^*S + S$ from Example~\ref{Example_first}, which leads to
a renormalized function that is, in fact, a multivalued Nevanlinna
family. The situation is made more concrete by treating as a special
case the second order differential operator $S=-D^2$ on $L^2[0,1]$.

   \begin{example}\label{Example_third}
(i) Let $S$ and $M(z)= z S^*S + S$ be as in
Example~\ref{Example_first}. Consider another renormalization of
$M(\cdot)$ using the bounded operator $G_2:=(S_F)^{-1}$. For
simplicity we assume in addition that $\overline{\dom S^2}=\cH$,
which implies that $S^*S = (S^2)_F$. Since $S^2\subset (S_F)^2$ one
concludes that $S^*S\geq (S_F)^2$ and, in particular, $\dom S=\dom
(S^*S)^\half \subset \dom S_F$ which shows that $\ran G_2\supset\dom
{\overline{\mathfrak t}}_{M(z)}=\dom S$. It is easily seen that
$S(S_F)^{-1}=I\uphar \ran S$ and, since
$((S_F)^{-1}S^*)^*=S(S_F)^{-1}$, the operator
\[
 T_1 := (S_F)^{-1}S^* S(S_F)^{-1}
\]
is nonnegative, nondensely defined and nonclosable. The closure of
$T_1$ is a nonnegative selfadjoint relation given by
\[
 \clos T_1= I_{\ran S} \oplus (\{0\}\times \ker S^*).
\]
On the other hand, the operator $T_0$,
$$
T_0f := (S_F)^{-1}S(S_F)^{-1}f = S^{-1}f, \quad f\in \cH,
$$
is bounded and nondensely defined. Hence $M_2(z):= T_1 z +T_0$ as a
nondensely defined unbounded operator function, whose closure
\begin{equation}\label{renorm3}
 \clos (M_2(z))=z\, (T_1)^{**} + S^{-1} = (z I_{\ran S}+ P_{\ran S}S^{-1}) \oplus (\{0\}\times \ker S^*)
\end{equation}
is a Nevanlinna family in the Hilbert space $\cH$.

It should be noted that here the corresponding $\gamma$-fields would
be $\gamma(z)\equiv S$ and $\gamma_2(z)\equiv S(S_F)^{-1}=I\uphar
\ran S$, and this last one is closed but nondensely defined; cf.
Remark~\ref{essrem1}.

(ii)
As a special case consider $S=-D^2,\  \dom S=H^2_0[0,1]$. Then
\begin{eqnarray*}
S_F = -D^2,\quad  \dom S_F = H^2[0,1]\cap H^1_0[0,1],\\
S^* S = (S^2)_F = D^4,\quad \dom S^*S =  H^4[0,1]\cap H^2_0[0,1],
 \end{eqnarray*}
and the operator valued function $M(\cdot)$ given by
~\eqref{5.35_Nev_func} is a Nevanlinna function.

It is easily seen that
   \begin{equation*}
(S_F)^{-1}f = (x-1)\int^x_0 t f(t)dt + x\int^1_x(t-1)f(t)dt, \quad
f\in H^0[0,1] = L^2[0,1],
   \end{equation*}
and
   \begin{equation*}
S^* S(S_F)^{-1}f = f'', \quad \dom\bigl(S^*
S(S_F)^{-1}\bigr)=\left\{f\in H^2[0,1]:\ f\perp 1,f\perp t\right\}.
  \end{equation*}
Finally,
 \begin{equation*}
T_1f = (S_F)^{-1}S^* S(S_F)^{-1}f=f(x)-(1-x)f(0)-x f(1)
\end{equation*}
and $\dom T_1 = \dom\bigl(S^* S(S_F)^{-1}\bigr)= \left\{f\in
H^2[0,1]:\ f\perp 1,f\perp t\right\}$.

Clearly, $T_1$  is nondensely defined and nonclosable in $L^2[0,1]$.
On the other hand, the operator  $T_0$,
$$
T_0f := (S_F)^{-1}S(S_F)^{-1}f = S^{-1}f, \quad \dom T_0 = \{f\in
L^2[0,1]:\ f\perp 1,f\perp t\}.
$$
is nondensely defined while it is bounded.

Thus, $M_1(z) = z\, T_1 + T_0$ is not a Nevanlinna function.
However, its closure is a Nevanlinna family of the form
~\eqref{renorm3} whose multivalued part is spanned by the functions
$g_0(t)\equiv 1$ and $g_1(t)=t$ in $L^2[0,1]$.
     \end{example}
%
\begin{remark}
The situation treated in Example~\ref{Example_third} could be
recovered with a slightly more general variant of
Theorem~\ref{essThm2} that would result from the following relaxed
assumptions on the renormalizing operator: $G$ is bounded, its range
satisfies $\ran G\supset {\overline{\mathfrak t}}_{M(z)}$, and the
renormalization of the $\gamma$-field, i.e. $\gamma(z)G$, should be
bounded and densely defined.

Notice that the operator $G_2$ in Example~\ref{Example_third} admits
all these properties apart from the last one: $\gamma_2(z)\equiv
S(S_F)^{-1}=I\uphar \ran S$ is closed but nondensely defined. To
recover this it suffices to replace $G_2$ by some suitable bounded
operator $\wt G_2=G_2K$, where $K$ is e.g. the orthogonal projection
onto $\ran S$ or a restriction to $\ran S$: this would give a
renormalization as in~\eqref{renorm3} where the multivalued part is
projected away.
\end{remark}

\section{Some classes of $ES$-generalized boundary
triples}\label{sec6}

\subsection{Transforms of $B$-generalized boundary
triples}\label{sec6.1}
Let $\Gamma$ be an isometric relation from the Kre\u{\i}n space
$(\sH^2,J_\sH)$ to the Kre\u{\i}n space $(\cH^2,J_\cH)$ and
decompose $\Gamma=\{\Gamma_0,\Gamma_1\}$ according to the Cartesian
decomposition of its range space $\cH\times\cH$ as in~\eqref{g01}.
Then the \textit{transposed boundary pair} defined as a composition
of two isometric relations via
\[
 \Gamma^\top=\begin{pmatrix}0& 1\\-1 & 0 \end{pmatrix}\Gamma
 =\{\Gamma_1,-\Gamma_0\}
\]
is again an isometric relation from $(\sH^2,J_\sH)$ to
$(\cH^2,J_\cH)$ with $\dom \Gamma^\top=\dom \Gamma$. It is well
known that in the particular case of an ordinary boundary triple
$\{\cH,\Gamma_0,\Gamma_1\}$ for $\ZA^*$, also the transposed boundary
triple $\{\cH,\Gamma_1,-\Gamma_0\}$ is an ordinary boundary triple
for $\ZA^*$. Moreover, if $W$ is any bounded $J_\cH$-unitary operator
in the Kre\u{\i}n space $(\cH^2,J_\cH)$, then the composition
\[
 \binom{\Gamma_0^W}{\Gamma_1^W}=W\binom{\Gamma_0}{\Gamma_1}
\]
is also an ordinary boundary triple for $\ZA^*$ and, conversely, all
ordinary boundary triples of $\ZA^*$ are connected via some
$J_\cH$-unitary operator $W$ to each other in this way; see
\cite{DHMS1,DHMS09}.

The situation changes essentially when $\{\cH,\Gamma_0,\Gamma_1\}$
is not an ordinary boundary triple for $\ZA^*$. In this section we
treat the simplest case of a $B$-generalized boundary triple and
show that a simple $J_\cH$-unitary transform can produce an
$ES$-generalized boundary triple for $\ZA^*$ whose Weyl function
need not be domain invariant, however, according to
Theorem~\ref{essThm1} it is still form domain invariant.

The next result shows how an arbitrary $B$-generalized boundary
triple $\{\cH,\Gamma_0,\Gamma_1\}$ for $\ZA^*$, which is not an
ordinary boundary triple, can be transformed to an $ES$-generalized
boundary triple, whose $\gamma$-field becomes unbounded.

\begin{theorem} \label{thmBtoE}
Let $\{\cH,\Gamma_0,\Gamma_1\}$ be a $B$-generalized boundary triple
for $\ZA^*$ with $\ZA_*=\dom\Gamma\subset \ZA^*$, $\ZA_*\neq \ZA^*$,
let $M(\cdot)$ and $\gamma(\cdot)$ be the corresponding Weyl
function and $\gamma$-field, and let $A_0=\ker\Gamma_0$. Then:
\begin{enumerate}\def\labelenumi {\textit{(\roman{enumi})}}
\item for every fixed $\nu\in\cmr$ the transform
\begin{equation}\label{Etrans1}
    \begin{pmatrix}
    \Gamma^\nu_0 \\ \Gamma^\nu_1
    \end{pmatrix}=
    \begin{pmatrix}
    -\RE M(\nu)  & I \\ -I & 0
    \end{pmatrix}\begin{pmatrix}
    \Gamma_0 \\ \Gamma_1
    \end{pmatrix}
\end{equation}
defines a unitary boundary triple for $\ZA^*$ whose Weyl function and
$\gamma$-field  are given by
\begin{equation}\label{Mnu}
 M_\nu(\lambda)=-(M(\lambda)-\RE M(\nu))^{-1}, \quad
 \gamma_\nu(\lambda)=\gamma(\lambda)(M(\lambda)-\RE M(\nu))^{-1},
\end{equation}
and, moreover, $M_\nu(\lambda)$ and $\gamma_\nu(\lambda)$ are
unbounded operators for every $\lambda\in\cmr$;

\item $\{\cH,\Gamma^\nu_0,\Gamma^\nu_1\}$ is an $ES$-generalized boundary
triple for $\ZA^*$ with $\dom\Gamma^\nu=\ZA_*$ and, hence,
$M_\nu(\lambda)$ is form domain invariant and $\gamma_\nu(\lambda)$
is closable for every $\lambda\in \cmr$;

\item the Weyl function $M_\nu(\cdot)$ (equivalently the
$\gamma$-field $\gamma_\nu(\cdot)$) is domain invariant on $\cmr$ if
and only if
\begin{equation}\label{MnuDomInv}
 \sN_\mu(\ZA_*) \subset \ran (A_{0,\nu}-\lambda) \quad \text{for all } \lambda, \mu \in
 \cmr,
\end{equation}
where $A_{0,\nu}=\ker \Gamma^\nu_0$.
\end{enumerate}
\end{theorem}

\begin{proof}
(i) \& (ii) Since $\{\cH,\Gamma_0,\Gamma_1\}$ is a $B$-generalized
boundary triple for $\ZA^*$, we have $M\in \cR^s[\cH]$, see
\cite[Proposition~5.7]{DHMS06}, i.e., $M$ is a strict Nevanlinna
function whose values $M(\lambda)$ are bounded operators on $\cH$
with $\ker\IM M(\lambda)=0$ for every $\lambda\in\rho(A_0)$. In
particular, the real part $\RE M(\nu)$ is a bounded operator when
$\lambda\in\rho(A_0)$. Therefore, $\Gamma^\nu$ is a standard
$J_\cH$-unitary transform of $\Gamma$. According to
\cite[Proposition~3.11]{DHMS09} this implies that $\Gamma^\nu$ is a
unitary boundary triple (a boundary relation in the terminology of
\cite{DHMS09}) with $\dom \Gamma^\nu=\dom\Gamma$ whose Weyl function
and $\gamma$-field are given by~\eqref{Mnu}. The assumption
$\ZA_*\neq \ZA^*$ is equivalent to $\ran\Gamma\neq \cH^2$ and
therefore $0\not\in\rho(\IM M(\lambda))$, $\lambda\in\rho(A_0)$; see
\cite[Section~2]{DHMS06}. It follows from~\eqref{Mnu} that
\begin{equation}\label{Mnu02}
 M^\nu(\nu)=i(\IM M(\nu))^{-1}=-M^\nu(\nu)^*
\end{equation}
and then~\eqref{Green3} shows that for all $h,k\in \dom
M^\nu(\nu)=\dom\gamma^\nu(\nu)$,
\begin{equation}\label{Green3B}
 (\nu-\bar\nu)(\gamma^\nu(\nu)h,\gamma^\nu(\nu)k)_\sH
 =(M^\nu(\nu)h,k)_\cH-(h,M^\nu(\nu)k)_\cH=2i((\IM
 M(\nu))^{-1}h,k)_\cH.
\end{equation}
Hence, $M^\nu(\nu)$ and $\gamma^\nu(\nu)$ are unbounded operators at
the point $\nu\in\cmr$. In this case $M^\nu(\lambda)$ is an
unbounded operator for all $\nu\in\cmr$; see
\cite[Proposition~4.18]{DHMS06}.

Next consider the $\gamma$-field $\gamma_\nu(\cdot)$. Since
$M(\lambda)-\RE M(\nu)$ is bounded, it follows from~\eqref{Mnu} that
\[
  \gamma_\nu(\lambda)^*=(M(\bar\lambda)-\RE M(\nu))^{-1}\gamma(\lambda)^*,
  \quad \lambda\in\cmr.
\]
This combined with~\eqref{Mnu02} shows that
\begin{equation}\label{Mnu03}
  \gamma_\nu(\nu)^*=i(\IM M(\nu))^{-1}\gamma(\nu)^*,\quad
  \gamma_\nu(\bar\nu)^*=-i(\IM M(\nu))^{-1}\gamma(\bar\nu)^*.
\end{equation}
Since
\[
 \gamma(\nu)^*\gamma(\nu)=\gamma(\bar\nu)^*\gamma(\bar\nu)=(\IM \nu)^{-1}\IM
 M(\nu),
\]
it follows from~\eqref{Mnu03} that
\[
 \ran \gamma(\nu)\oplus \ker \gamma(\nu)^*\subset \dom
 \gamma_\nu(\nu)^*, \quad
 \ran \gamma(\bar\nu)\oplus \ker \gamma(\bar\nu)^*\subset \dom
 \gamma_\nu(\bar\nu)^*.
\]
Hence, $\gamma_\nu(\nu)^*$ and $\gamma_\nu(\bar\nu)^*$ are densely
defined operators, which means that $\gamma_\nu(\nu)$ and
$\gamma_\nu(\bar\nu)$ are closable operators. According to
Theorem~\ref{essThm1} $A_{0,\nu}=\ker \Gamma^\nu_0$ is essentially
selfadjoint and the assertions in (ii) hold. The fact that
$\gamma_\nu(\lambda)$ is an unbounded operator for every
$\lambda\in\cmr$ is seen e.g. from the identity~\eqref{ceq1} in
Lemma~\ref{prop:C111}. Thus, all the assertions in (i) are proven.

(iii) This assertion is obtained directly from
Proposition~\ref{domMchar}.
\end{proof}

Theorem~\ref{thmBtoE} will now be specialized to a situation that
appears often in system theory and in PDE setting where typically
the underlying minimal symmetric operator $\ZA$ is nonnegative; the
simplest situation occurs when the lower bound $\mu(\ZA)$ is
positive. The first part of the next result follows the general
formulation given in \cite[Proposition~7.41]{DHMS12} which was
motivated by the papers of V. Ryzhov; see \cite{Ryzhov2009} and the
references therein.

\begin{proposition} \label{propBtoE}
Let $A_0^{-1}$ and $E$ be selfadjoint operators in $\sH$ and $\cH$,
respectively, and let the operator $G:\cH\to \sH$ be bounded and
everywhere defined with $\ker G=\{0\}$. Moreover, let
\begin{equation}\label{TplusG}
  \ZA_*=\{\{A_0^{-1}f'+G\varphi,f'\}:\, f'\in \ran A_0,\, \varphi\in \dom E\}
\end{equation}
and define the operators $\Gamma_0,\Gamma_1:\ZA_*\to \cH$ by
\begin{equation}\label{eq:5.14}
    \Gamma_0 \widehat f=\varphi,\quad
 \Gamma_1 \widehat f=G^*f'+E \varphi; \quad
  \widehat f=\{A_0^{-1}f'+G\varphi,f'\}\in \ZA_*.
\end{equation}
Then:
\begin{enumerate}\def\labelenumi {\textit{(\roman{enumi})}}
\item $\{\cH,\Gamma_0,\Gamma_1\}$ is an $S$-generalized boundary triple
for $\ZA^*=\overline{\ZA_*}$ with $\ker\Gamma_0=A_0$. For
$\lambda\in\rho(A_0)$ and $\varphi\in \dom E$ the corresponding
$\gamma$-field and the Weyl function are given by
\begin{equation*}
 \gamma(\lambda)\varphi=(I-\lambda A_0^{-1})^{-1}G \varphi, \quad
 M(\lambda)\varphi=E \varphi+ \lambda G^*(I-\lambda A_0^{-1})^{-1}G \varphi;
\end{equation*}

\item $\{\cH,\Gamma_0,\Gamma_1\}$ is a $B$-generalized boundary triple
for $\ZA^*$ if and only if $E$ is a bounded selfadjoint operator;

\item $\{\cH,\Gamma_0,\Gamma_1\}$ is an ordinary boundary triple
for $\ZA^*$ if and only if $E$ is bounded and $G^*(\ran A_0)=\cH$, in
particular, then $\ran G$ must be closed;

\item the transform $\{\Gamma_1-E\Gamma_0,-\Gamma_0\}$
defines an isometric boundary triple for $\ZA^*$ whose closure
$\{\cH,\wt \Gamma_0,\wt \Gamma_1\}$ is a unitary boundary triple for
$\ZA^*$ which is defined by
\begin{equation}\label{Etrans}
    \begin{pmatrix}
    \wt\Gamma_0 \\ \wt\Gamma_1
    \end{pmatrix}\wh f
    =\begin{pmatrix}
     G^*f' \\ -\varphi
    \end{pmatrix},
\quad \wh f\in \dom \wt\Gamma=\{\{A_0^{-1}f'+G\varphi,f'\}:\, f'\in
\ran A_0,\, \varphi\in \cH\},
\end{equation}
and whose Weyl function and $\gamma$-field are given by
\begin{equation}\label{M0inv}
 \wt M(\lambda)=-(M_0(\lambda))^{-1}, \quad
 \wt\gamma(\lambda)=\overline{\gamma(\lambda)}(M_0(\lambda))^{-1},
\end{equation}
where $M_0(\lambda)=\overline{(M(\lambda)-E)}=\lambda G^*(I-\lambda
A_0^{-1})^{-1}G$ and $\overline{\gamma(\lambda)}=(I-\lambda
A_0^{-1})^{-1}G$, and the corresponding transposed boundary triple
is $B$-generalized with Weyl function $M_0(\cdot)$;

\item if $0\in\rho(A_0)$ then $\{\cH,\wt\Gamma_0,\wt\Gamma_1\}$
is an $ES$-generalized boundary triple for $\ZA^*$ and it is
$S$-generalized if and only if $\ran G$ is closed, or equivalently,
$\dom\wt\Gamma$ in~\eqref{Etrans}
is closed, i.e., if and only if
$\{\cH,\wt\Gamma_0,\wt\Gamma_1\}$ is an ordinary boundary triple for
$\ZA^*$.

\item the Weyl function $\wt M$ (equivalently the
$\gamma$-field $\wt\gamma(\cdot)$) is domain invariant on $\cmr$ if
and only if
\begin{equation}\label{invEtrans}
 \ran P_G(I-\mu A_0^{-1})^{-1}G=\ran P_G(I-\lambda A_0^{-1})^{-1}G \quad \text{for all } \lambda, \mu \in
 \cmr,
\end{equation}
where $P_G$ stands for the orthogonal projection onto $\cran G$.
\end{enumerate}
\end{proposition}
\begin{proof}
(i) It was proved in \cite[Proposition~7.41]{DHMS12} that
$(\cH,\Gamma_0,\Gamma_1)$ is a unitary boundary triple for
$\ZA^*=\overline{\ZA_*}$. It is clear from the definition of
$\Gamma_0$ that $\ker\Gamma_0=A_0$, which by assumption is a
selfadjoint relation as an inverse of a selfadjoint operator. Hence,
this boundary triple is $S$-generalized.

(ii) \& (iii) The formula for $\Gamma_0$ shows that
$\ran\Gamma_0=\cH$ precisely when $\dom E=\cH$ or equivalently, $E$
is bounded. Since
\begin{equation}\label{ranges}
    \begin{pmatrix}
    \Gamma_0 \\ \Gamma_1
    \end{pmatrix}\wh f=
    \begin{pmatrix}
    I  & 0 \\ E & G^*
    \end{pmatrix}\begin{pmatrix}
    \varphi \\ f'
    \end{pmatrix}
    =\begin{pmatrix}
    I  & 0 \\ E & I
    \end{pmatrix}
    \begin{pmatrix}
    I  & 0 \\ 0 & G^*
    \end{pmatrix}
    \begin{pmatrix}
    \varphi \\ f'
    \end{pmatrix}
\end{equation}
and in the last product the triangular operator is bounded with
bounded inverse when $E$ is bounded, we conclude that
$\ran\Gamma=\cH\times\cH$ if and only if $\dom E=\cH$ and the
diagonal operator in~\eqref{ranges} is surjective, i.e., $G^*(\ran
A_0)=\cH$; in this case $\ran G^*=\cH$ and $\ran G$ is closed.

(iv) It is clear from~\eqref{eq:5.14} that the transform
$\{\Gamma_1-E\Gamma_0,-\Gamma_0\}$ has the same domain $T$ as
$\Gamma$. Moreover, using~\eqref{eq:5.14} it is straightforward to
check that the closure
$\{\wt\Gamma_0,\wt\Gamma_1\}=\clos\{\Gamma_1-E\Gamma_0,-\Gamma_0\}$
is given by~\eqref{Etrans}. In fact, the transposed boundary triple
$\{\wt\Gamma_1,-\wt\Gamma_0\}$ is $S$-generalized and of the same
form as $\Gamma$ in~\eqref{eq:5.14} when $E=0$, i.e., in view of
(ii) it is even $B$-generalized. Applying (i) to this transposed
boundary triple one also concludes that the Weyl function and
$\gamma$-field of the boundary triple $\{\wt\Gamma_0,\wt\Gamma_1\}$
are given by~\eqref{M0inv}.

(v) It follows from~\eqref{Etrans} that
\begin{equation}\label{tildeA0}
 \wt A_0=\{\{A_0^{-1}f'+G\varphi,f'\}:\, f'\in \ran A_0,\, G^*f'=0,\,\varphi\in \cH\}.
\end{equation}
Using graph expressions one can write $\wt
A_0=A_0\cap(\cH\times \ker G^*)\wh{+} (\ran G\times\{0\})$ and now
using the properties of adjoints it is seen that
\[
 \wt A_0^*=\clos(A_0 \wh{+} \cran G\times \{0\}) \cap (\cH\times \ker
 G^*).
\]
Observe that $A_0\cap (\cran G\times \{0\})=0$, since $\ker
A_0=\{0\}$. If $0\in\rho(A_0)$ then $A_0 \wh{+} \cran G\times \{0\}$
is a closed subspace of $\sH^2$ and this implies that $\wt
A_0^*=\overline{\wt A_0}$. Hence, $\wt A_0$ is essentially
selfadjoint. Since $0\in\rho(A_0)$, it is clear from~\eqref{tildeA0}
that $\ran G$ is closed if and only if $\wt A_0=\ker \wt\Gamma_0$ is
closed, or equivalently, $\dom\wt\Gamma$ in~\eqref{Etrans}
is closed.

(vi) Using for $\wt A_0$ the formula in~\eqref{tildeA0} and the
equalities $\sN_\mu(\dom \wt \Gamma)=\ran \wt\gamma(\mu)=\ran
\overline{\gamma(\mu)}$ the domain invariance condition in
Proposition~\ref{domMchar}~(i) can be rewritten as follows: for
every $h\in\cH$ there exist $h_0\in\cH$ and $f'\in\ran A_0\cap\ker
G^*$ such that
\[
 (I-\mu A_0^{-1})^{-1}Gh=(I-\lambda A_0^{-1})f'+Gh_0
\]
or, equivalently,
\[
 (I-\lambda A_0^{-1})^{-1}(I-\mu A_0^{-1})^{-1}Gh=f'+(I-\lambda
 A_0^{-1})^{-1}Gh_0,
\]
$\mu,\lambda\in\cmr$. Applying resolvent identity to the product
term it is seen that the previous condition is equivalent to
\begin{equation}\label{invEtrans2}
 (I-\mu A_0^{-1})^{-1}Gh=f_1'+(I-\lambda
 A_0^{-1})^{-1}Gh_1,
\end{equation}
for some $h_1\in\cH$ and $f_1'\in\ran A_0\cap\ker G^*$. This
condition is equivalent to the inclusion
\[
 \ran P_G(I-\mu A_0^{-1})^{-1}G\subset \ran P_G(I-\lambda
 A_0^{-1})^{-1}G.
\]
Since $\lambda,\mu\in\cmr$ are arbitrary, this last condition
coincides with the condition~\eqref{invEtrans}.
\end{proof}

\begin{remark}\label{RemBtoA}
(i) The boundary triples $\Gamma$ and $\wt\Gamma$ are completely
determined by $A_0\,(=\ker \Gamma_0=\ker\wt\Gamma_1)$ and the
operators $G$ and $E=E^*$. If, in particular, $0\in\rho(A_0)$, then
the Weyl function $\wt M(\cdot)$ in~\eqref{M0inv} is form domain
invariant (see Theorem~\ref{essThm1}) and the $\gamma$-field
$\gamma(\cdot)$ and the Weyl function $M(\cdot)$ as well as $\wt
M(\cdot)$ (in the resolvent sense) admit holomorphic continuations
to the origin $\lambda=0$ with
\[
 \gamma(0)=G, \quad M(0)=E.
\]
If, in addition, $E$ is bounded and $G$ has closed range, then $\wt
\Gamma=\{\Gamma_1-E\Gamma_0,-\Gamma_0\}$ is an ordinary boundary
triple and the condition~\eqref{invEtrans} is satisfied. Indeed, in
this case $\dom M(\lambda)=\ran M_0(\lambda)=\cH$ for all
$\lambda\in\cmr$.

(ii) If $E$ is bounded, no closure is needed in part (iv), i.e.,
$\wt \Gamma=\{\Gamma_1-E\Gamma_0,-\Gamma_0\}$. In this case,
$\Gamma$ is a $B$-generalized boundary triple and Proposition
\ref{propBtoE} can be seen as an extension of Theorem \ref{thmBtoE}
to a point on the real line. Here the results are formulated for
$\nu=0$. They can easily be reformulated also for $\nu\in\dR$. In
addition, for $\nu=\infty$ the results in Proposition \ref{propBtoE}
can be translated to analogous results when treating range
perturbations (instead of the domain perturbations as in Proposition
\ref{propBtoE}); for general background see \cite[Section
7.5]{DHMS12}. For $\nu=\infty$ the operator $E$ appears as the limit
value $M(\infty)$, while $A_0$ and $\ZA_*$ should be replaced by
their inverses; see~\eqref{domPertur} below.

(iii) The criterion~\eqref{invEtrans} for domain invariance of $\wt
M$ can be derived also directly using $\dom \wt M(\lambda)=\ran
M_0(\lambda)$ and the explicit formula for $M_0(\lambda)$ given in
part (iv) of Proposition~\ref{propBtoE}; see also the equivalent
condition in~\eqref{invEtrans2}.
\end{remark}

If $\{\cH,\Gamma_0,\Gamma_1\}$ is not an ordinary boundary triple
for $\ZA^*$, then the condition~\eqref{invEtrans} fails to hold in
general.
In particular, if  $\ran A_0\cap\ker G^*=\{0\}$ (if e.g. $\ker
G^*=\{0\}$), then the condition~\eqref{invEtrans} is equivalent to
\begin{equation}\label{invEtrans1a}
 \ran (I-\mu A_0^{-1})^{-1}G=\ran (I-\lambda A_0^{-1})^{-1}G \quad \text{for all } \lambda, \mu \in
 \cmr.
\end{equation}
Multiplying this identity from the left by
$\frac{\mu}{\lambda}(I-\lambda A_0^{-1})$ it is seen that
~\eqref{invEtrans1a} implies
\begin{equation}\label{invEtrans1b}
 \ran (I-\mu A_0^{-1})^{-1}G\subset \ran G \quad \text{for all } \mu \in
 \cmr.
\end{equation}
Similarly it can be seen that~\eqref{invEtrans1b} implies
~\eqref{invEtrans1a}. Thus, if $\ran A_0\cap\ker G^*=\{0\}$ then $\wt
M$ is domain invariant if and only if the operator range $\ran G$ is
invariant under the resolvent $(I-\mu A_0^{-1})^{-1}$ for all
$\mu\in\cmr$.

\begin{corollary}\label{cor6.4}
Let $A_0$ in Proposition \ref{propBtoE} be a selfadjoint operator
with $\ker A_0=\{0\}$ and assume that $\ran A_0\cap\ker G^*=\{0\}$.
If $S=(\ZA_*)^*$ is densely defined, then the Weyl function $\wt
M(\cdot)$ defined in~\eqref{M0inv} is not domain invariant.
\end{corollary}
\begin{proof}
Since $\ran A_0\cap\ker G^*=\{0\}$, $\wt M(\cdot)$ is domain
invariant if and only if~\eqref{invEtrans1b} holds. In other words,
for every $\varphi\in\cH$ there exists $h\in\cH$ such that $(I-\mu
A_0^{-1})^{-1}G\varphi = Gh$, or, equivalently,
\begin{equation}\label{invEtrans1c}
 (I+\mu (A_0-\mu)^{-1})G\varphi=Gh \quad \Leftrightarrow\quad
 \mu (A_0-\mu)^{-1}G\varphi=G(h-\varphi).
\end{equation}
If $\ZA$ is densely defined, then  $\ZA^*\supset \ZA_*$ is an operator.
Since $\ker A_0=\{0\}$ one concludes from~\eqref{TplusG} that $\ZA_*$
is an operator if and only if $\dom A_0\cap \ran G=\{0\}$. This
condition applied to~\eqref{invEtrans1c} implies that $\varphi=0$
and $h-\varphi=0$, since $\ker G=\{0\}$. This proves the claim.
\end{proof}

In the case that $A_0$ in Proposition \ref{propBtoE} is nonnegative,
one can specify further the type of the Weyl function as follows.

\begin{corollary}\label{cor6.5}
Assume that in Proposition \ref{propBtoE} $A_0=A_0^*\geq 0$ and
$E=E^*\le 0$. Then the Weyl functions
\[
  M(\lambda)\varphi=E \varphi+ \lambda G^*(I-\lambda A_0^{-1})^{-1}G
  \varphi, \quad
 M_0(\lambda)=\lambda G^*(I-\lambda A_0^{-1})^{-1}G, \quad
 \lambda\in\rho(A_0),
\]
are domain invariant inverse Stieltjes functions, while the Weyl
function $\wt M(\cdot)=-M_0(\cdot)^{-1}$ in~\eqref{M0inv} is a
Stieltjes function.
\end{corollary}
\begin{proof}
Since $A_0$ is a nonnegative selfadjoint operator with $\ker
A_0=\{0\}$ and $E=E^*\le 0$, the Weyl function $M(\lambda)=E +
\lambda G^*(I-\lambda A_0^{-1})^{-1}G$ admits a holomorphic
continuation to the negative real line and, moreover,
\[
 M(x)=E + x G^*(I-x A_0^{-1})^{-1}G=M(x)^*\leq 0 \quad
 \text{for all }
 x < 0.
\]
Hence, $M(\cdot)$ and, in particular, $M_0(\cdot)$ are an inverse
Stieltjes functions with $\ker M(x)=\{0\}$ and $\ker M_0(x)=\{0\}$,
since $\ker G=\{0\}$. In view of $(\wt M(\lambda)-\mu
I)^{-1}=-(I+\mu M_0(\lambda))^{-1}M_0(\lambda)$ ($\mu\in\cmr$) also
the function $\wt M(\cdot)=-M_0(\cdot)^{-1}$ admits a holomorphic
continuation to the negative real line with nonnegative selfadjoint
values therein, i.e., it is a Stieltjes function.
\end{proof}

Let us also mention that analogously the function
\[
 \wh M(\lambda)=-M(1/\lambda)= G^*(A_0^{-1}-\lambda)^{-1}G-E
\]
admits a holomorphic continuation to the negative real line and
\begin{equation}\label{domPertur}
 \wh M(x)=G^*(A_0^{-1}-x)^{-1}G-E=M(x)^*\geq 0 \quad
 \text{for all }
 x < 0
\end{equation}
with $\ker \wh M(x)=\{0\}$.  Hence, $\wh M(\cdot)$ is a Stieltjes
function and the transposed function $\wh M^\top(\cdot)=-\wh
M(\cdot)^{-1}$ is an inverse Stieltjes function. Observe, that $\wh
M(\cdot)$ is the Weyl function corresponding to the boundary triple
$\{\cH,\wh\Gamma_0,\wh\Gamma_1\}$ given by
\[
    \wh\Gamma_0 \widehat f=\varphi,\quad
 \wh\Gamma_1 \widehat f=-G^*f'-E \varphi; \quad
  \widehat f=\{f',A_0^{-1}f'+G\varphi\}\in T^{-1}
\]
with $\ker \wh\Gamma_0=A_0^{-1}$ and $\ker \wh\Gamma_1=A_1^{-1}$.

%

We now assume that $0\in\rho(A_0)$ and make explicit the
renormalization for the $ES$-generalized boundary triple
$\{\cH,\wt\Gamma_0,\wt\Gamma_1\}$ in Proposition~\ref{propBtoE}~(v).
This also yields a representation for the form domain invariant Weyl
function $\wt M(\cdot)$ in~\eqref{M0inv}. To state the result
decompose the bounded inverse $A_0^{-1}$ according to $\sH=\cran
G\oplus (\ran G)^\perp$ as $A_0^{-1}=(A^-_{ij})_{i,j=1}^2$. This
generates the following expression for an associated Schur
complement of the resolvent $(A_0^{-1}-1/\lambda)^{-1}$,
\begin{equation}\label{schur0}
 S_0(\lambda)=A^-_{11}-1/\lambda I-(A^-_{21})^*(A^-_{22}-1/\lambda I)^{-1}A^-_{21},
 \quad \lambda\in\rho(A_0).
\end{equation}

\begin{theorem}\label{propBetE2}
Let the notations and assumptions be as in
Proposition~\ref{propBtoE}, let $A_0$ be a selfadjoint operator with
$0\in\rho(A_0)$ and assume that $\ran G$ is not closed, so that the
$ES$-generalized boundary triple $\{\cH,\wt\Gamma_0,\wt\Gamma_1\}$
is not an ordinary boundary triple. Then:
\begin{enumerate}\def\labelenumi {\textit{(\roman{enumi})}}
\item the closure of the $\gamma$-field $\wt\gamma$ satisfies $\dom\overline{\wt\gamma(\lambda)}=\ran G^*$,
$\lambda\in\rho(A_0)$;

\item the renormalized boundary triple $\{\cran G,\Gamma_{0,r},\Gamma_{1,r}\}$
is an ordinary boundary triple for $\ZA^*=A_0\hplus (\cran
G\times\{0\})$ and is determined by
\begin{equation}\label{Etrans00}
    \begin{pmatrix}
    \Gamma_{0,r} \\ \Gamma_{1,r}
    \end{pmatrix}(f+h)
    =\begin{pmatrix}
     P_GA_0f \\ -h
    \end{pmatrix},
\quad f\in \dom A_0,\quad h\in \cran G,
\end{equation}
where $P_G$ denotes the orthogonal projection onto $\ker \ZA^*=\cran
G$;

\item the Weyl function $M_r(\cdot)$ of the renormalized boundary triple coincides
with the Schur complement in~\eqref{schur0},
\[
  M_r(\lambda)= S_0(\lambda), \quad \lambda\in \rho(A_0)
\]
and the form domain invariant Weyl function $\wt M(\cdot)$ in
~\eqref{M0inv} has the form
\begin{equation}\label{Mrenorm2}
  \wt M(\lambda)=G^{-1}S_0(\lambda)G^{-(*)}, \quad \lambda\in
  \rho(A_0),
\end{equation}
where $G^{(*)}$ is the adjoint when $G$ is treated as an operator
from $\cH$ into $\cran G$.
\end{enumerate}
\end{theorem}
\begin{proof}
(i) By Proposition~\ref{propBtoE}
$\wt\gamma(\lambda)=\overline{\gamma(\lambda)}(M_0(\lambda))^{-1}$.
Using the expressions for $M_0(\lambda)$ in~\eqref{M0inv} and
$S_0(\lambda)$ in~\eqref{schur0} one obtains
\begin{equation}\label{MRenorm}
 \wt M(\lambda) = G^{-1}S_0(\lambda)G^{-(*)}, \quad
 \wt\gamma(\lambda) = -(I-\lambda A_0^{-1})^{-1} I_{\ran G}
 S_0(\lambda) G^{-(*)},
\end{equation}
where $G^{-(*)}$ stands for the inverse of $G^*$, when $G^*$ is
treated as an injective mapping from $\cran G$ to $\cH$. Since
$(I-\lambda A_0^{-1})^{-1}$, $I_{\ran G}$, and $S_0(\lambda)$ are
bounded with bounded inverse for $\lambda\in\rho(A_0)$, we conclude
that the form domain of $\wt M(\lambda)$ is equal to $\ran G^*$ and
that the closure of the $\gamma$-field is given by
\begin{equation}\label{gamtilde}
  \overline{\wt\gamma(\lambda)}=\frac{1}{\lambda}\,(A_0^{-1}-\frac{1}{\lambda}\, I)^{-1} P_GS_0(\lambda) G^{-(*)}
  =\frac{1}{\lambda}\, \begin{pmatrix} I \\ -(A^-_{22}-\frac{1}{\lambda}\,I)^{-1}A^-_{21}\end{pmatrix}G^{-(*)},
  \quad \lambda\in\rho(A_0).
\end{equation}
Here the last identity uses the standard block formula for the
inverse $(A_0^{-1}-1/\lambda)^{-1}$.

(ii) The assumption $0\in\rho(A_0)$ implies that the closure $\ZA^*$
of $T$ is $\ZA^*=A_0\hplus (\cran G\times\{0\})$. In view of (i) one
can use $G^*:\cran G\to \cH$ as the renormalizing operator in
Theorem \ref{essThm2}. Since $A_0$ is an operator one can rewrite
the renormalization of the boundary triple~\eqref{Etrans} in the
form $\{\cran G,\Gamma_{0,r},\Gamma_{1,r}\}$, where
\begin{equation}\label{ERenorm}
    \begin{pmatrix}
    \Gamma_{0,r} \\ \Gamma_{1,r}
    \end{pmatrix}\wh f
    =\begin{pmatrix}
     P_G A_0f \\ -G\varphi
    \end{pmatrix},
\quad \wh f\in \{\{f+G\varphi,A_0f\}:\, f\in \dom A_0,\, \varphi\in
\cH\}.
\end{equation}
The final expression for the renormalized boundary triple is
obtained by taking closure in~\eqref{ERenorm}; this leads to
~\eqref{Etrans00}, since $0\in \rho(A_0)$. Now it is clear that $\dom
\Gamma_r=\ZA^*$ and $\ran \Gamma_r=\cran G\times \cran G$, i.e.,
$\{\cran G,\Gamma_{0,r},\Gamma_{1,r}\}$ is an ordinary boundary
triple for $\ZA^*$.

(iii) This is obtained from~\eqref{MRenorm}. In particular, the
equality $M_r(\lambda)= S_0(\lambda)$ follows by taking closure of
$G\wt M(\lambda)G^*\uphar \cran G$.
\end{proof}

According to Theorem \ref{propBetE2} $A_{0,r}=\ker\Gamma_{0,r}$ is
selfadjoint. Clearly, $A_{0,r}$ coincides with the closure of $\wt
A_0=\ker \wt \Gamma_0$ in Proposition \ref{propBtoE}; see
~\eqref{tildeA0}. If, in particular, $A_0$ is strictly positive, then
$A_{0,r}=\ker\Gamma_{0,r}$ is the \emph{Kre\u{\i}n-von Neumann
extension} $\ZA_K$ of $\ZA$ and we have the following identities
\begin{equation}\label{KvNRenorm}
    \overline{\ker \wt \Gamma_0}=\overline{\wt A_0}= A_{0,r}=\ker \Gamma_{0,r}= S \hplus (\cran G\times\{0\})=\ZA_K,
\end{equation}
where $\ZA$ is the range restriction of $A_0$: $\ZA=\{\{f,A_0f\}: f\in
\dom A_0, \, G^*A_0f=0\,\}$. Observe, that $\ZA$ is densely defined if
and only if $\ZA^*$ is an operator, i.e.,
\[
 \cdom \ZA=\sH \quad \Leftrightarrow\quad \cran G\cap \dom A_0=\{0\}.
\]
By \eqref{Mrenorm2} $\wt M(\cdot)$ is domain invariant if and only
if the dense set $S_0(\lambda)^{-1}(\ran G)$ does not depend on
$\lambda$; in the particular case $\ker G^*=\{0\}$ this also leads
to Corollary~\ref{cor6.4}.

In Proposition \ref{propBtoE} we regularized the $S$-generalized
boundary triple $\{\cH,\Gamma_0,\Gamma_1\}$ via the transform
$\{\Gamma_0,\Gamma_1-E\Gamma_0\}$ before transposing the mappings
and closing up. In fact, the closure of this regularized boundary
triple \index{Boundary triple!regularized}
$\clos\{\Gamma_0,\Gamma_1-E\Gamma_0\}$ is of the same form
as $\Gamma$ in~\eqref{eq:5.14} with $E=0$ and it is $B$-generalized;
see item (iv) in Proposition \ref{propBtoE}.

The next example shows what happens for the boundary triple
$\{\cH,\Gamma_0,\Gamma_1\}$ in Proposition \ref{propBtoE} if it is
transposed without the indicated regularization of the mapping
$\Gamma_1$.

\begin{example}\label{example6.5}
Let $\{\cH,\Gamma_0,\Gamma_1\}$ be the boundary triple as defined
in~\eqref{eq:5.14}. Then
\[
    A_1=\{\{A_0^{-1}f'+G\varphi,f'\}:\, G^*f'+E \varphi=0, \, f'\in \ran A_0,\, \varphi\in \dom E\}
\]
and $\ZA=T^*=\ker \Gamma$ is given by
\[
    \ZA=\{\{A_0^{-1}f',f'\}:\, f'\in \ker G^* \}=\{\{f,f'\}\in A_0:\,  f'\in \ker G^* \}.
\]
If, in particular, $A_0$ is an operator then $\ZA$ is a standard range
restriction of $A_0$ to $\ker G^*$. The defect numbers of $\ZA$ are
equal to $\dim (\ran G)$.

Now, assume that $\ker E=\{0\}$ and $\ran G^*\cap \ran E=\{0\}$.
Then the identity $G^*f'+E \varphi=0$ implies that
$G^*f'=E\varphi=0$ and, consequently, $\varphi=0$ and this means
that $A_1=\ZA$. This means that $A_1$ is not essentially selfadjoint
and thus the transposed boundary triple $\{\cH,\Gamma_1,-\Gamma_0\}$
is not $ES$-generalized. The corresponding Weyl function is given by
\[
 M^\top(\lambda)= -(E + \lambda G^*(I-\lambda A_0^{-1})^{-1}G)^{-1}
\]
and according to Theorem \ref{essThm1} it cannot be form domain
invariant.

If, in addition, $\ker G^*=\{0\}$, then
\[
 \dom M^\top(\lambda)\cap \dom M^\top(\mu)=\{0\}, \quad \text{for  all }
 \lambda\neq \mu,\quad \lambda,\mu\in \rho(A_0).
\]
To see this assume that $g=(E + \lambda G^*(I-\lambda
A_0^{-1})^{-1}G)f_1=(E + \mu G^*(I-\mu A_0^{-1})^{-1}G)f_2$ holds
for some $g,f_1,f_2\in\cH$. Then
\begin{equation}\label{eq:domMtop}
 E(f_2-f_1)=G^*[\lambda (I-\lambda A_0^{-1})^{-1}Gf_1-\mu (I-\mu
A_0^{-1})^{-1}Gf_2]
\end{equation}
and the assumptions $\ran G^*\cap \ran E=\{0\}$ and $\ker E=\{0\}$
imply $f_1=f_2$. Now $\ker G^*=\{0\}$ and an application of the
resolvent identity on the righthand side of~\eqref{eq:domMtop}
yields $g=0$.

If, in particular, $A_0$ is a nonnegative selfadjoint operator with
$\ker A_0=\{0\}$ and $E=E^*\le 0$, then the function $M(\cdot)$ is
an inverse Stieltjes function and the transposed function
$M^\top(\cdot)=-M(\cdot)^{-1}$ is a Stieltjes function, which need
not be form domain invariant; cf. Corollary \ref{cor6.5}.
Analogously the function
\[
 -M(1/\lambda)= G^*(A_0^{-1}-\lambda)^{-1}G-E
\]
is a Stieltjes function and the transposed function $\wt
M(1/\lambda)^{-1}$ is an inverse Stieltjes function, which need not
be form domain invariant.
\end{example}

Finally, it should be mentioned that later, in Section \ref{sec7},
it is shown how the standard Dirichlet and Neumann trace operators
on smooth, as well as on Lipschitz, domains can be included in the
abstract boundary triple framework constructed in Proposition
\ref{propBtoE}; hence all the previous results on them will have
immediate applications in concrete PDE setting.

\subsection{$ES$-generalized boundary triples and graph continuity of a component mapping}
\label{sec6.2}

It is known that for a boundary triple $\{\cH,\Gamma_0,\Gamma_1\}$
(as well as for a boundary pair $\{\cH,\Gamma\}$) to be an ordinary
boundary triple it is necessary and sufficient that both boundary
mappings $\Gamma_0$ and $\Gamma_1$ are continuous on $\ZA^*$  (with
the graph norm on $\dom \ZA^*$ in case $\ZA$ is densely defined). In
general the mappings $\Gamma_0$ and $\Gamma_1$ both can be unbounded
when $\dim\cH=\infty$. In this section we establish analytic
criteria for $\Gamma_0$ or $\Gamma_1$ to be continuous with the aid
of the associated Weyl function. Recall that the kernels
$A_0=\ker\Gamma_0$ and $A_1=\ker \Gamma_1$ are always symmetric and
it is possible that $A_0=\ZA$ or $A_1=\ZA$; see e.g. Example
\ref{example6.5}.

The next result characterizes boundedness of the mapping $\Gamma_1$
for an $ES$-generalized boundary triple.

\begin{proposition}\label{essProp3}
For a unitary boundary triple $\Pi=\{\cH,\Gamma_0,\Gamma_1\}$ with
$A_*=\dom \Gamma$ the following conditions are equivalent:
\begin{enumerate}\def\labelenumi {\textit{(\roman{enumi})}}
\item $A_0=\ker \Gamma_0$ is essentially selfadjoint and $\Gamma_1$ is a bounded operator (w.r.t. the graph norm) on $A_*$;
\item $A_0$ is selfadjoint and the restriction $\Gamma_{1}\uphar \wh\sN_\lambda(A_*)$ is a bounded operator for some
(equivalently for every) $\lambda\in\cmr$;
\item the form associated with $\IM(-M^{-1}(\lambda))$ has a positive lower bound for some
(equivalently for every) $\lambda\in\cmr$.
\end{enumerate}
If one of these conditions is satisfied, then the triple
$\Pi=\{\cH,\Gamma_0,\Gamma_1\}$ is $B$-generalized.
\end{proposition}
\begin{proof}
(i) $\Rightarrow$ (ii)   If $\Gamma_1$ is bounded, the also the
restrictions $\Gamma_{1}\uphar A_0$ and $\Gamma_{1}\uphar
\wh\sN_\lambda(A_*)$ are bounded. Now by Corollary \ref{cor:C3}
$A_0$ is closed and, therefore, $A_0=A_0^*$.

(ii) $\Rightarrow$ (iii) Observe that $(\Gamma_{1}\uphar
\wh\sN_\lambda(A_*))^{-1}=\wh\gamma^\top(\lambda)$ is the
$\gamma$-field of the transposed boundary triple
$\Pi^\top=\{\cH,\Gamma_1,-\Gamma_0\}$. Hence the condition that
$\Gamma_{1}\uphar \wh\sN_\lambda(A_*)$ is bounded means that
$(\gamma^\top(\lambda))^*\gamma^\top(\lambda)$ has a positive lower
bound or, equivalently, that the form corresponding to
$\IM(-M^{-1}(\lambda))$ has a positive lower bound (cf.
~\eqref{Green3U} and Definition~\ref{DefFormdomain}).

(iii) $\Rightarrow$ (i) As shown in the previous implication, the
assumption on $\IM(-M^{-1}(\lambda))$ means that the restriction
$\Gamma_{1}\uphar \wh\sN_\lambda(A_*)$ is bounded. On the other
hand, if the form corresponding to $\IM(-M^{-1}(\lambda))$ has a
positive lower bound, say $c>0$, then
\[
 \| M^{-1}(\lambda) f\|_\cH \|f\|_\cH \geq |(M^{-1}(\lambda) f,f)|
 \geq \IM (-(M^{-1}(\lambda) f,f)) \geq c \|f\|^2_\cH.
\]
Consequently, $\|M(\lambda)\|\leq c^{-1}$, i.e., $M(\cdot)$ is a
bounded Nevanlinna function. Now by Theorem~\ref{prop:C6B} $A_0$ is
selfadjoint and hence according to Corollary \ref{cor:C3} the
restriction $\Gamma_1\uphar A_0$ is bounded. Moreover, by
selfadjointness of $A_0$, one has the decomposition $A_*=A_0 \hplus
\wh\sN_\lambda(A_*)$. Since the angle between $A_0$ and
$\wh\sN_\lambda(A_*)$ is positive, one concludes that $\Gamma_1$ is
bounded on $A_*$. This completes the proof of the implication.

Finally, if one of the equivalent conditions (i)--(iii) holds then,
as shown above, $M(\cdot)$ is a bounded Nevanlinna function. This is
a necessary and sufficient condition for the boundary triple $\Pi$
to be $B$-generalized.
\end{proof}

By passing to the transposed boundary triple gives the following
analog of Proposition \ref{essProp3}.

\begin{proposition}\label{essProp3B}
For a unitary boundary triple $\Pi=\{\cH,\Gamma_0,\Gamma_1\}$ with
$A_*=\dom \Gamma$ the following conditions are equivalent:
\begin{enumerate}\def\labelenumi {\textit{(\roman{enumi})}}
\item $A_1=\ker \Gamma_1$ is essentially selfadjoint and $\Gamma_0$ is a bounded operator (w.r.t. the graph norm) on $A_*$;
\item $A_1$ is selfadjoint and the restriction $\Gamma_{0}\uphar \wh\sN_\lambda(A_*)$ is a bounded operator for some
(equivalently for every) $\lambda\in\cmr$;
\item the form associated with $\IM M(\lambda)$ has a positive lower bound for some
(equivalently for every) $\lambda\in\cmr$.
\end{enumerate}
If one of these conditions is satisfied, then the transposed
boundary triple $\Pi^\top=\{\cH,\Gamma_1,-\Gamma_0\}$ is
$B$-generalized.
\end{proposition}

\begin{remark}\label{rem:7.22}
(i) For infinite direct sums of ordinary boundary triples the
extensions $A_j=\ker \Gamma_j$, $j=1,2$, are automatically
essentially selfadjoint; see \cite[Theorem~3.2]{KosMMM}. If, in
addition, $\Gamma_1$ is bounded, then
$\Pi=\{\cH,\Gamma_0,\Gamma_1\}$ is a $B$-generalized boundary triple
for $\ZA^*$ by Proposition \ref{essProp3}; this implication was
proved in another way in \cite[Proposition~3.6]{KosMMM}; see also
Corollary \ref{newCor1} below.

(ii) Note that $\Pi=\{\cH,\Gamma_0,\Gamma_1\}$ is a $B$-generalized
boundary triple if and only if the composition
$\Gamma_1\wh\gamma(\lambda)$ ($=M(\lambda)$) is bounded for some
(equivalently for all) $\lambda\in\cmr$. In particular, if
$\Gamma_1\wh\gamma(\lambda)$ is bounded, then also the
$\gamma$-field $\gamma(\lambda)$ itself is bounded (see
~\eqref{Green3U}), $A_0=A_0^*$ (by Theorem~\ref{prop:C6B}) and the
restriction $\Gamma_{1}\uphar A_0$ is also bounded (by
Corollary~\ref{cor:C3}). However, in this case $\Gamma_1$ need not
be bounded. Therefore, the conditions in Proposition \ref{essProp3}
are sufficient, but not necessary, for $\Pi$ to be a $B$-generalized
boundary triple. An example is any $B$-generalized boundary triple
$\Pi$, which is not an ordinary boundary triple, such that also the
transposed boundary triple $\Pi^\top$ is $B$-generalized, since then
$\Pi^\top$ cannot be an ordinary boundary triple. Then the second
condition in (iii) of Proposition \ref{essProp3} is not satisfied.
For an explicit example of such a $B$-generalized boundary triple,
see for instance the direct sum of Dirac operators treated below in
Proposition~\ref{prop3.5direcsum}.
\end{remark}

The boundedness of the component mappings $\Gamma_0$ and $\Gamma_1$
can be used to derive the following new characterization of ordinary
boundary triples.

\begin{proposition}\label{charOBT}
For a unitary boundary triple $\Pi=\{\cH,\Gamma_0,\Gamma_1\}$ with
$A_*=\dom \Gamma$ the following conditions are equivalent:
\begin{enumerate}\def\labelenumi {\textit{(\roman{enumi})}}
\item $\Gamma_0$ is bounded and $\ran \Gamma_0=\cH$;
\item $\Gamma_1$ is bounded and $\ran \Gamma_1=\cH$;
\item $\Pi=\{\cH,\Gamma_0,\Gamma_1\}$ is an ordinary boundary
triple.
\end{enumerate}
\end{proposition}
\begin{proof}
(i) $\Rightarrow$ (iii) By Corollary \ref{cor:C5} $\ran
\Gamma_0=\cH$ implies that $A_0=A_0^*$ and $\Pi$ is a
$B$-generalized boundary triple. In particular, the Weyl function
$M(\cdot)$ of $\Pi$ belongs to the class $R^s[\cH]$ of bounded
strict Nevanlinna functions. On the other hand,
$(\Gamma_{0}\uphar\wh\sN_\lambda(A_*))^{-1}=\wh\gamma(\lambda)$,
$\lambda\in\cmr$. Now $\Gamma_0\uphar \wh\sN_\lambda(A_*)$ is
bounded and this means that $\gamma(\lambda)^*\gamma(\lambda)$ has a
positive positive lower bound or, equivalently, that
$0\in\rho(\IM(M(\lambda))$. Hence, $M(\cdot)\in R^u[\cH]$ and $\Pi$
is an ordinary boundary triple; see \cite[Proposition~2.18]{DHMS06}.

(ii) $\Rightarrow$ (iii) Apply the previous implication to the
transposed boundary triple.

(iii) $\Rightarrow$ (i), (ii) This is clear, since for ordinary
boundary triple $\Gamma:A_*\to \cH^2$ is bounded and surjective.
\end{proof}


\subsection{Real regular points and analytic extrapolation principle for Weyl functions}
\label{sec6.3}

The main result here contains an analytic extrapolation principle
for Weyl functions in the case when the underlying minimal operator
$A$ admits a regular type point on the real line $\dR$. The main
tool for getting this result relies on the the main transform of
boundary relations (called here boundary pairs) that was introduced
in \cite{DHMS06}. The result is partially motivated by the analytic
criterion for an isometric boundary triple to be unitary which can
be found in \cite[Proposition~3.6]{DHMS06} and
\cite[Theorem~7.51]{DHMS12}.

The main transform \index{Main transform} makes a connection between subspaces of the
Hilbert space $(\sH\oplus\cH)^2$ and linear relations from the
Kre\u{\i}n space $(\sH^2,J_\sH)$ to the Kre\u{\i}n space
$(\cH^2,J_\cH)$. It is a linear mapping $\cJ$ from
$\sH^2\times\cH^2$ to $(\sH \oplus \cH)^2$ defined by the formula
\[
 \cJ:\left\{ \begin{pmatrix} f \\ f' \end{pmatrix},
         \begin{pmatrix} h \\ h' \end{pmatrix}
     \right\}
 \mapsto
     \left\{ \begin{pmatrix} f \\ h \end{pmatrix},
         \begin{pmatrix} f' \\ -h' \end{pmatrix}
     \right\},
\quad f,f'\in\sH,\,\, h,h'\in\cH.
\]
The mapping $\cJ$ establishes a one-to-one correspondence between
the (closed) linear relations $\Gamma:\sH^2 \to \cH^2$ and the
(closed) linear relations $\wt A$ in $\wt \sH=\sH\oplus\cH$ via
\begin{equation}\label{awig}
 \Gamma\mapsto\wt A :=\cJ(\Gamma)=
 \left\{\,
 \left\{ \begin{pmatrix} f \\ h \end{pmatrix},
         \begin{pmatrix} f' \\ -h' \end{pmatrix}
 \right\} :\,
 \left\{ \begin{pmatrix} f \\ f' \end{pmatrix},
         \begin{pmatrix} h \\ h' \end{pmatrix}
 \right\}
 \in \Gamma
 \,\right\}.
\end{equation}
According to \cite[Proposition~2.10]{DHMS06} the main transform
$\cJ$ establishes a one-to-one correspondence between the sets of
contractive, isometric, and unitary  relations $\Gamma$ from
$(\sH^2,J_\sH)$ to $(\cH^2,J_\cH)$ and the sets of dissipative,
symmetric, and selfadjoint relations $\wt A$ in $\sH \oplus \cH$,
respectively. Recall that a boundary pair $\{\cH,\Gamma\}$ is called
\textit{minimal}, if \index{Boundary pair!unitary!minimal}
\[
 \sH=\sH_{min}:=\cspan\{\,\sN_\lambda(A_*):\,\lambda\in\dC_+\cup\dC_-\,\}.
\]

The next result shows usefulness of the main transform for analytic
extrapolation of Weyl functions $M(\cdot)$ from a single real point
$x\in\dR$ to the complex plane, when $x$ is a regular type point of
the minimal operator $A$. In the special case when the analytic
extrapolation of $M(x)$ is a uniformly strict Nevanlinna function
the extrapolation principle formulated for Weyl functions in the
next theorem, yields a solution to the following general inverse
problem: given a pair $\{\Gamma_0,\Gamma_1\}$ of boundary mappings
from $A^*$ to $\cH$ determine the selfadjoint extension $A_\Theta$
of $A$ (up to unitary equivalence) when the boundary condition
$\Gamma_1\wh f=\Theta \Gamma_0 \wh f$ is fixed by some operator
$\Theta$ acting on $\cH$. It is emphasized that this result arises
basically from the stated regularity claim for $M(x)$ at the single
point $x\in \hat\rho(A)$.

\begin{theorem}\label{MRealreg}
Let $\{\Gamma,\cH\}$ be an isometric boundary pair for $A^*$ with
domain $A_*=\dom \Gamma$, $\clos A_*=A^*$, (i.e. Green's identity
~\eqref{Greendef1} holds for $\wh f, \wh g\in A_*$), let $\wt
A=\cJ(\Gamma)$ be the main transform~\eqref{awig} of $\Gamma$.
Assume that there exists a selfadjoint extension $H\subset A_*=\dom
\Gamma$ of $A$ with $x\in \rho(H)\cap \dR$ and let the mapping
$M(x)$ at this point $x$ be defined by~\eqref{Weylreal}. Then the
following assertions hold:
\begin{enumerate}\def\labelenumi {\textit{(\roman{enumi})}}
\item The following two conditions are equivalent:
\begin{enumerate}
\def\labelenumi{\textit{(\roman{enumi})}}
\item $M(x)$ is selfadjoint in $\cH$ and $0\in \rho(M(x)+xI)$;
\item $x\in\rho(\wt A)\cap \dR$.
\end{enumerate}

\item If the conditions (a), (b) in (i) hold then $\{\Gamma,\cH\}$ is a unitary boundary
pair for $A^*$ and $M(x)$ admits an analytic extrapolation \index{Analytic extrapolation of the Weyl family}from the
point $x$ to the half-planes $\dC_\pm$ as the Weyl family $M(\cdot)$
which necessarily belongs to the class $\wt\cR(\cH)$ of Nevanlinna
families.

\item If the boundary pair $\{\Gamma,\cH\}$ is minimal then all the
intermediate extensions $A_\Theta$ of $A$ given by~\eqref{eq:SA_ext}
are, up to unitary equivalence, uniquely determined by $M(\cdot)$.
\end{enumerate}
\end{theorem}
\begin{proof}
(i) To prove the equivalence of (a) and (b) consider the main
transform $\wt A$ of $\Gamma$ in~\eqref{awig}. The range of $\wt
A-x$ is given by
\begin{equation}
\label{awig1}
 \ran(\wt A -xI)=
 \left\{\,
 \begin{pmatrix} f'-x f \\ -h'-x h  \end{pmatrix}
         :\,
 \left\{ \begin{pmatrix} f \\ f' \end{pmatrix},
         \begin{pmatrix} h \\ h' \end{pmatrix}
 \right\}
 \in \Gamma
 \,\right\}.
\end{equation}

(a) $\Rightarrow$ (b) For $\wh f_x\in \wh\sN_{x}(\ZA_*)$ one has
$f_x'=xf_x$ and $\{h,h'\}\in M(x)$ by the definition in
~\eqref{Weylreal}. Since $-x\in \rho(M(x))$ and $-h'-xh \in\ran
(-M(x)-xI)=\cH$ it follows from~\eqref{awig1} that
\begin{equation}
\label{awig2}
 \begin{pmatrix} 0 \\ \cH \end{pmatrix}\subset \ran(\wt A -xI).
\end{equation}
Since $H\subset \dom \Gamma$ and $x\in \rho(H)\cap \dR$ one has
$\ran (H-x)=\sH$ which combined with~\eqref{awig1} and~\eqref{awig2}
shows that $\ran(\wt A -xI)=\sH\oplus\cH$. This implies that $\wt A$
is a selfadjoint relation in $\sH\oplus\cH$, since $\wt A$ is
symmetric by isometry of $\Gamma$; cf.
\cite[Proposition~2.10]{DHMS06}. In particular, $x\in\rho(\wt A)\cap
\dR$.

(b) $\Rightarrow$ (a) If $x\in\rho(\wt A)\cap \dR$ then $\ran(\wt A
-xI)=\sH\oplus\cH$ and, in particular,~\eqref{awig2} is satisfied.
In view of~\eqref{awig} and~\eqref{Weylreal} this means that
$\{f,f'\}\in \wh\sN_{x}(\ZA_*)$ and $\{h,h'\}\in M(x)$ and therefore
$\ran (-M(x)-xI)=\cH$. On the other hand, it follows from
~\eqref{Weylreal} and Green's identity~\eqref{Greendef1} that $M(x)$
is symmetric, i.e., $(h',h)=(h,h')$ for all $\{h,h'\}\in M(x)$.
Therefore, $M(x)$ is selfadjoint and $-x\in \rho(M(x))$.

(ii) The proof of (i) shows that if (a) or, equivalently, (b) holds
then $\wt A$ is a selfadjoint relation in $\sH\oplus\cH$. Then,
equivalently, its (inverse) main transform $\{\Gamma,\cH\}$ is a
unitary boundary pair for $A^*$. By the main realization result
stated in Theorem~\ref{GBTNP} we conclude that $M\in\wt \cR(\cH)$.

(iii) To prove this assertion first recall that according to the
main realization result in \cite[Theorem 3.9]{DHMS06} (cf.
Theorem~\ref{GBTNP}) the Weyl function uniquely determines the
$\Gamma$, as well as $\wt A$, by the minimality of $\Gamma$.
Uniqueness of $\Gamma$ here means that if there exists another
minimal boundary pair $\{\cH,\wt\Gamma\}$ associated with the
symmetric operator $\wh A=\ker\Gamma$ in some Hilbert space $\wh\sH$
having the same Weyl function $M(\cdot)$, then there exists a
standard unitary operator $U:\sH\to \wh\sH$ such that
\begin{equation}
\label{UNIT1} \wh\Gamma=\left\{ \left\{\begin{pmatrix}Uf \\
Uf'\end{pmatrix},
\begin{pmatrix}h \\ h'\end{pmatrix}\right\}:
\left\{\begin{pmatrix}f \\ f'\end{pmatrix},
\begin{pmatrix}h\\h'\end{pmatrix}\right\}
\in\Gamma \right\}.
\end{equation}
Hence, if the extension $A_\Theta$ of $A$ in the Hilbert space $\sH$
and the extension $\wh A_\Theta$ of $\wh A$ in the Hilbert space
$\wh\sH$ are associated with the same ``boundary condition''
$\Theta$ via~\eqref{eq:SA_ext} then~\eqref{UNIT1} implies that
\[
 \wh A_\Theta = \{\{Uf,Uf'\}:\, \{f,f'\} \in A_\Theta \}= U A_\Theta U^{-1}.
\]
This means that the linear relations $A_\Theta$ and $\wh A_\Theta$
are unitarily equivalent via the same unitary operator $U$ for every
linear relation $\Theta$ in $\cH$.
\end{proof}

\begin{remark}
The proof of item (i) in Theorem~\ref{MRealreg} shows that (b)
$\Rightarrow$ (a) without the assumption on the existence of a
selfadjoint extension $H\subset A_*$ with $x\in\rho(H)$.

As to items (ii) and (iii) of Theorem~\ref{MRealreg} it should be
mentioned that if the analytic extrapolation $M(\cdot)$ belongs to
the class $\cR^u[\cH]$, then the spectrum $\sigma(A_\Theta)$ of
every selfadjoint extension $A_\Theta$ ($\Theta=\Theta^*$) of $A$
can be completely characterized by the spectral function of the
corresponding Weyl function $M_\Theta(\cdot)\in \cR[\cH]$ which is
obtained via a fractional linear transforms from the function
$M(\cdot)$. For details we refer to \cite{DM91,DM95,DHMS1,DHMS09}.
Some further developments concerning uniqueness of boundary triples
and connections between $\sigma(A_\Theta)$ and the spectral
functions $\Sigma(t)$ of the associated Weyl functions can be found
in \cite{HMM2013}.
\end{remark}

Theorem~\ref{MRealreg} offers also a useful analytic tool to check
whether an isometric boundary triple (or boundary pair) is actually
unitary or, equivalently, if the Weyl function of some isometric
boundary triple is in fact from the class $\cR(\cH)$ of Nevanlinna
functions. We use this result to construct a unitary boundary pair
for Laplacians defined on rough domains in Section~\ref{example8.3}
and to associate unitary boundary triples with boundary pairs of
nonnegative forms in the next subsection.



\subsection{Boundary pairs of nonnegative operators and unitary boundary
triples}\label{sec6.4}

The notion of boundary pairs involves initially only one boundary
map associated with a closed nonnegative form $\sh$ or a pair of
nonnegative selfadjoint operators. The purpose in this section is to
show that, after introducing a second boundary map $\Gamma_1$ (via
the first Green's identity), the boundary pair $(\cH,\wt\Gamma_0)$
generates a unitary boundary triple $\{\cH,\Gamma_0,\Gamma_1\}$.
Furthermore, various special cases of boundary pairs are connected
to specific classes of unitary boundary triples.
%
%
In applications to PDE's $\sh$ is often the Neumann form and in
abstract setting e.g. the form $\sh_K$ associated to the Kre\u{\i}n's
extension $A_K$, which is the smallest nonnegative selfadjoint
extension of $A$.
The notion of a boundary pair can be seen to arise from the works of
Kre\u{\i}n and Birman (also Vishik?) and has been treated in later papers
by G. Grubb (PDE setting) and Yu. M. Arlinskii (abstract setting).

A (basic) positive boundary pair\index{Positive boundary pair} $\{\cH,\wt\Gamma_0\}$ involving the
form domain of the Kre\u{\i}n extension was introduced in
\cite{Arlin96}. This notions leads to positive boundary triples
$\{\cH,\wt\Gamma_0,\Gamma_1\}$, where $\ker \wt\Gamma_0=A_F$ and
$A_K=\ker \Gamma_1$ are the Friedrichs and the Kre\u{\i}n extension
of a nonnegative operator $A$; see \cite{Koc79}, \cite{Arlin88} and
also \cite[Chapter~3]{HdSSz2012} for some further details and
literature. Boundary pairs which lead to $B$-generalized boundary
triples appear in \cite{ArHa15}. A more general class of boundary
pairs $(\cH,\wt\Gamma_0)$ has been studied recently by O. Post
\cite{Post16}; who relaxed the surjectivity condition on
$\wt\Gamma_0$ and replaced it by the weaker requirement that $\ran
\wt\Gamma_0$ is dense in $\cH$. We recall the definition more
explicitly here (using present notations):

\begin{definition}[\cite{Post16}]\label{BP}
Let $\sh$ be a closed nonnegative form on a Hilbert space $\sH$ and
let $\wt\Gamma_0$ be a bounded linear map from $\sH^1:=(\dom
\sh,\|\cdot\|_1)$, where $\|f\|_1^2=\sh(f)+\|f\|^2$, into another
Hilbert space $\cH$. Then $(\cH,\wt\Gamma_0)$ is said to be a
\emph{boundary pair associated with the form $\sh$}, if:

(a) $(\sH^{1,D}:=)\, \ker\wt\Gamma_0$ is dense in $\sH$;

(b) $(\cH^{1/2}:=)\, \ran \wt\Gamma_0$ is dense in $\cH$.

\noindent A pair $(\cH,\wt\Gamma_0)$ is said to be \emph{bounded}\index{Positive boundary pair!bounded} if
$\ran \wt\Gamma_0=\cH$, otherwise it is said to be unbounded.
\end{definition}

Since $\wt\Gamma_0$ is bounded its kernel defines a closed
restriction of the form $\sh$, which we denote here by
$\sh_0(f)=\sh(f)$, $f\in\ker \wt\Gamma_0$. By assumption (a) the
forms $\sh_0$ and $\sh$ are densely defined in $\sH$ and we denoted
by $H_0$ and $H$ the selfadjoint operators associated with the
closed forms $\sh_0$ and $\sh$, respectively. Next we associate a
symmetric operator and its adjoint with the boundary pair
$(\cH,\wt\Gamma_0)$ via
\[
 A:=H_0\cap H, \quad A^*=\clos\left(H_0\hplus H\right).
\]
In general, $A$ need  not be densely defined, in which case $A^*$ is
multivalued; in what follows we assume that $A$ is densely defined.
By definition $H_0$ and $H$ are disjoint selfadjoint extensions of
$A$. Recall that $\dom A^*=\dom H_0\dot{+}\ker(A^*-\lambda I)$,
$\lambda\in\rho(H_0)$, and there is similar decomposition with $H$.
Since $\sh_0\subset \sh$, one has $H\leq H_0$ or, equivalently,
$(H+a)^{-1}\geq (H_0+a)^{-1}$ for all $a>0$. Hence, see
\cite[Lemma~2.2]{HSSW07}, one can write
\[
 \dom H^{1/2}=\dom H_0^{1/2}+ \ran((H+a)^{-1}-(H_0+a)^{-1})^{1/2},
\]
and since clearly $\ran((H+a)^{-1}-(H_0+a)^{-1})\subset \ker
(A^*+a)$, one obtains
\[
 \dom\sh=\dom \sh_0 \,{+}\, (\sN_{-a}\cap \dom \sh), \quad a>0.
\]
This sum is not in general direct, since $\sN_{-a}\cap \dom \sh_0$
is nontrivial, whenever $H_0\neq A_F$; see
\cite[Proposition~2.4]{HSSW07}. This sum can be made direct with an
additional restriction on $\sN_\lambda$. As shown in
\cite[Propositions~2.9]{Post16} the set of so-called weak solutions
with a fixed $\lambda\in\dC$ defined by
\begin{equation}\label{Post00}
 \sN^1_\lambda:=\left\{ f\in \sH^1:\, \sh(f,g)-\lambda (f,g)_\sH=0,\, \forall \, g\in \dom \sh_0  \right\}
\end{equation}
leads to the following direct sum decomposition for every
$\lambda\in \rho(H_0)$:
\begin{equation}\label{Post0}
 \dom\sh=\dom \sh_0 \,\dot{+}\, \sN^1_\lambda.
\end{equation}
Here $\sN^1_\lambda\,(\subset \sN_\lambda\cap \dom \sh)$ is closed
in $\sH^1$, $\sN^1_\lambda$ is dense in $\ker(A^*-\lambda)$, and
$\sN^1_\lambda\cap \dom \sh_0=\{0\}$. The restriction
$\wt\Gamma_0\uphar\sN^1_\lambda$ is a bounded operator from
$\sN^1_\lambda$ into $\cH$ and the decomposition~\eqref{Post0}
implies that it is injective and its range is equal to $\ran
\wt\Gamma_0$. The inverse operator
\[
 S(\lambda):=(\wt\Gamma_0\uphar\sN^1_\lambda)^{-1}: \cH^{1/2}\to \sN^1_\lambda
\]
is closed as an operator from $\cH$ to $\sH^1$ with domain
$\cH^{1/2}=\ran\wt\Gamma_0$.

\begin{definition}[\cite{Post16}]\label{DefPost}
The boundary pair $(\cH,\wt\Gamma_0)$ associated with the form $\sh$\index{Positive boundary pair!elliptically regular}
is said to be \emph{elliptically regular}, if the operator
$S:=S(-1)$ is bounded as an operator from $\cH$ to $\sH$, i.e.
$\|Sh\|_\sH\leq C\|h\|_\cH$ for all $h\in\cH^{1/2}$ and some $C\geq
0$. Moreover, the boundary pair $(\cH,\wt\Gamma_0)$ is said to be\index{Positive boundary pair! uniformly positive}
\emph{(uniformly) positive}, if there is a constant $c>0$, such that
$\|Sh\|_\sH\geq c\|h\|_\cH$ for all $h\in\cH^{1/2}$.
\end{definition}

Let $\lambda=-1$ and define the form $\sfl[h,k]$ on $\cH$ by
\[
 \sfl[h,k]=(Sh,Sk)_{\sH^1}, \quad h,k\in \cH^{1/2}.
\]
The form $\sfl$ is closed in $\cH$, since $S:\cH\to\sH^1$ is a
closed operator. Hence, associated with $\sfl$ there is a unique
selfadjoint operator $\Lambda$ in $\cH$ characterized by the
equality
\[
 \sfl[h,k]=(\Lambda h,k)_\cH,\qquad h\in \dom \Lambda, \quad k\in \dom
 \sfl=\cH^{1/2}.
\]
It is clear that $\Lambda=S^*S$, where $S^*:\sH^1\to\cH$ is the
usual Hilbert space adjoint. The operator $\Lambda$ is called the
\emph{Dirichlet-to-Neumann operator} at the point $\lambda=-1$
associated with the boundary pair $(\cH,\wt\Gamma_0)$. The
\emph{(strong) Dirichlet-to-Neumann operator} at a point $\lambda\in
\rho(H_0)$ is defined as follows (\cite[Section 2.4]{Post16}):
\begin{equation}\label{defLambda}
 \dom \Lambda(\lambda):= \left\{ \varphi\in \cH^{1/2}:\,
 \exists \psi\in\cH \text{ such that }
 (\sh-\lambda)(S(\lambda)\varphi,S\eta)=(\psi,\eta)_\cH, \quad \forall \eta\in\cH^{1/2} \, \right\}
\end{equation}
and then $\Lambda(\lambda)\varphi:=\psi$. The
operator $\Lambda(\lambda)$ is closed in $\cH$ and it has bounded
inverse operator $\Lambda(\lambda)^{-1}\in \cB(\cH)$ for all
$\lambda\in\rho(H_0)$; see \cite[Proposition~2.17]{Post16}.

Next consider the restriction of $A^*$ to the form domain of
$\sh$ 
\begin{equation}\label{Post1}
 \sH_0^1:=\left\{ f\in \sH^1\cap \dom A^*:\, \sh(f,g)=(A^*f,g)_\sH,\, \forall \, g\in \dom \sh_0
 \right\}
\end{equation}
be equipped with the norm defined by
$\|f\|^2_{\sH_0^1}=\sh(f)+\|f\|^2+\|A^*f\|^2$, which makes $\sH_0^1$
a Hilbert space. Now using the rigged Hilbert space
$\cH^{1/2}\subset\cH\subset\cH^{-1/2}$ introduce a bounded operator
$\check{\Gamma}_1:\sH_0^1\to \cH^{-1/2}$ such that
\begin{equation}\label{Post2}
 (\check{\Gamma}_1f,\wt\Gamma_0g)_{-1/2,1/2}=(A^*f,g)_\sH-\sh(f,g)
\end{equation}
holds for all $f\in \sH_0^1$ and $g\in\sH^1$; this map is well
defined by the formulas~\eqref{Post1},~\eqref{Post2}. Finally, we
introduce the restriction $A_*$ of $A^*$ by
\[
 \dom A_*:=\left\{ g\in \sH_0^1:\, \check{\Gamma}_1 g
 \in\cH \right\}
\]
and denote $\Gamma_0=\wt\Gamma_0\uphar \dom A_*$,
$\Gamma_1=\check{\Gamma}_1\uphar \dom A_*$. By definition (the first
Green's identity)
\begin{equation}\label{GreenEq1}
 \sh(f,g)=(A^*f,g)_\sH-(\Gamma_1f,\wt\Gamma_0g)_\cH
\end{equation}
holds for all $f\in\dom A_*$ and $g\in\sH^1$. In what follows the
triple $\{\cH,\Gamma_0,\Gamma_1\}$ with the domain $\dom A_*=\dom
\Gamma_0\cap\dom \Gamma_1$ is called a boundary triple generated by
the boundary pair $(\cH,\wt\Gamma_0)$. The next result characterizes
the central properties of the boundary pair $(\cH,\wt\Gamma_0)$ by
means of the boundary triple $\{\cH,\Gamma_0,\Gamma_1\}$. In
particular, it shows that the notion of boundary pair in Definition
\ref{BP} can be included in the framework of unitary boundary
triples whose Weyl function are Nevanlinna functions from the class
$\cR^s(\cH)$.

\begin{theorem}\label{ThmFpair}
Let $(\cH,\wt\Gamma_0)$ be a boundary pair for the closed
nonnegative form $\sh$ in $\sH$ and let $\{\cH,\Gamma_0,\Gamma_1\}$
be the corresponding triple as defined above. Then:
\begin{enumerate}\def\labelenumi {\textit{(\roman{enumi})}}

\item $\{\cH,\Gamma_0,\Gamma_1\}$ is a unitary boundary triple for
$A^*$;

\item $A_0:=A^*\uphar\ker \Gamma_0$ is a symmetric restriction of
$H_0$, while $A_1:=A^*\uphar\ker \Gamma_1$ is selfadjoint and it is
equal to $H$;

\item the $\gamma$-field and the Weyl function $M(\cdot)$ of the boundary triple $\{\cH,\Gamma_0,\Gamma_1\}$
are given by
\[
 \gamma(\lambda)=S(\lambda)\uphar\dom \Lambda(\lambda), \quad
 M(\lambda)=-\Lambda(\lambda), \quad \lambda\in \rho(H_0);
\]

\item the transposed triple $\{\cH,\Gamma_1,-\Gamma_0\}$ is
a $B$-generalized boundary triple for $A^*$;

\item $\{\cH,\Gamma_0,\Gamma_1\}$ is $ES$-generalized, i.e., $\clos A_0=H_0$ if and only if
$S(\lambda)$ is closable when treated as an operator from
$\cH\to\sH$ for some (equivalently for all) $\lambda\in\rho(H_0)$;

\item $(\wt\Gamma_0,\cH)$ is elliptically regular if and only if
$\{\cH,\Gamma_0,\Gamma_1\}$ is an $S$-generalized boundary triple;

\item $(\wt\Gamma_0,\cH)$ is uniformly positive if and only if
$\Gamma_0:A_*\to \cH$ is a bounded operator (w.r.t. the graph norm
on $A_*$) or, equivalently, the form $\st_{M(\lambda)}$ has a
positive lower bound for some (equivalently for every)
$\lambda\in\cmr$;

\item $(\wt\Gamma_0,\cH)$ is bounded if and only if $\{\cH,\Gamma_0,\Gamma_1\}$ is a $B$-generalized boundary
triple;

\item $(\wt\Gamma_0,\cH)$ is bounded and uniformly positive if and only if
$\{\cH,\Gamma_0,\Gamma_1\}$ is an ordinary boundary triple.
\end{enumerate}
\end{theorem}
\begin{proof}
(i) First observe that the first Green's identity~\eqref{GreenEq1}
applied to $h[f,g]$ and $\overline{h[g,f]}$ with $f,g\in\dom A_*$
leads to the second Green's identity~\eqref{Greendef1} by symmetry
of the form $\sh$. The second Green's identity~\eqref{Greendef1}
implies that the restrictions $A_0=A^*\uphar\ker \Gamma_0$ and
$A_1=A^*\uphar\ker \Gamma_1$ are symmetric operators extending $A$.

Next we prove that the (graph) closure of $A_*$ is $A^*$ and that
$\{\cH,\Gamma_0,\Gamma_1\}$ is a unitary boundary triple for $A^*$.
It is clear from~\eqref{Post1} that the set of weak solutions
$\sN^1_\lambda$ belongs to $\sH_0^1$. Since $H_0$ is the selfadjoint
operator associated with the form $\sh_0$ by the first
representation theorem of Kato and $\sh_0\subset \sh$, we conclude
from~\eqref{Post1} that $\dom H_0\subset\sH_0^1$. Similarly $H$ is
the selfadjoint operator associated with the form $\sh$ and, hence
also $\dom H\subset\sH_0^1$. Now applying~\eqref{Post2} with
$f\in\dom H$ and $g\in\sH^1$ taking into account that $\ran\Gamma_0$
is dense in $\cH$ by assumption (b) in Definition~\ref{BP} we
conclude that $\check{\Gamma}_1f=0$. Hence, $\dom H\subset \dom A_*$
and $\Gamma_1(\dom H)=\{0\}$. Thus, $H\subset A_1$ and since $A_1$
is symmetric this implies that $A_1=H$ is selfadjoint. Now consider
the operator $\Lambda=S^*S$. Since $\dom\Lambda$ is a core for the
form $\sfl$ it is also a core for the operator $S$. This implies
that $S(\dom \Lambda)$ is dense in $\sN_{-1}^1$ w.r.t. the topology
in $\sH^1$, since $S$ has bounded inverse. We claim that $S(\dom
\Lambda)\subset \dom A_*$. To see this we consider the form
\begin{equation}\label{s1}
 \sh(f,g)-(A^*f,g)_\sH, \qquad f\in \sH_0^1, \quad
 g\in\sH^1.
\end{equation}
Notice that $\sN^1_\lambda\subset \sH_0^1$, see~\eqref{Post00},
~\eqref{Post1}, and that the decomposition~\eqref{Post0} for
$\lambda=-1$ is orthogonal in $\sH^1$. Hence, one can write
$g=g_0+g_1\in \dom \sh_0 \oplus_1 \sN^1_{-1}$, $g\in\sH^1=\dom \sh$.
Now for $h\in\dom \Lambda$ one has $Sh\in \sN^1_{-1}$ and for all
$g=g_0\in \dom \sh_0=\ker\Gamma_0$,
\[
 \sh(Sh,g_0)-(A^*Sh,g_0)_\sH=\sh(Sh,g_0)+(Sh,g)_\sH=(Sh,g_0)_{\sH^1}=0.
\]
On the other hand, when $g=g_1\in \sN^1_{-1}$, then
$k=\Gamma_0g_1\in \cH^{1/2}$ satisfies $g_1=Sk$. This leads to
\[
 \sh(Sh,g_1)-(A^*Sh,Sk)_\sH=\sh(Sh,Sk)+(Sh,Sk)_\sH=(Sh,Sk)_{\sH^1}=(\Lambda h,k)_{\sH}=(\Lambda h,\Gamma_0g_1)_{\sH}.
\]
We conclude that for $f=Sh$, $h\in\dom \Lambda$, and all $g\in
\sH^1$ the form~\eqref{s1} can be rewritten as follows
\[
 \sh(Sh,g)-(A^*Sh,g)_\sH=(\Lambda h,\Gamma_0g)_{\sH}.
\]
Comparing this formula with~\eqref{Post2} we conclude that
$\check{\Gamma}_1Sh=\Gamma_1Sh=-\Lambda h\in \cH$, which proves the
claim $S(\dom \Lambda)\subset \dom A_*$.

Since $S(\dom \Lambda)$ is dense in $\sN_{-1}^1$ and $\dom H\subset
\dom A_*$, the closure of $A_*$ is equal to the closure of
$H+\wh\sN_{-1}^1$, which coincides with $A^*$. Hence, the domain of
$\{\Gamma_0,\Gamma_1\}$ is dense in $\dom A^*$ w.r.t. the graph
topology. As was shown above $\Gamma_1Sh=-\Lambda h$ for all
$h\in\dom \Lambda$ and, in addition, $\Gamma_0Sh=h$. Since
$S(\dom\Lambda)\subset \sN_{-1}(A_*)$ this implies that for the
regular point $\lambda=-1\in \rho(H)$ one has
\[
 -\Lambda \subset M(-1).
\]
Here equality $M(-1)=-\Lambda $ prevails, since $M(-1)$ is
necessarily symmetric by Green's identity~\eqref{Greendef1}.
Clearly, $M(-1)-I= -\Lambda-I\le -I$ and thus $0\in\rho(M(-1)-I)$.
Therefore, we can apply Theorem~\ref{MRealreg} to conclude that
$\{\cH,\Gamma_0,\Gamma_1\}$ is a unitary boundary triple for $A^*$
with dense domain $A_*$.

(ii) The equality $A_1=H$ was already proved in item (i). Next we
prove the inclusion $A_0\subset H_0$. The first Green's identity
~\eqref{GreenEq1} shows that
\begin{equation}\label{GreenEq1B}
 \sh(f,g)=(A^*f,g)_\sH, \quad \text{for all } f\in \dom A_*,\quad g\in
\ker\wt\Gamma_0=\dom\sh_0.
\end{equation}
If, in particular, $f\in\dom A_0$ i.e. $\Gamma_0f=0$, then $f\in
\dom\sh_0$ and~\eqref{GreenEq1B} can be rewritten as
\[
 \sh_0(f,g)=(A_0f,g)_\sH, \quad \text{for all } g\in \dom\sh_0.
\]
Now by the first representation theorem (see \cite{Kato}) one
concludes that $f\in\dom H_0$ and $A_0f=H_0f$. Therefore,
$A_0\subset H_0$.

(iii) It was shown in part (i) that $\ran
S(\lambda)=\sN^1_\lambda\subset \sH_0^1$ for each
$\lambda\in\rho(H_0)$. Now assume in addition that $h\in\dom
\Lambda(\lambda)$ and let $g\in\sH^1$. Then the definition of
$\Lambda(\lambda)$ shows that
\[
 \sh(S(\lambda)h,g)-(A^*S(\lambda)h,g)_\sH=(\sh-\lambda I)[S(\lambda)h,g]
 =(\Lambda(\lambda) h,\Gamma_0g)_{\sH}.
\]
Comparing this formula with~\eqref{Post2} we conclude that
$\check{\Gamma}_1S(\lambda)h=\Gamma_1S(\lambda)h=-\Lambda(\lambda)h\in
\cH$, which shows that $S(\dom \Lambda(\lambda))\subset \dom A_*$
and, moreover, that $M(\lambda)h=-\Lambda(\lambda)h$. Therefore,
\[
 -\Lambda(\lambda)\subset M(\lambda), \quad \lambda\in\rho(H_0).
\]
Equivalently, $\Lambda(\lambda)^{-1}\subset -M(\lambda)^{-1}$ and
since $M(\cdot)$ is the Weyl function of a single valued unitary
boundary triple, $M(\cdot)\in\cR^s(\cH)$, in particular, $\ker
M(\lambda)=\{0\}$; see~\eqref{stric-unb}. On the other hand,
$\Lambda(\lambda)^{-1}\in\cB(\cH)$ and, hence, the equality
$\Lambda(\lambda)^{-1}= -M(\lambda)^{-1}$ follows. The equality
$\gamma(\lambda)=S(\lambda)\uphar \dom M(\lambda)$ is clear, and the
formulas for $\gamma(\lambda)$ and $M(\lambda)$ are proven.

(iv) Since $\Lambda(\cdot)^{-1}\in \cB(\cH)$ and
$-M(\lambda)^{-1}=\Lambda(\cdot)^{-1}$ by part (iii) the transposed
boundary triple is $B$-generalized; see Theorem~\ref{thm:WF_BG_BT}.

(v)  By definition $\{\cH,\Gamma_0,\Gamma_1\}$ is $ES$-generalized
if and only if $A_0$ is essentially selfadjoint, which in view of
(ii) means that $\clos A_0=H_0$. On the other hand, by
Theorem~\ref{essThm1} and Remark~\ref{rem:Mderiv}
$\{\cH,\Gamma_0,\Gamma_1\}$ is $ES$-generalized if and only if
$\gamma(\lambda)$ is closable for some (equivalently for all)
$\lambda\in\rho(H_0)$. 

Since $\gamma(\lambda)\subset S(\lambda)$, it is clear that if
$S(\lambda)$ is closable then also $\gamma(\lambda)$ is closable. On
the other hand, it follows from \cite[Theorems~2.11,
Proposition~2.17]{Post16} that $\dom \Lambda(\lambda)$ is dense
w.r.t. the $\cH^{1/2}$-topology on $\cH^{1/2}$ and that
\[
 \overline{S(\lambda)\uphar\dom\Lambda(\lambda)}^{\,\cH^{1/2}\to\sH^1}=S(\lambda),
\]
since $S(\lambda):\cH^{1/2}\to\sN^1_\lambda$ is a topological
isomorphism. Since the topologies on $\cH^{1/2}$ and $\sH^1$ are
stronger than the topologies on $\cH$ and $\sH$ it follows that if
$\gamma(z):\cH\to\sH$ is closable, then also $S(z):\cH\to\sH$ is
closable and
\[
 \overline{\gamma(\lambda)}^{\,\cH\to\sH}
 =\overline{S(\lambda)}^{\,\cH\to\sH}.
\]

(vi) When $(\wt\Gamma_0,\cH)$ is elliptically regular, then
$S:\sH^1\to \cH$ is a bounded operator. Then equivalently the
$\gamma$-field  $\gamma(\lambda)$ is bounded for all
$\lambda\in\rho(H_0)$, cf. \cite[Theorem~2.11]{Post16}, and the
statement is obtained from Theorem~\ref{prop:C6}.

(vii) If $(\wt\Gamma_0,\cH)$ is (uniformly) positive then
$S(\lambda)$, $\lambda\in\rho(H_0)$ is bounded from below; cf.
\cite[Theorem~2.11]{Post16}. In view of~\eqref{Green3U} this means
that the form $\st_{M(\lambda)}$ has a positive lower bound. Now the
statement follows from Proposition~\ref{essProp3B}, since $A_1=H$ is
selfadjoint by part (iii).

(viii) If $(\wt\Gamma_0,\cH)$ is bounded, i.e.,
$\ran\wt\Gamma_0=\cH^{1/2}=\cH$, then $S:\cH\to\sH^1$ is closed (as
an inverse of a bounded operator $\wt\Gamma_0\uphar\sN^1_{-1}$),
everywhere defined, and bounded by the closed graph theorem. In
particular, $\{\cH,\Gamma_0,\Gamma_1\}$ is a $B$-generalized
boundary triple. On the other hand, we conclude that the form
$(\sh+1)(Sh,Sk)$ is closed and defined everywhere on $\cH$. Now it
follows from~\eqref{defLambda} that $\dom \Lambda(-1)=\cH$. This
implies that $M(\cdot)\in\cR^s[\cH]$; see e.g.~\eqref{eq:M_S_gen} in
Theorem~\ref{prop:C6}. Therefore, $\{\cH,\Gamma_0,\Gamma_1\}$ is a
$B$-generalized boundary triple by Theorem~\ref{thm:WF_BG_BT}.

The converse statement is clear, since $\ran\Gamma_0=\cH$ implies
that also $\ran\wt\Gamma_0=\cH$.

(ix) This follows directly e.g. from Proposition~\ref{charOBT}.
Alternatively, by (vi) and (vii) the conditions mean that
$M(\cdot)\in\cR^u[\cH]$, and then the result follows from
Theorem~\ref{thm:WF_Ord_BT}.
\end{proof}

\begin{remark}
(a) Characterizations (viii) and (ix) have been announced (without
proofs) in \cite[Theorem~1.8]{Post16}. Moreover, elliptic regularity
has been characterized in \cite[Theorem~1.8]{Post16} using
equivalence to quasi boundary triples. However, as indicated the
conditions defining a quasi boundary triple are not sufficient to
guarantee that the corresponding Weyl function belongs to the class
of Nevanlinna functions. In this sense the characterization of
elliptic regularity presented in (vi) is more precise and complete.
As to (vii) a characterization of positive boundary pairs via
uniform positivity of the form valued function $z\to-\sfl_z$ appears
in \cite[Theorem~3.13]{Post16}, while the other characterization
that $\Gamma_0:A_*\to \cH$ is a bounded operator, as well as the
statements (i) -- (v) in Theorem~\ref{ThmFpair} are obviously new.

(b) Since $H_0$  and $H$ are nonnegative selfadjoint operators, the
Weyl functions $M(\cdot)$ and $-M(\cdot)^{-1}$ admit analytic
continuations (in the resolvent sense) to the negative real line. In
fact, $M(\cdot)$ belongs to the class of operator valued (in general
unbounded) \emph{inverse Stieltjes functions}, while
$-M(\cdot)^{-1}$ belongs to the class of operator valued
\emph{Stieltjes functions}. Essentially these facts follow from the
following formula:
\[
 (M(x)h,h)=(\sh-x)[(H_0+1)(H_0-x)^{-1}h,h] \leq 0, \quad h\in\dom M(x),\quad
 x<0.
\]
\end{remark}

\section{Applications to Laplace operators}\label{sec7}

In this section the applicability of the abstract theory developed
in the preceding sections is demonstrated for the analysis of some
classes of differential operators. First we consider the most
standard case of elliptic PDE by treating Laplacians in smooth\index{Laplacian}\index{Laplacian!on smooth domain}
bounded domains; in this case many of the abstract results
take a rather explicit form.

\subsection{The Kre\u{\i}n - von Neumann Laplacian}\label{sec7.1}
\index{Kre\u{\i}n - von Neumann Laplacian}
Let $\Omega$ be a bounded domain in $\dR^d$ $(d\ge 2)$ with a smooth
boundary $\partial\Omega$. Consider the differential expression
$\ell:=-\Delta$, where $\Delta$ is a Laplacian operator in  $\Omega$
and denote by $\ZA_{\min}$ and $\ZA_{\max}$ the minimal and the maximal
differential operators generated in $H^0(\Omega):=L^2(\Omega)$ by
the differential expression $\ell$. Let $\gamma_D$ and $\gamma_N$ be
the Dirichlet and Neumann trace operators \index{Trace operators} defined for any $f\in
 H^2(\Omega)$  by
  \begin{equation}\label{eq:3.24}
\gamma_D :\left.\,f\mapsto f \right|_{\partial\Omega}
,\quad  \gamma_N :\,\left.f\mapsto \frac{\partial f}{\partial n}\right
|_{\partial\Omega} = \sum_{j=1}^d n_j \gamma_D\frac{\partial f}{\partial x_j}
   \end{equation}
where $n = (n_1, \ldots, n_d)$ is the outward unit normal  to the boundary
$\partial\Omega$.
Then the mapping 
\begin{equation}\label{eq:Trace_Thm}
    \left(\begin{array}{cc}
\gamma_D \\ \gamma_N \\
\end{array}
\right):\,f\in H^2(\Omega) \mapsto \left(\begin{array}{cc}
\gamma_D f\\ \gamma_N f \\
\end{array}
\right)\in \left(\begin{array}{cc}
H^{3/2}(\partial\Omega)\\ H^{1/2}(\partial\Omega)
\end{array}
\right)
\quad\textup{is bounded and onto}
\end{equation}
(see 
\cite[Thm 1.8.3]{LionsMag72}). It is
known (see, for instance, \cite{Ber}) that  $\ZA_{\max} =
\ZA_{\min}^*\,(=\ZA^*)$ and
\[
\dom \ZA_{\min} = \{f\in H^{2}_0(\Omega):\,\gamma_D f=\gamma_N f=0\,\}.
\]
%
%
Clearly, $\dom \ZA_{\max} \supset  H^{2}(\Omega)$. However, an
explicit description of $\dom \ZA^*$ is unknown while Lions and
Magenes  \cite{LionsMag72} have shown that the mappings $\gamma_D$
and $\gamma_N$ defined on $H^{2}(\Omega)$ extend to continuous
mappings from the domain of the maximal operator,
\begin{equation}\label{eq:TraceThm0}
\gamma_D : \dom \ZA^*   \to H^{-1/2}(\partial\Omega), \quad \gamma_N :
\dom \ZA^*   \to H^{-3/2}(\partial\Omega)
\end{equation}
and these mappings are surjective.

The differential expression $\ell$ admits two classical selfadjoint realizations, the
Dirichlet Laplacian $-\Delta_D$ \index{Dirichlet Laplacian}and the Neumann Laplacian \index{Neumann Laplacian } $-\Delta_N$, given by $\ell$
on the domains
\begin{equation}\label{eq:dom_Delta_D}
    \dom \Delta_D = \{f\in H^{2}(\Omega):\,\gamma_D f=0\,\} \quad \text{and}
\quad \dom \Delta_N = \{f\in H^{2}(\Omega):\,\gamma_N f=0\,\},
\end{equation}
respectively.

General, not necessarily local, boundary value problems for elliptic operators have been
studied in the pioneering works of Vi\v{s}ik~\cite{Vi52} and Grubb~\cite{Gr68} (see
also~\cite{Gr09}, \cite{MMM10}, \cite{GeMi11}, \cite{BeMi14} for further developments
and applications).

Denote by $H_\Delta^{s}(\Omega)$ the following space
%
  \begin{equation}\label{eq:S*}
H_\Delta^{s}(\Omega):=  H^{s}(\Omega)\cap \dom \ZA_{\max} =
\left\{f\in H^{s}(\Omega):\Delta f\in L^2(\Omega)\right\},\quad 0\le
s \le 2,  
   \end{equation}
%
and equip it with the graph norm
$\|f\|_{H_\Delta^{s}(\Omega)}=(\|f\|^2_{H^{s}}+\|\ZA_{\max}f\|^2_{L^2(\Omega)})^{1/2}$
of $-\Delta$ on $H^{s}(\Omega)$.

According to the Lions-Magenes result (\cite[Theorem 2.7.3]{LionsMag72})  the trace
operators $\gamma_D$ and $\gamma_N$ admit continuous extensions to the operators
%
  \begin{equation}\label{eq:3.24A}
 \gamma^s_D :\,H_\Delta^{s}(\Omega) \to H^{s-1/2}(\partial\Omega),\quad
 \gamma^s_N :\,H_\Delta^{s}(\Omega) \to H^{s-3/2}(\partial\Omega),\quad 0< s \le 2, 
   \end{equation}
which are surjective. It is emphasized that the values $s= 1/2$  and
$s= 3/2$ are not excluded here. At the same time the traces
$\gamma^s_D :\,H^{s}(\Omega) \to H^{s-1/2}(\partial\Omega)$ and
$\gamma^s_N :\,H^{s}(\Omega) \to H^{s-3/2}(\partial\Omega)$  are
continuous mappings if and only if  $s> 1/2$ and $s > 3/2$,
respectively, (see (\cite[Theorems 1.9.4, 1.9.5]{LionsMag72} and
\cite{Agran2015}). In the latter case both mappings
in~\eqref{eq:3.24A} are surjective. Moreover, for $s>3/2$ the
mapping $\gamma^s_D\times \gamma^s_N :\,H^{s}(\Omega) \to
H^{s-1/2}(\partial\Omega)\times H^{s- 3/2}(\partial\Omega)$ is also
surjective.

When treating the traces $\gamma^s_D$ and $\gamma^s_N$ as mappings
into $L^2(\partial\Omega)$ a natural choice for the index is
$s=3/2$; see Remark~\ref{remBTs} below.
The restriction of $\ZA^*$ to the domain
\begin{equation}\label{eq:wtS*0}
 \dom \ZA_*=H_\Delta^{3/2}(\Omega)
\end{equation}
is called a pre-maximal operator and is denoted by $\ZA_*$.

The Dirichlet Laplacian $-\Delta_D$
is an invertible selfadjoint operator in $L^2(\Omega)$ with a
discrete spectrum $\sigma_p(-\Delta_D)$. Define a solution operator
$\cP(z):\, L^2(\partial\Omega)\to H^{1/2}(\Omega)$ for
$z\in\dC\setminus\sigma_p(-\Delta_D)$. Let $\varphi \in
L^2(\partial\Omega)$ and let $f_z\in\dom \ZA_{\max}$ be the unique
solution of the Dirichlet  problem
  \begin{equation}\label{eq:Pz}
-\Delta f_z-zf_z=0,\quad \gamma_Df = \varphi  
  \end{equation}
Then the operator $\cP(z):\,  \varphi \mapsto f_z$  is continuous as
an operator from $L^2(\partial\Omega)$ to $H^{1/2}(\Omega)$ and it
maps $H^1(\partial\Omega)$ into $H^{3/2}(\Omega)$; see \cite{Gr68}.
Hence the Poincar\'e-Steklov operator $\Lambda(z)$ defined by\index{Poincar\'e-Steklov operator}
\begin{equation}\label{eq:PSt}
    \Lambda (z)\varphi := \gamma_N\cP(z) \varphi,
\end{equation}
maps $H^1(\partial\Omega)$ into $L^2(\partial\Omega)$ with
continuous extension from $H^{-1/2}(\partial\Omega)$ to
$H^{-3/2}(\partial\Omega)$. Moreover, the Dirichlet-to-Neumann map
\index{Dirichlet-to-Neumann map} $\Lambda:=\Lambda(0)$ treated as an
operator   in $L^2(\partial\Omega)$ is selfadjoint on the domain
$\dom \Lambda=H^1(\partial\Omega)$; see \cite{MMM10}.

Following \cite{Vi52} and~\cite{Gr68} we introduce the regularized
trace operators \index{Trace operators!regularized} as follows
  \begin{equation}\label{eq:GBTL}
\wt\Gamma_{0,\Omega}f = (\gamma_N - \Lambda(0)\gamma_D)f, \qquad \wt
\Gamma_{1,\Omega}f=\gamma_D f,\quad f\in \dom S_{*}.
 \end{equation}
%
%
It is proved in \cite{Gr68} that the mappings $\wt \Gamma_{0,\Omega}$ and $\wt
\Gamma_{1,\Omega}$ are well defined and
\begin{equation}\label{ranRegDN}
\wt \Gamma_{0,\Omega}:\dom \ZA_{\max}\to
H^{1/2}(\partial\Omega),\qquad \wt \Gamma_{1,\Omega}:\dom \ZA_{\max}\to
H^{-1/2}(\partial\Omega) .
\end{equation}
In fact, the effect of regularization appearing in $\wt
\Gamma_{0,\Omega}$ follows from the decomposition $\dom \ZA_{\max} =
\dom \Delta_D \dotplus \ker \ZA^*$ ($0\in\rho(-\Delta_D))$: $u\in \dom
\ZA_{\max}$ admits a decomposition $u = u_D + u_0$ with $u_D \in \dom
\Delta_D \subset H^{2}(\Omega)$ and $u_0\in \ker \ZA^*$. Now an
application of~\eqref{eq:PSt} and~\eqref{eq:3.24A} gives $\wt
\Gamma_{0,\Omega}u = \gamma_Nu_D\in H^{1/2}(\partial\Omega)$ which
yields~\eqref{ranRegDN}. Since $\gamma^{2}_D\times \gamma^{2}_N
:\,H_\Delta^{2}(\Omega) \to H^{3/2}(\partial\Omega)\times
H^{1/2}(\partial\Omega)$ is surjective, one has $\{0\}\times
H^{1/2}(\partial\Omega)\subset \ran \gamma^{2}_D\times
\gamma^{2}_N$, which shows that $\wt \Gamma_{0,\Omega}$ in
~\eqref{ranRegDN} is also surjective. In addition, it is also closed
when $\dom \ZA_{\max}=H_\Delta^{0}(\Omega)$ is equipped with the
$L^2$-graph norm. Indeed, by the continuity properties of the traces
$\gamma_D$, $\gamma_N$, and the Poincar\'e-Steklov operator
$\Lambda(z)$ (see~\eqref{eq:TraceThm0},~\eqref{eq:PSt})
$\wt\Gamma_{0,\Omega} = (\gamma_N - \Lambda(0)\gamma_D)$ is a
continuous mapping from $\dom \ZA_{\max}$ into
$H^{-3/2}(\partial\Omega)$. Now, if $f_n\in \dom
\wt\Gamma_{0,\Omega}$ and $f_n\to f$ in $H_\Delta^{0}(\Omega)$ and
$g_n=\wt\Gamma_{0,\Omega}f_n\to g$ in $H^{1/2}(\partial\Omega)$,
then $g_n\to g$ also in $H^{-3/2}(\partial\Omega)$, $f\in \dom
\wt\Gamma_{0,\Omega}$ and $g=\wt\Gamma_{0,\Omega}f$ by the
$H^{-3/2}(\partial\Omega)$-continuity of $\wt\Gamma_{0,\Omega}$.
Hence, $\wt\Gamma_{0,\Omega}:\dom \ZA_{\max}\to
H^{1/2}(\partial\Omega)$ is closed. Finally, by the closed graph
theorem $\wt \Gamma_{0,\Omega}$ in~\eqref{ranRegDN} is bounded; cf.
\cite[Theorem III.1.2]{Gr68} where these results on $\wt
\Gamma_{0,\Omega}$ are derived in a more general elliptic setting.

With these preliminaries we are ready to give first applications of
the abstract results for Laplacians on smooth bounded domains.

Let $\wt \ZA_*$ be a restriction of $\ZA_{\max}$ to the domain
\begin{equation}\label{eq:wtS*}
\dom \wt \ZA_*=\{f\in\dom \ZA_{\max}:\,\gamma_D f\in L^2(\partial\Omega)\}.
\end{equation}
%

\begin{proposition}\label{prop:Laplacian}
Let the operators $\gamma_N$, $\gamma_D$, $\cP(z)$, $\Lambda(z)$,
$\ZA_*$ and $\wt \ZA_*$ be defined by \eqref{eq:TraceThm0},
\eqref{eq:Pz}, \eqref{eq:PSt},\eqref{eq:wtS*0}, and~\eqref{eq:wtS*}.
Then:
\begin{enumerate}\def\labelenumi {\textit{(\roman{enumi})}}
  \item $\{L^2(\partial\Omega), \gamma_D\uphar{\dom \ZA_*}, -\gamma_N\uphar{\dom \ZA_*}\}$ is an $S$-generalized boundary triple for $\ZA^*$,
  and the corresponding Weyl function $M(\cdot)$ coincides with $-\Lambda(\cdot)$;
  \item $\{L^2(\partial\Omega), \wt\Gamma_{0,\Omega}\uphar{\dom \wt \ZA_*}, \wt\Gamma_{1,\Omega}\uphar{\dom \wt \ZA_*}\}$
  is an $ES$-generalized boundary triple for $\ZA^*$, the corresponding Weyl function is the $L^2(\partial\Omega)$-closure
\[
 \wt M(z)=\clos(\Lambda(z)-\Lambda(0))^{-1};
\]
  \item the extension $\wt A_0 := \wt S_{*}\uphar{\ker \wt\Gamma_{0, \Omega}}$ is essentially selfadjoint
and its closure coincides with the Kre\u{\i}n - von Neumann
extension of the operator $\ZA_{\min}$.
\end{enumerate}
\end{proposition}
\begin{proof}
(i) The triple $\{L^2(\partial\Omega), \gamma_N\uphar{\dom \ZA_*},
\gamma_D\uphar{\dom \ZA_*} \}$ is a $B$-generalized boundary triple
for the operator $\ZA^*$ (see~\cite{DHMS12}), since Green's identity
holds, the mapping $\gamma^s_N$ with $s=3/2$ is surjective by
~\eqref{eq:3.24A}, and by the descriptions~\eqref{eq:dom_Delta_D},
\[
 \ran(\gamma_N\uphar{\dom \ZA_*})=  H^0(\partial\Omega),\qquad
 \ker(\gamma_N\uphar{\dom \ZA_*}) \supset \ker(\gamma_N\uphar{\sH^2(\Omega)})=\dom \Delta_N.
\]
Here the inclusion in the second relation holds as an equality,
since for any isometric boundary triple $\{\cH,\Gamma_0,\Gamma_1\}$
the kernels $\ker \Gamma_0$ and $\ker \Gamma_1$ determine symmetric
restrictions of $\ZA_{\max}$ (by Green's identity); cf.
\cite[Proposition~2.13]{DHMS06}.

Clearly, the boundary triple $\Pi = \{H^0(\partial\Omega),
\gamma_D\uphar{\dom \ZA_*}, -\gamma_N\uphar{\dom \ZA_*}\}$ is
unitary, since it is transposed to the $B$-generalized boundary
triple $\{H^0(\partial\Omega), \gamma_N\uphar{\dom
\ZA_*},\gamma_D\uphar{\dom \ZA_*}\}$. Moreover, as above from
~\eqref{eq:dom_Delta_D} one concludes that $\ker(\gamma_D\uphar{\dom
\ZA_*})=\dom \Delta_D$. Since $-\Delta_D = -\Delta_D^*$, the triple
$\Pi$ is an $S$-generalized boundary triple for $\ZA^*$. By
definition, the corresponding Weyl function coincides with
$-\Lambda(z)$.

Another proof of the statement (i) can be extracted  from
Proposition~\ref{propBtoE}. Indeed, take $A_0=-\Delta_D$ and fix the
mappings $G := \cP(0)$ and $E := -\Lambda(0)$; see Remark
\ref{RemBtoA}. By definition $\gamma_D G \varphi = \varphi$ and
$\gamma_D A_0^{-1}f = 0$ for all $\varphi \in  H^0(\partial\Omega)$
and $f \in L^2(\Omega)$. Moreover, a direct calculation (see e.g.
\cite{Ryzhov2009} with smooth functions $f$) leads to
\[
 G^*f = -\gamma_N A_0^{-1}f, \quad f \in L^2(\Omega),
\]
cf.~\eqref{Hlambda},~\eqref{ABGG1}. Therefore, the abstract boundary
mappings $\Gamma_0$ and $\Gamma_1$ defined in~\eqref{eq:5.14}
coincide with the trace operators $\gamma_D\uphar{\dom \ZA_*}$ and
$-\gamma_N\uphar{\dom \ZA_*}$, respectively.

Moreover,
the mappings $\Gamma_0=\gamma_D\uphar{\dom \ZA_*}$ and
$\Gamma_1=-\gamma_N\uphar{\dom \ZA_*}$ in~\eqref{eq:3.24A} with
$s=3/2$ are surjective. By Proposition~\ref{propBtoE}
$\{H^0(\partial\Omega),\gamma_D\uphar{\dom
\ZA_*},-\gamma_N\uphar{\dom \ZA_*}\}$ is an $S$-generalized boundary
triple and since $E=-\Lambda(0)$ is unbounded, this triple is not
$B$-generalized.


(ii)  Next we apply Proposition \ref{propBtoE} to the closure  $\wt
\Gamma$ of the transformed boundary triple as defined in
~\eqref{Etrans}. By definition $\wt\Gamma:\dom \ZA^*\supset \ZA_* \to
(L^{2}(\partial\Omega))^2$ is closed and in view of~\eqref{Etrans}
$\wt \Gamma_1$ maps onto $L^{2}(\partial\Omega)$. It is clear that
$\wt \Gamma_0$ and $\wt \Gamma_1$ coincide (up to inessential change
of signs) with the regularized trace operators given
by~\eqref{eq:GBTL}, $\wt \Gamma_0 = \wt\Gamma_{0, \Omega}$ and $\wt
\Gamma_1 = \wt\Gamma_{1, \Omega}$ on the initial domain
$H^{3/2}(\Omega)$. On the other hand,
$\wt\Gamma_{0,\Omega}\times\wt\Gamma_{1,\Omega}:\sH^{1/2}_\Delta(\Omega)\to
(L^{2}(\partial\Omega))^2$ and hence
$\wt\Gamma_{\Omega}\uphar\sH^{1/2}_\Delta(\Omega)\subseteq
\wt\Gamma_{\Omega}\uphar\dom \wt \ZA_*$; here equality holds, since
$\wt\Gamma_{1,\Omega}:\sH^{1/2}_\Delta(\Omega)\to
L^{2}(\partial\Omega)$ and $\wt\Gamma_{1,\Omega}:\dom \wt \ZA_*\to
L^{2}(\partial\Omega)$ both are surjective, see~\eqref{eq:TraceThm0}
and~\eqref{eq:wtS*}, and their kernels are equal to $\dom \Delta_D$.
Moreover, $\wt\Gamma_{\Omega}\uphar\sH^{1/2}_\Delta(\Omega)$ is
closed in $\dom \ZA^*\times (L^{2}(\partial\Omega))^2$. Hence $\wt
\Gamma\subset\wt\Gamma_{\Omega}\uphar\sH^{1/2}_\Delta(\Omega)$ and
here equality holds, since by Proposition \ref{propBtoE} $\wt
\Gamma_1$ is surjective and has kernel $\dom \Delta_D$.

By Proposition \ref{propBtoE} $\{\cH,\wt\Gamma_0,\wt\Gamma_1\}$ is
an $ES$-generalized boundary triple for $\ZA_{\max}$. Since $\ran G
\subset H^{1/2}(\Omega)$, it is not closed in $H^0(\Omega)$, hence
the triple $\{\cH,\wt\Gamma_0,\wt\Gamma_1\}$ is not $S$-generalized.
The  statement concerning the Weyl function is obtained
from~\eqref{M0inv}.

(iii) Since $\{\cH,\wt\Gamma_0,\wt\Gamma_1\}$ is not
$S$-generalized, $\wt A_0$ is not selfadjoint. It follows from~\eqref{eq:GBTL} that 
\begin{equation}\label{eq:KNext}
    \dom \wt A_0=\{f\in H^{3/2}_\Delta(\Omega):\, (\gamma_N-\Lambda(0)\gamma_D)f=0\}
\end{equation}
contains the set
  \[
  \dom S\dotplus\cP(0)(H^1(\partial\Omega))(\subset H_\Delta^{3/2}(\Omega)),
  \]
which is dense in the domain of the Kre\u{\i}n-von Neumann extension $S_{K}$,
  \begin{equation}\label{eq:Krein_7.11}
 \dom \ZA_K =\dom S\dotplus \ker \ZA^*,
     \end{equation}
equipped with the graph norm.  This completes the proof.
\end{proof}

\begin{remark}\label{remBTs}
Using the above mentioned properties of the traces $\gamma^s_D$ and
$\gamma^s_N$, it is easy to see that for the values $3/2\le s\le 2$
the boundary triple $\{L^2(\partial\Omega), \gamma^s_D\uphar{\dom
\ZA_*}, -\gamma^s_N\uphar{\dom \ZA_*}\}$ well as the transposed
boundary triple $\{L^2(\partial\Omega), \gamma^s_N\uphar{\dom
\ZA_*}, \gamma^s_D\uphar{\dom \ZA_*}\}$ are quasi boundary triples
(compare \cite[Theorem~6.11]{BeLa2012}) and hence, in particular,
$AB$-generalized boundary triples. Indeed, since Green's identity
holds for $s=3/2$ (by Proposition~\ref{prop:Laplacian}), it holds
also for $3/2< s\le 2$. This combined with~\eqref{eq:3.24A} leads to
$\dom \Delta_D=\ker\gamma^s_D$, $\dom \Delta_N=\ker\gamma^s_N$, and
by the surjectivity of $\gamma^s_D\times \gamma^s_N :\,H^{s}(\Omega)
\to H^{s-1/2}(\partial\Omega)\times H^{s-3/2}(\partial\Omega)$ the
range of $\gamma^s_D\times \gamma^s_N$ is dense in
$L^2(\partial\Omega)$.
\end{remark}

More precisely for every $3/2\leq s\leq 2$ all the quasi boundary
triples in Remark \ref{remBTs} are in fact essentially unitary; the
choice $s=3/2$ in Proposition~\ref{prop:Laplacian} is also motivated
by the next corollary.

\begin{corollary}\label{CorLaplace}
For every $3/2\leq s\leq 2$ the closure of $\gamma^s_D\times
\gamma^s_N$ in $(\dom \ZA_{\max})\times (L^2(\partial\Omega))^2$
coincides with $\gamma^{3/2}_D\times \gamma^{3/2}_N$, where $\dom
\ZA_{\max}$ is equipped with the $L^2$-graph norm of $\dom
\ZA_{\max}$. The closure is an $S$-generalized boundary triple for
$\ZA_{\max}$.

\end{corollary}
\begin{proof}
Since for $3/2\leq s\leq 2$ the mapping $\gamma^s_D\times
\gamma^s_N:H^s_\Delta(\Omega)\to (L^2(\partial\Omega))^2$ is
continuous by~\eqref{eq:3.24A} and the inclusions
$H^s(\Omega)\subset H^{3/2}(\Omega)$ are dense, it follows that the
closure of $\gamma^s_D\times \gamma^s_N$ in
$H^{3/2}_\Delta(\Omega)\times (L^2(\partial\Omega))^2$ coincides
with $\gamma^{3/2}_D\times \gamma^{3/2}_N$. Since the $L^2$-graph
norm of $\ZA_{\max}$ is majorized by the $H^s_\Delta(\Omega)$-norm,
the closure of $\gamma^s_D\times \gamma^s_N$ in $(\dom
\ZA_{\max})\times (L^2(\partial\Omega))^2$ contains
$\gamma^{3/2}_D\times \gamma^{3/2}_N$. However, by
Proposition~\ref{prop:Laplacian} (i) $\gamma^{3/2}_D\times
\gamma^{3/2}_N$ defines an $S$-generalized boundary triple for $\dom
\ZA_{\max}$, which is unitary in the Kre\u{\i}n space sense (see
Definitions~\ref{GBT},~\ref{SgenBT}). Therefore,
$\gamma^{3/2}_D\times \gamma^{3/2}_N$ is also closed in $(\dom
\ZA_{\max})\times L^2(\partial\Omega)$, i.e., the closures coincide.
\end{proof}

When applying form methods, it is often convenient to consider the
above traces on $H^1(\Omega)$. In this case $\gamma^s_N$ maps onto
$H^{-1/2}(\partial\Omega)$ and one needs (Sobolev) dual parings of
the boundary spaces for Green's identity. Of course, if one
restricts such boundary mappings on the side of the range to
$L^2(\partial\Omega)\times L^2(\partial\Omega)$ one gets again the
mapping $\gamma^{3/2}_D\times \gamma^{3/2}_N$ by continuity and
surjectivity of $\gamma^s_N$ onto $H^{s-3/2}(\partial\Omega)$,
$0<s<3/2$; see~\eqref{eq:3.24A}.

The results concerning the $L^2$-closure of the $\gamma$-field in
Section~\ref{sec5} (see in particular Proposition~\ref{gammaclos},
Lemma~\ref{prop:C111}, Theorem~\ref{essThm1}) are now specialized to
the $ES$-generalized boundary triple appearing in part (ii) of
Proposition~\ref{prop:Laplacian}.

\begin{proposition}\label{Prop:Laplacegamma}
Let $\{L^2(\partial\Omega), \wt\Gamma_{0,\Omega}\uphar{\dom \wt
\ZA_*}, \wt\Gamma_{1,\Omega}\uphar{\dom \wt \ZA_*}\}$ be the
$ES$-generalized boundary triple for $\ZA^*$ in
Proposition~\ref{prop:Laplacian} and let $\wt M(\cdot)$ and
$\wt\gamma(\cdot)$ be the corresponding Weyl function and the
$\gamma$-field. Then:
\begin{enumerate}\def\labelenumi {\textit{(\roman{enumi})}}
\item the closure of the boundary mapping
$\wt\Gamma_{0,\Omega}\uphar{\dom\wt \ZA_*}$ coincides with
$\wt\Gamma_{0,\Omega}$ in~\eqref{ranRegDN},
\[
 \wt\Gamma_{0,\Omega}:\dom \ZA_{\max}\to H^{1/2}(\partial\Omega),
\]
it maps bijectively and continuously
$\sN_z(\ZA_{\max})=\ker(\ZA_{\max}-zI)$, $z\in\rho(\ZA_K)$, onto
$H^{1/2}(\partial\Omega)$ and $\ker\wt\Gamma_{0,\Omega}=\dom \ZA_K$,
where $\ZA_K$ is the Kre\u{\i}n-von Neumann extension of $\ZA_{\min}$;

\item the closure of the $\gamma$-field is given by
\[
 \overline{\wt\gamma(z)}=(\wt\Gamma_{0,\Omega}\uphar\sN_z(\ZA_{\max}))^{-1},
\]
it is an unbounded and domain invariant operator with $\dom
\overline{\wt\gamma(z)}=H^{1/2}(\partial\Omega)$ and, furthermore,
$\ran \overline{\wt\gamma(z)}=\sN_z(\ZA_{\max})$, $z\in\cmr$;

\item the domain decomposition~\eqref{ceq1a}
for the closures in Lemma~\ref{prop:C111} reads as
\[
 (\dom \ZA_{\max}=)\, \dom \wt\Gamma_{0,\Omega}=\dom \ZA_K \dot{+} \ran
 \overline{\wt\gamma(z)}, \quad z\in\cmr,
\]
i.e., here~\eqref{ceq1a} reduces to
the second von Neumann formula for $\dom \ZA_{\max}$.
\end{enumerate}
\end{proposition}

\begin{proof}
(i) Recall that $\wt\Gamma_{0,\Omega}:\dom \ZA_{\max}\to
H^{1/2}(\partial\Omega)$ is defined everywhere on $\dom \ZA_{\max}$
and that it is continuous and surjective; see discussion following
~\eqref{ranRegDN}. Therefore, $\wt\Gamma_{0,\Omega}:\dom \ZA_{\max}\to
L^{2}(\partial\Omega)$ is also bounded and closed. By
Lemma~\ref{prop:C111} $\clos \wt
A_0=\ZA_{\max}\uphar\ker\wt\Gamma_{0,\Omega}$ is selfadjoint and
coincides with the Kre\u{\i}n - von Neumann extension $\ZA_K$ of
$\ZA_{\min}$; see Proposition~\ref{prop:Laplacian}. Consequently, the
mapping $\wt\Gamma_{0,\Omega}:\sN_z(\ZA_{\max})\to
H^{1/2}(\partial\Omega)$ is bijective and continuous for all
$z\in\rho(\ZA_K)$.

(ii) The formula for the closure of the $\gamma$-field is obtained
by combining (i) with its abstract description~\eqref{ceq0} in
Proposition~\ref{gammaclos}. Moreover, part (i) shows that
$\overline{\wt\gamma(z)}:H^{1/2}(\partial\Omega)\to \sN_z(\ZA_{\max})$
is bijective and continuous w.r.t. $H^{1/2}(\partial\Omega)$
topology, and closed and unbounded as an operator from
$L^2(\partial\Omega)$ onto $\sN_z(\ZA_{\max})$ with bounded inverse.

(iii) This follows immediately from Lemma~\ref{prop:C111} and the
descriptions in items (i) and (ii).
\end{proof}

     \begin{remark}\label{remBTs2}
The description of the Kre\u{\i}n - von Neumann Laplacian in part
(iii) of Proposition~\ref{Prop:Laplacegamma} by means of trace
operators essentially goes back to \cite{Vi52}; see
also~\cite[Section~12.3]{Maurin65}. For Lipschitz domains a similar
description of the Kre\u{\i}n- von Neumann Laplacian in terms of
extended trace operators was recently given in~\cite{BeGeMM15}; see
also Section~\ref{sec7.3} below for another construction.
  \end{remark}

The next result characterizes the Weyl functions of various boundary
triples appearing in Proposition~\ref{prop:Laplacian} more precisely.

\begin{proposition}\label{Prop:LaplaceWeyl}
Let the notations and assumptions be as in
Proposition~\ref{prop:Laplacian}.
\begin{enumerate}\def\labelenumi {\textit{(\roman{enumi})}}
\item the Weyl function $M(z)=-\Lambda(z)$ of the triple
  $\{L^2(\partial\Omega), \gamma_D\uphar{\dom \ZA_*},  \gamma_N\uphar{\dom \ZA_*}\}$ is domain invariant with $\dom M(z)=H^1(\partial\Omega)$
  and belongs to the class of inverse Stieltjes functions of unbounded operators while the inverse $\Lambda(z)^{-1}$
  belongs to the class of Stieltjes functions of compact operators;
\item the Weyl function $\wt M(\cdot)$ of the boundary triple
 $\{L^2(\partial\Omega), \wt\Gamma_{0,\Omega}\uphar{\dom \wt \ZA_*}, \wt\Gamma_{1,\Omega}\uphar{\dom \wt \ZA_*}\}$
 is form domain invariant with form domain $H^{1/2}(\partial\Omega)$ and belongs to the class of Stieltjes functions of unbounded operators
 while the inverse $-\wt M(\cdot)^{-1}=\clos(\Lambda(0)-\Lambda(\cdot))$ belongs to the class of inverse Stieltjes functions of bounded operators.
\end{enumerate}
\end{proposition}

\begin{proof}
(i) Since $\ker(\gamma_D\uphar{\dom \ZA_*})=\dom \Delta_D$ and
$-\Delta_D = -\Delta_D^*$, $M(\lambda)$ is domain invariant and
moreover $\dom M(\lambda)=\gamma_D({\dom \ZA_*})$; see
Lemma~\ref{domM}, Theorem~\ref{ABGthm}~(v). Now~\eqref{eq:wtS*0} and
surjectivity of $\gamma_D^{3/2}$ in~\eqref{eq:3.24A} gives
$\gamma_D({\dom \ZA_*})=H^1(\partial\Omega)$. The transposed
boundary triple $\{L^2(\partial\Omega), \gamma_N\uphar{\dom
\ZA_*},\gamma_D\uphar{\dom \ZA_*}\}$ is $B$-generalized and
$M(\lambda)^{-1}:L^2(\partial\Omega)\to H^1(\partial\Omega)$,
$\lambda\in\rho(M)$, is bounded. Since the embedding
$H^1(\partial\Omega)\to L^2(\partial\Omega)$ is compact,
$M(\lambda)^{-1}$ is a compact operator in $L^2(\partial\Omega)$. To
get the statements concerning Stieltjes and inverse Stieltjes
classes we apply Green's first identity:
\[
  \int_\Omega (\Delta u) \overline{v} \,dx+\int_\Omega \nabla u\cdot \overline{\nabla v} \,dx=(\gamma_N u, \gamma_D
  v)_{L^2(\partial\Omega)}, \quad u,v\in C^2(\overline\Omega).
\]
With $u,v\in \ker \Delta$ this leads to $(\Lambda(0) \gamma_Du,
\gamma_D v)_{L^2(\partial\Omega)}=(\gamma_N u, \gamma_D
v)_{L^2(\partial\Omega)}\ge 0$ and by denseness of $u,v$ in $\dom
\gamma_D$ one concludes that $\Lambda(0)\geq 0$. 
Equivalently, $M(0)=-\Lambda(0)\leq 0$ and, since $M(\cdot)$ belongs
to the class $\cR^s(\cH)$ of Nevanlinna functions by
Theorem~\ref{prop:C6B}~(iv) and moreover, $0\in\rho(\Delta_D)$,
$M(\cdot)$ is holomorphic and monotone on the interval
$(-\infty,0]$. Therefore, $M(x)\leq 0$ for all $x\leq 0$ and thus
$M(\cdot)$ is an inverse Stieltjes function. Consequently, the
inverse function $-M(\cdot)^{-1}$ is a Stieltjes function.

(ii) It was shown in the proof of Proposition~\ref{prop:Laplacian}
that $\wt\Gamma_{\Omega}\uphar{\dom \wt \ZA_*}$ coincides with
$\wt\Gamma_{0,\Omega}\times\wt\Gamma_{1,\Omega}:H^{1/2}_\Delta(\Omega)\to
(L^{2}(\partial\Omega))^2$. In view of~\eqref{eq:3.24A} (with
$s=1/2$) $\wt\Gamma_{1,\Omega}$ is surjective and, moreover,
$\ker\wt\Gamma_{1,\Omega}=\dom \Delta_D$. Hence, again the
transposed boundary triple
\[
\{L^2(\partial\Omega), \wt\Gamma_{1,\Omega}\uphar{\dom \wt \ZA_*},
-\wt\Gamma_{0,\Omega}\uphar{\dom \wt \ZA_*}\}
\]
is $B$-generalized and the values $-\wt M(z)^{-1}$,
$z\in\rho(\Lambda)$, are bounded operators; cf.
Proposition~\ref{propBtoE}~(iv). By Proposition~\ref{prop:Laplacian}
$\wt M(0)^{-1}=0$ and since $-\wt M(\cdot)^{-1}$ is a Nevanlinna
function and holomorphic on $(-\infty,0]$, one concludes by
monotonicity that $-\wt M(x)^{-1}\le 0$ for all $x<0$. Hence, $-\wt
M(\cdot)^{-1}$ is an inverse Stieltjes function and its inverse $\wt
M(\cdot)$ is a Stieltjes function.

As to the form domain invariance of $\wt M(\cdot)$ notice that
\begin{equation}\label{Lformdom}
 \dom \overline{\st_{\wt M(\lambda)}}
 =\dom\overline{\wt\gamma(\lambda)}=H^{1/2}(\partial\Omega),
 \quad \lambda\in \dC\setminus\dR,
\end{equation}
by Theorem~\ref{essThm1} and Proposition \ref{Prop:Laplacegamma}.
\end{proof}

Notice that by Proposition~\ref{Prop:LaplaceWeyl} and
Corollary~\ref{CorLaplace} the closure of the Weyl function of the
quasi boundary triple $\{L^2(\partial\Omega), \gamma^s_D\uphar{\dom
\ZA_*}, -\gamma^s_N\uphar{\dom \ZA_*}\}$ for all $3/2\le s\le 2$ is
just the Weyl function $M(\cdot)$ of the boundary triple
$\{L^2(\partial\Omega),\gamma^{3/2}_D\uphar{\dom \ZA_*},
-\gamma^{3/2}_N\uphar{\dom \ZA_*}\}$.

Finally, the renormalization result in Theorem~\ref{essThm2} is
spelialized to the case of the $ES$-generalized boundary triple
appearing in Proposition~\ref{prop:Laplacian}. For this purpose we
need a bounded operator $G$ in $L^2(\partial\Omega)$ with $\ran
G=H^{1/2}(\partial\Omega)$; see~\eqref{Lformdom}. Let
$\Delta_{\partial\Omega}$ be the Laplace-Beltrami operator on
$L^2(\partial\Omega)$. Then $-\Delta_{\partial\Omega}+I\ge 0$ and
the operator
\[
 G:=(-\Delta_{\partial\Omega}+I)^{-1/4}
\]
is a nonnegative contraction in $L^2(\partial\Omega)$ with $\ran
G=H^{1/2}(\partial\Omega)$. With this choice of $G$
Theorem~\ref{essThm2} leads to the following result.

\begin{corollary}\label{Lapcor}
Let $\{L^2(\partial\Omega), \wt\Gamma_{0,\Omega}\uphar{\dom \wt
\ZA_*}, \wt\Gamma_{1,\Omega}\uphar{\dom \wt \ZA_*}\}$ be the
$ES$-generalized boundary triple in
Proposition~\ref{prop:Laplacian}~(ii). Then the renormalized
boundary triple $\{L^2(\partial\Omega),\Gamma_{0,G}, \Gamma_{1,G}\}$
defined by~\eqref{ess01} is an ordinary boundary triple for
$\ZA_{\max}$ given by
\[
 \Gamma_{0,G}=G^{-1}\wt\Gamma_{0,\Omega}=G^{-1}(\gamma_N -
 \Lambda(0)\gamma_D),\quad
 \Gamma_{1,G}=\overline{G\wt\Gamma_{1,\Omega}}=\overline{G\gamma_D},
\]
i.e., $\Gamma_{0,G}\times\Gamma_{1,G}:\dom \ZA_{\max}\to
(L^{2}(\partial\Omega))^2$ is surjective.

The corresponding Weyl function is given by
\[
 M_G(z)=\overline{G(\Lambda(z)-\Lambda(0))^{-1}G},
\]
it is a uniformly strict Nevanlinna function belonging also to the
class of Stieltjes functions.
\end{corollary}

It follows from Corollary \ref{Lapcor} for instance that
$\ZA^*\uphar\ker \Gamma_{0,G}=\ZA_K$ and $\ZA^*\uphar\ker
\Gamma_{0,G}=\ZA_F$ and the formula
\begin{equation}\label{bcTheta}
 \wt \ZA \to \Gamma_G(\dom \wt \ZA)=\{\{\Gamma_{G,0}f,\Gamma_{G,1}f\}:\, f\in \dom \wt \ZA\,\}=:\Theta
\end{equation}
establishes a one-to-one correspondence between selfadjoint
(nonnegative) realization of the Laplacian operator $-\Delta$ and
the selfadjoint (nonpositive) relations $\Theta$ in
$L^2(\partial\Omega)$ via boundary conditions as expressed in
~\eqref{bcTheta}.

For a general class of elliptic operator in bounded and unbounded
domains \cite[Proposition~3.5,~5.1]{MMM10} an ordinary boundary
triple for $\ZA_{\max}$ was constructed via the transposed boundary
triple $\{L^2(\partial\Omega), \wt\Gamma_{1,\Omega}\uphar{\dom \wt
\ZA_*}, -\wt\Gamma_{0,\Omega}\uphar{\dom \wt \ZA_*}\}$ which is
$B$-generalized. A similar regularization method for bounded domains
$\Omega$ appears already in \cite{Gr68} without a general formalism
of boundary triples. In \cite[Proposition~5.1]{MMM10} the
constructed boundary triple is the transposed boundary triple
$\{L^2(\partial\Omega),\Gamma_{1,G}, -\Gamma_{0,G}\}$ and resulted
in the Weyl function
\[
 -M_G(z)^{-1}=G^{-1}(\Lambda(0)-\Lambda(z))G^{-1};
\]
(for the sign change, notice that in \cite{MMM10} interior, instead
of exterior, normal derivatives to $\partial\Omega$ are being used).
It is emphasized that the construction here relies on the general
renormalization result proved for abstract operators in
Theorem~\ref{essThm2}.

As shown in \cite{MMM10} with a renormalized boundary triple which
is ordinary it is possible to carry out spectral analysis for the
selfadjoint realizations of $-\Delta$ with the aid of Kre\u{\i}n's
resolvent formula. Alternatively, one can apply in the study a
Kre\u{\i}n type resolvent formula for the more general classes of
boundary triples as in Theorem~\ref{Kreinformula} or
Theorem~\ref{Kreinformula2}; see \cite[Example~7.27]{DHMS12} for a
discussion which uses the renormalization in \cite{MMM10}, and for
various related contributions in the study of such nonlocal boundary
conditions, see e.g. \cite{AmrPear04,BeLa07,BeLa2012,GeMi11,Gr68}.

\subsection{Mixed boundary value problem for Laplacian.}\label{ex:Mixed}
Let $\Omega$ be a bounded open set in $\dR^d$ $(d\ge 2)$ with a
smooth boundary $\partial\Omega$. Let $\Sigma_+$ be a compact smooth
submanifold of $\partial\Omega$ and $\Sigma_-
:=\partial\Omega\setminus\Sigma_+^\circ$, so that
$\Sigma=\Sigma_+\cup\Sigma_-$. Here $\Sigma_+^\circ$ is the interior
of $\Sigma_+$. Let $-\Delta_Z$ be the Zaremba
Laplacian,\index{Zaremba Laplacian} i.e. the restriction of the
maximal operator $\ZA_{\max}$ to the set of functions, which satisfy
Dirichlet boundary condition on $\Sigma_-$ and Neumann boundary
condition on $\Sigma_+$.

Let  $H^1_{\Sigma_+}(\Omega) =\{u\in  H^1(\Omega): \supp \gamma_D
u\subset \Sigma_+\}.$ It is  known (see for instance \cite{Gr11}),
that the operator $-\Delta_Z$ is associated with the  nonnegative
closed quadratic form
$$
\frak a_{\Sigma_+}(u,v) = \int_\Omega \nabla u\cdot \nabla \overline{v}\, dx
,\quad \dom \frak a_{\Sigma_+} = H^1_{\Sigma_+}(\Omega),
$$
hence it is  selfadjoint in $H^0(\Omega)$. Clearly, its spectrum
$\sigma(-\Delta_Z)$ is discrete.

Here we construct  an $ES$-generalized boundary triple, associated
with the Zaremba Laplacian. 

Let $\ZA_{\min}$ and $\ZA_*$ be the minimal  and  pre-maximal operators,
respectively, associated with $-\Delta$,    $\dom(\ZA_*)=
H_\Delta^{3/2}(\Omega) = H^{3/2}(\Omega)\cap \dom \ZA_{\max}$ (see
Section~\ref{sec7.1}).
%
%
Let $S_{*,+}$ be a realization of $-\Delta$  given by
\begin{equation}\label{eq:S*+-}
\dom S_{*,+}=\{f\in H_\Delta^{3/2}(\Omega):(\gamma_Nf)|_{\Sigma_+}=0\}.
\end{equation}
Using  the Green formula and then applying the regularity result for the realization
$-\Delta_N$ we derive that  $S_{+}=(S_{*,+})^*$ is a symmetric realization of the
Laplacian $-\Delta$ on the domain
  \begin{equation}\label{eq:S_+}
\begin{split}
 \dom S_{+}&=\{f\in H_\Delta^{3/2}(\Omega):\, \gamma_Nf=(\gamma_Df)|_{\Sigma_{-}}=0\}\\
& =\{f\in H^{2}(\Omega):\, \gamma_Nf = (\gamma_Df)|_{\Sigma_{-}}=0\} \subset \dom
\Delta_N.
\end{split}
   \end{equation}
Hence
\[
\ZA_{\min}\subset S_{+}\subset S_{*,+}\subset (S_{+})^* \subset \ZA_{\max},
\]
in particular,  $S_{+}$  is an intermediate extension of $S=\ZA_{\min}$ in the sense of
\cite{DHMS09}. More precisely we have the following result.

%

\begin{proposition}\label{prop:Mixed}
Let the operator $S_{*,+}$  be defined by~\eqref{eq:S*+-} and let $S_{+}=(S_{*,+})^*$. Then:
\begin{enumerate}
  \item [(i)]  $\Pi^+=(L^2(\Sigma_{+}),P_{L^2(\Sigma_{+})}\gamma_N,P_{L^2(\Sigma_{+})}\gamma_D)$  is a
  $B$-generalized boundary triple for $(S_{+})^*$;
  \item [(ii)] the Weyl function corresponding to the boundary triple $\Pi^+$ equals to
  \begin{equation}\label{eq:M-}
    \Lambda_+(z)=P_{L^2(\Sigma_{+})}\Lambda(z)^{-1}\uphar{L^2(\Sigma_{+})},
  \end{equation}
  where $\Lambda(z)^{-1}$ is the Neumann-to-Dirichlet map;\index{Neumann-to-Dirichlet map}
  \item [(iii)]
   $(\Pi^+)^\top=(L^2(\Sigma_{+}),P_{L^2(\Sigma_{+})}\gamma_D,-P_{L^2(\Sigma_{+})}\gamma_N
   )$  is an $ES$-generalized boundary triple for $(S_{+})^*$.
\end{enumerate}
\end{proposition}
\begin{proof}
(i)  As is proved in Proposition  \ref{prop:Laplacian}(i) the triple
$\Pi=(L^2(\partial\Omega),\gamma_N,\gamma_D)$ is a $B$-generalized
boundary triple  for $\ZA^*$.  Since $S_{+}$ is an intermediate
extension of $\ZA$,~\cite[Proposition~4.1]{DHMS09} implies that the
triple
$\Pi^+=(L^2(\Sigma_{+}),P_{L^2(\Sigma_{+})}\gamma_N,P_{L^2(\Sigma_{+})}\gamma_D)$
is a $B$-generalized boundary triple for $(S_{+})^*$. Notice that
\[
\ran(P_{L^2(\Sigma_{+})}\gamma_N)= L^2(\Sigma_{+}),
\]
and the operator $A_{0,+}$ defined as the restriction of $-\Delta$ to the domain
\[
\begin{split}
\dom A_{0,+}&=\ker\Gamma_0^+=\{f\in \dom(S_{*,+}): P_{L^2(\Sigma_{-})}\gamma_Nf=0\}\\
    &=\{f\in H_\Delta^{3/2}(\Omega): \gamma_Nf=0\} = \{f\in H^{2}(\Omega):\, \gamma_Nf =0\} = \dom (-\Delta_N).
\end{split}
\]
is selfadjoint, since it coincides with the Neumann Laplacian.

(ii) This statement is implied by the fact that the Weyl function of the operator $\ZA$,
corresponding to the boundary triple $\Pi=(L^2(\partial\Omega),\gamma_N,\gamma_D)$,
coincides with $\Lambda(z)^{-1}$; see \cite[Proposition~4.1]{DHMS09}.

(iii) Consider the operator $A_{1,+}$ defined as the restriction of $-\Delta$ to the domain
\[
\begin{split}
    \dom A_{1,+}&=\{f\in \dom(S_{*,+}):P_{L^2(\Sigma_{-})}\gamma_Df)=0\}\\
            &=\{f\in H_\Delta^{3/2}(\Omega): (\gamma_Nf)|_{\Sigma_+}=(\gamma_Df)|_{\Sigma_{-}}=0\}.
\end{split}
\]
Note that  $\dom(-\Delta_Z)\subset H^{3/2-\varepsilon}(\Omega)$ for each $\varepsilon
>0$ while $\dom(-\Delta_Z)\not\subset H^{3/2}(\Omega)$ for certain configurations of $\Sigma_{+}$
(see~\cite{Gr11}). Therefore, for such subsets $\Sigma_{+}$  the
operator $A_{1,+}$ is a proper symmetric restriction of Zaremba
Laplacian $-\Delta_Z$,  hence $A_{1,+}$ is not selfadjoint.

To prove the statement it suffices to show that the operator $A_{1,+}$ is essentially
selfadjoint. Assuming the contrary one finds $\lambda_0
=\bar\lambda_0 \not\in \sigma_p(-\Delta_Z)$ and a vector
$g\in L^2(\Omega)$ such that $g\perp\ran(A_{1,+}-\lambda_0)$, i.e.
%
\begin{equation}\label{eq:ort_g}
(g,(-\Delta - \lambda_0)f)_{L^2(\Omega)}=0\quad\mbox{for all}\quad
f\in\dom A_{1,+}.
  \end{equation}
This relation with $f\in \dom S$ implies  $g\in\dom(\ZA_{\max})$ and
$(-\Delta-\lambda_0)g=0$. Letting $f\in\dom S_{+}$  and  applying the  Green formula one
obtains from~\eqref{eq:S_+} and~\eqref{eq:ort_g}  that
%
%
\begin{equation}\label{eq:ort_g2}
    \begin{split}
        0&=(g,-\Delta f)_{L^2(\Omega)}-(\lambda_0 g,f)_{L^2(\Omega)} \\
        &=(g,-\Delta f)_{L^2(\Omega)}-(-\Delta g,f)_{L^2(\Omega)}\\
        &=\left\langle\gamma_D g,\gamma_N f\right\rangle_{-1/2,1/2}-\left\langle\gamma_N g,\gamma_D f\right\rangle_{-3/2,3/2}=\left\langle
        -(\gamma_N g)|_{\Sigma_{+}},(\gamma_Df)|_{\Sigma_{+}}\right\rangle_{-3/2,3/2},
    \end{split}
\end{equation}
Here $\left\langle\cdot,\cdot\right\rangle_{-s,s}$ denotes duality between
$H^{-s}(\partial\Omega)$ and $H^{s}(\partial\Omega)$ $(s\in\dR)$. It follows from~\eqref{eq:Trace_Thm} that
$\gamma_D(\dom \Delta_N)=H^{3/2}(\partial\Omega)$. Hence
$\gamma_D
(\dom S_{+})=H^{3/2}(\Sigma_+)$ and the latter
 implies  
\begin{equation}\label{eq:gamma_N+}
    (\gamma_Ng)|_{\Sigma_+}=0.
\end{equation}

Similarly, it follows
from~\eqref{eq:Trace_Thm} that $\gamma_N(\dom \Delta_D)=H^{1/2}(\partial\Omega)$.
For a subset $\cL$ of $\dom A_{1,+}$
\[
 \cL=\{f\in H^2(\Omega):\, (\gamma_Nf)|_{\Sigma_+}=\gamma_Df=0\}.
\]
one obtains
\begin{equation}\label{eq:gamma_NL}
    \gamma_N\cL=H^{1/2}(\Sigma_{-}).
\end{equation}
Let now $f\in\cL$. Then using the Green formula the equality~\eqref{eq:ort_g} can be
rewritten as
\begin{equation}\label{eq:ort_g2+}
0=(g,-\Delta f)_{L^2(\Omega)}-(\lambda_0 g,f)_{L^2(\Omega)} =
\left\langle\gamma_D g,\gamma_N f\right\rangle_{-1/2,1/2}
\end{equation}
and~\eqref{eq:gamma_NL},~\eqref{eq:ort_g2+} lead to
\begin{equation}\label{eq:gamma_D-}
(\gamma_Dg)|_{\Sigma_{-}}=0.
\end{equation}
Since $g\in\dom(\ZA_{\max})$, relations~\eqref{eq:gamma_N+} and~\eqref{eq:gamma_D-} mean
that $g\in\dom(-\Delta_Z)$. Thus  $g\in\ker(-\Delta_Z-\lambda_0)=\{0\}$, hence  $g=0$.
This completes the proof.
\end{proof}

\begin{remark}
As follows from Theorem~\ref{essThm1} the statement (iii) in
Proposition~\ref{prop:Mixed} is equivalent to the fact that the
$\gamma$-field $\gamma(\lambda)$ admits a single-valued closure for
all $\lambda\in\dC_+\cup \dC_-$ with constant domain and the
$M$-function $-\Lambda_+(z)^{-1}$ is form domain invariant. As was
mentioned in the proof the operator $A_{1,+}$ is essentially
selfadjoint while is not selfadjoint. By
Theorems~\ref{prop:C6B},~\ref{prop:C6}, this implies that the
operators $\gamma(\lambda)$ are not bounded; this fact was
apparently first mentioned in~\cite[Theorem~6.23]{Post16}. In
particular, the corresponding boundary triple $(\Pi^+)^\top$ is
neither $S$-generalized, nor an $AB$-generalized or a quasi boundary
triple in the sense of \cite{BeLa07}.
\end{remark}

\subsection{Laplacians on Lipschitz domains}\label{sec7.3}\index{Laplacian!on Lipschitz domain}

Here the smoothness properties on $\Omega$ are relaxed; it is
assumed that $\Omega$ is a bounded Lipschitz domain. In this case
the Dirichlet and Neumann traces $\gamma_D$ and $\gamma_N$
\[
 \gamma_D: H^s_\Delta(\Omega)\to H^{s-1/2}(\partial\Omega), \quad
 \gamma_N: H^s_\Delta(\Omega)\to H^{s-3/2}(\partial\Omega),
\]
are still continuous operators for all $1/2\leq s\leq 3/2$ and, in
addition, both are surjective when $s=1/2$ and $s=3/2$; see
\cite[Lemmas~3.1,~3.2]{GeMi11}. In this case the results, which are
analogous to those in Section~\ref{sec7.1}, will be derived directly
from the abstract setting treated in Section~\ref{sec6.1}.

The following analog of Proposition~\ref{prop:Laplacian} is obtained
from Proposition~\ref{propBtoE} using the $3/2$ regularity of the
selfadjoint extensions $-\Delta_D$ and $-\Delta_N$; see
\cite{JeKe81,JeKe95,GeMi11}. Since $0\in\rho(-\Delta_D)$ one can
decompose
\[
 \dom \ZA_{\max}=\dom \Delta_D \dot{+} \ker \ZA_{\max}.
\]

\begin{proposition}\label{prop:Laplacian2}
Let $\Omega\subset \dR^n$, $n\ge 2$, be a bounded Lipschitz domain.
Let the operators $\gamma_N$, $\gamma_D$, $\cP(z)$, $\Lambda(z)$ and
$\ZA_*$ be defined by~\eqref{eq:TraceThm0},~\eqref{eq:Pz},
~\eqref{eq:PSt}, and~\eqref{eq:wtS*0}. Then:
\begin{enumerate}\def\labelenumi {\textit{(\roman{enumi})}}
  \item $\{L^2(\partial\Omega), \gamma_D\uphar{\dom \ZA_*}, -\gamma_N\uphar{\dom \ZA_*}\}$
  is an $S$-generalized boundary triple for $\ZA^*$ with domain $\dom \ZA_*=H^{3/2}_\Delta(\Omega)$,
  the transposed boundary triple is $B$-generalized, moreover,
  the corresponding $\gamma$-field $\gamma(\cdot)$ is bounded and coincides with $\cP(z)$
  and the Weyl function $M(\cdot)$ coincides with $-\Lambda(\cdot)$;
  \item $\{L^2(\partial\Omega), \Gamma_{0,\Omega}, \Gamma_{1,\Omega}\}$, where
\begin{equation}\label{Etrans2}
    \begin{pmatrix}
    \Gamma_{0,\Omega} \\ \Gamma_{1,\Omega}
    \end{pmatrix}(f+\overline{\gamma(0)}h)
    =\begin{pmatrix}
     -\gamma(0)^*\Delta_D f \\ -h
    \end{pmatrix},
 \quad f\in \dom \Delta_D, \quad h \in L^2(\partial\Omega)\},
\end{equation}
defines an $ES$-generalized boundary triple for $\ZA_{\max}$ with
dense domain $\ZA_*=\dom\Delta_D+\ran\overline{\gamma(0)}\subset
\ZA^*$, the transposed boundary triple is $B$-generalized, and the
corresponding Weyl function is the $L^2(\partial\Omega)$-closure
\[
 \wt M(z)=\clos(\Lambda(z)-\Lambda(0))^{-1};
\]
  \item the extension $\wt A_0 := \ZA_{\max}\uphar{\ker \wt\Gamma_{0, \Omega}}$ is essentially selfadjoint
and its closure coincides with the Kre\u{\i}n - von Neumann
extension of the operator $\ZA_{\min}$.
\end{enumerate}
\end{proposition}
\begin{proof}
(i) Green's identity holds: this can be obtained for instance from
the formula (3.21) in \cite{GeMi11} (cf. proof of Proposition
\ref{ABGrough} below). According to \cite{JeKe81,JeKe95,GeMi11}
\[
    \Delta_D = \Delta\uphar\{y\in H^{3/2}_\Delta(\Omega):\,\gamma_D y=0\,\} \quad \text{and}
\quad \Delta_N = \Delta\uphar\{y\in H^{3/2}_\Delta(\Omega):\,\gamma_N
y=0\,\},
\]
are selfadjoint operators in $L^2(\partial\Omega)$ and, in addition,
$0\in\rho(-\Delta_D)$. Hence, $\dom M(\cdot)=\ran
\gamma_D=H^{1}(\partial\Omega)$ and $\ran M(\cdot)=\ran
\gamma_N=H^{0}(\partial\Omega)$. Thus, $\{L^2(\partial\Omega),
\gamma_D\uphar{\dom \ZA_*}, -\gamma_N\uphar{\dom \ZA_*}\}$ is an
$AB$-generalized boundary triple. Moreover, according to
\cite[Theorem~5.7]{GeMi11} the corresponding Weyl function
$M(\cdot)$ is a bounded operator from $H^{1}(\partial\Omega)$ to
$L^2(\partial\Omega)$. Since $M(z)$, $z\in\rho(-\Delta_D)$, is
surjective, the inverse $M(z)^{-1}$, $z\in\rho(-\Delta_D)\cap
\rho(-\Delta_N)$, is bounded from $L^2(\partial\Omega)$ to
$H^{1}(\partial\Omega)$; in particular, $-M(z)^{-1}$ is bounded in
$L^2(\partial\Omega)$. From Corollary~\ref{QBRcor}~(i) one concludes
that the $AB$-generalized boundary triple $\{L^2(\partial\Omega),
\gamma_D\uphar{\dom \ZA_*}, -\gamma_N\uphar{\dom \ZA_*}\}$ is a
unitary, i.e., it is $S$-generalized. The assertion concerning the
$\gamma$-field is obtained from Theorem~\ref{prop:C6B}. The
transposed boundary triple is $B$-generalized, since $\gamma_N:
H^{3/2}_\Delta(\Omega)\to H^{0}(\partial\Omega)$ is surjective, or
since the corresponding Weyl function $-M(z)^{-1}$ is bounded.

(ii) This result is obtained directly from Proposition~\ref{propBtoE} with $A_0=-\Delta_D$, $G :=\overline{\gamma(0)}$ which is bounded by item (i)
and $E:= -\Lambda(0)$ which is selfadjoint, since $0\in\rho(-\Delta_D)$.

(iii) This follows from Proposition~\ref{propBtoE} and Theorem
\ref{propBetE2}; see~\eqref{KvNRenorm}.
\end{proof}

Next we apply the renormalization result in Theorem~\ref{propBetE2}
to the $ES$-generalized boundary triple in
Proposition~\ref{prop:Laplacian2}.

\begin{proposition}\label{propBetE3}
Let the notations and assumptions be as in
Proposition~\ref{prop:Laplacian2}. Moreover, let
$\{L^2(\partial\Omega), \Gamma_{0,\Omega}, \Gamma_{1,\Omega}\}$ be
the $ES$-generalized boundary triple with the Weyl function $\wt
M(\cdot)$ and let $P_0$ be the orthogonal projection onto
$\sN_0:=\ker \ZA_{\max}$. Then:
\begin{enumerate}\def\labelenumi {\textit{(\roman{enumi})}}
\item the Weyl function $\wt M(z)=\clos(\Lambda(z)-\Lambda(0))^{-1}$
is form domain invariant,
\[
 \dom \overline{\st_{\wt M(z)}}=\ran \gamma(0)^*, \quad z\in \rho(-\Delta_D);
\]

\item the renormalized boundary triple $\{\sN_0,\Gamma_{0,r},\Gamma_{1,r}\}$, where
\[
    \begin{pmatrix}
    \Gamma_{0,r} \\ \Gamma_{1,r}
    \end{pmatrix}(f+h)
    =\begin{pmatrix}
      -P_0\Delta_Df\\ -h
    \end{pmatrix},
\quad f\in \dom \Delta_D,\quad h\in \sN_0,
\]
is an ordinary boundary triple for $\ZA_{\max}$;

\item the corresponding Weyl function is given by
\[
 M_r(\lambda)=A^-_{11}-1/\lambda-(A^-_{21})^*(A^-_{22}-1/\lambda)^{-1}A^-_{21},
 \quad \lambda\in\rho(-\Delta_D).
\]
where $-\Delta_D^{-1}=(A^-_{ij})_{i,j=1}^2$ is decomposed according
to $\sH=\sN_0\oplus (\sN_0)^\perp$.
\end{enumerate}
\end{proposition}
\begin{proof}
The result is obtained by applying Theorem~\ref{propBetE2} to
Proposition \ref{prop:Laplacian2} with the choices $A_0=-\Delta_D$
and $G :=\overline{\gamma(0)}$.
\end{proof}

As a consequence one has the following result:

\begin{corollary}
The inverse of the regularized Dirichlet-to-Neumann map\index{Dirichlet-to-Neumann map!regularized}
$\clos(\Lambda(z)-\Lambda(0))$ has the form
\[
 \wt M(z)=\clos(\Lambda(z)-\Lambda(0))^{-1}=\gamma(0)^{(-1)}M_r(z)\gamma(0)^{-(*)}
\]
and, consequently, the Dirichlet-to-Neumann map has the
representation
\[
 \Lambda(z)=\Lambda(0)+\gamma(0)^{*}M_r(z)^{-1}\gamma(0), \quad z\in
 \rho(-\Delta_D).
\]
\end{corollary}

Notice that here by definition $M_r(0)^{-1}=(\infty^{-1}=)\,0$.

Comparing Proposition~\ref{prop:Laplacian2}~(ii) with
Proposition~\ref{propBetE3}~(i) we get the following equality
\[
 \ran\Gamma_{0,\Omega}=\dom \overline{\st_{\wt M(z)}}=\ran \gamma(0)^*, \quad z\in
 \rho(-\Delta_D).
\]
Furthermore, it is clear from~\eqref{Etrans2} that
\[
 \ran \Gamma_{0,\Omega} \times \Gamma_{1,\Omega}
    =\ran \gamma(0)^*\times L^2(\partial\Omega).
\]
In particular, one can renormalize the regularized boundary mappings
$\Gamma_{0,\Omega}=\gamma_N-\Lambda(0)\gamma_D$,
$\Gamma_{1,\Omega}=\gamma_D$ also by any bounded operator $G$ acting
in the original boundary space $L^2(\partial\Omega)$ satisfying
$\ran G=\ran \gamma(0)^*$ and $\ker G= \{0\}$, and this leads to an
isomorphic copy of the results in Proposition~\ref{propBetE3}. In
this case the parametrization of all intermediate extensions of
$\ZA_{\min}$ can be expressed via boundary conditions involving
$G^{-1}(\gamma_N-\Lambda(0)\gamma_D)$ and $G^*\gamma_D$; cf.
Corollary~\ref{Lapcor}.

\subsection{Laplacian on rough domains}\label{example8.3}\index{Laplacian! on rough domain}
Let $\Omega$ be a bounded domain in $\dR^d$ $(d\ge 2)$ whose
boundary $\partial\Omega$ is equipped with a finite
$(d-1)$-dimensional Hausdorff measure \index{Hausdorff measure} $\sigma$,
$\sigma(\partial\Omega)<\infty$. To construct an analog for the
boundary triple appearing in Proposition~\ref{prop:Laplacian}~(i) in
nonsmooth domains $\Omega$ we make use of some results established
in \cite{Daners2000} and \cite{ArEl11,ArEl12,ArW03}. Following
\cite[Definition~3.1]{ArEl11} we first recall the notion of a trace
$\varphi\in L^2(\sigma)$ for a class of functions $u\in
H^1(\Omega)$.

\begin{definition} \label{def:trace}
  A function $\varphi\in L^2(d\sigma)$ is said to be a trace of $u\in H^1(\Omega)$, if there is a sequence $u_n\in H^1(\Omega)\cap C(\overline{\Omega})$, such that
  \[
  \lim_{n\to\infty}u_n=u\quad(\text{in }H^1(\Omega))\quad \text{ and } \quad
  \lim_{n\to\infty}u_n|_{\partial\Omega}=\varphi\quad(\text{in }L^2(\sigma)).
  \]
\end{definition}
Denote by $H_\sigma^1(\Omega)$ the set of elements of $H^1(\Omega)$ for which there exists a trace. In general, the trace is not uniquely defined. It is possible that $u\mid \Omega=0$
while its trace $\gamma_D u=u\mid\partial\Omega$ in $L^2(\sigma)$ is
nontrivial; for an example see e.g. \cite[Example 4.4]{ArEl11}.
Define the linear relation $\gamma_D$ by
\[
 \gamma_D:=\{\{u,\varphi\}:\, u\in H_\sigma^1(\Omega),\,\varphi\in L^2(\sigma),\,\varphi\textup{ is a trace of }u\}.
\]
Then $\gamma_D$ can be considered as a mapping from $H^1(\Omega)$ to
$L^2(\sigma)$, which is linear but in general multivalued on the
domain $H_\sigma^1(\Omega)$ and it has dense range in $L^2(\sigma)$;
cf. \cite{ArEl11}. If $u$ and $\varphi$ are as in
Definition~\ref{def:trace} we shall write
\[
\varphi\in\gamma_Du.
\]
The space $H_\sigma^1(\Omega)$ coincides with the closure of
$H^1(\Omega)\cap C(\overline{\Omega})$ in the norm
\begin{equation}\label{H1L2}
 \|u\|_{1,\sigma}^2= \|u\|^2_{H^1(\Omega)} + \int_{\partial\Omega} |u|^2\, d\sigma.
\end{equation}
Following \cite{ArEl11} denote by $\widetilde H^1(\Omega)$ the closure of $H^1(\Omega)\cap C(\overline{\Omega})$ in $H^1(\Omega)$.
In view of~\eqref{H1L2} $H_\sigma^1(\Omega)$ is a subset of $\widetilde H^1(\Omega)$. Without additional conditions on $\Omega$
the space $\widetilde H^1(\Omega)$ need not be dense in $H^1(\Omega)$.
Some sufficient conditions, like $\Omega$ being starshaped or having a continuous
boundary, can be found e.g. in \cite[Section~1.1.6]{Mazya}. Consequently,
$H_\sigma^1(\Omega)$ is not necessarily a dense subset of $H^1(\Omega)$.

For associating an appropriate boundary triple in this setting, we impose the following additional assumption.

\begin{assumption}\label{assumeT}
 $H_\sigma^1(\Omega)=\widetilde H^1(\Omega)$.
\end{assumption}

A list of conditions equivalent to Assumption~\ref{assumeT} is given in~\cite[Theorem~6.1]{ArEl11}.
Notice that the space $H_\cH^1(\Omega)$ appearing in \cite[Section 5]{ArEl11} has a norm
which is equivalent to norm of $H_\sigma^1(\Omega)$ defined in~\eqref{H1L2} due to
the following special case of Maz'ya inequality: there exists a
constant $c_M>0$ such that
\begin{equation}\label{eq:Friedr_Ineq}
 \int_\Omega |u|^2 \,dx \leq c_M\left( \int_\Omega |\nabla u|^2\,dx
 + \int_{\partial\Omega} |u|^2\, d\sigma \right)
\end{equation}
holds for all $u\in H^1(\Omega)\cap C(\overline{\Omega})$;
see \cite[Section 3.6]{Mazya}, \cite[eq. (5)]{ArEl11}.
The inequality~\eqref{eq:Friedr_Ineq} is a generalization of Friedrichs inequality to the case of rough domains.

In \cite[Definition 3.2]{ArEl11} the (weak) normal derivative is
defined implicitly via Green's (first) formula as follows: a
function $u\in H^1(\Omega)$ with $\Delta u\in L^2(\Omega)$ is said
to have a weak normal derivative in $L^2(\sigma)$ if there exists
$\psi\in L^2(\sigma)$ such that
\begin{equation}\label{eq:1stGreen}
   \int_\Omega (\Delta u) \overline{v}\,dx +\int_\Omega \nabla u\cdot \overline{\nabla v}\,dx=\int_{\partial\Omega} \psi \overline{v} \,d\sigma
\end{equation}
holds for all $v\in H^1(\Omega)\cap C(\overline\Omega)$, where $\Delta u$ denotes the Laplacian understood in distributional sense.
Since the functions $v\uphar \partial\Omega$, $v\in H^1(\Omega)\cap C(\overline\Omega)$,
form a dense set in $L^2(\sigma)$, the function $\psi\in L^2(\sigma)$ is uniquely
determined by $u$ and the mapping $u\to \psi$ is denoted by $\gamma_N$:
\[
 \gamma_Nu:=\psi,\quad u\in\dom\gamma_N\subset H^1(\Omega)\cap\dom A_{\rm max}.
\]

Assume that for some $ \varphi,\psi\in L^2(\sigma)$, $ u\in H^1(\Omega)$, and  $x\le 0$ one has
\begin{equation}\label{eq:DtoN_map}
  (-\Delta -xI)u=0,\quad \varphi\in\gamma_D u,  \quad \psi=\gamma_N u,\quad  x\le 0.
\end{equation}
The operator $\Lambda(x)$ which maps $\varphi$ to $\psi$ is called  the {\it Dirichlet-to-Neumann map}.
A slight modification of the proof of \cite[Theorem 3.3]{ArEl11} shows, that  $\Lambda(x)$ is a nonnegative selfadjoint operator on $L^2(\sigma)$ which is uniquely determined by the three properties listed in~\eqref{eq:DtoN_map}.



Now consider the differential expression
$-\Delta$, where $\Delta=\nabla\cdot\nabla$ is the (distributional) Laplacian operator
in $\Omega$.
Recall (see \cite[Example~3.1]{ArW03}) that for an open set $\Omega$ (without any regularity on the
boundary) the Dirichlet Laplacian $-\Delta_D$ is defined as the
selfadjoint operator associated with the closed (Dirichlet) form
\[
 \tau_D(f,g)=\int_\Omega \nabla f \cdot\overline{\nabla g} \,dx, \quad
 \dom \tau_D=H_0^1(\Omega).
\]
Similarly the Neumann Laplacian $-\Delta_N$ is defined as the
selfadjoint operator associated with the closed form (see~\cite[Example~3.2]{ArW03})
\begin{equation}\label{NeumanLap}
 \tau_N(f,g)=\int_\Omega \nabla f \cdot \overline{\nabla g}\,dx, \quad
  \dom \tau_N=\wt H^1(\Omega).
\end{equation}

\begin{proposition}\label{ABGrough}
Let $\Omega\subset \dR^d$, $d\geq 2$, whose boundary
$\partial\Omega$ is equipped with a finite $(d-1)$-dimensional
Hausdorff measure $\sigma$, let Assumption \ref{assumeT} be in force, and let the linear relation $\Gamma$ be defined by \begin{equation}\label{eq:Gamma_rough}
  \Gamma=\left\{\left\{f,\begin{pmatrix}
   \varphi \\ -\psi
   \end{pmatrix} \right\}:\,\begin{array}{ccc}
                        f \in \wt H^1(\Omega)\cap\dom \gamma_N, & \varphi,\psi\in L^2(\sigma), & \Delta f\in L^2(\Omega)  \\
                        \varphi\in\gamma_D f, & \psi=\gamma_N f &
                      \end{array}
  \right\}.
\end{equation}
Then: 
\begin{enumerate}\def\labelenumi {\textit{(\roman{enumi})}}
\item[(i)]
the pair $\{L^2(\sigma),\Gamma\}$ is a positive unitary boundary pair for $-\Delta$ on $\ZA_*:=\dom \Gamma$;
\item [(ii)]
for every $x<0$ the Weyl function $M(x)$ corresponding to the pair $\{L^2(\sigma),\Gamma\}$ coincides (up to the sign) with the Dirichlet-to-Neumann map $\Lambda(x)$:
\begin{equation}\label{eq:Weyl_rough}
M(x)=-\Lambda(x),\quad x<0,
\end{equation}
in particular, the function $M(\cdot)$ is an inverse Stieltjes function whose values
$M(z)$, $z\in \dC\setminus [0,\infty)$, are (unbounded) operators with $\ker M(z)=\mul \Gamma_0$;

\item [(iii)]
the operator  $A_1:=-\Delta\uphar {\ker \Gamma_1}$ coincides with the Neumann Laplacian $-\Delta_N$;
\item [(iv)]
the transposed pair $\{L^2(\sigma),\Gamma^\top\}$ is $S$-generalized
and the corresponding Weyl function $-M(\cdot)^{-1}$ is a
multivalued domain invariant Stieltjes function.
\end{enumerate}
\end{proposition}
\begin{proof}
(i)--(iii) If $f\in \dom \Gamma$, then the (first) Green's identity
~\eqref{eq:1stGreen} holds with $u=f$ and $v\in H^1(\Omega)\cap
C(\overline\Omega)$. Then in view of~\eqref{H1L2} this identity can
be extended to hold for all $v\in H_\sigma^1(\Omega)$. Thus, in
particular, it holds for all $g:=v\in \dom \Gamma$:
\begin{equation}\label{eq:1stGreen2}
   \int_\Omega (\Delta f) \overline{g}\,dx +\int_\Omega \nabla f\cdot \overline{\nabla g}\,dx=\int_{\partial\Omega} \psi \overline{\wt\varphi} \,d\sigma, \quad \psi=\gamma_Nf,\,\, \wt\varphi\in\gamma_Dg.
\end{equation}
Similarly, one gets from~\eqref{eq:1stGreen} with $u=g\in\dom \Gamma$ and $v=f\in \dom \Gamma$:
\[
   \int_\Omega (\Delta g) \overline{f}\,dx +\int_\Omega \nabla g\cdot \overline{\nabla f}\,dx =\int_{\partial\Omega} \wt\psi \overline{\varphi} \,d\sigma, \quad \wt\psi=\gamma_Ng,\,\, \varphi\in\gamma_Df.
\]
Taking conjugates in the last identity and subtracting the identity~\eqref{eq:1stGreen2}
from that leads to Green's (second) formula in~\eqref{Green1} for $-\Delta$ with $f,g\in \ZA_*=\dom \Gamma$.
This means that  $\{L^2(\sigma),\Gamma\}$ is an isometric boundary pair.

To prove that  $\{L^2(\sigma),\Gamma\}$ is a unitary boundary pair, we proceed by proving (ii) and (iii).
With $x<0$ it follows from~\eqref{eq:Gamma_rough} that $\varphi\in\dom M(x)$ and $M(x)\varphi=-\psi$ precisely when there exists $u\in\wt H^1(\Omega)\cap\dom\gamma_N$, such that
\[
-\Delta u-xu=0,\quad \varphi\in\gamma_D u, \quad \psi=\gamma_N u.
\]
In view of~\eqref{eq:DtoN_map} this means that the operator $-M(x)$
coincides with the Dirichlet-to-Neumann map $\Lambda(x)$, which is a
nonnegative selfadjoint operator in $L^2(\sigma)$. This proves
~\eqref{eq:Weyl_rough}. The definition of $\Gamma$ shows that $\mul
\Gamma=\mul\Gamma_0\times\{0\}$ and hence by Lemma~\ref{Weylmul}
$\ker M(z)=\mul \Gamma_0$ does not depend on $\lambda\in\cmr$. The
assertion that $M(\cdot)$ is an inverse Stieltjes function is a
consequence of $M(x)\leq 0$, $x<0$, (the nonnegativity of the main
transform $\wt A$, which is shown below, implies that $M(x)$ is also
holomorphic at $x<0$). This proves (ii).

By definition every $f\in \dom(-\Delta_N)$ belongs to $\wt H^1(\Omega)$. On the other hand, by Assumption~\ref{assumeT} $\wt H^1(\Omega)=H_\sigma^1(\Omega)=\dom \gamma_D$ and hence, in particular, for every
$f\in \dom(-\Delta_N)$ there exists a Dirichlet trace $\varphi\in\gamma_D f$. Next it is shown that
for every $f\in \dom(-\Delta_N)$ also the Neumann trace $\gamma_D u$ exists. Indeed,
by definition the Neumann Laplacian $-\Delta_N$ is the
selfadjoint operator associated with the closed form~\eqref{NeumanLap}.
Hence,~\eqref{NeumanLap} implies that for all $f\in \dom(-\Delta_N)$ and $g\in \wt H^1(\Omega)$ ,
\[
 \int_\Omega \nabla f \cdot \overline{\nabla g}\,dx
 =\int_\Omega (-\Delta f) \overline{g} \,dx.
\]
Comparing this identity with the definition of $\gamma_N$ it is seen
that the equality~\eqref{eq:1stGreen} is satisfied with the choice
$\psi=0$. Therefore, $f\in \dom \gamma_N$ and $\gamma_N f=0$. This
implies that $ \dom(-\Delta_N)\subset\dom A_*$ and, moreover, that
$\dom(-\Delta_N)\subset\ker\gamma_N=\dom A_1$. Since $-\Delta_N$ is
a selfadjoint operator in $L^2(\Omega)$ and $A_1$  is symmetric (see
Section \ref{sec3.3}), the equality $A_1=-\Delta_N$ follows. This
proves the assertion (iii).


Next we complete the proof of (i) by showing that $\{L^2(\sigma),\Gamma\}$ is a positive unitary boundary pair,
i.e., that the main transform $\wt A$ of $\Gamma$ given by
\begin{equation}
\label{awig0}
 \wt A :=
 \left\{\,
 \left\{ \begin{pmatrix} f \\ \varphi \end{pmatrix},
         \begin{pmatrix} -\Delta f \\ \psi \end{pmatrix}
 \right\} :\,
 \left\{ \begin{pmatrix} f \\ -\Delta f \end{pmatrix},
         \begin{pmatrix}  \varphi \\ -\psi \end{pmatrix}
 \right\}
 \in \Gamma
 \,\right\}
\end{equation}
is a nonnegative selfadjoint relation in $L^2(\Omega)\times
L^2(\sigma)$; see~\eqref{awig}. Nonnegativity of $\wt A$ follows
immediately from~\eqref{eq:1stGreen2}. On the other hand, by item
(ii) the Weyl function satisfies $-M(x)=\Lambda(x)\geq 0$, $x<0$,
and hence it is a nonpositive selfadjoint operator with
$-x\in\rho(M(x))$. Since $\dom(-\Delta_N)\subset \dom A_*$ and
$-\Delta_N\geq 0$ is selfadjoint it follows from
Theorem~\ref{MRealreg} that $x\in \rho(\wt A)$ and hence $\wt A=\wt
A^*\ge 0$, which proves the claim.

(iv) Since $A_1=-\Delta_N$ is selfadjoint, the transposed pair
$\{L^2(\sigma),\Gamma^\top\}$ is $S$-generalized; see Definition
\ref{SgenBT}. Moreover, the corresponding Weyl function
$-M(x)^{-1}\ge 0$ of $\{L^2(\sigma),\Gamma^\top\}$ is a nonnegative
selfadjoint relation in $L^2(\sigma)$ for every $x<0$. This implies
that $-M(\cdot)^{-1}$ is a (multivalued) Stieltjes family. It is
domain invariant by Theorem~\ref{prop:C6}.
\end{proof}

In this general setting, the multivalued part of $\Gamma$ can be
nontrivial, since the trace $\gamma_D$ need not be uniquely
determined. For unitary boundary pairs the multivalued part is
described in \cite[Lemma~4.1]{DHMS06}; see also Lemma~\ref{Weylmul}.
In the present setting a more explicit description of the
multivalued part can be given with the aid of a result of Daners in
\cite{Daners2000}; see also \cite{ArW03} for an other proof of
Daners result via capacity arguments.

\begin{corollary}
There exists a Borel set $B_0\subset \partial\Omega$, such that
\[
 \mul \gamma_D=L^2(B_0),\quad \mul \Gamma=\mul
 \gamma_D\times\{0\}.
\]
and, in particular, $\mul\gamma_D=\ker M(\lambda)$, $\lambda\in\dC\setminus[0,\infty)$.
\end{corollary}

Hence, $\Gamma$ is single-valued if and only if $L^2(B_0)=\{0\}$,
i.e., $\sigma(B_0)=0$. The set $B_0$ is unique up to
$\sigma$-equivalence $\sigma(B_0\Delta \wt B_0)=0$. Since
$\mul\gamma_D\neq 0$ corresponds to $\sigma(B_0)>0$, $B_0$ can be
considered to represent an irregular part of the boundary.

\begin{remark}
In this general setting we do not know if the operator
$A_0:=-\Delta\uphar {\ker \Gamma_0}$ coincides with the Dirichlet
Laplacian $-\Delta_D$. In other words, we do not know if the Neumann
trace $\gamma_Nu$ exists for every $u\in \dom (-\Delta_D)$.
%
%
\end{remark}

\section{Applications to differential operators with local point interactions}\label{sec8}

\subsection{Abstract results on direct sums of boundary triples \index{Direct sum of boundary triples} and their Weyl functions}
\label{sec8.1}
%

A general class of unitary boundary triples, which are more general
than generalized boundary triples is obtained by considering an
infinite orthogonal sum of ordinary boundary triples. Here we mainly
follow the considerations in \cite{KosMMM}; see also the references
given therein.

Let $S_n$ be a densely defined symmetric operator with equal defect
numbers $n_+(S_n)=n_+(S_n)$ in the Hilbert space $\sH_n$, $n\in\dN$.
Consider the operator $A=\bigoplus_{n=1}^\infty S_n$ in the Hilbert
space $\sH:=\bigoplus_{n=1}^\infty \sH_n=\{\, \oplus_{n=1}^\infty
f_n:\, f_n\in\sH_n,\,\, \sum_{n=1}^\infty \|f_n\|^2<\infty\}$. Then
$A$ is symmetric with equal defect numbers and its adjoint $A^*$ is
given by
   \begin{equation}\label{eq:8.A^*}
 A^*=\bigoplus_{n=1}^\infty S_n^*, \quad
\dom A^*= \left\{\,{\displaystyle \bigoplus_{n=1}^\infty f_n}\in\sH :\, f_n\in\dom
S_n^*,\,\, \sum_{n=1}^\infty \|S_n^*f_n\|^2<\infty\right\}.
\end{equation}
Now let $\Pi_n=\{\cH_n,\Gamma_0^{(n)},\Gamma_1^{(n)}\}$ be an ordinary boundary triple
for $S_n^*$, $n\in\dN$. Let $\cH=\bigoplus_{n=1}^\infty \cH_n$, $\Gamma^{(n)} :=
\{\Gamma^{(n)}_0, \Gamma^{(n)}_1\}$ and let the mapping $\Gamma_0'$ and $\Gamma_1'$
be defined by
%
%
   \begin{equation}\label{III.1_02-second}
{\Gamma_j}'  := \bigoplus_{n=1}^{\infty} \Gamma^{(n)}_j,\quad \dom \Gamma_j' =
\bigl\{{\displaystyle \bigoplus_{n=1}^\infty f_n} \in\dom A^*: \ \sum_{n\in \N}\|\Gamma^{(n)}_j
f_n\|^2_{\cH_n} <\infty \bigr\},  \quad  j\in\{0,1\}.
   \end{equation}

We also put 
   \begin{equation}\label{III.1_02}
\Gamma=\begin{pmatrix}\Gamma_0\\\Gamma_1\end{pmatrix}:=\begin{pmatrix}\Gamma_0'\\\Gamma_1'\end{pmatrix} \upharpoonright \dom \Gamma, \quad\textup{where}\quad\dom \Gamma =  \dom
\Gamma_1'\cap\dom \Gamma_0'.
   \end{equation}
 Then
${\Gamma_j}' = \overline {\Gamma_j}$, $j=0,1$. \ Denote by $\sH_+$
the domain  $\dom A^*$  equipped with the graph norm of $A^*$.
Clearly,  $\dom \Gamma $  is dense in $\gH_+$. Define the operators
$S_{n,j} := S_{n}^*\upharpoonright \ker \Gamma^{(n)}_j$ and ${A}_j
:= \bigoplus^{\infty}_{n=1}S_{n,j}$, $j\in \{0,1\}$. Then ${A}_0$
and ${ A}_1$ are selfadjoint extensions of $A$. Note that ${A}_0$
and ${A}_1$ are disjoint but not necessarily transversal.

Finally, we  set
\begin{equation}\label{III.1_03}
A_* := A^*\upharpoonright\dom \Gamma\quad \text{and}\quad A_{*j} :=
A_*\upharpoonright\ker\Gamma_j,\quad  j\in\{0,1\}.
   \end{equation}
Clearly, $\overline{A_{*j}} = A_{j}$, hence  $A_{*j}$ is essentially
selfadjoint, $j\in \{0,1\}$.

The following result is contained in \cite[Theorem~3.2]{KosMMM} (and
stated here in the terminology of the present paper).
  \begin{theorem}[\cite{KosMMM}]\label{otriple}
Let $\Pi_n=\{\cH_n,\Gamma_0^{(n)},\Gamma_1^{(n)}\}$ be an ordinary
boundary triple for $S_n^*$, let $S_{n,j}=S_n^*\uphar\ker
\Gamma_j^{(n)}$, $j\in \{0,1\}$, and let $M_n(\cdot)$, $n\in\dN$, be
the corresponding Weyl function. Moreover, let the operators $A^*$, $\Gamma_j'$ and $\Gamma_j$, $j\in \{0,1\}$, be given by~\eqref{eq:8.A^*},
~\eqref{III.1_02-second} and~\eqref{III.1_02}.
Then:
\begin{enumerate}\def\labelenumi {\textit{(\roman{enumi})}}
\item $\Pi=\{\cH,\Gamma_0,\Gamma_1\}$ is a unitary boundary triple for $A^*$;
\item the corresponding Weyl function is the orthogonal sum $M(z)=\bigoplus_{n=1}^\infty M_n(z)$;
\item the mapping $\Gamma_j:\sH_+\to \cH$  is closable and $\overline{\Gamma_j}=\Gamma_j'$, $j\in
\{0,1\}$;
\item The operator $A_{*j}$ 
given by~\eqref{III.1_03}  is essentially selfadjoint and
$\overline{A_{*j}}=\bigoplus_{n=1}^\infty S_{n,j} = A_j$, $j\in
\{0,1\}$.
  \end{enumerate}
\end{theorem}
The following result characterizes selfadjointness of $A_j=\ker
\Gamma_j$, $j\in \{0,1\}$,  and completes Theorem~3.2 from
\cite{KosMMM}.
   \begin{proposition}\label{A01selfadj}
Let the assumptions be as in Theorem~\ref{otriple} and let $A_j=\ker
\Gamma_j$, $j\in \{0,1\}$. Then
  \begin{equation}\label{SA01crit}
A_j=\bigoplus_{n=1}^\infty S_{n,j} \quad \Longleftrightarrow \quad
\Gamma_{j'}\uphar A_j \quad \text{is bounded }\quad
(j'=1-j\in\{0,1\}).
\end{equation}
In particular, $A_0$ satisfies~\eqref{SA01crit} (i.e. $A_0=A_0^*$)
if and only if the corresponding Weyl function $M(\cdot)$ and the
$\gamma$-field $\gamma(\cdot)$ satisfy one of the equivalent
conditions in Theorem~\ref{prop:C6B}.

Similarly, $A_1$ satisfies~\eqref{SA01crit} if and only if the Weyl
function $-M^{-1}(\cdot)$ and $\gamma$-field
$\gamma(\cdot)M^{-1}(\cdot)$ corresponding to the (unitary)
transposed boundary triple $\Pi^\top = \{\cH,\Gamma_1,-\Gamma_0\}$
satisfy one of the equivalent conditions listed  in
Theorem~\ref{prop:C6B}.
     \end{proposition}
\begin{proof}
The statements follow from Lemma~\ref{cor:GH} and
Proposition~\ref{prop:C3}. Indeed, by Proposition~\ref{prop:C3} (i)
$\Gamma_1H(\lambda)=\gamma(\bar\lambda)^*$ and hence
$\Gamma_1H(\lambda)$ is closed. Since $A_0$ is essentially
selfadjoint, the equivalence $A_0=A_0^*$ $\Longleftrightarrow$
$\Gamma_1\uphar A_0$ is bounded, is obtained from Lemma~\ref{cor:GH}
(iii), (v). All the other equivalent conditions for $A_0=A_0^*$ hold
by Theorem~\ref{prop:C6B}.

The criterion~\eqref{SA01crit} and the other equivalent statements
for $A_1=A_1^*$ are obtained by passing to the transposed boundary
triple $\{\cH,\Gamma_1,-\Gamma_0\}$.
\end{proof}

\begin{remark}\label{rem:7.20}
The criterion~\eqref{SA01crit} implies the sufficient conditions for
$A_0$ and $A_1$ to be selfadjoint as established in
\cite[Theorem~3.2]{KosMMM}. Namely, if $\Gamma_{1}$ or $\Gamma_{0}$
is bounded, then also the restriction $\Gamma_{1}\uphar A_0$ or
$\Gamma_{0}\uphar A_1$, respectively, is bounded. Moreover, if $A_0$
and $A_1$ are transversal, i.e. $\dom A_0 +\dom A_1=\dom A^*$, then
clearly $\Gamma_{j'}\uphar A_j$ is bounded $\Leftrightarrow$
$\Gamma_{j'}$ is bounded, since $\ker \Gamma_j=\dom A_j$
($j'=1-j\in\{0,1\}$).
\end{remark}

A criterion for  a  direct sum of ordinary boundary triples to form
also an ordinary boundary triple can be formulated in terms of the
corresponding  Weyl functions (see~\cite{MalNei12}, \cite{KosMMM},
\cite{CarMalPos13}).
  \begin{theorem}\label{th_criterion(bt)}
Let  $\Pi_n=\{\cH_n, \Gamma_0^{(n)}, \Gamma_1^{(n)} \}$ be a
boundary triple for $S_{n}^*$ and let $M_n(\cdot)$  be the
corresponding Weyl function, $n\in \dN $.

\begin{enumerate}\def\labelenumi {\textit{(\roman{enumi})}}
\item  The direct sum $\Pi=\oplus_{n=1}^{\infty}\Pi_n$ forms  an ordinary boundary triple for
the operator \linebreak $A^* =\oplus_{n=1}^{\infty}S_n^*$ if and only if
        \[
C_1= \sup_n\|M_n(i)\|_{\cH_n}  < \infty \quad\text{and}\quad C_2=
\sup_n\|(\IM M_n(i))^{-1}\|_{\cH_n} < \infty.
   \]
\item  The direct sum $\Pi=\oplus_{n=1}^{\infty}\Pi_n$ is a B-generalized boundary triple for
the operator \linebreak $A^* =\oplus_{n=1}^{\infty}S_n^*$ if and only if
$C_1<\infty$.

\item
If, in addition, the operators $\{S_{n,0}\}_{n\in \Bbb N}$ have a
common gap $(a-\varepsilon,a+\varepsilon)$,   then the direct sum
$\Pi= \bigoplus^{\infty}_{n=1} \Pi_n$ is a B-generalized boundary
triple for $A^*= \bigoplus^{\infty}_{n=1}S_n^*$ if and only if
    \begin{equation}\label{III.2.2_02}
C_3 := \sup_{n\in\N}\|M_n(a)\|_{\cH_n} < \infty \qquad \text{and}
\qquad C_4:= \sup_{n\in\N}\|M'_n(a)\|_{\cH_n} < \infty,
     \end{equation}
where  $M'_n(a):=({dM_n}(z)/{dz})|_{z=a}$.

\item   $\Pi= \bigoplus^{\infty}_{n=1} \Pi_n$ is an ordinary  boundary
triple for $A^*= \bigoplus^{\infty}_{n=1}S_n^*$ if and only if in
addition to~\eqref{III.2.2_02}  the following condition is fulfilled
    \begin{equation}\label{III.2.2_02NEW}
 C_5 := \sup_{n\in\N}\|\bigl(M'_n(a)\bigr)^{-1}\|_{\cH_n} < \infty.
     \end{equation}
\end{enumerate}
         \end{theorem}
The next statement contains analogous characterization for
S-generalized boundary triples.

\begin{proposition}\label{OrdinBTprop}
Assume the conditions of Theorem~\ref{otriple}. Then the direct sum
$\Pi=\bigoplus_{n=1}^{\infty}\Pi_n$ forms an S-generalized boundary
triple for $A^*= \bigoplus^{\infty}_{n=1}S_n^*$ if and only if
   \begin{equation}\label{SBT_criterion}
\sup_n\|\IM M_n(i)\|_{\cH_n}  < \infty.
   \end{equation}
Similarly, if the operators $(S_{n,0})$ have a common gap
$(a-\varepsilon,a+\varepsilon)$, then $\Pi$ forms  an S-generalized
boundary triple for
 $A^*$ if and only if $C_4 < \infty$ where $C_4$ is given by~\eqref{III.2.2_02}.
\end{proposition}
   \begin{proof}
The condition~\eqref{SBT_criterion}   means that $\IM M(z)$
($\Longleftrightarrow$ the gamma-field $\gamma(z)$) is bounded for
some (equivalently for every) $z\in \C_{\pm}$. By
Theorem~\ref{prop:C6B}, this amounts  to saying that  $\Pi$  is  an
S-generalized boundary triple for $A^*$.

Similarly, in case of a common spectral gap
$(a-\varepsilon,a+\varepsilon)$ the condition~\eqref{SBT_criterion}
is equivalent to the condition $C_4<\infty$ in~\eqref{III.2.2_02} as
can be seen by the same argument that was used in
Remark~\ref{rem:Mderiv}.
   \end{proof}
The next result is immediate by combining
Propositions~\ref{essProp3}  in~\eqref{III.2.2_02}
with~\ref{essProp3B}.
  \begin{corollary}\label{newCor1}
Let $\Pi=\{\cH,\Gamma_0,\Gamma_1\}=\bigoplus_{n=1}^\infty \Pi_n$
 be a direct sum of ordinary boundary triples
$\Pi_n=\{\cH_n,\Gamma_0^{(n)},\Gamma_1^{(n)}\}$, let
$M(\cdot)=\bigoplus_{n=1}^\infty M_n(\cdot)$ be the corresponding
Weyl function, and let $A_*:=\dom \Gamma$. Then the following
conditions are equivalent:
\begin{enumerate}
\item [(i)] $\Gamma_0:\,A_*\to \cH$ is bounded;
\item [(ii)] the following condition is satisfied
\[
  C_2=\sup_n\|(\IM M_n(i))^{-1}\|_{\cH_n} < \infty.
\]
\end{enumerate}
In this case the transposed boundary triple
$\Pi^\top=\{\cH,\Gamma_1,-\Gamma_0\}$ is B-generalized.

Similarly, the following conditions are equivalent:
\begin{enumerate}
\item [(i)$'$] $\Gamma_1:\,A_*\to \cH$ is bounded;
\item [(ii)$'$] the following condition is satisfied
\[
 C^\top_2:=\sup_n\|(\IM (M^{-1}_n(i)))^{-1}\|_{\cH_n} < \infty.
\]
\end{enumerate}
In this case the triple  $\Pi$ is a B-generalized boundary triple.
\end{corollary}
  \begin{proof}
By Theorem \ref{otriple} (see \cite[Theorem 3.2]{KosMMM}) $\Pi$ is a
unitary boundary triple such that $A_0=\ker \Gamma_0$ and $A_1=\ker
\Gamma_1$ are essentially selfadjoint. Now the first part of the
statement follows easily from Proposition \ref{essProp3B}, while the
second part is implied by Proposition \ref{essProp3}.
\end{proof}

\subsection{Momentum operators\index{Momentum operator} with local point interactions}\label{sec8.2}
$X=\{x_n\}^{\infty}_1$ be a strictly increasing sequence of positive
numbers satisfying $\lim_{n\to\infty}x_n = \infty$   and let $d_n$,
$d_*$ and $d^*$ be defined by~\eqref{dn**}.
Define a symmetric differential operator $D_n$ in $\cH_n :=
L^2([x_{n-1},x_n])$ by
   \be\la{r3.6New}
D_n= -i\frac{d}{dx}, \qquad \dom D_n=W^{1,2}_0([x_{n-1},x_n]),\qquad
n\in \N.
   \ee
In quantum mechanics this operator in 1-D case  appears in the form
$-i\hslash \frac{d}{dx}$, where $\hslash=h/2\pi$ is the reduced
Planck constant
and whose eigenvalues are measuring the momentum of a particle.

The adjoint of the operator $D_n$ is given by $D_n^*=-i\frac{d}{dx}$
with $\dom D_n^* = W^{1,2}([x_{n-1},x_n])$, $n \in \dN$. Following
\cite{MalNei12} associate with $D_n^*$ a boundary triple
$\Pi_n=\{\C,\Gamma^{(n)}_{0},\Gamma^{(n)}_{1}\}$ by setting
  \be\la{r3.8Anew}
\Gamma^{(n)}_{0}f_n := i\frac{f_n(x_n-0) -
f_n(x_{n-1}+0)}{\sqrt{2}},\quad \Gamma^{(n)}_{1}f_n :=
\frac{f_n(x_n-0) + f_n(x_{n-1}+0)}{\sqrt{2}}.
  \ee
The Weyl function $M_{n}(\cdot)$ corresponding to the triple
$\Pi_n$ is given by
    \begin{equation}\label{3.10new}
M_n(z)=-i\frac{e^{izx_n} + e^{izx_{n-1}}}{e^{izx_n} - e^{izx_{n-1}}}
= -\ctg({2^{-1}zd_n}),\qquad z\in\Bbb C_{\pm}.
\end{equation}
Let $D_X:=\bigoplus_{1}^\infty D_n$. Then $D^*_X =
\bigoplus_{1}^\infty D^*_n$ and
\begin{equation}\label{space_W-1,2}
   \dom D^*_X =  W^{1,2}(\R_+\setminus X) = \bigoplus^\infty_{n=1}W^{1,2}([x_{n-1},x_n]).
\end{equation}
%
%

Next we describe the main properties of a boundary triple $\Pi :=
\bigoplus^\infty_{n=1}\Pi_n$  assuming that $d_* = 0$ partially
treated in \cite{MalNei12}.
%
To this end  we first recall a complete trace characterization  of
the space $W^{1,2}(\R_+\setminus X)$ (see \cite[Proposition
3.5]{CarMalPos13}).
Due to the embedding theorem, the trace mappings 
     \begin{equation}\label{3.31}
\pi_{\pm}:  W^{1,2}(\R_+\setminus X) \to l^2(\N), \qquad
\pi_+(f)=\{f(x_{n-1}+)\}^{\infty}_1, \quad
\pi_-(f)=\{f(x_n-)\}^{\infty}_1,
     \end{equation}
are well defined for functions with compact supports, i.e. for $f\in
\bigoplus_1^N W^{1,2}[x_{n-1},x_n],$  $N\in \Bbb N.$ We  assume
$\pi_{\pm}$ to be defined on its maximal domain $\dom(\pi_{\pm}):=
\{f\in W^{1,2}(\R_+\setminus X):\ \pi_{\pm}f\in l^2(\N)\}$. Clearly,
$\dom(\pi_{\pm})$ is dense in  $W^{1,2}(\R_+\setminus X)$ although,
in general, $\dom(\pi_{\pm}) \not =  W^{1,2}(\R_+\setminus X).$
   \begin{lemma}[\cite{CarMalPos13}] \label{prop3.6_CMPos}
Let $X =\{x_{n}\}_{n=1}^\infty$ be as above with $x_{0}=0$ and
$X\subset\overline\R_+$. Then:
\begin{enumerate}\def\labelenumi {\textit{(\roman{enumi})}}
\item For any pair of sequences $a^{\pm} = \{a^{\pm}_n\}_1^{\infty}$ satisfying
    \begin{equation}\label{3.36}
a^{\pm} = \{a^{\pm}_n\}^{\infty}_1  \in l^2({\Bbb N};\{d_n\}) \qquad
\text{and} \qquad \{a^+_n-a^-_n\}^{\infty}_1\in l^2({\Bbb
N};\{d^{-1}_n\}),
   \end{equation}
there exists a (non-unique) function $f\in W^{1,2}(\R_+\setminus X)$
such that $\pi_{\pm}(f) = a^{\pm}$.  Moreover,  the mapping
$\pi_+-\pi_-: W^{1,2}(\R_+\setminus X) \to l^2({\Bbb
N};\{d^{-1}_n\})$ is surjective and contractive, i.e.
    \begin{equation}\label{3.30AA}
\sum_{n\in\N} d_n^{-1}{|f(x_n-) - f(x_{n-1}+)|^2} \le
\|f\|^2_{W^{1,2}({\R}_+\setminus X)},\qquad f\in
W^{1,2}(\R_+\setminus X).
    \end{equation}
\item Assume in addition, that $d^{*}<\infty.$ Then the mapping $\pi_{\pm}$
can be extended  to a bounded surjective  mapping from
$W^{1,2}(\R_+\setminus X)$ onto $l^2({\mathbb N};\{d_n\})$.
Moreover,  the following estimate holds
   \begin{equation}\label{3.30AB}
\sum_{n\in\N}d_n \left( |f(x_{n-1}+)|^2 + |f(x_n-)|^2\right)\le
4\left((d^{*})^2 \|f'\|^2_{L^{2}({\R}_+)} +
\|f\|^2_{L^{2}({\R}_+)}\right) \le
C_1\|f\|^2_{W^{1,2}({\R}_+\setminus X)},
    \end{equation}
for any $f\in W^{1,2}(\R_+\setminus X)$  where  $C_1 :=
4\max\{(d^{*})^2, 1\}$. Besides, the traces $a^{\pm} :=
\pi_{\pm}(f)$ of each $f\in W^{1,2}({\R}_+\setminus X)$ satisfy
conditions~\eqref{3.36}.  Moreover, the assumption $d^{*}<\infty$ is
necessary  for the inequality~\eqref{3.30AB} to hold with some
$C_1>0$.
\end{enumerate}
  \end{lemma}

Now we are ready to state and prove the main result of this
subsection.
   \begin{proposition}\label{MN2012direcsum}
Let $X$ be as above, let $d_{*}=0$ and $d^{*} < \infty$, let
${\Pi}^{(n)}=\big\{\C,{\Gamma}_{0}^{(n)}, {\Gamma}_{1}^{(n)}\big\}$
be the boundary triple for the operator $D^*_n$ defined by
~\eqref{r3.8Anew}. Let
 ${\Pi}=\bigoplus_{n=1}^\infty {\Pi}^{(n)}
=:\big\{\cH,{\Gamma}_{0},{\Gamma}_{1}\big\}$,
 where $\cH  = l^{2}(\N)$,
$D_X := \oplus_{n=1}^\infty
D_n$,
$D_{X,*} = D_{X}^* | \dom \Gamma$ and let the operators $\Gamma_j'$ and $\Gamma_j$, $j\in \{0,1\}$, be given by~\eqref{III.1_02-second} and~\eqref{III.1_02}, respectively.
Then:
\begin{enumerate}\def\labelenumi {\textit{(\roman{enumi})}}
\item   The mapping $\Gamma_0' \times \Gamma_1'$ can be extended to the mapping
  \begin{equation}
\Gamma_0'' \times \Gamma_1'':   W^{1,2}(\Bbb R_+\setminus X) \ \mapsto
l^2\bigl(\Bbb N;\{d_n^{-1}\}\bigr)\times l^2\bigl(\Bbb
N;\{d_n\}\bigr),
  \end{equation}
which is well defined  and surjective. Besides,  $\ker(\Gamma_0'' \times
\Gamma_1'') = W^{1,2}_0(\Bbb R_+\setminus X).$

\item   The mapping
  \begin{equation}
\Gamma_0 \times \Gamma_1:   \dom D_{X,*} = \dom \Gamma \ \mapsto
l^2\bigl(\Bbb N;\{d_n^{-1}\}\bigr)\times l^2\bigl(\Bbb N\bigr)
 (\subset l^2\bigl(\Bbb N\bigr)\otimes\Bbb C^2),
    \end{equation}
%
%
%
%
is well defined  and surjective.  Moreover, $\Gamma_0$ boundedly
maps $\dom D_{X,*}$ in $l^2(\Bbb N)$;


\item
The Weyl function $M(\cdot)$ is domain invariant and its domain is
given by
  \begin{equation}
\dom  M(z) = l^2\bigl(\Bbb N;\{d_n^{-2}\}\bigr) \bigl(\subsetneq
\ran\Gamma_0 = \Gamma_0(\dom A_{*}) = l^2\bigl(\Bbb
N;\{d_n^{-1}\}\bigr)\bigr), \quad z\in \Bbb C_{\pm}.
   \end{equation}
\item The domain of the form $\st_{M(z)}$ associated with  the imaginary part $\IM M(z)$ is given by
\[
 \dom \st_{M(z)}=\left\{ \{a_n\}_{n=1}^\infty \in
{l}^2(\dN)\otimes \dC^2:\, \{a_n\}_{n=1}^\infty\in
{l}^2(\dN;\{d_n^{-1}\})\right\}, \quad z\in \C_{\pm}.
\]
\item  The triple $\Pi$ is an ES-generalized boundary triple for
$D^*_X$ and  $A_0\not =A_0^*$.  Moreover,  the imaginary part $\IM
M(\cdot)$ of the Weyl function $M(\cdot)$ takes values in
 $\cC(\cH)\setminus \cB(\cH)$.
%
%
%
%
\item  The  transposed triple $\Pi^\top$ is B-generalized not an ordinary
boundary triple for the operator $D^*_X$.  In particular, the Weyl
function $-M(\cdot)^{-1}$ takes values in $\cB(\cH)$, and $A_1
=A_1^*.$

%
\end{enumerate}
\end{proposition}
\begin{proof}
(i)
The proof is immediate from Lemma \ref{prop3.6_CMPos}(1).

(ii) Since $d^*<\infty$, the space  $l^2(\Bbb N)$ is  (continuously)
embedded in $l^2\bigl(\Bbb N;\{d_n\}\bigr)$. Therefore the
surjectivity is immediate from $(i)$.
By Lemma \ref{prop3.6_CMPos}(i), the mapping  $\Gamma_0:\dom D_{X,*} 
\mapsto l^2\bigl(\Bbb N;\{d_n\}^{-1}\bigr)$ is bounded. To prove the
boundedness of $\Gamma_0:\ \dom D_{X,*} =\dom \Gamma \mapsto
l^2(\Bbb N)$ it remains to note that the embedding $l^2\bigl(\Bbb
N;\{d_n\}^{-1}\bigr) \hookrightarrow l^2(\Bbb N)$ is continuous
since $d^*<\infty$.

(iii) In accordance with~\eqref{3.10new} $M_n(z)= - \cot(2^{-1}d_n
z)$. Therefore the description of $\dom M(\cdot)$ follows from the
obvious relation
    \begin{equation}\label{asymp_cot}
\cot(2^{-1}z d_n)\sim 2z^{-1} d^{-1}_n \qquad \text{as}\qquad d_n
\to 0,\qquad z\in\Bbb C_{\pm}.
    \end{equation}

(iv)   Notice that $\{a_n\}_{n=1}^{\infty}\in\dom \st_{M(z)}$ if and
only if $ \sum^{\infty}_{n=1} \left(\IM M_n(z)a_n,a_n\right)
<\infty$.
It follows from~\eqref{3.10new}  and~\eqref{asymp_cot} that $\IM
M_n(x+iy) \sim \frac{2y}{x^2+y^2}d_n^{-1}$ as $n\to \infty$.
Therefore
\begin{equation}\label{Mformdom_momentum_op}
 \sum^{\infty}_{n=1} \left(\IM M_n(z)a_n,a_n\right)<\infty  \Longleftrightarrow  \sum^{\infty}_{n=1} |a_n|^2d_n^{-1}  <\infty.
\end{equation}

(v)  Being a direct sum of ordinary boundary triples,  the triple
${\Pi}=\bigoplus_{n=1}^\infty {\Pi}^{(n)}$ is an ES-generalized
boundary triple in accordance with  Theorem \ref{otriple}(iv). The
relation $A_0\not =A_0^*$ is implied by item (iii) since the
inclusion $\dom  M(z) \subsetneq \ran\Gamma_0$  is strict.

Furthermore, relation~\eqref{3.10new} implies
$M_n(i)=i\,\textup{cth} (2^{-1}d_n)$. It follows that $\IM
M_n(i)=\textup{cth} (2^{-1}d_n)$, $n\in\dN$. Hence the values of
imaginary part $\IM M(\cdot)$ are unbounded,  $\IM M(\cdot)\in
\cC(\cH)\setminus \cB(\cH)$. Due to Theorem~\ref{prop:C6}(v) (see
also Theorem~~\ref{prop:C6B}) this last property gives another proof
for the fact that the triple $\Pi$ is not S-generalized.
%
%

(vi) It follows from~\eqref{3.10new} that $-M_n^{-1}(z)=
\tan(2^{-1}d_n z)$. Therefore the Weyl function $-M^{-1}(\cdot) =
\oplus_1^{\infty}(-M_n^{-1}(\cdot)) \in R^s[\cH]$. By Theorem
\ref{thm:WF_Ord_BT}  the transposed triple $\Pi^\top$ is
B-generalized.
\end{proof}

  \begin{remark}
(i)  Note that statements (iii)--(vi) remain valid for $d^*=\infty$.

(ii)  Assuming that $d^*<\infty$  it is shown in  \cite{MalNei12} that the triple $\Pi=
\oplus_{n\in \Bbb N}\Pi_n$ is an ordinary boundary triple for $D^*_X$ if and only if
$d_*>0$. This result remains true also in the case $d^*=\infty$.

(iii)  Let $G = \diag \{({\wt d_1})^{1/2}, \ldots, ({\wt d_n})^{1/2}, ...\}$ be the
diagonal operator defined on $\cH =l^{2}(\dN)$, with  $\wt d_n=\min\{1,d_n\}$,
$n\in\dN$. In accordance with Theorem \ref{MN2012direcsum}(iv)
%
\[
 \ran G= \dom G^{-1} = \dom \st_{M(i)}.
\]
Hence the renormalization in Theorem \ref{essThm2} is determined via the formulas $\wt
\Gamma_0=G^{-1}\Gamma_0$, $\wt \Gamma_1=G\Gamma_0$ and the corresponding Weyl function
is given by
\[
 M_G(z)=G^* M(z) G=-\bigoplus_{n=1}^\infty\, \wt d_n
 \ctg({2^{-1}zd_n}).
\]
Since $\wt d_n\IM M_n(i)\to 2$ as $d_n\to 0$, we conclude that (the closure of)
$M_G(\cdot)$ is a bounded uniformly strict Nevanlinna function, $M_G(\cdot)\in
R^u[\cH]$. Thus, the renormalization procedure in this case leads to  an ordinary
boundary triple for $D_X^*$. In the case $d^*<\infty$ this renormalization procedure was
firstly applied  in \cite{MalNei12} to construct  the above mentioned ordinary boundary
triple for $D_X^*$; see Examples 3.2, 3.8 and Theorem 3.6 in \cite{MalNei12}.
%
\end{remark}

\subsection{Schr\"odinger operators \index{Schr\"odinger operator with local point interactions} with local point interactions}\label{sec8.3} 

Let  $X=\{x_n\}_{n=0}^\infty\subset \dR_+$  be  a strictly increasing sequence
satisfying $\lim_{n\to\infty}x_n = \infty$.
Let also $\rH_n$  be a minimal operator associated with expression $-\frac{\rD^2}{\rD
x^2}$ in $L^{2}[x_{n-1},x_{n}]$.  Clearly, $\rH_n$ is a closed symmetric,
$n_{\pm}(\rH_n)=2$, and its domain is $\dom(\rH_{n})= W^{2,2}_0[x_{n-1},x_{n}].$

%

It is easily seen that a boundary  triple  $\Pi_n=\{\Bbb C^2,
\Gamma_0^{(n)},\Gamma_1^{(n)}\}$  for  $\rH_n^*$ can be chosen  as
      \begin{equation}\label{IV.1.1_05-new}
\Gamma_0^{(n)}f:= \left(\begin{array}{c}
    f'(x_{n-1}+)\\
    f'(x_{n}-)
\end{array}\right),
\quad  \Gamma_1^{(n)}f:= \left(\begin{array}{c}
    -f(x_{n-1}+)\\
    f(x_{n}-)
\end{array}\right), \qquad
f\in W_2^2[x_{n-1},x_n].
\end{equation}
The corresponding Weyl  function $M_n$ is given by
\begin{equation}\label{eq:M_n-new}
    M_n(z) = \dfrac{-1}{\sqrt{z}}
    \begin{pmatrix}
        \cot(\sqrt{z}d_n) & -\frac{1}{\sin(\sqrt{z}d_n)}\\
        -\frac{1}{\sin(\sqrt{z}d_n)} & \cot(\sqrt{z}d_n)
    \end{pmatrix}.
\end{equation}

Consider in $L^2(\dR_+)$  the direct sum of symmetric operators
$\rH_n$,
$\rH := {\rm H}_{\min}  =\oplus_{n=1}^{\infty} \rH_n,$  
$\dom(\rH_{\min})=W^{2,2}_0(\dR_+\setminus X) =
\bigoplus_{n=1}^{\infty} W^{2,2}_0[x_{n-1},x_{n}].$

Denoting by $\frak H_+$ the domain $\dom(\rH^*)$ equipping with the graph norm, we note
that $\dom \Gamma$ is dense in $\frak H_+$ while  in general it is narrower  than $\frak
H_+$. As was shown in~\cite{Koc_89}, the triple
${\Pi}=\oplus_{n\in\dN}{\Pi}_n:=\{\cH,\Gamma_0,\Gamma_1\}$ is  an ordinary boundary
triple for the operator $\rH_{\max}:= \rH_{\min}^*$ whenever
    \begin{equation}\label{IV.1.1_03}
0\ < \ d_*=\inf_{n\in\dN}  d_n\ \ \leq\ \ d^*=\sup_{n\in \dN} d_n\
<\ +\infty.
       \end{equation}
%
%
The converse statement is also true  (see~\cite{KosMMM}): the condition $d_* >0$ is
necessary for the direct sum ${\Pi}=\oplus_{n\in\dN}{\Pi}_n$ to form  a boundary triple
for $\rm H_{\max}:= \rm H_{\min}^*$.

Such type  triples  have naturally arisen in investigation of
spectral properties of the Hamiltonian $\rH_{X,\gA}$  associated in
$L^2(\Bbb R_+)$ with  a formal differential expression
   \begin{equation}\label{I_01}
\ell_{X,\alpha}:=-\frac{\rD^2}{\rD x^2}+\sum_{x_{n}\in X}\gA_n\delta(x-x_n), \qquad
\alpha=(\alpha_n)_{n=0}^\infty\subset \Bbb R,
  \end{equation}
when treating $\rH_{X,\gA}$  as an extension of   $\rH_{\min}$ (see~\cite{Koc_89},
\cite{KosMMM}, and Remark \ref{rem8.15}  below).

Next we present extended and completed version  
of Theorem~\ref{th_criterion(bt)Ex}.
 Assertion (iii) of Theorem~\ref{th_criterion(bt)Ex} will be proved after
a preparatory lemma.

  \begin{theorem}\label{th_Scrod_criter}
Let $\Pi_n$, $n\in \dN,$ be the boundary triple  given
by~\eqref{IV.1.1_05-new}, let $M_n(\cdot)$ be the corresponding Weyl
function, $\cH={l}^2(\dN)\otimes \dC^2$, and let $\Pi := \bigoplus_{n=1}^{\infty} \Pi_n = \{\cH,
\Gamma_0,\Gamma_1 \}$ be the direct sum of triples $\Pi_n$ given by~\eqref{III.1_02-second} and~\eqref{III.1_02}. Assume
also that   $d_*=0$  and   $d^* \le \infty$. Then the following
statements hold:
\begin{enumerate}\def\labelenumi {\textit{(\roman{enumi})}}
\item  The triple $\Pi$ is an ES-generalized boundary triple for $H_{\min}^*$ such that $A_0\not =A_0^*$.
\item 
The Weyl function $M(\cdot)$ is domain invariant and its domain is
given by
  \begin{equation}\label{dom_M(z)_Schrod}
\dom M(z)= \left\{ \left\{\binom{a_n}{b_n}\right\}_{n=1}^\infty \in
{l}^2(\dN)\otimes \dC^2:\, \{a_n- b_n\}_{n=1}^\infty\in
{l}^2(\dN;\{d_n^{-2}\})\right\}, \quad z\in \C_{\pm}.
  \end{equation}
\item
Let in addition $d^*<\infty$. Then the  range of $\Gamma_0$ is given
by
\begin{equation}\label{7.69_ran_of_G0}
 \ran\Gamma_0
 = \left\{ \binom{a_n}{b_n}_{n=1}^\infty \in {l}^2(\dN) \otimes
\dC^2:\, \{a_n - b_n\}_{n=1}^\infty\in
 {l}^2(\dN;\{d_n^{-1}\})\right\} \supsetneqq \dom M(\pm i).
  \end{equation}
\item The domain of the form $\st_{M(z)}$ generated by the imaginary part $\IM M(z)$ is given by
\[
 \dom \st_{M(z)}=\left\{ \left\{\binom{a_n}{b_n}\right\}_{n=1}^\infty \in
{l}^2(\dN)\otimes \dC^2:\, \{a_n- b_n\}_{n=1}^\infty\in
{l}^2(\dN;\{d_n^{-1}\})\right\}, \quad z\in \C_{\pm}.
\]
In particular, if $d^*<\infty$, then $\dom \st_{M(z)} = \ran\Gamma_0$.

\item  The  transposed triple $\Pi^\top$  is an S-generalized boundary triple for $H_{\min}^*$, i.e. $A_1 =A_1^*$. However, it is not
      a B-generalized boundary triple for $H_{\min}^*$.
\item
The Weyl function $M^{\top}(\cdot) =
 -M(\cdot)^{-1}$ corresponding to the transposed triple ${\Pi}^\top$
is domain invariant and its domain is given by
\begin{equation*}
 \dom M^{\top}(z)=
\left\{ \left\{\binom{a_n}{b_n}\right\}_{n=1}^\infty \in
{l}^2(\dN)\otimes \dC^2:\, \{a_n + b_n\}_{n=1}^\infty\in
{l}^2(\dN;\{d_n^{-2}\})\right\}, \quad z\in \C_{\pm}.
\end{equation*}
\item The domain of the form $\st_{M^{\top}(z)}$ generated by the imaginary part $\IM M^\top(z)$ is given by
\[
 \dom \st_{M^{\top}(z)}=\left\{ \left\{\binom{a_n}{b_n}\right\}_{n=1}^\infty \in
{l}^2(\dN)\otimes \dC^2:\, \{a_n + b_n\}_{n=1}^\infty\in
{l}^2(\dN;\{d_n^{-1}\})\right\}, \quad z\in \C_{\pm}.
\]
  \end{enumerate}
%
  \end{theorem}
\begin{proof}

(i) By Theorem~\ref{otriple}(iv), the  triple $\Pi$ is an
ES-generalized boundary triple for $\rm H_{\min}^*$.
Fix  $z\in\cmr$.
It follows from~\eqref{eq:M_n-new}  that
\begin{equation}\label{M(0)}
 \lim_{d_n\to 0} d_n M_n(z) = \dfrac{-1}{z}
    \begin{pmatrix}
        1 & -1\\
         -1 & 1
    \end{pmatrix},
\qquad
 \lim_{d_n\to 0} d_n \IM M_n(i) =
    \begin{pmatrix}
        1 & -1\\
         -1 & 1
    \end{pmatrix}.
\end{equation}
Since $d_*=0$,  the last relation yields  $\sup_n\|\IM M_n(i)\| = \infty$. Therefore
Proposition~\ref{OrdinBTprop} implies $A_0\ne A_0^*$.

(ii)  By Theorem \ref{otriple}(ii), the Weyl function of $\Pi =
\bigoplus \Pi_n$ is $M(\cdot)= \bigoplus_{n=1}^\infty M_n(\cdot)$,
where $M_n(\cdot)$ is given by ~\eqref{eq:M_n-new}.  By definition,
$\{h_n\}_{n=1}^\infty\in\dom M(z)$ if and only if
\begin{equation}\label{Mdom}
 \sum^{\infty}_{n=1} \|M_n(z)h_n\|^2 <\infty; \quad
 \{h_n\}_{n=1}^\infty =\left\{\binom{a_n}{b_n}\right\}_{n=1}^\infty \in {l}^2(\dN)\otimes \dC^2.
\end{equation}
It follows from~\eqref{eq:M_n-new} that  $\|M_n(z)\|$ as a function of $d_n$ is bounded
on the intervals $[\delta,\infty)$, $\delta > 0$.

Combining this fact with the first relation in~\eqref{M(0)}  and noting that $d_* = 0$
and $\frac{\sin(\sqrt{z}d_n)}{\sqrt{z}d_n} \sim 1$ as $d_n \to 0$, one concludes that
the convergence of the series in~\eqref{Mdom} is equivalent to
    \begin{equation}\label{7.73_sum-ineq}
\sum^{\infty}_{n=1}\frac{|a_n-b_n|^2}{d^2_n}<\infty,
    \end{equation}
i.e. to the inclusion $\{a_n- b_n\}_{n=1}^\infty\in
{l}^2(\dN;\{d_n^{-2}\})$.
%
%
%
%

(iii) The proof is postponed after Lemma \ref{Prop_W2,2_Traces}.

(iv) The proof is similar to that of the item (ii). First notice that
$\{h_n\}_{n=1}^{\infty}\in\dom \st_{M(z)}$ if and only if
\begin{equation}\label{Mformdom}
 \sum^{\infty}_{n=1} \left(\IM M_n(z)h_n,h_n\right) <\infty, \quad
 \{h_n\}_{n=1}^\infty =\left\{\binom{a_n}{b_n}\right\}_{n=1}^\infty \in {l}^2(\dN)\otimes \dC^2.
\end{equation}
Note that $\IM M_n(z)$ as a function of $d_n$ is bounded on the intervals
$[\delta,\infty)$, $\delta > 0$. Rewriting  the first of relations in~\eqref{M(0)} as
%
%
\begin{equation}\label{Mdom0}
 \lim_{d_n\to 0} \left( M_n(z)+\frac{1}{d_n z} K\right)=0, \quad \text{where}
 \quad K =
   \begin{pmatrix}
        1 & -1\\
         -1 & 1
    \end{pmatrix},
\end{equation}
%
%
we derive  that the convergence of the series in~\eqref{Mformdom} is equivalent to
\[
\sum^{\infty}_{n=1} \IM\left(\frac{\left(K h_n,h_n\right)}{zd_n}\right) 
= \frac{y}{x^2+y^2} \sum^{\infty}_{n=1}\frac{|a_n-b_n|^2}{d_n}<\infty, \quad z=x +iy \in
\dC_+.
\]
This proves  the statement.
%

 (v) The Weyl function  $M^\top(\cdot)$ corresponding to the transposed boundary triple $\Pi^\top$ is
$M^\top(\cdot) = \oplus M^\top_n(\cdot)$, where $M^\top_n(\cdot) =
-M_n^{-1}(\cdot)$ is given by
\begin{equation}\label{eq:M_ntrans}
    M^\top_n(z) = -\sqrt{z}
    \begin{pmatrix}
        \cot(\sqrt{z}d_n) & \frac{1}{\sin(\sqrt{z}d_n)}\\
        \frac{1}{\sin(\sqrt{z}d_n)} & \cot(\sqrt{z}d_n)
    \end{pmatrix}.
\end{equation}
It follows that
\begin{equation}\label{eq:M_n(-1)trans}
   \lim_{d_n\to \infty}  M^\top_n(z)= \pm i \sqrt{z}I_2,
\qquad
   \lim_{d_n\to 0} d_n M^\top_n(z) =
   - \begin{pmatrix}
        1 & 1\\
        1 & 1
    \end{pmatrix}, \quad
   z\in\dC_\pm.
\end{equation}
Since $d_*=0$, the last relation  shows that the Weyl function $M^\top(\cdot)$ takes
unbounded  values.

On the other hand, using the Laurent series expansions for $\cot z$ and $(\sin z)^{-1}$
at $0$ gives
\begin{equation}\label{limM'(-1)tr}
    \lim_{d_n\to 0}d_n^{-1} \IM M^\top_n(z) =(\IM z)\begin{pmatrix}
        \frac{1}{3} & -\frac{1}{6}\\
        -\frac{1}{6} & \frac{1}{3}
    \end{pmatrix},
    \quad z\in\dC_\pm.
\end{equation}
Hence, $\IM M^\top_n(z)$ is uniformly bounded as a function of $d_n\in(0,\infty)$ for
every $z\in\cmr$. Therefore  Proposition~\ref{OrdinBTprop} ensures that  the transposed
boundary triple $\Pi^\top$ is S-generalized. At the same time $\Pi^\top$ is  not
B-generalized, since $M^\top(\cdot)$ takes values in $\cC(\cH)\setminus \cB(\cH).$

(vi) The proof is similar to that of the statement (ii). One should only use relations
~\eqref{eq:M_n(-1)trans}  instead of~\eqref{M(0)}.

(vii) The proof is similar to that of (iv).
  \end{proof}
 \begin{remark}
Here we show that the triple
${\Pi}=\oplus_{n\in\dN}{\Pi}_n:=\{\cH,\Gamma_0,\Gamma_1\}$ is an
ordinary boundary triple for the operator $\rH_{\max}:=
\rH_{\min}^*$  if and only if $d_*>0$. This statement extends the
corresponding results  from~\cite{Koc_89}, \cite{KosMMM}, to the
case $d^*=\infty$.

Denote by $\sqrt{z}$ the branch of the multifunction defined in $\dC$ with the cut along
the non-negative semiaxes $\dR_+$ and fixed by the condition $\sqrt{1} =1$. It is easily
seen that $\sqrt{\cdot}:\  \dC \mapsto \dC_+$.

Consider the behavior of $M_n(z)$ as $d_n\to \infty$. The functions $\cot(\sqrt{z}d)$
and $(\sin(\sqrt{z}d))^{-1}$ depend continuously on $d\in (0,\infty)$  and $\lim_{d\to
\infty}\cot(\sqrt{z}d) = - i$ and $\lim_{d\to \infty}(\sin(\sqrt{z}d))^{-1} = 0.$
Therefore for any  fixed $z\in\cmr$ the matrix function $M_n(z)$ in~\eqref{eq:M_n-new}
is continuous and bounded in  $d_n\in [\delta,\infty)$ for every $\delta>0$.

Further, 
clearly,  $\lim_{d_n\to \infty} \IM M_n(z)=\pm I_2$  for $z\in \dC_\pm$
and this implies that for every fixed $z\in\dC_+$ there exists $c_\delta(z)>0$ such that
\[
 \IM M_n(z) \geq c_\delta(z) I_2, \quad d_n\in [\delta,d^*], \quad \delta>0.
\]
Thus, by Theorem \ref{th_criterion(bt)}~(i) $\Pi$ is an ordinary
boundary triple for the operator $H_{\min}^*$, whenever $d_*>0$ and,
in particular, $A_0=A_0^*$ and $A_1=A_1^*$ are transversal
extensions of $H_{\min}$ in this case.
  \end{remark}
%
%

It remains to prove the assertion (iii) of Theorem~\ref{th_Scrod_criter}. It is more
involved and to this end we describe traces of functions  $f\in W^{2,2}(\Bbb R_+
\setminus X)$ as well as traces of their first derivatives  and prove  an analog of
Lemma \ref{prop3.6_CMPos}.
   \begin{lemma}\label{Prop_W2,2_Traces}
Let $X =\{x_{n}\}_{n=1}^\infty$ be as above and let $0\leq d_*\leq d^* < \infty$. Then
the mapping $\Gamma_0'': W^{2,2}(\Bbb R_+ \setminus X)\otimes\Bbb C^2\to {l}^2(\dN;\{d_n^3\})$
defined by
\[
\Gamma_0'': f\to \left\{\left(\begin{array}{c}
    f'(x_{n-1}+)\\
    f'(x_{n}-)
\end{array}\right)\right\}_{n=1}^{\infty}
\]
is well defined and bounded and its range $\ran \Gamma_0''$ is given by
\begin{equation}\label{ran_0_on_W2,2}
\Gamma_0''\bigl(W^{2,2}(\Bbb R_+ \setminus X)\bigr)  = \left\{
\binom{a_n}{b_n}_{n=1}^\infty \in {l}^2(\dN;\{d_n^3\}) \otimes \dC^2:\, \{a_n-
b_n\}_{n=1}^\infty\in {l}^2(\dN;\{d_n^{-1}\})\right\}.
\end{equation}
   \end{lemma}
  \begin{proof}
Denote temporarily the right-hand side of~\eqref{ran_0_on_W2,2}  by $\cR(\Gamma_0'')$. First we
prove the inclusion $\ran \Gamma_0'' = \Gamma_0''(W^{2,2}(\R_+\setminus X))\subset \cR(\Gamma_0'')$.
 Let  $f\in W^{2,2}(\R_+\setminus X)$. This  inclusion implies  $f\in W^{2,2}[x_{n-1}, x_n]$
for each $n\in \Bbb N$ and, it is easy to check that
\begin{equation}\label{weq01}
 d_n|f(x)|^2 \leq 2\left(\|f\|^2_{L^2(\Delta_n)}+d_n^2
 \|f'\|^2_{L^2(\Delta_n)}\right),\qquad  x\in\Delta_n:=[x_{n-1},x_n], \quad  n\in
 \N.
\end{equation}
Moreover, by the Sobolev embedding theorem (cf. e.g.
\cite{Agran2015}, \cite[p. 192]{Kato}) there are exist  constants
$c_0,c_1>0$ not depending on $f$ and $n\in \N$ such that
   \begin{equation}\label{weq02}
 \|f'\|^2_{L^2(\Delta_n)} \leq  c_1 d_n^2
 \|f''\|^2_{L^2(\Delta_n)}+c_0 d_n^{-2}\|f\|^2_{L^2(\Delta_n)}, \qquad  x\in\Delta_n, \quad  n\in
 \N.
\end{equation}
By applying~\eqref{weq01} to $f'$ and combining the result with
~\eqref{weq02} shows that
\[
d^2_n|f'(x)|^2\le C_1 d^3_n\|f''\|^2_{L^2(\Delta_n)}   +
 C_0 d^{-1}_n\|f\|^2_{L^2(\Delta_n)},\qquad  x\in\Delta_n, \quad  n\in \N,
\]
where $C_0$ and $C_1$ do not depend on $f$ and $n\in \N$. Therefore,
  \begin{eqnarray}\label{7.70_bound_est-tes}
\sum_n d^3_n\bigl(|f'(x_n-)|^2+|f'(x_{n-1}+)|^2\bigr)\le  2C_1
\sum_n d^4_n\|f''\|^2_{L^2(\Delta_n)}
+ 2C_0 \sum_n \|f\|^2_{L^2(\Delta_n)}\nonumber   \\
\le 2C_1(d^*)^4\|f''\|^2_{L^2(\Bbb R_+)} + 2C_0\|f\|^2_{L^2(\Bbb
R_+)} \le C_3 \|f\|^2_{W^{2,2}(\R_+\setminus X)},
\end{eqnarray}
where $C_3 = 2\max\{C_0,C_1(d^*)^4\}$. Hence, the mapping $\Gamma_0''$ is bounded.

Furthermore, since $f\in W^{2,2}[x_{n-1}, x_n]$,  $n\in \Bbb N$, and
$f''\in L^{2}(\R_+)$, one gets
    \begin{eqnarray}\label{7.71_New}
\sum_{n\in\N} \frac{ |f'(x_n-) - f'(x_{n-1}+)|^2}{d_n} =
\sum_{n\in\N}\frac{1}{d_n} \left|\int_{x_{n-1}}^{x_n} f''(x)\,dx
\right|^2 \le \sum_{n\in\N} \int_{x_{n-1}}^{x_n}
|f''(x)|^2\,dx \nonumber   \\
= \int_{\Bbb R_+} |f''(x)|^2\,dx \le
\|f\|^2_{W^{2,2}({\R}_+\setminus X)}.
    \end{eqnarray}
Combining~\eqref{7.70_bound_est-tes} with~\eqref{7.71_New} yields the inclusion $\ran
\Gamma_0'' = \Gamma_0''(W^{2,2}(\R_+\setminus X))\subset \cR(\Gamma_0'')$.

To prove the reverse inclusion we choose any vector
$\{\binom{a_n}{b_n}\}_{n\in\Bbb N}\in {l}^2(\dN;\{d_n^3\})
\otimes\Bbb C^2$ satisfying $\{a_n-
 b_n\}_{n=1}^\infty\in {l}^2(\dN;\{d_n^{-1}\})$.  Setting
    \begin{equation}\label{7.72_test_func}
g_n(x) = a_n(x-x_{n-1}) + 2^{-1}d^{-1}_n (x-x_{n-1})^2(b_n -
a_n),\quad x\in[x_{n-1},x_n]
    \end{equation}
and $g :=\oplus^{\infty}_1 g_n$ one easily checks that
$$
\|g_n\|^2_{L^2(\Delta_n)} \le d_n^3 \left[ \frac{2}{3}|a_n|^2 +
\frac{1}{10}|b_n -a_n|^2 \right]  \le d_n^3 \left( |a_n|^2 + |b_n|^2
\right),
$$
hence $g=\oplus^{\infty}_1 g_n \in L^{2}(\Bbb R_+)$. Moreover, the
condition  $\{a_n- b_n\}_{n=1}^\infty\in {l}^2(\dN;\{d_n^{-1}\})$
yields the inclusion $g''\in L^{2}(\Bbb R_+)$.   Thus $g\in
W^{2,2}(\Bbb R_+ \setminus X)$. To complete the proof it remains to
note that
\begin{equation}\label{gDerivates}
 g'_n(x_{n-1}+) = a_n, \quad   g'_n(x_n-) = a_n +(b_n - a_n)= b_n,
\end{equation}
i.e. $\Gamma_0'' g   = \left\{\binom{a_n}{b_n}\right\}_{n=1}^\infty$.
  \end{proof}
  \begin{remark}
Notice that the relation~\eqref{ran_0_on_W2,2} cannot be extracted
from Proposition \ref{prop3.6_CMPos}(i) applied to the derivative
$f'$, since the embedding $W^{2,2}(\Bbb R_+ \setminus X) \to
W^{1,2}(\Bbb R_+ \setminus X)$ holds if and only if $d_*>0$ (see
\cite{KosMMM_2013}).
    \end{remark}
We are now ready to prove the  assertion (iii) in Theorem~\ref{th_Scrod_criter}, i.e. to
prove relation~\eqref{7.69_ran_of_G0}.
%
%
%
\begin{proof}[Proof of item (iii) in Theorem~\ref{th_Scrod_criter}]
Let the righthand side of \eqref{7.69_ran_of_G0} be denoted
temporarily by $\cR_0(\Gamma_0)$. The inclusion $\ran(\Gamma_0) =
\Gamma_0(\dom \rH_*)\subset \cR_0(\Gamma_0)$ is immediate from Lemma
\ref{Prop_W2,2_Traces}.
%

To prove the reverse inclusion we choose any vector
$\{\binom{a_n}{b_n}\}_{n\in\Bbb N}\in l^2(\Bbb N)\otimes\Bbb C^2$
that satisfies $\{a_n- b_n\}_{n=1}^\infty\in
{l}^2(\dN;\{d_n^{-1}\})$ and consider the functions $g_n$
and $g=\oplus^{\infty}_1 g_n$ as defined in~\eqref{7.72_test_func}.
As shown in Lemma \ref{Prop_W2,2_Traces} $g\in W^{2,2}(\Bbb R_+
\setminus X)$ and $g'$ satisfies the equalities~\eqref{gDerivates}.
Besides,
$$
g_n(x_{n-1}+)=0 \quad \text{and} \quad g_n(x_n -)= a_n d_n +
2^{-1}d_n(b_n - a_n) = 2^{-1}(a_n + b_n)d_n\in l^2(\Bbb N).
$$
Note that the latter inclusion holds since $d^*<\infty$.  Summing up
we get
%
%
\[
\Gamma_0 g  = \left(\begin{array}{c}
    g'(x_{n-1}+)\\
    g'(x_{n}-)
 \end{array}\right) = \left\{\binom{a_n}{b_n}\right\}_{n=1}^\infty \quad \text{and}\quad
 \Gamma_1 g=  \left\{\binom{0}{2^{-1}(a_n + b_n)d_n}\right\}_{n=1}^\infty \in l^2(\Bbb N)\otimes\Bbb C^2.
\]
Thus, $g\in \dom\Gamma_0'\cap \dom\Gamma_1'= \dom \rH_*$ and this
completes the proof.
%
\end{proof}

One gets from Lemma~\ref{Prop_W2,2_Traces} a description for the
ranges of the closures of $\Gamma_0$ and $\Gamma_1$.

\begin{corollary}
Assume the conditions of  Theorem~\ref{th_Scrod_criter} and $d^* < \infty$. Then the
range of the closure of $\Gamma_0$ is
  \begin{equation}\label{ranG0cl}
 \ran\overline{\Gamma_0} = \left\{ \binom{a_n}{b_n}_{n=1}^\infty \in {l}^2(\dN) \otimes
\dC^2:\, \{a_n - b_n\}_{n=1}^\infty\in
 {l}^2(\dN;\{d_n^{-1}\})\right\} = \ran{\Gamma_0}.
  \end{equation}
and
  \begin{equation}\label{ranG1cl}
 \ran\overline{\Gamma_1} = {l}^2(\dN) \otimes
\dC^2.
  \end{equation}
\end{corollary}
\begin{proof}
Recall that, by definition,  $\dom \Gamma_0 = \dom \Gamma_1 = \dom \rH_*$.   Clearly,
$\Gamma_0 = \Gamma_0'' \upharpoonright \dom \rH_*$ and
  \begin{equation}\label{8.41}
\ran\Gamma_0\subseteq\ran\overline{\Gamma}_0\subseteq\ran \Gamma_0''\cap\bigl(l^2(\Bbb
N)\otimes\Bbb C^2\bigr).
  \end{equation}
On the other hand, it follows from Lemma \ref{Prop_W2,2_Traces}  and Theorem
\ref{th_Scrod_criter}(iii)  that
\begin{equation*}
\ran\Gamma_0  = \ran\Gamma_0''\cap\bigl(l^2(\Bbb N)\otimes\Bbb C^2\bigr).
\end{equation*}
Combining this relation with~\eqref{8.41} and applying Theorem
\ref{th_Scrod_criter}(iii) yields~\eqref{ranG0cl}.

The second relation  is proved similarly.
\end{proof}

\begin{remark}\label{rem8.15}
(i) Recall that according to Theorem~\ref{prop:C6} the condition
\begin{equation}\label{ranInv}
  \ran\Gamma_0=\dom M(z),\quad z\in\dC\setminus\dR,
\end{equation}
ensures selfadjointness of $A_0=\ker \Gamma_0$. Theorem~\ref{th_Scrod_criter}~(iii)
gives an explicit example showing that condition~\eqref{ranInv} cannot be replaced by
the weaker domain invariance condition
\[
 \dom M(z)=\dom M(i)\,(\subsetneqq  \ran \Gamma_0), \quad z\in\cmr.
\]
In other words, domain invariance property does not imply the property of a boundary
triple to be S-generalized (see also Example~\ref{example5.1}). Such Weyl functions
cannot be written in the form~\eqref{eq:M_S_gen} without a renormalization of the
boundary triple as in Theorem \ref{essThm2}.

(ii) In the case  $d_*=0$ and $d^*<\infty$ an abstract  regularization procedure
from~\cite{MalNei12,KosMMM} has first  been applied in~\cite{KosMMM} to the direct sum
$\Pi = \oplus_{n=1}^{\infty} \Pi_n =\{\cH, \Gamma_0,\Gamma_1\}$ of triples
~\eqref{IV.1.1_05-new} for  $\rH_n^*$   to obtain a (regularized) ordinary boundary
triple $\Pi^{r} =\{\cH, \Gamma^{r}_0,\Gamma^{r}_1\}$ satisfying $\ker \Gamma_0 = \ker
\Gamma^{r}_0$.  A special construction  of
a regularized triple $\Pi^{r}$  in~\cite{KosMMM} has been motivated by the following
circumstance: the boundary operator $B_{X,\alpha}$ corresponding to the Hamiltonian
$\rH_{X,\gA}$ of the form~\eqref{I_01}, i.e. operator satisfying $\dom(\rH_{X,\gA}) =
\ker (\Gamma^{r}_1 - B_{X,\alpha}\Gamma^{r}_0)$, is a Jacobi matrix. It is shown in
\cite{KosMMM}  that certain spectral properties of $\rH_{X,\gA}$ strictly correlate with
that of $B_{X,\alpha}$.
  \end{remark}
%

Next we apply the renormalization result in Theorem~\ref{propBetE2}
to the ES-generalized boundary triple $\Pi$ in
Theorem~\ref{th_Scrod_criter}. The transposed boundary triple
$\Pi^\top$ can be renormalized by a suitable modification of
Theorems~\ref{QBTthm},~\ref{thm:2.1} using a regular point (here
$z=-1$) on the real line; cf. Proposition~\ref{propBtoE}.

  \begin{proposition}\label{PIcor_second}
Let $\Pi_n$ be the boundary triple for  $\rH_n^*$ given by~\eqref{IV.1.1_05-new}, let
$M_n(\cdot)$, $n\in \Bbb N$, be the corresponding Weyl function given by
~\eqref{eq:M_n-new}, and let $\wt d_n=\min\{d_n,1\}$. Then:

\begin{enumerate}\def\labelenumi {\textit{(\roman{enumi})}}
\item  The orthogonal sum $\wt\Pi = \oplus_{n=1}^{\infty} \wt\Pi_n$ of boundary triples $\wt\Pi_n=\{\Bbb
C^2, \Gamma_0^{(n)},\Gamma_1^{(n)}\}$  with  
the mappings  $\wt\Gamma_j^{(n)}:W_2^2[x_{n-1},x_n]\to \dC^2$, $n\in\dN$, $j\in
\{0,1\}$, given by
\[
 \wt\Gamma_0^{(n)}f:= \left(\begin{array}{c}
    {\wt d_n}^{-1/2} f'(x_{n-1}+)\\
    {\wt d_n}^{-1/2} f'(x_{n}-)
\end{array}\right),
\quad  \wt\Gamma_1^{(n)}f:= \left(\begin{array}{c}
    -{\wt d_n}^{1/2}f(x_{n-1}+)\\
    {\wt d_n}^{1/2}f(x_{n}-)
\end{array}\right),
\]
forms  a B-generalized boundary triple for $H_{\min}^*$. Moreover, $\wt\Pi$ is  an
ordinary boundary triple if and only if $d_*>0$.
\item  The orthogonal sum $\Pi^{(r)}= \oplus_{n=1}^{\infty} \Pi^{(r)}_n$ of the boundary triples
$\Pi^r_n=\{\Bbb C^2, \Gamma_0^{(r,n)},\Gamma_1^{(r,n)}\}$ with the
mappings $\Gamma_j^{(r,n)}:W_2^2[x_{n-1},x_n]\to \dC^2$, $n\in\dN$,
$j\in \{0,1\}$, given by
\[
 \Gamma_0^{(r,n)}f:= \left(\begin{array}{c}
    {\wt d_n}^{1/2}f(x_{n-1}+)\\
    -{\wt d_n}^{1/2}f(x_{n}-)
\end{array}\right), \quad
 \Gamma_1^{(r,n)}f:= \left(\begin{array}{c}
    {\wt d_n}^{-1/2} f'(x_{n-1}+)+{\wt d_n}^{-3/2}(f(x_{n-1}+)-f(x_{n}-)\\
    {\wt d_n}^{-1/2} f'(x_{n}-)+{\wt d_n}^{-3/2}(f(x_{n-1}+)-f(x_{n}-)
\end{array}\right),
\]
is an ordinary boundary triple for $H_{\min}^*$.
\end{enumerate}
   \end{proposition}
The proof is similar to that of Theorem~\ref{th_Scrod_criter} and is
omitted.

\subsection{Dirac operators with local point interactions}\index{Dirac operator with local point interactions}

Let $D$ be a  differential expression
\begin{equation}\label{1.2}
D=-i\,c\,\frac{d}{dx}\otimes\sigma_{1}+\frac{c^{2}}{2}\otimes\sigma_{3}=\left(\begin{array}{cc}{c^{2}}/{2}
& -i\,c\,\frac{d}{dx} \\-i\,c\,\frac{d}{dx} &
-{c^{2}}/{2}\end{array}\right)
\end{equation}
acting on $\C^2$-valued functions of a real variable. Here
\begin{equation}\label{pauli}
\sigma_{1}=\left(\begin{array}{cc}0 & 1 \\1 &
0\end{array}\right),\quad\sigma_{2}=\left(\begin{array}{cc}0 & -i
\\i & 0\end{array}\right), \quad\sigma_{3}=\left(\begin{array}{cc}1
& 0 \\0 & -1\end{array}\right),
\end{equation}
are the Pauli matrices in $\C^{2}$ and $c>0$ denotes the velocity of
light.

Further, let $D_{n}$ be the minimal operator generated in
$L^{2}[x_{n-1},x_{n}]\otimes\C^{2}$ by the differential
expression (\ref{1.2}) 
  \begin{equation}\label{3.6A}
D_{n}= D \upharpoonright \dom(D_n),
\qquad\dom(D_{n})=W^{1,2}_{0}[x_{n-1},x_{n}] \otimes\C^{2}.
\end{equation}
Denote $d_n:=x_{n}-x_{n-1}>0$. Recall that $D_{n}$ is a symmetric
operator with deficiency indices $n_{\pm}(D_{n})=2$ and its adjoint
$D_{n}^*$ is given by
\[
D_{n}^*=D\upharpoonright\dom(D_{n}^*),
 \qquad \dom(D_{n}^*)=W^{1,2}[x_{n-1},x_{n}]\otimes\C^{2}.
\]

Next following \cite{CarMalPos13} we recall the construction of a
boundary triple for $D_n^*$ and compute the corresponding Weyl
function.  Namely, the triple ${\Pi}^{(n)}=\big\{\C^{2},
{\Gamma}_{0}^{(n)},{\Gamma}_{1}^{(n)}\big\}$, where
  \begin{equation}\label{triple2}
{\Gamma}_{0}^{(n)}f:= {\Gamma}_{0}^{(n)}
   \begin{pmatrix}f_{1}\\f_{2}\end{pmatrix}=
\left(\begin{array}{c}
                 f_{1}(x_{n-1}+)\\
                 i\,c\, f_{2}(x_{n}-)
                       \end{array}\right), \quad 
 {\Gamma}_{1}^{(n)}f:=
 {\Gamma}_{1}^{(n)}\begin{pmatrix}f_{1}\\f_{2}\end{pmatrix}
 =\left(\begin{array}{c}
  i\,c\,f_{2}(x_{n-1}+)\\
  f_{1}(x_{n}-)
  \end{array}\right)\,,
 \end{equation}
forms a boundary triple  for $D^{*}_n$. Clearly,
${D}_{n,0}:=D_n^{*}\upharpoonright\ker {\Gamma}^{(n)}_{0}=
{D}_{n,0}^*$ and
    \begin{equation}\label{3.8}
\dom(D_{n,0}) = 
\{\{f_1, f_2\}^{\tau}\in W^{1,2}[x_{n-1},x_{n}]\otimes\C^{2}:
f_{1}(x_{n-1}+)= f_{2}(x_{n}-)=0\}.
     \end{equation}
Moreover, the spectrum of the operator ${D}_{n,0}$ is discrete,
\begin{equation}\label{spetcom} \sigma({D}_{n,0})=
\sigma_d({D}_{n,0})= \left\{\pm\sqrt{\frac{
c^2\pi^2}{d^2_n}\,\left(j+\frac12\right)^{2}+\left(\frac{c^{2}}{2}\right)^{2}}\,,\quad
j\in \Bbb N \right\}.
\end{equation}
%
%
The defect subspace $\mathfrak N_z :=\text{\rm ker}(D_{n}^*-z)$ is
spanned by the vector functions $f_n^\pm(\cdot, z)$,
  \begin{equation}\label{3.7}
f_n^{\pm}(x, z) := \begin{pmatrix}e^{\pm i\, k(z)\,x}\\\pm
k_{1}(z)e^{\pm i\,k(z)\,x}\end{pmatrix}.
  \end{equation}
Moreover, the Weyl function ${M}_{n}(\cdot)$ corresponding to the
triple $\Pi^{(n)}$ is (cf. \cite{CarMalPos13})
\begin{equation}\label{IV.1.1_09.2}
  {M}_{n}(z)=\frac{1}{\cos(d_n\,k(z))}\begin{pmatrix}c\,k_1(z)\,\sin(d_n\,k(z))
   & 1 \\1 & (c\,k_1(z))^{-1}\sin(d_n\,k(z))\end{pmatrix},\quad z\in\rho(D_{n,0}),
\end{equation}
where
\begin{equation}\label{1.5}
 k(z):=c^{-1}\sqrt{z^{2}-\left({c^{2}}/{2}\right)^{2}},\quad z\in\C,
\end{equation}
and
  \begin{equation}\label{1.6}
k_{1}(z):=\frac{c\,k(z)}{z+{c^{2}}/{2}}
=\sqrt{\frac{z-{c^{2}}/{2}}{z+{c^{2}}/{2}}},\quad z\in\C.
\end{equation}

Next we construct a boundary triple for the operator $D_X^*
:=\bigoplus_{n=1}^\infty D_n^*$ in the general case $0 \leq
d_{*}<d^{*} \leq \infty$. It appears that the result in the case
$d^*=\infty$ remains analogous to what was obtained in
\cite{CarMalPos13} for the case $d^*<\infty$.

Define $D_X:= \bigoplus_{1}^\infty D_n,$
  $$
\dom(D_X^*)=  W^{1,2}(\Bbb R_+\setminus X)\otimes\C^{2} =
\bigoplus_{1}^\infty W^{1,2}[x_{n-1},x_{n}]\otimes\C^{2}.
  $$

Next following \cite{CarMalPos13} we collect certain properties
of the direct sum ${\Pi} := \bigoplus_{n=1}^\infty
 {\Pi}^{(n)}$   of boundary triples ${\Pi}^{(n)}$  given by~\eqref{triple2}.

%
   \begin{proposition}\label{prop3.5direcsum}
Let $X$ be as above, let $0 \leq d_{*}<d^{*} \leq \infty$, and let $
{\Pi}^{(n)}=\big\{\C^{2},
 {\Gamma}_{0}^{(n)}, {\Gamma}_{1}^{(n)}\big\}$ be
the boundary triple for the operator $D_n^*$ defined in
~\eqref{triple2}. Let
$\cH = l^{2}(\N)\otimes\C^{2}$ and
 ${\Pi} := \bigoplus_{n=1}^\infty
 {\Pi}^{(n)} = \big\{\cH, {\Gamma}_{0},
 {\Gamma}_{1}\big\}$, where
the operators  $\Gamma_j$, $j\in \{0,1\}$ are given by~\eqref{III.1_02},
i.e.
\begin{equation}\label{triple2NEW}
 {\Gamma}_{0} \begin{pmatrix}f_{1}\\f_{2}\end{pmatrix}=
\left\{\left(\begin{array}{c}
                 f_{1}(x_{n-1}+)\\
                 i\,c\, f_{2}(x_{n}-)
                       \end{array}\right)\right\}_{n\in \N}, \quad 
 {\Gamma}_{1}\begin{pmatrix}f_{1}\\f_{2}\end{pmatrix}
=\left\{\left(\begin{array}{c}
 i\,c\,f_{2}(x_{n-1}+)\\
 f_{1}(x_{n}-)
 \end{array}\right)\right\}_{n\in \N}, 
\end{equation}
where  $f= \binom{f_1}{f_2}\in \dom D_{X,*}:= \dom\Gamma$  and
$D_{X,*}:= D_{X}^* \upharpoonright \dom D_{X,*}$.
Then:
\begin{enumerate}\def\labelenumi {\textit{(\roman{enumi})}}

\item The domain $\Gamma$ is given by
$\dom D_{X,*}:= \dom\Gamma 
(=\dom\Gamma_0 = \dom\Gamma_1).$
%
%

\item The direct sum ${\Pi} := \bigoplus_{n=1}^\infty
 {\Pi}^{(n)}$ 
 forms a $B$-generalized boundary triple  for $D^*_X$.

\item The transposed triple $ {\Pi}^\top =
\{\cH, {\Gamma}^\top_0, {\Gamma}^\top_1\}:=\{\cH, {\Gamma}_1,-
{\Gamma}_0\}$ also forms a $B$-generalized boundary triple for
$D_X^*$.

\item The triple $ {\Pi}$ (equivalently the triple $ {\Pi}^\top $)
is an ordinary boundary triple for the operator $D_X^* =
\bigoplus_{n=1}^\infty D_{n}^{*}$ if and only if $d_{*}>0$ (with
$d^{*}\leq \infty$).
\end{enumerate}
  \end{proposition}
\begin{proof}
(i), (ii) The Weyl function of the boundary triple $\{\cH, \Gamma_0,
\Gamma_1\}$ is the orthogonal sum $  M=\oplus   M_n$ of the Weyl
functions defined by~\eqref{IV.1.1_09.2}. It follows from
~\eqref{1.5} and~\eqref{1.6} that  $k(0)=i\,c/2$ and $k_1(0)=i$ and
hence
  \begin{equation}\label{for-la_Mn(0)}
M_n(0) =
\frac{1}{\textup{ch}(d_n\,c/2)}\begin{pmatrix}-c\,\textup{sh}(d_n\,c/2)
   & 1 \\ 1 & c^{-1}\textup{sh}(d_n\,c/2)\end{pmatrix}.
\end{equation}
Hence
%
%
  \begin{equation}\label{limM(0)}
   M_n(0)\to \begin{pmatrix}0 & 1 \\ 1 & 0\end{pmatrix} \quad\text{as}\quad d_n\to 0
 \quad\text{and}\quad
    M_n(0)\to \begin{pmatrix}-c & 0 \\ 0 & c^{-1}\end{pmatrix} \quad \text{as}\quad d_n\to
 \infty.
\end{equation}
It follows that the sequence  $\{M_n(0)\}_{n\in \Bbb N}$ is  bounded. 

Furthermore, one gets from~\eqref{1.5}  and~\eqref{1.6} that
$k'(0)=0$, $k_1'(0)=-i\,2/c^2$, and
    \begin{equation}\label{for-la_for_M'n(0)}
 M_n'(0)=\begin{pmatrix}\frac{2}{c}\,\textup{th}(d_n\,c/2)
 & 0 \\ 0 & \frac{2}{c^3}\,\textup{th}(d_n\,c/2)\end{pmatrix} \geq 0, \qquad n\in \Bbb N.
   \end{equation}
This description implies that
  \begin{equation}\label{lim_M'n(0)}
M_n'(0)\to \begin{pmatrix} 0 & 0 \\ 0 & 0\end{pmatrix} \quad
\text{as}\quad d_n\to 0
 \quad\text{and}\quad
M_n'(0)\to \begin{pmatrix} \frac{2}{c} & 0 \\ 0 &
\frac{2}{c^3}\end{pmatrix} \quad\text{as}\quad d_n\to  \infty.
  \end{equation}
Thus, the sequence  $\{M_n'(0)\}_{n\in \Bbb N}$ is  bounded too.
Combining formulas~\eqref{limM(0)} with~\eqref{lim_M'n(0)} and
applying Theorem \ref{th_criterion(bt)}(iii) one concludes that $\Pi
= \{\cH, \Gamma_0, \Gamma_1\}$ is a $B$-generalized boundary triple
for $D^*_X$.

(iii)  It follows from~\eqref{for-la_Mn(0)} that  $\det( M_n(0))=
-1$, hence  the sequence of inverses $\{M_n(0)^{-1}\}_{n\in \Bbb N}$
is bounded alongside the sequence $\{M_n(0)\}_{n\in \Bbb N}$.
Combining this fact with boundedness of the sequence
$\{M_n'(0)\}_{n\in \Bbb N}$ of the derivatives
and using the identities
\[
-(  M_n^{-1})'(0)=M_n^{-1}(0)\,M_n'(0)\, M_n^{-1}(0),\quad n\in \Bbb
N,
\]
we obtain that the sequence   $\{( M_n^{-1})'(0)\}_{n\in \Bbb N}$ is
bounded too. It remains to apply
Theorem \ref{th_criterion(bt)}(iii).

(iv) 
It follows from~\eqref{for-la_for_M'n(0)} that the sequence
$\{M_n'(0)\}_{n\in \Bbb N}$ of the derivatives  is uniformly
positive  if and only if $d_*>0$.  One completes the proof by
combining  Theorem \ref{th_criterion(bt)}(iv) with the above proved
items (ii), (iii).
\end{proof}
 \begin{remark}
Note that if $d^*=\infty$ then in view of~\eqref{spetcom} $\pm
\frac{c^{2}}{2}\in \sigma(D_0)$, while
$(-\frac{c^{2}}{2},\frac{c^{2}}{2})\subset \rho(D_0)$. Therefore as
distinguished from the considerations in \cite{CarMalPos13} treating
the case $d^* < \infty$, here we consider the behavior of the Weyl
function at $z=0\in \rho(D_0)$.
  \end{remark}

We now apply a modification of  Theorem~\ref{thmBtoE} and Proposition \ref{propBtoE} to
produce an $ES$-generalized boundary triple for $D^*_X$ from the $B$-generalized
boundary triple $ {\Pi} := \bigoplus_{n=1}^\infty
 {\Pi}^{(n)}=\big\{\cH, {\Gamma}_{0},
 {\Gamma}_{1}\big\}$. In this modification we subtract from
the Weyl function $  M_n$ the limit value $\lim_{d_n\to 0} M_n(0)$,
instead of the value $  M_n(0)$, to get a transform of boundary
mappings in a simple form.
  \begin{proposition}\label{prop3.5sum2}
Let $X$ be as above, let $0 \leq d_{*}<d^{*} \leq \infty$, let $\wt{\Pi}^{(n)}=\big\{\C^{2},
 \wt{\Gamma}_{0}^{(n)}, \wt{\Gamma}_{1}^{(n)}\big\}$ be
the boundary triple for the operator $D_n^*$ defined by
\[
 \wt{\Gamma}^{(n)}_{0} \begin{pmatrix}f_{1}\\f_{2}\end{pmatrix}
 =\left(\begin{array}{c}
  i\,c\,(f_{2}(x_{n}-)-f_{2}(x_{n-1}+))\\
  f_{1}(x_{n-1}+)-f_{1}(x_{n}-)
 \end{array}\right),
 \quad 
 \wt{\Gamma}^{(n)}_{1}\begin{pmatrix}f_{1}\\f_{2}\end{pmatrix}
 =\left(\begin{array}{c}
                 f_{1}(x_{n-1}+)\\
                 i\,c\, f_{2}(x_{n}-)
                       \end{array}\right),\quad n\in\dN,
\]
let $\wt{\Gamma}_{j}' =\bigoplus_{n=1}^\infty \wt{\Gamma}^{(n)}_{j},$
$j\in\{0,1\},$ and
let $ \wt\Pi = \bigoplus_{n=1}^\infty \wt\Pi_n =
\big\{\cH, \wt{\Gamma}_{0}, \wt{\Gamma}_{1}\big\}$ be the boundary triple for $D^*_X$,  where
\[
\wt{\Gamma}_{j} :=\wt{\Gamma}_{j}' \upharpoonright  \dom(D_{X,*}),\quad \dom(D_{X,*}):=  \dom {\wt\Gamma} :=
\dom \wt\Gamma_0' \cap \dom \wt\Gamma_1'.
\]
\begin{enumerate}\def\labelenumi {\textit{(\roman{enumi})}}
\item
The mapping  $\wt\Gamma_0 \times \wt\Gamma_1$ is naturally extended to the mapping $\wt\Gamma_0'' \times \wt\Gamma_1''$ defined by the same formulas on $W^{1,2}(\Bbb R_+\setminus X)\otimes\Bbb
C^2$. Moreover, the mapping
  \begin{equation}\label{ran_Gamma_on_Sob_space}
\wt\Gamma_0''\times \wt\Gamma_1'':   W^{1,2}(\Bbb R_+\setminus X)\otimes\Bbb C^2 \ \to (l^2\bigl(\Bbb
N;\{d_n^{-1}\}\bigr)\otimes\Bbb C^2)\times (l^2\bigl(\Bbb N;\{d_n\}\bigr)\otimes\Bbb
C^2)
  \end{equation}
is well defined  and surjective.

\item The mapping
  \begin{equation}\label{ran_Gamma_on_subdomain}
\wt\Gamma_0 \times \wt\Gamma_1:   \dom D_{X,*} \ \to (l^2\bigl(\Bbb
N;\{d_n^{-1}\}\bigr)\otimes\Bbb C^2)\times (l^2\bigl(\Bbb
N\bigr)\otimes\Bbb C^2)
 (\subset l^2\bigl(\Bbb N\bigr)\otimes\Bbb C^4),
    \end{equation}
%
%
%
is well defined  and surjective.  Moreover,  
$\dom D_{X,*} =\dom {\wt\Gamma} = \dom \overline {{\wt\Gamma}_1}$, while  $\dom
\overline {{\wt\Gamma}_0} = \dom D_{X}^* = W^{1,2}(\Bbb R_+\setminus X)\otimes\Bbb C^2$.

\item  The Weyl function is of the form $\wt M(\cdot)=\bigoplus_{n=1}^\infty
\wt M_n(\cdot)$, where
   \begin{equation}\label{W-F2_for_Dirac}
 \wt M_n(z)= - 2^{-1}
 \begin{pmatrix}\dfrac{\sin(d_nk(z))}{ck_1(z)(1- \cos(d_nk(z)))} & -1 \\
 -1 & \dfrac{ck_1(z)\sin(d_nk(z))}{1- \cos(d_nk(z))}\end{pmatrix}
   \end{equation}
 and it is domain invariant with
  \begin{equation}\label{dom_M(z)}
\dom \wt M(z) = l^2\bigl(\Bbb N;\{d_n^{-2}\}\bigr)\otimes\Bbb
C^2\subseteq \wt\Gamma_0(\dom D_{X,*}) = l^2\bigl(\Bbb
N;\{d_n^{-1}\}\bigr)\otimes\Bbb C^2 \quad \text{for}\quad z\in \Bbb
C_{\pm}.
   \end{equation}
Here the strict inclusion  $\dom M (z) \subsetneqq \wt\Gamma_0(\dom
D_{X,*})$ holds if and only if $d_{*}=0$.

\item  The Weyl function $\wt M(\cdot)$ is also form-domain invariant with
  \begin{equation}\label{form_dom_M(z)}
 \dom \st_{\wt M(z)} = l^2\bigl(\Bbb N;\{d_n^{-1}\}\bigr)\otimes\Bbb
C^2 = \wt\Gamma_0(\dom D_{X,*}) \quad \text{for}\quad z\in \Bbb
C_{\pm}.
   \end{equation}

\item $\wt\Pi=\big\{\cH, \wt{\Gamma}_{0},  \wt{\Gamma}_{1}\big\}$ forms an $ES$-generalized
boundary triple for $D^*_X$. Moreover, $\wt\Pi$  is an
$S$-generalized boundary triple for $D^*_X$ if and only if $d_{*} >
0$ and in this case $\wt\Pi$ is in fact and ordinary boundary triple
for $D^*_X$.

\item The transposed triple $ {\Pi}^\top = \{\cH, \wt{\Gamma}^\top_0,
\wt{\Gamma}^\top_1\}:= \{\cH, \wt{\Gamma}_1,- \wt{\Gamma}_0\}$ is  a
$B$-generalized boundary triple for $D_X^*$. In particular, $A_1 =
D_X^*\upharpoonright \ker \wt{\Gamma}_1$ is selfadjoint.
\end{enumerate}
   \end{proposition}
   \begin{proof}
(i) The proof is immediate from Lemma \ref{prop3.6_CMPos}.

(ii) Due to $d^{*} < \infty$ one has the following chain of
continuous embeddings
  \begin{equation}\label{l-2-embeddings}
 l^2\bigl(\Bbb N;\{d_n^{-1}\}\bigr)\otimes\Bbb C^2 \subset l^2\bigl(\Bbb
N\bigr)\otimes\Bbb C^2  \subset  l^2\bigl(\Bbb
N;\{d_n\}\bigr)\otimes\Bbb C^2.
  \end{equation}
Since $l^2\bigl(\Bbb N\bigr)\otimes\Bbb C^2$ is a part of $l^2\bigl(\Bbb
N;\{d_n\}\bigr)\otimes\Bbb C^2$, the surjectivity of the mapping $\wt\Gamma = (\wt\Gamma_0''\times \wt\Gamma_1'')\upharpoonright \dom D_{X,*}$
is immediate from (i).  The inclusion in~\eqref{ran_Gamma_on_subdomain} as well as the
relation $\dom{\wt\Gamma} = \dom \overline {{\wt\Gamma}_1}$   is implied by the first
inclusion in~\eqref{l-2-embeddings}.
%
%
%
%

(iii) The Weyl function  corresponding to $\wt\Pi$  is the direct
sum $\wt M(\cdot) =\bigoplus_{n=1}^\infty \wt M_n(\cdot)$, where
\[
 \wt M_n(z):= \left(\begin{pmatrix}0 & 1 \\ 1 &
 0\end{pmatrix} - M_n(z)\right)^{-1}, \quad z\in \rho(M_n).
\]
This immediately  leads to  formula~\eqref{W-F2_for_Dirac} for $\wt M_n(z)$.
Using~\eqref{1.5},~\eqref{1.6}, and the Taylor series expansions for $\sin(z)$ and
$\cos(z)$  we easily derive
  \begin{equation}\label{wtMexp}
\wt M_n(z)
 +\frac{1}{d_n}\begin{pmatrix} (z-c^2/2)^{-1} & 0 \\ 0 & c^2(z+c^2/2)^{-1} \end{pmatrix}
 \to \frac{1}{2}\begin{pmatrix}0 & 1 \\ 1 & 0 \end{pmatrix}
\quad\text{as} \quad d_n\to 0, \quad z\in \Bbb C_{\pm}.
 \end{equation}
This formula shows that $\wt M(z)$, as well as $\IM\wt M(z)$, is bounded if and only if
$d_*>0$, $z\in \Bbb C_{\pm}$. Moreover, it follows from~\eqref{wtMexp} that
 $\{\binom{a_n}{b_n}\}_{n\in \dN}\in\dom \wt M(z)$, $z\in
\Bbb C_{\pm}$,  precisely when
    \begin{equation*}
\sum^{\infty}_{n=1}\frac{|a_n|^2+|b_n|^2}{d^2_n} < \infty.
    \end{equation*}
The inclusion (in fact the continuous embedding) in~\eqref{dom_M(z)}
follows from the estimate
    \begin{equation*}
\sum^{\infty}_{n=1}\frac{|a_n|^2 + |b_n|^2}{d_n}  \le
d^*\sum^{\infty}_{n=1}\frac{|a_n|^2 + |b_n|^2}{d^2_n}.
    \end{equation*}
Note that the converse inequality holds if and only if $d_*>0$. Indeed, writing down the
reverse inequality and inserting  here $\{a_n\}=\{\delta_{jn}\}_{n\in\Bbb N}$ and
 $\{b_n\}=\{0\}_{n\in\Bbb N}$, one arrives at the inequalities
$$
1\le c d_j, \qquad j\in\Bbb N,
$$
showing that  $d_* \ge 1/c >0$.

(iv) By definition, $\{h_n\}_{n=1}^{\infty}\in\dom \st_{\wt M(z)}$ if and only if
\begin{equation}\label{wtMformdom}
 \sum^{\infty}_{n=1} \left(\IM \wt M_n(z)h_n,h_n\right) <\infty; \quad
 \{h_n\}_{n=1}^\infty =\left\{\binom{a_n}{b_n}\right\}_{n=1}^\infty \in {l}^2(\dN)\otimes \dC^2.
\end{equation}
As a function of $d_n$ the imaginary part $\IM \wt M_n(z)$ is
bounded on the intervals $[\delta,\infty)$, $\delta > 0$, and hence
it follows from~\eqref{wtMexp} that the convergence of the series in
~\eqref{wtMformdom} is equivalent to
\[
 \sum^{\infty}_{n=1} \frac{\left(\IM K(z) h_n,h_n\right)}{d_n} <\infty; \quad
 \{h_n\}_{n=1}^\infty =\left\{\binom{a_n}{b_n}\right\}_{n=1}^\infty \in {l}^2(\dN)\otimes
 \dC^2,
\]
where $K(z)$ denotes the diagonal matrix function in the left-hand side of
~\eqref{wtMexp}. Clearly, $\IM K(z)$ is bounded with bounded inverse for each
$z\in\dC_\pm$ and this yields the stated description of $\dom \st_{\wt M(z)}$.

(v) By Theorem  \ref{otriple}(iv), the triple  $\wt\Pi$ being  a
direct sum of ordinary boundary triples,  is an $ES$-generalized
boundary triple. On the other hand, by (iii) the strict inclusion
$\dom M(z) \subsetneqq \wt\Gamma_0(\dom D_{X,*})$ is equivalent to
$d_{*}=0$. Therefore, Theorem \ref{prop:C6}  applies and ensures
that in the latter case  $\wt\Pi$ is not an $S$-generalized boundary
triple.

(vi) The Weyl function corresponding to the transposed boundary
triple ${\Pi}^\top$ is $-\wt M(\cdot)^{-1} = \bigoplus_1^\infty
(-\wt M_n(\cdot)^{-1})$. In particular, one gets from~\eqref{wtMexp}
(or from~\eqref{IV.1.1_09.2}) that
 \begin{equation*}
- \wt M_n(z)^{-1} =  M_n(z) - \begin{pmatrix}0 & 1 \\ 1 & 0
\end{pmatrix}  \sim  d_n
\begin{pmatrix}
z-c^2/2  & 0\\
0  & c^{-2} (z+c^2/2)
\end{pmatrix}\quad \text{as}\quad d_n \to 0.
   \end{equation*}
%
%
This shows that $-\wt M(\cdot)^{-1}\in \cR^s[\cH]$, which is
equivalent to the required  statement: $\Pi^\top$ is a
$B$-generalized boundary triple (see \cite[Chapter 5]{DM95}).
   \end{proof}

\begin{remark}
Apart from statements (ii) and the formula for $\wt\Gamma_0(\dom
D_{X,*})$ in statement (iii) the results in
Proposition~\ref{prop3.5sum2} remain valid for $d^* = \infty$.
Indeed, statement (i) is still immediate from Proposition
\ref{prop3.6_CMPos}(i)  which holds  in this case, too. All the
other statements can easily be extracted from the fact that the
limit value of the Weyl function $\wt M_n(z)$ as well as its inverse
$\wt M_n(z)^{-1}$ remain bounded when $d_n\to \infty$.
\end{remark}


Let $\alpha=\{\alpha_n\}_{n\in\dN}$ be a sequence from $\dR$.
Gesztesy-\v{S}eba realization of Dirac operator \index{Gesztesy-\v{S}eba realization of Dirac operator}(see~\cite{GeSe87})
is defined by $D_{X,\alpha}=D|_{\dom{D_{X,\alpha}}}$, where
\begin{equation}\label{eq:GS_real}
             \dom{D_{X,\alpha}}=
            \left\{
                \begin{array}{c}
                    f\in W_{comp}(\dR_+\setminus X)\otimes\dC^2:\,f_1\in AC_{loc}(\dR_+), f_2\in AC_{loc}(\dR_+\setminus X)\\
                    f_2(a+)=0,\, f_2(x_n+)-f_2(x_{n}-)=-\frac{i\alpha_n}{c}f_1(x_n),\, n\in\dN
                \end{array}
            \right\}.
\end{equation}

As was shown in~\cite{GeSe87}, ~\cite{CarMalPos13} the
Gesztesy-\v{S}eba realization $D_{X,\alpha}$ is always selfadjoint.
The operators $D_{X,\alpha}$ are parametrized in the boundary triple
$\Pi=\{\cH,\wt\Gamma_0,\wt \Gamma_1\}$ via selfadjoint
three-diagonal matrices
\[
    J_\alpha=  \begin{pmatrix}
             0  & -1 &          &    &           &    & \\
             -1 & 0 & 1         &    &           &    &\\
                & 1 & \alpha_1  & -1 &           &    &\\
                &   &-1         &  0 & 1         &    &\\
                &   &           & 1  & \alpha_2  & -1  &\\
                &   &           &    & \ddots    & \ddots & \ddots\\
           \end{pmatrix}
\]
\begin{proposition}
  Let $\alpha=\{\alpha_n\}_{n\in\dN}\subset\dR$, let $D_{X,\alpha}$ be the Gesztesy-\v{S}eba realization of the Dirac operator given by~\eqref{eq:GS_real}, let 
  $\Pi=\{\cH,\Gamma_0,\Gamma_1\}$  be the boundary triple defined by~\eqref{triple2NEW} and let
\[
M(\lambda)= \bigoplus_{n=1}^{\infty}M_n(\lambda),  
\quad \gamma(\lambda)= \bigoplus_{n=1}^{\infty}\gamma_n(\lambda),\quad
Q=\bigoplus_{n=1}^{\infty} Q_n,\quad
Q_n=\begin{pmatrix}
                     0 & 1 \\
                     1 & 0
                   \end{pmatrix},\quad{n\in\dN},
\]
with $M_n(\lambda)$ and $\gamma_n(\lambda)$ given
by~\eqref{IV.1.1_09.2} and~\cite[(3.11)]{CarMalPos13}, respectively.
Then:
\begin{equation}\label{eq:GS_real_Gamma}
  \textup{dom}\, D_{X,\alpha} 
=\ker(\Gamma_1-(B_\alpha+Q)\Gamma_0).
\end{equation}
Moreover,
\begin{equation}\label{eq:sigma_p}
  \lambda\not\in\sigma_p(D_{X,\alpha})\Longleftrightarrow 0\not\in\sigma_p(B_\alpha+Q-M(\lambda)),
\end{equation}
and the following Kre\u{\i}n-type formula  holds
\begin{equation}\label{eq:Krein_formula_D}
(D_{X,\alpha}-\lambda)^{-1}=(D_0-\lambda)^{-1}
+\gamma(\lambda)\bigl(B_\alpha+Q-M(\lambda)\bigr)^{-1}\gamma(\bar\lambda)^*,\qquad
\lambda\in\rho(D_{X,\alpha})\cap\rho(D_0).
\end{equation}
\end{proposition}
\begin{proof}
The equality~\eqref{eq:GS_real_Gamma} is implied
by~\eqref{eq:GS_real} and~\eqref{triple2NEW}. The
formulas~\eqref{eq:sigma_p} and~\eqref{eq:Krein_formula_D} follow
from Theorem~\ref{Kreinformula} or Theorem~\ref{Kreinformula2}.
\end{proof}

\printindex

\end{document}